\documentclass[a4paper,11pt]{amsart}
\usepackage{amssymb, amscd, amsxtra, xypic, amsmath}
\usepackage{textcomp}
\usepackage[all]{xy}
\usepackage{color, colortbl}
\usepackage{array}
\usepackage{longtable}
\usepackage{multirow}
\usepackage{url}

\definecolor{lightgray}{gray}{0.91}

\numberwithin{equation}{section} 

\makeatletter
\renewcommand\subsubsection{\@secnumfont}{\bfseries}%
\renewcommand\subsubsection{\@startsection{subsubsection}{3}
  \z@{.5\linespacing\@plus.7\linespacing}{-.5em}%
  {\normalfont\bfseries}}
\makeatother

\usepackage[pagebackref=true, linktocpage=true]{hyperref}
\hypersetup{%
bookmarksnumbered=true,%
colorlinks=true,%
linkcolor=blue,%
citecolor=blue,%
urlcolor=blue,%
setpagesize=false,%
pdftitle={},%
pdfauthor={}}

\theoremstyle{plain}
\newtheorem{Thm}{Theorem}[subsection]
\newtheorem{Lem}[Thm]{Lemma}
\newtheorem{Cor}[Thm]{Corollary}
\newtheorem{Prop}[Thm]{Proposition}
\newtheorem{Conj}[Thm]{Conjecture}
\newtheorem{Claim}{Claim}

\theoremstyle{definition}
\newtheorem{Def}[Thm]{Definition}
\newtheorem{Def-Lem}[Thm]{Definition-Lemma}

\newtheorem{Rem}[Thm]{Remark}

\newtheorem*{Ack}{Acknowledgments}
\newtheorem{Ex}[Thm]{Example}

\newtheorem{Question}[Thm]{Question}

\newcommand{\Aut}{\operatorname{Aut}}
\newcommand{\Bir}{\operatorname{Bir}}
\newcommand{\codim}{\operatorname{codim}}

\newcommand{\Proj}{\operatorname{Proj}}
\newcommand{\prt}{\partial}
\newcommand{\Sing}{\operatorname{Sing}}
\newcommand{\Spec}{\operatorname{Spec}}

\newcommand{\Cl}{\operatorname{Cl}}

\newcommand{\Int}{\operatorname{Int}}
\newcommand{\bNE}{\operatorname{\overline{NE}}}
\newcommand{\Bs}{\operatorname{Bs}}

\newcommand{\mult}{\operatorname{mult}}

\newcommand{\wt}{\operatorname{wt}}

\newcommand{\ord}{\operatorname{ord}}
\newcommand{\lct}{\operatorname{lct}}
\newcommand{\Supp}{\operatorname{Supp}}

\newcommand{\EI}{\mathrm{EI}}

\newcommand{\ntimes}{\! \times \!}
\newcommand{\Q}{\mathrm{QI}}
\newcommand{\E}{\mathrm{EI}}
\newcommand{\sm}{\operatorname{sm}}

\newcommand{\Sm}{\operatorname{Sm}}
\newcommand{\QSm}{\operatorname{QSm}}
\newcommand{\coeff}{\operatorname{coeff}}
\newcommand{\omult}{\operatorname{omult}}
\newcommand{\wf}{\operatorname{wf}}
\newcommand{\Diff}{\operatorname{Diff}}
\newcommand{\mov}{\operatorname{mov}}
\newcommand{\BR}{\mathrm{BR}}
\newcommand{\BSR}{\mathrm{BSR}}

\newcommand{\mbA}{\mathbb{A}}

\newcommand{\mbC}{\mathbb{C}}

\newcommand{\mbP}{\mathbb{P}}
\newcommand{\mbQ}{\mathbb{Q}}
\newcommand{\mbR}{\mathbb{R}}

\newcommand{\mbZ}{\mathbb{Z}}

\newcommand{\mcC}{\mathcal{C}}
\newcommand{\mcD}{\mathcal{D}}

\newcommand{\mcF}{\mathcal{F}}

\newcommand{\mcH}{\mathcal{H}}
\newcommand{\mcI}{\mathcal{I}}

\newcommand{\mcM}{\mathcal{M}}

\newcommand{\mcO}{\mathcal{O}}

\newcommand{\mcU}{\mathcal{U}}

\newcommand{\msp}{\mathsf{p}}
\newcommand{\msq}{\mathsf{q}}

\newcommand{\msi}{\mathsf{i}}
\newcommand{\msI}{\mathsf{I}}

\newcommand{\ratmap}{\dashrightarrow}

\newcommand{\bmu}{\boldsymbol{\mu}}

\makeatletter
\def\imod#1{\allowbreak\mkern10mu({\operator@font mod}\,\,#1)}
\makeatother

\title[K-stability of Fano 3-fold hypersurfaces]{K-stability of birationally superrigid \\ Fano 3-fold weighted hypersurfaces}
\author{In-Kyun~Kim \and Takuzo~Okada \and Joonyeong~Won}
\address{Center for Mathematical Challenges, Korea Institute for Advanced Study, Seoul 02455, Korea}
\email{soulcraw@gmail.com}
\address{Department of Mathematics, Faculty of Science and Engineering, Saga University, Saga 840-8502 Japan}
\email{okada@cc.saga-u.ac.jp}
\address{Department of Mathematics, Ewha Womans University, Seoul 03760, Korea}
\email{leonwon@ewha.ac.kr}
\subjclass[2020]{14J45 \and 53C25}
\date{}

\begin{document}

\begin{abstract}
We prove that the alpha invariant of a quasi-smooth Fano 3-fold weighted hypersurface of index $1$ is greater than or equal to $1/2$. 
Combining this with the result of Stibitz and Zhuang \cite{SZ19} on a relation between birational superrigidity and K-stability, we prove the K-stability of a birationally superrigid quasi-smooth Fano 3-fold weighted hypersurfaces of index $1$.
\end{abstract}

\maketitle

\tableofcontents

\section{Introduction} \label{sec:intro}

Throughout the article, the ground field is assumed to be the complex number field $\mbC$.

\subsection{K-stability, birational superrigidity, and a conjecture}
\label{sec:KstBSRConj}

The notion of K-stability was introduced by Tian \cite{Tia97} as an attempt to characterize the existence of K\"ahler--Einstein metrics (KE metrics, for short) on Fano manifolds.
Later K-stability was extended and reformulated by Donaldson \cite{Don02} in algebraic terms.
The notion of K-stability emerged in the study of KE metrics (see \cite{Don02}, \cite{Tia97}), and it gives a characterization of the existence of a KE metric for smooth Fano manifolds (see \cite{CDS}, \cite{Tia15}).   

Birational (super)rigidity means the uniqueness of a Mori fiber space in the birational equivalence class (see Definition \ref{def:BSR}), and this notion has its origin in the rationality problem of Fano varieties.
Specifically it grew out of the study of birational self-maps of smooth quartic 3-folds by Iskovskikh and Manin \cite{IM71} (see \cite{Puk13} and \cite{Che05} for surveys). 

K-stability and birational superrigidity have completely different origins and we are unable to find a similarity in their definitions.
However, both of them are closely related to some mildness of singularities of pluri-anticanonical divisors (or linear systems).
For example, it is proved by Odaka and Sano \cite{OS12} (see also \cite{Tia87}) that a Fano variety $X$ of dimension $n$ is K-stable if $\alpha (X) > n/(n+1)$.
Here 
\[
\alpha (X) = \sup \{\, c \in \mbQ_{> 0}  \mid \text{$(X, c D)$ is log canonical for any $D \in \left| - K_X \right|_{\mbQ}$} \,\} 
\] is called the {\it alpha invariant} of $X$ and it measures singularities of pluri-anticanonical divisors.
We refer readers to \cite{Fuj19b}, \cite{Li17}, \cite{FO18} and \cite{BJ20} for criteria for K-stability in terms of beta and delta invariants which are more or less related to singularities of pluri-anticanonical divisors.
On the other hand, it is known that a Fano variety of Picard number one is birationally superrigid if and only if the pair $(X, \lambda \mcM)$ is canonical for any $\lambda \in \mbQ_{> 0}$ and any movable linear system $\mcM$ such that $\lambda \mcM \sim_{\mbQ} - K_X$ (see Theorem \ref{thm:charactbsr}).
With these relations in mind, one may expect a positive answer to the following.

\begin{Conj} \label{conj:BSRKst}
A birationally superrigid Fano variety is K-stable.
\end{Conj}

Actually we expect stronger conjectures to hold (see \S \ref{section:BRKst} for discussions).
The main aim of this article is to verify Conjecture \ref{conj:BSRKst} for quasi-smooth Fano 3-fold weighted hypersurfaces.

\subsection{Evidences for the conjecture}

\subsubsection{Smooth Fano manifolds}
\label{sec:introsmFano}

Smooth quartic $3$-folds and double covers of $\mbP^3$ branched along a smooth hypersurface of degree $6$ (or equivalently smooth weighted hypersurfaces of degree $6$ in $\mbP (1, 1, 1, 1, 3)$) are the only smooth Fano 3-fold that are birationally superrigid (see \cite{IM71}, \cite{Isk80}, \cite{Che05}).
K-stability (and hence the existence of a KE metric) is proved for  smooth quartic 3-folds (\cite[Corollary 1.4]{Fuj19a}) and for smooth weighted hypersurfaces of degree $6$ in $\mbP (1, 1, 1, 1, 3)$ (\cite[Corollary 3.4]{CPW14}).

We have evidences in arbitrary dimension $n \ge 3$.
After the results established in low dimensional cases in \cite{IM71}, \cite{Puk87} and \cite{dFEM03}, it is finally proved by de Fernex \cite{dF13} that any smooth hypersurface of degree $n+1$ in $\mbP^{n+1}$ is birationally superrigid for $n \ge 3$.
On the other hand, it is proved by Fujita \cite{Fuj19a} that any such hypersurface is K-stable (hence admits a KE metric).
It is also proved in \cite{Zhu20b} that a smooth Fano complete intersection $X \subset \mbP^{n+r}$ of Fano index $1$, codimension $r$ and dimension $n \ge 10 r$ is birationally superrigid and K-stable.

\subsubsection{Fano 3-fold weighted hypersurfaces}
\label{sec:introFanoWH}

By a {\it quasi-smooth Fano $3$-fold weighted hypersurface}, we mean a Fano 3-fold (with only terminal singularities) embedded as a quasi-smooth hypersurface in a well-formed weighted projective 4-space $\mbP (a_0, \dots, a_4)$ (see \S \ref{sec:wfqsm} for quasi-smoothness and well-formedness).
Let $X = X_d \subset \mbP (a_0, \dots, a_4)$ be a quasi-smooth Fano 3-fold weighted hypersurface of degree $d$.
Then, the class group $\Cl (X)$ is isomorphic to $\mbZ$ and is generated by $\mcO_X (1)$ (see for example \cite[Remark 4.2]{Oka19}).
By adjunction, we have $\mcO_X (-K_X) \cong \mcO_X (\iota_X)$, where
\[
\iota_X := \sum_{i=0}^4 a_i - d \in \mbZ_{> 0}.
\]
We call $\iota_X$ the {\it Fano index} (or simply {\it index}) of $X$.

By \cite{IF00} and \cite{CCC11}, quasi-smooth Fano 3-fold weighted hypersurfaces of index $1$ are classified and they consist of 95 families. 
Among them, quartic 3-folds and weighted hypersurfaces of degree $6$ in $\mbP (1, 1, 1, 1, 3)$ are smooth and the remaining 93 families consist of singular Fano 3-folds (with terminal quotient singularities).
The descriptions of these 93 families are given in Table \ref{table:main}.

\begin{Thm}[\cite{CP17}, \cite{CPR00}] \label{thm:BRWH}
Any quasi-smooth Fano $3$-fold weighted hypersurface of index $1$ is birationally rigid.
\end{Thm}

Among the 95 families, any quasi-smooth member of each of specific 50 families is birationally superrigid.
The 50 families consist of 48 families in Table \ref{table:main} which do not admit singularity with ``QI" or ``EI" in the 4th column plus the 2 families of smooth Fano weighted hypersurfaces.
For each of the remaining 45 families, a general quasi-smooth member is strictly birationally rigid (meaning that it is not birationally superrigid) but some special quasi-smooth members are birationally superrigid (see \S \ref{sec:95fam} for details).

\begin{Thm}[{\cite[Corollary 1.45]{Che09}}] \label{thm:KstWHgen}
A general quasi-smooth member of each of the 95 families is K-stable and admits a KE metric.
\end{Thm}

The generality assumption is crucial in Theorem \ref{thm:KstWHgen}. 
In particular, it is highly likely that birationally superrigid special members of each of the above mentioned 45 families are not treated in Theorem \ref{thm:KstWHgen}.
Note that openness of K-stability is known (see \cite{Oda13}, \cite{Don15}, and \cite{BL22}), and this implies the difficulty in determining which Fano varieties (in a given family) are K-stable.
Although Theorems \ref{thm:BRWH} and \ref{thm:KstWHgen} give a strong evidence for Conjecture \ref{conj:BSRKst}, it is very important to consider special (quasi-smooth) members for Conjecture \ref{conj:BSRKst}. 

\subsubsection{Conceptual evidences}

Apart from evidences by concrete examples given in \S \ref{sec:introsmFano} and \S \ref{sec:introFanoWH}, we have conceptual results supporting Conjecture \ref{conj:BSRKst}. 

The notion of slope stability for polarized varieties was introduced by Ross and Thomas (cf.\ \cite{RT07}).
For a Fano variety $X$, slope stability of $(X, -K_X)$ is a weaker version of K-stability.

\begin{Thm}[{\cite[Theorem 1.1]{OO13}}]
Let $X$ be a birationally superrigid Fano manifold of Fano index $1$.
If $\left| -K_X \right|$ is base point free, then $(X, -K_X)$ is slope stable.
\end{Thm}

As it is explained in \S \ref{sec:KstBSRConj}, K-stability of a Fano variety $X$ of dimension $n$ follows from the inequality $\alpha (X) > n/(n+1)$.
In practice, the computations of alpha invariants are very difficult and hence it is not easy to prove the inequality $\alpha (X) > n/(n+1)$.

\begin{Rem}
In fact, our results show that there exists a birationally superrigid Fano 3-fold $X$ such that $\alpha (X) < 3/4$ (see Example \ref{ex:No46degQI}).
\end{Rem}

Recently Stibitz and Zhuang relaxed the assumption on the alpha invariants significantly under the assumption of birational superrigidity, and obtained the following.

\begin{Thm}[{\cite[Theorem 1.2, Corollary 3.1]{SZ19}}] \label{thm:SZ}
Let $X$ be a birationally superrigid Fano variety.
If $\alpha (X) \ge 1/2$, then $X$ is K-stable.
\end{Thm}

Note that the assumption on the alpha invariant is $\alpha (X) > 1/2$ in \cite[Theorem 1.2]{SZ19}, but the equality is allowed by \cite[Corollary 3.1]{SZ19}. 
It is informed by C. Xu and Z. Zhuang that one can even conclude the uniform K-stability of $X$ in Theorem \ref{thm:SZ} under the same assumption.
The notion of uniform K-stability is originally introduced in \cite{Der16b} and \cite{BHJ17} (see also \cite{Fuj19b} and \cite{BJ20}) and it is stronger than K-stability\footnote{Soon after this paper is completed, it is proved in \cite{LXZ22} that uniform K-stability is equivalent to K-stability.}.
Moreover it is very important to mention that uniform K-stability implies the existence of a KE metric (\cite{LTW19}). 
Combining these results, we have the following.

\begin{Thm}[{\cite[Theorem 9.6]{Xu21}, \cite{SZ19}, \cite{LTW19}}] \label{thm:SXZ}
Let $X$ be a birationally superrigid Fano variety and assume that $\alpha (X) \ge 1/2$.
Then $X$ is uniformly K-stable.
In particular, $X$ is K-stable and it admits a KE metric.
\end{Thm}

\subsection{Main results}
\label{sec:mainresults}

 We state Main Theorem of this article.

\begin{Thm}[Main Theorem] \label{mainthm}
Let $X$ be a quasi-smooth Fano $3$-fold weighted hypersurface of index $1$.
Then $\alpha (X) \ge 1/2$. 
\end{Thm}

The following is a direct consequence of Theorems \ref{mainthm} and \ref{thm:SXZ}.

\begin{Cor} \label{maincor:Kst}
Any birationally superrigid quasi-smooth Fano $3$-fold weighted hypersurface of index $1$ is K-stable and admits a KE metric.
\end{Cor}

By \cite[Corollary 1.3]{ACP20}, a birationally superrigid quasi-smooth Fano 3-fold weighted hypersurface necessarily has Fano index $1$.
Thus we obtain the following.

\begin{Cor}
Conjecture \ref{conj:BSRKst} is true for quasi-smooth Fano 3-fold weighted hypersurfaces.
\end{Cor}

It is natural to consider a generalization of Conjecture \ref{conj:BSRKst} by relaxing the assumption of birational superrigidity to birational rigidity (see \S \ref{section:BRKst}), or to expect that the conclusion of Corollary \ref{maincor:Kst} holds without the assumption of birational superrigidity.
We are unable to relax the assumption of birational superrigidity to birational rigidity in Theorem \ref{thm:SZ} or \ref{thm:SXZ}, and thus we cannot conclude K-stability for strictly birationally rigid members as a direct consequence of Theorem \ref{mainthm}.
By the arguments delivered in this article, we are able to prove $\alpha (X) > 3/4$ for any quasi-smooth member $X$ of suitable families.
As a consequence, we have the following (see \S \ref{sec:KEmetric} for details). 

\begin{Thm}[= Theorem \ref{thm:qsmWHKE}] \label{thm:introKE}
Let $X$ be any quasi-smooth member of a family which is given ``KE"  in the right-most column of Table \ref{table:main}.
Then $X$ is K-stable and admits a KE metric.
\end{Thm}

We can also prove K-stability for any quasi-smooth member (which is not necessarily birationally superrigid) of suitable families.
 
\begin{Thm}[= Corollary \ref{cor:Kstqsm}] \label{thm:introKst}
Let $X$ be any quasi-smooth member of a family which is given ``K" or ``KE" in the right-most column of Table \ref{table:main}.
Then $X$ is K-stable.
\end{Thm}

We explain the organization of this article.
In Section~\ref{chap:prelim}, we recall definitions and basic properties of relevant notions such as birational (super)rigidity, log canonical thresholds, alpha invariants, and weighted projective varieties.
In Section~\ref{chap:methods}, we explain methods of computing log canonical thresholds and alpha invariants.
By applying these methods, we compute local alpha invariants $\alpha_{\msp} (X)$ for any point $\msp$ on a quasi-smooth Fano 3-fold weighted hypersurface $X$ of index $1$.
In Sections~\ref{chap:smpt} and \ref{chap:singpt}, we compute local alpha invariants at smooth and singular points, respectively.
At this stage, Theorem \ref{mainthm} is proved except for 7 specific families.
These exceptional families are families No. 2, 4, 5, 6, 8, 10 and 14, and we need extra arguments to prove $\alpha (X) \ge 1/2$, which will be done in Section~\ref{chap:exc}.  
In Section~\ref{chap:discuss}, we will consider and prove further results such as Theorems \ref{thm:introKE} and \ref{thm:introKst}.
 We will also discuss related problems that arise naturally through the experience of huge amount of computations.
 Finally, in Section~\ref{chap:table}, various information on the families of quasi-smooth Fano 3-fold weighted hypersurfaces of index $1$ are summarized, and we also make it clear in Remark \ref{rem:whatisleft} what is left about K-stability for quasi-smooth Fano $3$-fold weighted hypersurfaces of index $1$.

\subsection{Relevant results in K-stability}
\label{sec:relresults}

Recently both theoretical and explicit studies of K-stability of Fano varieties have been developed drastically.
We refer readers to \cite{Xu21} for up to date surveys.
Following the suggestion from the referee, we add this section \S \ref{sec:relresults} to explain some of them that are developed during the preparation or after the completion of this article.

One of the most striking one is the equivalence of the notions of K-stability and uniform K-stability for klt Fano varieties that is proved in \cite{LXZ22}.
It in particular follows that the K-stability implies the existence of K\"{a}hler--Einstein metric for klt Fano varieties.
As a consequence, we are now able to conclude in Theorem~\ref{thm:introKst} not only the K-stability of $X$ but also the existence of a KE metric on $X$.

It should be mentioned that currently there are various methods in hand to check K-stability: the most powerful methods at present are the induction method of Abban and Zhuang \cite{AZ22} computing (local) delta invariants, or the moduli method of Liu et al.\ (see e.g.\ \cite{Liu22}, \cite{LX19}).
These methods are developed parallel to the preparation of this article, and we do not use them.
As it is explained in \S \ref{sec:mainresults}, the proofs of the main results of this article rely on the computation of (local) alpha invariants.

This article aims the systematic study of singular Fano $3$-folds.
There are on-going work by Cheltsov and collaborators on smooth Fano $3$-folds.
In the book \cite{Calabi}, it is completely determined whether the general member of each of the 105 irreducible families of smooth Fano 3-folds admits a KE metric or not.
Very recently there have been a lot of works aiming to drop the generality assumption in the above result and to classify K-(poly)stable smooth Fano 3-folds in each family completely (see \cite{Liu23}, \cite{CP22}, \cite{CFKO22}, \cite{CFKP23}, \cite{BL22}, \cite{Den22}, \cite{CDF22}, \cite{Mal23}).

 \begin{Ack}
We would like to thank Professors C.~Xu and Z.~Zhuang for valuable comments, and especially for informing us of Theorem \ref{thm:SXZ}.
We also would like to thank the anonymous referee for his/her suggestion on adding \S \ref{sec:relresults} explaining some relevant results in K-stability that have been developed during the preparation and after the completion of this article, and Remark~\ref{rem:whatisleft} that makes it clear what is left about K-stability of quasi-smooth Fano $3$-fold weighted hypersurfaces of index $1$.
 The first author was supported by the National Research Foundation of Korea (NRF-2023R1A2C1003390 and NRF-2022M3C1C8094326).
The second author is partially supported by JSPS KAKENHI Grant Numbers JP18K03216 and JP22H01118.
The third author was supported by the National Research Foundation of Korea (NRF-2020R1A2C1A01008018 and NRF-2022M3C1C8094326).
\end{Ack}

\section{Preliminaries} \label{chap:prelim}

\subsection{Basic definitions and properties} \label{sec:basicdefs}

We refer readers' to \cite{KM98} for standard notions of birational geometry which are not explained in this article.

\begin{Def}
By a {\it Fano variety}, we mean a normal projective $\mbQ$-factorial variety with at most terminal singularities whose anticanonical divisor is ample.
\end{Def}

For a variety $X$, we denote by $\Sm (X)$ the smooth locus of $X$ and $\Sing (X) = X \setminus \Sm (X)$ the singular locus of $X$.
For a subset $\Gamma \subset X$, we define $\Sing_{\Gamma} (X) := \Sing (X) \cap \Gamma$.
Let $X$ be a normal variety and $D$ a Weil divisor (class) on $X$.
We denote by $|D|_{\mbQ}$ the set of effective $\mbQ$-divisors on $X$ which are $\mbQ$-linearly equivalent to $D$.
For a smooth point $\msp \in X$, we define $|\mcI_{\msp} (D)|$ to be the linear subspace of $|D|$ consisting of members of $|D|$ passing through $\msp$.

\subsubsection{Birational (super)rigidity of Fano varieties}

Let $X$ be a normal $\mbQ$-factorial variety, $D$ a $\mbQ$-divisor on $X$ and $\mcM$ a movable linear system on $X$.
For a prime divisor $E$ over $X$, we define $\ord_E (D)$ to be the coefficient of $E$ in $\varphi^*D$, where $\varphi \colon Y \to X$ is a birational morphism such that $E \subset Y$, and we set $m_E (\mcM) := \ord_E (M)$, where $M$ is a general member of $\mcM$.
For a positive rational number $\lambda$, we say that a pair $(X, \lambda \mcM)$ is {\it canonical} if 
\[
a_E (K_X) \ge \lambda m_E (\mcM)
\] 
for any exceptional prime divisor $E$ over $X$.

Let $X$ be a Fano variety of Picard number one.
Note that we can view $X$ (or more precisely the structure morphism $X \to \Spec \mbC$) as a Mori fiber space.

\begin{Def} \label{def:BSR}
We say that $X$ is {\it birationally rigid} if the existence of a Mori fiber space $Y \to T$ such that $Y$ is birational to $X$ implies that $Y$ is isomorphic to $X$ (and $T = \Spec \mbC$).
We say that $X$ is {\it birationally superrigid} if $X$ is birationally rigid and $\Bir (X) = \Aut (X)$.
\end{Def}

\begin{Def}
A closed subvariety $\Gamma \subset X$ is called a {\it maximal center} if there exists a movable linear system $\mcM \sim_{\mbQ} - n K_X$ and an exceptional prime divisor $E$ over $X$ such that $m_E (\mcM) > n a_E (K_X)$.
\end{Def}

\begin{Thm}[{\cite[Theorem 1.26]{CS}}] \label{thm:charactbsr}
A Fano variety $X$ of Picard number $1$ is birationally superrigid if and only if the pair $(X, \frac{1}{n} \mcM)$ is canonical for any movable linear system $\mcM$ on $X$, where $n \in \mbQ_{> 0}$ is such that $\mcM \sim_{\mbQ} - n K_X$, or equivalently if and only if there is no maximal center on $X$.
\end{Thm}

\subsubsection{Log canonical thresholds and alpha invariants}

\begin{Def} \label{def:glct}
Let $(X, \Delta)$ be a pair, $D$ an effective $\mbQ$-divisor on $X$, and let $\msp \in X$ be a point.
Assume that $(X, \Delta)$ has at most log canonical singularities.
We define the {\it log canonical threshold} (abbreviated as LCT) of $(X, \Delta; D)$ {\it at} $\msp$ and the {\it log canonical threshold} of $(X, \Delta; D)$ to be the numbers
\[
\begin{split}
\lct_{\msp} (X, \Delta; D) &= \sup \{\, c \in \mbQ_{\ge 0} \mid \text{$(X, \Delta + c D)$ is log canonical at $\msp$} \,\}, \\
\lct (X, \Delta; D) &= \sup \{\, c \in \mbQ_{\ge 0} \mid \text{$(X, \Delta + c D)$ is log canonical} \,\}, 
\end{split}
\]
respectively.
We set $\lct_{\msp} (X; D) = \lct_{\msp} (X, 0; D)$ and $\lct_{\msp} (X; D) = \lct (X, \Delta; D)$ when $\Delta = 0$.
Assume that $\left| -K_X \right|_{\mbQ} \ne \emptyset$.
Then we define the {\it alpha invariant of $X$ at $\msp$} and the {\it alpha invariant of $X$} to be the numbers  
\[
\begin{split}
\alpha_{\msp} (X) &= \inf \{\, \lct_{\msp} (X, D) \mid D \in \left| -K_X \right|_{\mbQ} \,\}, \\
\alpha (X) &= \inf \{\, \alpha_{\msp} (X) \mid \msp \in X \, \},
\end{split}
\]
respectively.
\end{Def}

The following fact is frequently used.

\begin{Rem} \label{rem:covex}
Let $\msp$ be a point on $X$, and let $D_1, D_2$ be effective $\mbQ$-divisors on $X$.
If both $(X, D_1)$ and $(X, D_2)$ are log canonical at $\msp$, then the pair 
\[
(X, \lambda D_1 + (1-\lambda) D_2)
\] 
is log canonical at $\msp$ for any $\lambda \in \mbQ$ such that $0 \le \lambda \le 1$.
In particular, if $\alpha_{\msp} (X) < c$ for some number $c > 0$, then there exists an irreducible $\mbQ$-divisor $D \in \left|-K_X \right|_{\mbQ}$ such that $(X, c D)$ is not log canonical at $\msp$.
Here a $\mbQ$-divisor is {\it irreducible} if its support $\Supp (D)$ is irreducible.
\end{Rem}

\subsubsection{Cyclic quotient singularities and orbifold multiplicities}

\begin{Def}
Let $r > 0$ and $a_1, \dots, a_n$ be integers.
Suppose that the cyclic group $\bmu_r$ of $r$th roots of unity in $\mbC$ acts on the affine $n$-space $\mbA^n$ with affine coordinates $x_1, \dots, x_n$ via
\[
(x_1, \dots, x_n) \mapsto (\zeta^{a_1} x_1, \dots, \zeta^{a_n} x_n),
\]
where $\zeta \in \bmu_r$ is a fixed primitive $r$th root of unity. 
We denote by $\bar{o} \in \mbA^n/\bmu_r$ the image of the origin $o \in \mbA^n$ under the quotient morphism $\mbA^n \to \mbA^n/\bmu_r$.
A singularity $\msp \in X$ is a {\it cyclic quotient singularity} of type $\frac{1}{r} (a_1, \dots, a_n)$ if $\msp \in X$ is analytically isomorphic to an analytic germ $\bar{o} \in \mbA^n/\bmu_r$.
In this case $r$ is called the {\it index} of the cyclic quotient singularity $\msp \in X$.
\end{Def}

\begin{Rem}
Let $\msp \in X$ be an $n$-dimensional cyclic quotient singular point.
Then we have a suitable action of $\bmu_r$ on $\mbA^n$ such that there is an analytic isomorphism $\bar{o} \in \mbA^n/\bmu_r \cong \msp \in X$ of (analytic) germs.
In the following, the germ $o \in \mbA^n$ is often denoted by $\check{\msp} \in \check{X}$.
By identifying $\msp \in X \cong \bar{o} \in \mbA^n/\bmu_r$, the quotient morphism $o \in \mbA^n \to \bar{o} \in \mbA^n/\bmu_r$ is denoted by $q_{\msp} \colon \check{X} \to X$ and is called the {\it quotient morphism} of $\msp \in X$.
\end{Rem}

Note that, by convention, the case $r = 1$ is allowed in the definition of cyclic quotient singularity.
A cyclic quotient singularity $\msp \in X$ of index $1$ is nothing but a smooth point $\msp \in X$ and  in that case the quotient morphism $q_{\msp} \colon \check{X} \to X$ is simply an isomorphism.

\begin{Def}
Let $\msp \in X$ be a cyclic quotient singularity and let $q_{\msp} \colon \check{X} \to X$ be its quotient morphism with $\check{\msp} \in \check{X}$ the preimage of $\msp$.
For an effective $\mbQ$-divisor $D$ on $X$, we define 
\[
\omult_{\msp} (D) := \mult_{\check{\msp}} (q_{\msp}^*D)
\]
and call it the {\it orbifold multiplicity} of $D$ at $\msp$.
By convention, we set $\omult_{\msp} (D) = \mult_{\msp} (D)$ when $\msp \in X$ is a smooth point.
\end{Def}

\subsubsection{Kawamata blowup}

Let $\msp \in V$ be a $3$-dimensional terminal quotient singularity.
Then it is of type $\frac{1}{r} (1, a, r-a)$, where $r$ and $a$ are coprime positive integers with $r > a$ (see \cite{MS84}).
Let $\varphi \colon W \to V$ be the weighted blowup of $V$ at $\msp$ with weight $\frac{1}{r} (1,a,r-a)$.
By \cite{Kawamata}, $\varphi$ is the unique divisorial contraction centered at $\msp$ and we call $\varphi$ the {\it Kawamata blowup} of $V$ at $\msp$.
If we denote by $E$ the $\varphi$-exceptional divisor, then $E \cong \mbP (1,a,r-a)$ and we have
\[
K_W = \varphi^*K_V + \frac{1}{r} E,
\]
and
\[
(E^3) = \frac{r^2}{a (r-a)}.
\]

\subsection{Weighted projective varieties}

We recall basic definitions of various notions concerning weighted projective spaces and their subvarieties.
We refer readers to \cite{IF00} for details.

\subsubsection{Weighted projective space}

Let $N$ be a positive integer.
For positive integers $a_0, \dots, a_N$, let 
\[
R (a_0, \dots, a_N) := \mbC [x_0, \dots, x_N]
\] 
be the graded ring whose grading is given by $\deg x_i = a_i$. 
We define
\[
\mbP (a_0,\dots, a_N) := \Proj R (a_0, \dots, a_N),
\]
and call it the {\it weighted projective space} with homogeneous coordinates $x_0, \dots, x_N$ (of degree $\deg x_i = a_i$).
We sometimes denote
\[
\mbP (a_0, \dots, a_N)_{x_0, \dots, x_N}
\]
in order to make it clear the homogeneous coordinates $x_0, \dots, x_N$.
For $i = 0, \dots, N$, we denote by 
\begin{equation} \label{eq:wpvvertex}
\msp_{x_i} = (0\!:\!\cdots\!:\!1\!:\!\cdots\!:\!0) \in \mbP (a_0, \dots, a_N)
\end{equation}
the coordinate point at which only the coordinate $x_i$ does not vanish.  
Let $f \in R := R (a_0, \dots, a_N) = \mbC [x_0, \dots, x_N]$ be a polynomial.
We say that $f$ is {\it quasi-homogeneous} (resp.\ {\it homogeneous}) if it is homogeneous with respect to the grading $\deg x_i = a_i$ (resp.\ $\deg x_i = 1$) for $i = 0, 1, \dots, N$.
For a polynomial $f \in \mbC [x_0, \dots, x_N]$ and a monomial $M = x_0^{m_0} \cdots x_N^{m_N}$, we denote by 
\[
\coeff_f (M) \in \mbC
\] 
the coefficient of $M$ in $f$, and, by a slight abuse of notation, we write $M \in f$ if $\coeff_f (M) \ne 0$.
For quasi-homogeneous polynomials $f_1, \dots, f_k \in R$, we denote by 
\[
(f_1 = \cdots = f_k = 0) \subset \mbP (a_0, \dots, a_N)
\]
the closed subscheme defined by the quasi-homogeneous ideal $(f_1, \dots, f_k) \subset R$.
Moreover, for a closed subscheme $X \subset \mbP (a_0, \dots, a_N)$ and quasi-homogeneous polynomials $g_1, \dots, g_l \in R$, we define
\[
(g_1 = \cdots = g_l = 0)_X := (g_1 = \cdots = g_l = 0) \cap X,
\]
which is a closed subscheme of $X$.
For $i = 0, \dots, N$, we define 
\begin{equation} \label{eq:wpvHU}
\begin{split}
\mcH_{x_i} &:= (x_i = 0) \subset \mbP (a_0, \dots, a_N), \\
\mcU_{x_i} &:= \mbP (a_0, \dots, a_N) \setminus \mcH_{x_i}.
\end{split}
\end{equation} 

\begin{Rem}
The weighted projective space $\mbP (a, b, c, d, e)$ with homogeneous coordinates $x$, $y$, $z$, $t$, $w$ of degrees $a$, $b$, $c$, $d$, $e$, respectively, is sometimes denoted by 
\[
\mbP (a, b, c, d, e)_{x, y, z, t, w}
\]
in order to emphasize the homogeneous coordinates.
For a coordinate $v \in \{x, \dots, w\}$, the point $\msp_v \in \mbP (a, b, c, d, e)$, the quasi-hyperplane $\mcH_v = (v = 0) \subset \mbP (a, b, c, d, e)$ and the open set $\mcU_v = \mbP (a, b, c, d, e) \setminus \mcH_v$ are similarly defined as in \eqref{eq:wpvvertex} and \eqref{eq:wpvHU}.
\end{Rem}

\subsubsection{Well-formedness and quasi-smoothness}
\label{sec:wfqsm}

\begin{Def}
We say that a weighted projective space $\mbP (a_0, \dots, a_N)$ is {\it well-formed} if 
\[
\gcd \{a_0, \dots, \hat{a}_i, \dots, a_N\} = 1
\]
for any $i = 0, 1, \dots, N$.
\end{Def}

\begin{Def}
Let $\mbP (a_0, \dots, a_N)$ be a weighted projective space such that $\gcd \{a_0, \dots, a_N\} = 1$.
For $j = 0, 1, \dots, N$, we set
\[
\begin{split}
l_j &:= \gcd \{ a_0, a_1, \dots, \hat{a}_j, \dots, a_N\}, \\
m_j &:= l_0 l_1 \cdots \hat{l}_j \cdots l_N, \\
b_j &:= \frac{a_j}{m_j}.
\end{split}
\]
We then define
\[
\mbP (a_0, \dots, a_N)^{\wf} := \mbP (b_0, \dots, b_N)
\]
and call it the {\it well-formed model} of $\mbP (a_0, \dots, a_N)$.
\end{Def}

\begin{Rem}
Any weighted projective space is isomorphic to a well-formed one (see e.g.\ \cite[Lemma 5.7]{IF00}).
More precisely, for a weighted projective space $\mbP = \mbP (a_0, \dots, a_N)$ with $\gcd \{a_0, \dots, a_N\} = 1$, there exists an isomorphism
\[
\phi \colon \mbP (a_0, \dots, a_N)_{x_0, \dots, x_N} \to \mbP^{\wf} = \mbP (b_0, \dots, b_N)_{y_0, \dots, y_N},
\]
such that $\phi^* \mcH_{y_i} = m_i \mcH_{x_i}$ for $i = 0, 1, \dots, N$, where $\mcH_{x_i} = (x_i = 0) \subset \mbP$ and $\mcH_{y_i} = (y_i = 0) \subset \mbP^{\wf}$.
\end{Rem}

In the following, we set $\mbP := \mbP (a_0, \dots, a_N)$ and we denote by 
\[
\Pi \colon \mbA^{N+1} \setminus \{o\} \to \mbP, \quad
(\alpha_0, \dots, \alpha_N) \mapsto (\alpha_0\!:\!\cdots\!:\!\alpha_N),
\]
the canonical projection.
Let $X \subset \mbP$ be a closed subscheme.
We set $C_X^* := \Pi^{-1} (X)$ and call it the {\it punctured affine quasi-cone} over $X$.
The {\it affine quasi-cone} $C_X$ over $X$ is the closure of $C_X^*$ in $\mbA^{N+1}$.
We set $\pi := \Pi|_{C_X^*} \colon C_X^* \to X$.

\begin{Def}
We say that a closed subscheme $X \subset \mbP$ is {\it well-formed} if $\mbP$ is well-formed and $\codim_X (X \cap \Sing (\mbP)) \ge 2$.
\end{Def}

\begin{Def}
Let $X \subset \mbP$ be a closed subscheme as above.
We define the {\it quasi-smooth locus} of $X$ as 
\[
\QSm (X) := \pi (\Sm (C^*_X)) \subset X.
\] 
Let $S$ be a subset of $X$.
We say that $X$ is {\it quasi-smooth along} $S$ if $S \subset \QSm (X)$.
We simply say that $X$ is {\it quasi-smooth} when $X = \QSm (X)$.
\end{Def}

\subsubsection{Orbifold charts}

Let $\mcU_{x_i}$ be the open subset of $\mbP = \mbP (a_0, \dots, a_N)_{x_0, \dots, x_N}$ as in \eqref{eq:wpvHU}, where $i \in \{0, 1, \dots, N\}$.
We call $\mcU_{x_i}$ the {\it standard affine open subset} of $\mbP$ containing $\msp_{x_i}$.
We denote by $\breve{\mcU}_{x_i}$ the affine $N$-space $\mbA^N$ with affine coordinates $\breve{x}_0, \dots, \widehat{\breve{x}_i}, \dots, \breve{x}_N$.
Consider the $\bmu_{a_i}$-action on $\breve{\mcU}_i$ defined by
\[
\breve{x}_j \mapsto \zeta^{a_j} \breve{x}_j, \quad \text{for $j = 0, \dots, \hat{i}, \dots, N$},
\] 
where $\zeta \in \bmu_{a_i}$ is a primitive $a_i$th root of unity.
Then the open set $\mcU_i$ can be naturally identified with the quotient $\breve{\mcU}_{x_i}/\bmu_{a_i}$.
In fact, this can be seen by the identification
\[
\breve{x}_j = \frac{x_j}{x_i^{a_j/a_i}}, \quad \text{for $j = 0, \dots, \hat{i}, \dots, N$}.
\]
The quotient morphism $\breve{\mcU}_{x_i} \to \breve{\mcU}_{x_i}/\bmu_{a_i} = \mcU_{x_i}$ is denoted by 
\[
\rho_{x_i} \colon \breve{\mcU}_{x_i} \to \mcU_{x_i}
\] 
and is called the {\it orbifold chart} of $\mbP$ containing $\msp_{x_i}$.

Let $X \subset \mbP$ be a subscheme.
Usually, we denote by $U_{x_i} \subset X$ the open set $\mcU_{x_i} \cap X$, and we call $U_{x_i}$ the {\it standard affine open subset} of $X$ containing $\msp_{x_i}$. 
In this case, we set $\breve{U}_{x_i} = \rho_{x_i}^{-1} (U_{x_i}) \subset \breve{\mcU}_{x_i}$. 
By a slight abuse of notation, the morphism $\rho_{x_i}|_{U_i} \colon \breve{U}_{x_i} \to U_{x_i}$ is also denote by 
\[
\rho_{x_i} \colon \breve{U}_{x_i} \to U_{x_i}
\] 
and is called the {\it orbifold chart} of $X$ containing $\msp_{x_i}$.
When we are using the notation $\msp = \msp_{x_i}$, the morphism $\rho_{x_i}$ is sometimes denoted by $\rho_{\msp}$.
Note that $\breve{U}_{x_i}$ is not necessary smooth in general.

Suppose that $X \subset \mbP$ is a closed subvariety containing the point $\msp = \msp_{x_i}$.
The preimage $\breve{\msp}$ of $\msp$ is the origin of $\breve{U}_{x_i} \subset \breve{\mcU}_{x_i} = \mbA^N$.
It is straightforward to see that $X$ is quasi-smooth at $\msp$ if and only if $\breve{U}_{x_i}$ is smooth at $\breve{\msp}$.
Suppose that $X$ is quasi-smooth at $\msp$.
A system of local coordinates of $U_{x_i}$ at $\breve{\msp}$ is called a system of {\it local orbifold coordinates} of $X$ at $\msp$.
In this case, $\msp \in X$ is a cyclic quotient singularity of index $a_i$ and $\rho_{x_i} \colon \breve{U}_{x_i} \to U_{x_i}$ can be identified with (or analytically equivalent to) the quotient morphism $q_{\msp}$ of $\msp \in X$ after shrinking $U_{x_i}$ and then $\breve{U}_{x_i}$.
Moreover, if $X$ is quasi-smooth, then $\breve{U}_i$ is smooth for any $i$.

\begin{Rem}
When we work with $\mbP = \mbP (a, b, c, d, e)_{x, y, z, t, w}$ and its closed subscheme $X \subset \mbP$, then $\mcU_{v} = \mbA^5_{\breve{x}, \dots, \hat{\breve{v}}, \dots, \breve{w}}$, $\rho_v \colon \breve{\mcU}_v \to \mcU_v$, $\breve{U}_{v} = \rho_v^{-1} (U_v) \subset \mcU_v$ and $\rho_v \colon \breve{U}_v \to U_v$ are similarly defined.
\end{Rem}

\subsubsection{Weighted hypersurfaces and quasi-tangent divisors}

As in the previous subsections, we work with $\mbP = \mbP (a_0, \dots, a_N)_{x_0, \dots, x_N}$.

\begin{Def}
A {\it quasi-linear polynomial} (or a {\it quasi-linear form}) in variables $x_0, \dots, x_{n+1}$ is a quasi-homogeneous polynomial $f = f (x_0,\dots, x_{n+1})$ such that $x_i \in f$ for some $i = 0, \dots, n+1$.
\end{Def}

\begin{Def}
We say that a subvariety $S \subset \mbP$ is a {\it quasi-linear subspace} of $\mbP$ if it is a complete intersection in $\mbP$ defined by quasi-linear equations of the form
\[
\ell_1 + f_1 = \ell_2 + f_2 = \cdots = \ell_k + f_k = 0,
\]
where $\ell_1, \ell_2, \dots, \ell_k$ are linearly independent linear forms in variables $x_0, \dots, x_{n+1}$ and $f_1, \dots, f_k \in \mbC [x_0, \dots, x_{n+1}]$ are quasi-homogeneous polynomials which are not quasi-linear.
A quasi-linear subspace of $\mbP$ of codimension $1$ (resp.\ dimension $1$) is called a {\it quasi-hyperplane} (resp.\ {\it quasi-line}) of $\mbP$.
\end{Def}

It is clear that a quasi-linear subspace of $\mbP$ is isomorphic to a weighted projective space.
In particular, a quasi-line is isomorphic to $\mbP^1$.

Let $X$ be a hypersurface in $\mbP = \mbP (a_0, \dots, a_N)$ defined by a quasi-homogeneous polynomial of degree $d$. 
We often denote it as $X = X_d \subset \mbP (a_0, \dots, a_N)$.
Suppose that $X$ is quasi-smooth at a point $\msp = \msp_{x_i}$.
Then the defining polynomial $F = F (x_0, \dots, x_N)$ of $X$ can be written as
\begin{equation}
\label{eq:qtangpoly}
F = x_i^m f + x_i^{m-1} g_{m-1} + \cdots + x_i g_1 + g_0,
\end{equation}
where $m \ge 0$, $f = f (x_0, \dots, x_N)$ is a quasi-homogeneous polynomial of degree $d - m a_i$ which is quasi-linear and $g_k = g_k (x_0, \dots,\hat{x}_i, \dots, x_N)$ is a quasi-homogeneous polynomial of degree $d - k a_i$ which is not quasi-linear for $0 \le k \le m - 1$.
Note that the expression \eqref{eq:qtangpoly} is uniquely determined once the homogeneous coordinates of $\mbP$ are fixed.

\begin{Def} \label{def:qtangdiv}
Under the notation and assumptions as above, we call $f$ the {\it quasi-tangent polynomial} of $X$ at $\msp$ and the divisor $(f = 0)_X$ on $X$ is called the {\it quasi-tangent divisor} of $X$ at $\msp$.
When $f = x_j$ for some $j$, then we also call $x_j$ as the {\it quasi-tangent coordinate} of $X$ at $\msp$.
\end{Def}

\begin{Rem}
Let $X = X_7 \subset \mbP (1, 1, 1, 2, 3)_{x, y, z, t, w}$ be a weighted hypersurface of degree $7$.
Suppose that its defining polynomial is of the form
\[
F = t^3 x + t^2 w + t g_5 + g_7,
\]
where $g_5, g_7 \in \mbC [x, y, z, w]$ are quasi-homogeneous polynomials of degree $5, 7$, respectively.
In this case $X$ is quasi-smooth at $\msp = \msp_t$.
The quasi-tangent polynomial of $X$ at $\msp$ is $t x + w$.
Note that $x$ is not a quasi-tangent coordinate of $X$ at $\msp$ because of the presence of $t^2 w \in F$.
\end{Rem}

\begin{Lem}
Let $X \subset \mbP$ be a weighted hypersurface of degree $d$.
Assume that $X$ is quasi-smooth at a point $\msp = \msp_{x_i}$ for some $i = 0, 1, \dots, N$ and let $x_j$ be a homogeneous coordinate such that $x_j \in f$, where $f$ is the quasi-tangent polynomial of $X$ at $\msp$.
Then, after a suitable choice of homogeneous coordinates $x_0, \dots, x_N$, the defining polynomial $F$ of $X$ can be written as
\[
F = x_i^m x_j + x_i^{m-1} g_{m - 1} + \cdots + x_i g_1 + g_0,
\]
where $g_k = g_k (x_0, \dots, \hat{x}_i, \dots, x_N)$ is a quasi-homogeneous polynomial of degree $d - k a_i$ which is not quasi-linear.
\end{Lem}

\begin{proof}
We can write $F = x_i^m f + g$, where $m \ge 0$, $f = f (x_0, \dots, x_N) \ni x_j$ is the quasi-tangent polynomial and $g$ is a quasi-homogeneous polynomial of degree $d$ which does not involve a monomial divisible by $x_i^m$ and which is contained in the ideal $(x_0, \dots, \hat{x}_i, \dots, x_N)^2$.
We write $g = x_j^e h_{e} + \cdots + x_j h_1 + h_0$, where $e \ge 0$ and $h_k$ is a quasi-homogeneous polynomial of degree $d - k a_j$ which does not involve the variables $x_j$.
By rescaling $x_j$, we may assume $f = x_j - \tilde{f}$, where $x_j \notin \tilde{f}$, and we write $\tilde{f} = x_i^n \tilde{f}_n + \cdots + x_i \tilde{f}_1 + \tilde{f}_0$, where $\tilde{f}_k$ is a quasi-homogeneous polynomial of degree $d - (m + k) a_i$ which does not involve the variable $x_i$. 
We consider the coordinate change $x_j \mapsto x_j + \tilde{f}$.
Then the new defining polynomial can be written as
\[
\begin{split}
F &= x_i^m x_j + (x_j + x_i^n \tilde{f}_n + \cdots)^e h_{e} + \cdots + (x_j + x_i^n \tilde{f}_n + \cdots) h_1 + h_0 \\
&= (x_i^{n e - m} \tilde{f}_n^e + \cdots + x_j) x_i^m + \cdots 
\end{split}
\]
It follows that the new quasi-tangent polynomial is $x_i^{n e - m} \tilde{f}_n^e + \cdots + x_j$.
We claim that $n e - m< n$.
We have $d = m a_i + a_j$, $d = e a_j + \deg h_e \ge e a_j$ and $a_j = n a_i + \deg \tilde{f}_n > n a_i$, which implies $n e - m < n$.
Thus, repeating the above coordinate change, we can drop the degree of the quasi-tangent coordinate with respect to $x_i$, and we may assume $F = x_i^m f + x_i^{m-1} g_{m-1} + \cdots + x_i g + g_0$, where $f \ni x_j$ and $g_k$ are quasi-homogeneous polynomials of degree $a_j$ and $d - k a_i$, respectively, which do not involve the variable $x_i$.
Moreover $g_k$ is not quasi-linear for $0 \le k \le m-1$.
Finally, replacing $x_j$, we may assume $f = x_j$ and this completes the proof.
\end{proof}

\begin{Rem}
Suppose that a weighted hypersurface $X \subset \mbP$ is quasi-smooth at $\msp = \msp_{x_i}$.
Then $\omult_{\msp} ((f = 0)_X) > 1$ for the quasi-tangent polynomial $f$ of $X$ at $\msp$.
Moreover, $x_j$ is a quasi-tangent coordinate of $X$ at $\msp$ if and only if $\omult_{\msp} (H_{x_j}) > 1$.
\end{Rem}

\subsection{The 95 families}
\label{sec:95fam}

\subsubsection{Definition of the families}

As it is explained in \S \ref{sec:introFanoWH}, quasi-smooth Fano 3-fold weighted hypersurfaces of index 1 are classified and they form 95 families.
According to the classification, the minimum of the weights of an ambient space is $1$.
Hence a family is determined by a quadruple $(a_1, a_2, a_3, a_4)$, which means that the family corresponding to a quadruple $(a_1, a_2, a_3, a_4)$ is the family of weighted hypersurfaces of degree $d = a_1 + a_2 + a_3 + a_4$ in $\mbP (1, a_1, a_2, a_3, a_4)$.
The 95 families are numbered in the lexicographical order on $(d, a_1, a_2, a_3, a_4)$, and each family is referred to as family No.~$\msi$ for $\msi \in \{1, 2, \dots, 95\}$.
Families No.~$1$ and $3$ are the families consisting of quartic 3-folds and degree $6$ hypersurfaces in $\mbP (1, 1, 1, 1, 3)$, respectively, and for any smooth member of these 2 families, K-stability  (and hence the existence of KE metrics) is known.

\begin{Def}
We set
\[
\msI := \{1, 2, \dots, 95\} \setminus \{1, 3\},
\]
and, for $\msi \in \msI$, we denote by $\mcF_{\msi}$ the family consisting of the quasi-smooth members of family No.~$\msi$.
\end{Def}

The main objects of this article is thus the members of $\mcF_{\msi}$ for $\msi \in \msI$.

We set
\[
\msI_1 := \{2, 4, 5, 6, 8, 10, 14\}.
\]
The set $\msI_1$ is characterized as follows: let $X = X_d \subset \mbP (1, a_1, a_2, a_3, a_4)$, $a_1 \le a_2 \le a_3 \le a_4$, be a member of a family $\mcF_{\msi}$ with $\msi \in \msI$. 
Then $\msi \in \msI_1$ if and only if $a_2 = 1$.
The computations of alpha invariants will be done in a relatively systematic way for families $\mcF_{\msi}$ with $\msi \in \msI \setminus \msI_1$ (see Sections~\ref{chap:smpt} and \ref{chap:singpt}), while the computations will be done separately for families $\mcF_{\msi}$ with $\msi \in \msI_1$ (see Section~\ref{chap:exc}).

We explain notation and conventions concerning the main objects of this article.
Let $X = X_d \subset \mbP (1, a_1, a_2, a_3, a_4) =: \mbP$ be a member of a family $\mcF_{\msi}$ with $\msi \in \msI$.
\begin{itemize}
\item Unless otherwise specified, we assume $a_1 \le a_2 \le a_3 \le a_4$. 
\item In many situations (especially when we treat a specific family), we denote by $x, y, z, t, w$ the homogeneous coordinates of $\mbP$ of degree respectively $1, a_1, a_2, a_3, a_4$.
\item We denote by $F$ the polynomial defining $X$ in $\mbP$, which is quasi-homogeneous of degree $d = a_1 + a_2 + a_3 + a_4$.
\item We set $A = -K_X$, which is the positive generator of of $\Cl (X) \cong \mbZ$.
Note that we have
\[
(-K_X)^3 = (A^3) = \frac{d}{a_1 a_2 a_3 a_4} = \frac{a_1 + a_2 + a_3 + a_4}{a_1 a_2 a_3 a_4}.
\]
\end{itemize}

\subsubsection{Definitions of QI and EI centers, and birational (super)rigidity}
\label{sec:defBI}

In this subsection, let
\[
X = X_d \subset \mbP (1, a_1, a_2, a_3, a_4)_{x, y, z, t, w}
\]
be a member of a family $\mcF_{\msi}$ with $\msi \in \msI$, where $a_1 \le a_2 \le a_3 \le a_4$.
We give definitions of QI and EI centers, which are particular singular points on $X$ and are important for understanding birational (super)rigidity of $X$.
For EI centers, we only give an ad hoc definition (see \cite[Section 4.10]{CPR00} and \cite[Section 4.2]{CP17} for more detailed treatments).

\begin{Def} \label{def:BIcenter}
Let $\msp \in X$ be a singular point.
We say that $\msp \in X$ is an {\it EI center} if the upper script $\EI$ is given in the 4th column of Table \ref{table:main}, or equivalently if $\msi$ and $\msp$ belong to one of the following.
\begin{itemize}
\item $\msi = 7$ and $\msp$ is of type $\frac{1}{2} (1, 1, 1)$.
\item $\msi \in \{23, 40, 44, 61, 76\}$ and $\msp = \msp_t$.
\item $\msi \in \{20, 36\}$ and $\msp = \msp_z$.
\end{itemize}

We say that $\msp \in X$ is a {\it QI center} if there are distinct $j$ and $k$ such that $d = 2 a_k + a_j$ and the index of the cyclic quotient singularity $\msp \in X$ coincides with $a_k$.

We say that $\msp \in X$ is a {\it BI center} if it is either an EI center or a QI center.
\end{Def}

\begin{Rem} \label{rem:maxcent}
Let $X$ be a member of $\mcF_{\msi}$ with $\msi \in \msI$.
Then the following are proved in \cite{CP17}.
\begin{enumerate}
\item No smooth point on $X$ is a maximal center.
\item A singular point $\msp \in X$ is a maximal center only if either $\msp$ is a BI center or $X$ is a member of $\mcF_{23}$ and $\msp = \msp_z$ is of type $\frac{1}{3} (1, 1, 2)$.
\end{enumerate}
Note that a BI center $\msp \in X$ is not always a maximal center (see \S \ref{sec:eqQI}, especially Remark \ref{rem:QImaxcent}, for the complete analysis for QI centers).
Note also that the $\frac{1}{3} (1, 1, 2)$ point $\msp_z$ on a member $X$ of $\mcF_{23}$ is not a maximal center if $X$ is general.
However $\msp_z \in X$ can be a maximal center and in that case there is a birational involution of $X$ (called an invisible involution) with center $\msp_z$ (see \cite[Section 4.3]{CP17}).
\end{Rem}

\begin{Def}
We define the subset $\msI_{\BSR} \subset \msI$ as follows: $\msi \in \msI_{\BSR}$ if and only if a member $X$ of $\mcF_{\msi}$ does not admit a BI center.
We then define $\msI_{\BR} = \msI \setminus \msI_{\BSR}$.
\end{Def}

Note that $|\msI_{\BSR}| = 48$ and $|\msI_{\BR}| = 45$.
The following is a more precise version of Theorem \ref{thm:BRWH}.

\begin{Thm}[{\cite{CP17}}]
Let $X$ be a member of $\mcF_{\msi}$ with $\msi \in \msI$.
\begin{enumerate}
\item If $\msi \in \msI_{\BSR}$, then any member of $\mcF_{\msi}$ is birationally superrigid.
\item If $\msi \in \msI_{\BR}$, then any member of $\mcF_{\msi}$ is birationally rigid while its general member is not birationally superrigid.
\end{enumerate}
\end{Thm}

We emphasize that a family $\mcF_{\msi}$, where $\msi \in \msI_{\BR}$, can contain (in fact does contain for most of $\msi \in \msI_{\BR}$) birationally superrigid Fano 3-folds as special members.

\subsubsection{Numerics on weights and degrees} \label{sec:smptnum}

Let 
\[
X = X_d \subset \mbP (1, a_1, a_2, a_3, a_4)_{x, y, z, t, w}
\] 
be a member of $\mcF_{\msi}$.
Throughout the subsection, we assume that $\msi \in \msI \setminus \msI_1$ and that $a_1 \le a_2 \le a_3 \le a_4$.
We collect some elementary numerical results on weights $a_1, \dots, a_4$, the degree $d = a_1 + a_2 + a_3 + a_4$ of the defining polynomial $F = F (x, y, z, t, w)$ of $X$, and the anticanonical degree $(A^3)$ of $X$ which will be repeatedly used in the rest of this article. 

\begin{Lem} \label{lem:smptHLdegwt}
One of the following happens. 
\begin{enumerate}
\item $d = 2 a_4$.
\item $d = 3 a_4$.
\item $d = 2 a_4 + a_j$ for some $j \in \{1, 2, 3\}$.
\end{enumerate}
\end{Lem}

\begin{proof}
We see that either $w^n \in F$ for some $n \ge 2$ or $x^n v \in F$ for some $n \ge 1$ and $v \in \{x, y, z, t\}$ by the quasi-smoothness of $X$.

Suppose $w^n \in F$ for some $n \ge 2$.
Then we have 
\[
d = n a_4 = a_1 + a_2 + a_3 + a_4 < 4 a_4.
\] 
Hence $n = 2, 3$ and we are in case (1) or (2).
Suppose $w^n v \in F$ for some $n \ge 1$ and $v \in \{y, z, t\}$.
Then we have $d = n a_4 + a_j$ and moreover we have
\[
a_4 + a_j < d = a_1 + a_2 + a_3 + a_4 < 3 a_4 + a_j.
\]
This shows $n = 2$, that is, $d = 2 a_4 + a_j$.

If $a_1 = 1$, then the proof is completed.
It remains to show that the case $d = 2 a_4 + 1$ does not take place assuming $a_1 \ge 2$.
Suppose $d = 2 a_4 + 1$ and $a_1 \ge 2$.
Then $w^2 x \in F$ and the singularity of $\msp_w \in X$ is of type $\frac{1}{a_4} (a_1, a_2, a_3)$. 
There exist distinct $i, j \in \{1,2,3\}$ such that $a_i + a_j$ is divisible by $a_4$ since $\msp_w \in X$ is terminal.
We have $a_i + a_j = a_4$ since $0 < a_i + a_j < 2 a_4$.
Let $k \in \{1,2,3\}$ be such that $\{i, j, k\} = \{1,2,3\}$.
Then 
\[
d = a_1 + a_2 +  a_3 + a_4 = a_k + 2 a_4.
\]
Combining this with $d = 2 a_4 + 1$, we have $a_k = 1$.
This is a contradiction since $a_k \ge a_1 \ge 2$.
\end{proof}

\begin{Lem} \label{lem:wtnumerics}
\begin{enumerate}
\item We have $\msi \in \{9, 17\}$ if and only if $d = 3 a_4$ and $a_1 = 1$.
\item We have $a_1 a_2 a_3 (A^3) \le 3$ and the equality holding if and only if $d = 3 a_4$. 
\item If $a_1 < a_2$, then we have $a_1 (A^3) < 1$.
\item If $1 < a_1 < a_2$, then $a_1 a_3 (A^3) \le 1$.
\item If $a_1 < a_2$ and $d > 2 a_4$, then $a_1 a_4 (A^3) \le 2$.
\item If $d$ is divisible by $a_4$ and $\msi \notin \{9, 17\}$, then $a_2 a_3 (A^3) \le 2$.
\item If $d$ is not divisible by $a_4$ and $a_1 \ge 2$, then $a_2 a_4 (A^3) \le 2$.
\end{enumerate}
\end{Lem}

\begin{proof}
We prove (1).
The ``only if" part is obvious.
Suppose $d = 3 a_4$ and $a_1 = 1$.
Then we have $2 a_4 = 1 + a_2 + a_3$.
This implies $a_2 = a_4 - 1$ and $a_3 = a_4$ since $a_2 \le a_3 \le a_4$.
Then, by setting $a = a_2 \ge 2$, $X$ is a weighted hypersurface in $\mbP (1, 1, a, a+1, a+1)$ of degree $3 (a+1)$.
Suppose $\msp_z \notin X$.
Then some power of $z$ is contained in $F$ and this implies that $3 (a+1)$ is divisible by $a$.
In particular, we have $a = 3$ and this case corresponds to $\msi = 17$.
Suppose $\msp_z \notin X$, then either $3 (a+1) \equiv 1 \pmod{a}$ or $3 (a+1) \equiv a+1 \pmod{a}$ by the quasi-smoothness of $X$.
In both cases, we have $a = 2$, ad hence $\msi = 9$.
Thus (1) is proved.

The assertion (2) follows immediately since we have
\[
a_1 a_2 a_3 (A^3) = \frac{d}{a_4} \le 3
\]
and $d \le 3 a_4$ by Lemma \ref{lem:smptHLdegwt}.

We prove (3).
Note that $2 \le a_2 \le a_3 \le a_4$.
Note also that $a_4 > a_2$ because otherwise $X$ has non-isolated singularity along $L_{xy}$ which is impossible. 
In particular, we have $a_1 + \cdots + a_4 < 4 a_4$ and $a_2 a_3 \ge 4$, and we have
\[
a_1 (A^3) = \frac{a_1 (a_1 + a_2 + a_3 + a_4)}{a_1 a_2 a_3 a_4} 
< \frac{4}{a_2 a_3}
\le 1,
\]
which proves (3).

We prove (4).
We have $a_2 \ge 3$ since $a_2 > a_1 > 1$ and thus
\[
a_1 a_3 (A^3) = \frac{d}{a_2 a_4} \le \frac{3}{a_2} \le 1.
\]

We prove (5).
We have $d > 2 a_4$ by assumption.
Then, by Lemma \ref{lem:smptHLdegwt}, we have $d = 2 a_4 + a_j$ for some $j \in \{1,2,3,4\}$, and combining this with $d = a_1 + a_2 + a_3 + a_4$, we have
\[
a_4 = a_1 + a_2 + a_3 - a_j \le a_2 + a_3.
\]
If $a_1 > 1$, then we have $a_2, a_3 \ge 3$ and thus
\[
a_1 a_4 (A^3) = \frac{a_1 + a_2 + a_3 + a_4}{a_2 a_3} < \frac{3 a_2 + 2 a_3}{a_2 a_3} = \frac{3}{a_3} + \frac{2}{a_2} \le \frac{5}{3}.
\]
Suppose $a_1 = 1$.
In this case $2 \le a_2 \le a_3$.
If $a_3 \ge 3$, then 
\[
a_4 (A^3) = \frac{1 + a_2 + a_3 + a_4}{a_2 a_3} \le \frac{1 + 2 a_2 + 2 a_3}{a_2 a_3} = \frac{1}{a_2 a_3} + \frac{2}{a_3} + \frac{2}{a_2} \le \frac{11}{6}.
\]
Suppose $a_3 = 2$, that is, $a_2 = a_3 = 2$.
Then we have $a_4 = 3$ and $d = 8$ since $d = 5 + a_4 > 2 a_4$ and $a_4$ is odd.
In this case we have $a_4 (A^3) = 2$.
This proves (5).

We prove (6).
By Lemma \ref{lem:smptHLdegwt} and (1), either $d = 2 a_4$ or $d = 3 a_4$ and $a_1 \ge 2$.
If $d = 2 a_4$ (resp.\ $d = 3 a_4$ and $a_1 \ge 2$), then
\[
a_2 a_3 (A^3) = \frac{2}{a_1} \le 2 \quad (\text{resp.\ } a_2 a_3 (A^3) = \frac{3}{a_1} \le 2).
\]
This proves (6).

We prove (7).
By Lemma \ref{lem:smptHLdegwt}, we have $d = 2 a_4 + a_j$ for some $j \in \{1, 2, 3\}$.
Then we have $a_4 = a_1 + a_2 + a_3 - a_j \le a_2 + a_3$.
If $a_1 \ge 3$, then
\[
a_2 a_4 (A^3) = \frac{a_1 + a_2 + a_3 + a_4}{a_1 a_3} \le \frac{a_1 + 4 a_3}{a_1 a_3} = \frac{1}{a_3} + \frac{4}{a_1} \le \frac{5}{3}.
\]
We continue the proof assuming $a_1 = 2$.
If in addition $a_2 < a_3$, then
\[
a_2 a_4 (A^3) = \frac{2 + a_2 + a_3 + a_4}{2 a_3} \le \frac{2 + 2 a_2 + 2 a_3}{2 a_3} \le \frac{4 a_3}{2 a_3} = 2.
\]
We continue the proof assuming $a_1 = 2$ and $a_2 = a_3$.
In this case, by setting $a = a_2 = a_3$ and $b = a_4$, $X$ is a weighted hypersurface of degree $d$ in $\mbP (1, 2, a, a, b)$ and either $d = 2 b + 2$ or $d = 2 b + a$.
If $d = 2 b + 2$, then $b = 2 a$ but this is impossible since $X$ has only terminal singularities.
Hence $d = 2 b + a$.
In this case $b = a + 2$ and $d = 3 a + 4$.
By the quasi-smoothness of $X$, we see that $d = 3 a + 4$ is divisible by $a$.
This implies that $a \in \{2, 4\}$.
This is impossible since $X$ has only terminal singularities.
Therefore (7) is proved.
\end{proof}

\subsubsection{How to compute alpha invariants?}

Let $X$ be a member of a family $\mcF_{\msi}$ with $\msi \in \msI$.
For the proof of Theorem \ref{mainthm}, it is necessary to show $\alpha_{\msp} (X) \ge 1/2$ for any point $\msp \in X$.
Let $\msp \in X$ be a point.
We briefly explain the most typical method of bounding $\alpha_{\msp} (X)$ from below, which goes as follows.
\begin{enumerate}
\item Choose and fix a divisor $S$ on $X$ which vanishes at $\msp$ to a relatively large (orbifold) multiplicity $m = \omult_{\msp} (S) > 0$.
In some cases $S = H_x$ (when $\msp \in H_x$) and in other cases $S$ is the quasi-tangent divisor of $X$ at $\msp$.
Let $a$ be the positive integer such that $S \sim a A$.
\item Let $D \in |A|_{\mbQ}$ be an irreducible $\mbQ$-divisor other than $\frac{1}{a} S$.
Then $D \cdot S$ is an effective $1$-cycle on $X$.
\item Find a $\mbQ$-divisor $T \in |e A|_{\mbQ}$ for some $e \in \mbZ_{> 0}$ such that $\mult_{\msp} (T) \ge 1$ and $\Supp (T)$ does not contain any component of $D \cdot S$.
We will find such a $\mbQ$-divisor $T$ by considering $\msp$-isolating set or class which will be explained in \S \ref{sec:isol}. 
\item Let $q = q_{\msp}$ be the quotient morphism of $\msp \in X$ and $\check{\msp}$ be the preimage of $\msp$ via $q$.
By the above choices, $\Supp (q^*D) \cap \Supp (q^*S) \cap \Supp (q^*T)$ is a finite set of points including $\check{\msp}$, and hence the local intersection number $(q^*D \cdot q^*S \cdot q^*T)_{\check{\msp}}$ is defined (see \S \ref{sec:intnumber}).
Then, we have the inequalities
\[
m \omult_{\msp} (D) \le (q_{\msp}^*D \cdot q_{\msp}^*S \cdot q_{\msp}^*T)_{\check{\msp}} \le r (D \cdot S \cdot T) = r a e (A^3),
\]
where $r$ is the index of the cyclic quotient singularity $\msp \in X$. Note that $q$ is the identity morphism and $r = 1$ when $\msp \in X$ is a smooth point.
By Lemma \ref{lem:multlct} which will be explained below, we have
\[
\lct_{\msp} (X;D) \ge \frac{r a e (A^3)}{m}.
\]
for any $D$ as in (2).
\item As a conclusion, we have
\[
\alpha_{\msp} (X) \ge \min \left\{ \lct_{\msp} (X; S), \ \frac{r a e (A^3)}{m} \right\}.
\]
\item It remains to bound $\lct_{\msp} (X;S)$ from below.
This is easy when $S$ is quasi-smooth at $\msp$ because in that case we have $\lct_{\msp} (X;S) = 1$.
The computation gets involved when $S$ is the quasi-tangent divisor, but will be done by considering suitable weighted blowups which will be explained in \S \ref{sec:compwbl}.  
\end{enumerate}

We need to consider variants of the above explained method, or other methods especially for points in special positions.
These will be explained in Section~\ref{chap:methods}.

\section{Methods of computing log canonical thresholds}
\label{chap:methods}

\subsection{Auxiliary results}

\subsubsection{Some results on multiplicities and log canonicity}
\label{sec:intnumber}

Let $V$ be an $n$-dimensional variety.
For effective Cartier divisors $D_1, \dots, D_n$ on $V$ and a point $\msp \in V$ which is an isolated component of $\Supp (D_1) \cap \cdots \cap \Supp (D_n)$, the {\it intersection multiplicity} 
\[
i (\msp, D_1 \cdots D_n;V)
\]
is defined (see \cite[Example 7.1.10]{Fulton}).
Suppose that $V$ is $\mbQ$-factorial.
Then this definition is naturally generalized to effective $\mbQ$-divisors $D_1, \dots, D_n$ as follows:
\[ 
i (\msp, D_1, \cdots, D_n;V) := \frac{1}{d^n} i (\msp, d D_1, \cdots, d D_n;V),
\]
where $d$ is a positive integer such that $d D_i$ is a Cartier divisor for any $i$.
In this paper, we set
\[
(D_1 \cdots D_n)_{\msp} := i (\msp, D_1 \cdots D_n;V)
\]
and call it the {\it local intersection number} of $D_1, \dots, D_n$ at $\msp$.

\begin{Rem}
If $\msp \in V$ is a smooth point, $D_1, \dots, D_n$ are effective divisors defined by $f_1, \dots, f_n \in \mcO_{V, \msp}$ around $\msp$, and $\msp$ is an isolated component of $\Supp (D_1) \cap \cdots \cap \Supp (D_n)$, then 
\[
(D_1 \cdots D_n)_{\msp} = \dim_{\mbC} \mcO_{V, P}/(f_1, \dots, f_n).
\]

If $X \subset \mbP (a_0, \dots, a_N)$ is an $n$-dimensional subvariety which is quasi-smooth at $\msp = \msp_{x_i} \in V$, $D_1 = (G_1 = 0)_X, \dots, D_n = (G_n = 0)_X$ are effective Weil divisors such that $\msp$ is an isolated component of $D_1 \cap \cdots \cap D_n$, where $G_i = G_i (x_0, \dots, x_N)$ is a quasi-homogeneous polynomial of degree $d_i$, then
\[ 
(D_1 \cdots D_n)_{\msp} 
= \frac{1}{a_i} (\rho^*D_1 \cdots \rho^* D_n)_{\breve{\msp}}
= \frac{1}{a_i} \dim_{\mbC} \mcO_{\breve{U}_{\msp}, \breve{\msp}}/(g_1, \dots, g_n),
\]
where $\rho = \rho_{\msp} \colon \breve{U}_{\msp} \to U_{\msp} := X \cap \mcU_{\msp}$ is the orbifold chart with $\breve{\msp} \in \breve{U}_{\msp}$ the preimage of $\msp$ and $g_i = G (\breve{x}_0, \dots, 1, \dots, \breve{x}_N)$ with $\breve{x}_j = x_j/x_i^{a_j/a_i}$ for $j \ne i$.
\end{Rem}

We will frequently use the following property of local intersection numbers.
Let $D_1, \dots, D_n$ be effective $\mbQ$-divisors on $X$ and $\msp \in X$ be a smooth point.
If $\msp$ is an isolated component of $\Supp (D_1) \cap \cdots \cap \Supp (D_n)$, then
\[
(D_1 \cdot \ldots \cdot D_n)_{\msp} \ge \prod_{i=1}^n \mult_{\msp} (D_i).
\]
We refer readers to \cite[Corollary 12.4]{Fulton} for a proof.
Although the following results are well known to experts, we include their proofs for readers' convenience.

\begin{Lem} \label{lem:multlct}
Let $\msp \in X$ be either a germ of a smooth variety or a germ of a cyclic quotient singular point and let $D$ be an effective $\mbQ$-divisor on $X$.
Then the inequality
\[
\frac{1}{\omult_{\msp} (D)} \le \lct_{\msp} (X,D)
\]
holds.
\end{Lem}

\begin{proof}
Let $q = q_{\msp} \colon \check{X} \to X$ be the quotient morphism of $\msp \in X$, which is \'{e}tale in codimension $1$, and let $\check{\msp} \in \check{X}$ be the preimage of $\msp$.
By \cite[20.4 Corollary]{FA}, we have $\lct_{\msp} (X;D) = \lct_{\check{\msp}} (\check{X};q^*D)$.
Note that $\check{\msp} \in \check{X}$ is smooth.
Hence, by \cite[8.10 Lemma]{Kol}, we have 
\[
\frac{1}{\omult_{\msp} (D)} = \frac{1}{\mult_{\check{\msp}} (q^*D)} \le \lct_{\check{\msp}} (\check{X};q^*D),
\]
and the proof is completed.
\end{proof}

\begin{Lem}[{$2n$-inequality, cf.\ \cite[Corollary 3.5]{Corti}}] \label{lem:2nineq}
Let $\msp \in X$ be a germ of a smooth $3$-fold, $D$ an effective $\mbQ$-divisor on $X$, $n > 0$ a rational number and let $\varphi \colon Y \to X$ be the blowup of $X$ at $\msp$ with exceptional divisor $E$.
If $(X, \frac{1}{n} D)$ is not canonical at $\msp$, then there exists a line $L \subset E \cong \mbP^2$ with the following property.
\begin{itemize}
\item For any prime divisor $T$ on $X$ such that $T$ is smooth at $\msp$ and that its proper transform $\tilde{T}$ contains $L$, we have $\mult_{\msp} (D|_T) > 2 n$.
\end{itemize}
\end{Lem}

\begin{proof}
We set $m = \mult_{\msp} (D)$.
By \cite[Corollary 3.5]{Corti}, one of the following holds.
\begin{enumerate}
\item $m > 2 n$.
\item There is a line $L \subset E$ such that the pair
\[
\left(Y, \left(\frac{m}{n} - 1\right)E + \frac{1}{n} \tilde{D} \right)
\]
is not log canonical at the generic point of $L$.
\end{enumerate}
Note that in \cite[Corollary 3.5]{Corti} the boundary is a movable linear system $\mcH$ but the same argument applies if we replace $\mcH$ by an effective $\mbQ$-divisor $D$.
We may assume $m \le 2n$ because otherwise $\mult_{\msp} (D|_T) > 2 n$ for any prime divisor $T$ which is smooth at $\msp$ and the assertion follows by choosing any line on $E$.
Thus the option (2) takes place.
Let $T$ be a prime divisor on $X$ such that $T$ is smooth at $\msp$ and $\tilde{T} \supset L$.
We have
\[
K_Y + \left( \frac{m}{n} - 1 \right) E + \frac{1}{n} \tilde{D} + \tilde{T} = \varphi^* \left( K_X + \frac{1}{n} D + T \right).
\]
Note that $E|_{\tilde{T}} = L$ and we can write $\tilde{D}|_{\tilde{T}} = \alpha L + G$, where $\alpha \ge 0$ is a rational number and $G$ is an effective $\mbQ$-divisor on $\tilde{T}$.
Thus, by restricting the above equation to $\tilde{T}$, we have
\[
K_{\tilde{T}} + \left( \frac{m}{n} - 1 + \alpha \right) L + G = \varphi^* \left(K_T + \frac{1}{n} D|_T \right),
\]
and the pair 
\[
\left( \tilde{T}, \left( \frac{m}{n} - 1 + \alpha \right) L + G \right)
\]
is not log canonical at the generic point of $L$.
This implies $\frac{m}{n} - 1 + \alpha > 1$ and we have 
\[
\frac{1}{n} \mult_{\msp} (D|_T) = \left( \frac{m}{n} - 1 + \alpha \right) + 1 > 2.
\]
Thus $\mult_{\msp} (D|_T) > 2 n$ and the proof is completed.
\end{proof}

\begin{Lem} \label{lem:lctP2cubic}
Let $D \in \left| \mcO_{\mbP^2} (3) \right|$ be a divisor on $\mbP^2$ which is not a triple line.
Then $\lct (\mbP^2; D) \ge 1/2$.
\end{Lem}

\begin{proof}
We have the following possibilities for $D$.
\begin{enumerate}
\item $D$ is irreducible and reduced.
\item $D = Q + L$, where $Q$ is an irreducible conic and $L$ is a line.
\item $D = L_1 + 2 L_2$, where $L_1, L_2$ are distinct lines.
\item $D = L_1 + L_2 + L_3$, where $L_1, L_2, L_3$ are mutually distinct lines.
\end{enumerate}
If we are in one of the cases (1), (2) and (3), then $\mult_{\msp} (D) \le 2$ for any point $\msp \in D$ and thus $(\mbP^2, \frac{1}{2} D)$ is log canonical.
If we are in case (3), then it is obvious that the pair $(\mbP^2, \frac{1}{2} D) = (\mbP^2, \frac{1}{2} L_1 + L_2)$ is log canonical.
\end{proof}

\begin{Lem} \label{lem:singnoncanbd}
Let $X$ be a Fano $3$-fold of Picard number one and let $\msp \in X$ be a cyclic quotient terminal singular point (which is not a smooth point).
If $\msp \in X$ is not a maximal center, then there is at most one irreducible $\mbQ$-divisor $D \in \left| -K_X \right|_{\mbQ}$ such that $(X, D)$ is not canonical at $\msp$.
\end{Lem}

\begin{proof}
Suppose that there are two distinct irreducible $\mbQ$-divisors $D_i \sim_{\mbQ} - K_X$ such that $(X, D_i)$ is not canonical at $\msp$ for $i = 1, 2$.
Let $r > 1$ be the index of the singularity $\msp \in X$ and let $\varphi \colon Y \to X$ be the Kawamata blowup at $\msp$ with exceptional divisor $E$.
By \cite{Kawamata}, we have $\ord_E (D_i) > 1/r$.
Take a positive integer $n$ such that $n D_1, n D_2$ are both integral and $n D_1 \sim n D_2$.
Then the pencil $\mcM \sim - n K_X$ generated by $n D_1$ and $n D_2$ is a movable linear system and we have $\ord_E (\mcM) \ge n/r$.
It follows that the pair $(X, \frac{1}{n} \mcM)$ is not canonical at $\msp$.
This is a contradiction since $\msp \in X$ is not a maximal center.
\end{proof}

\begin{Lem} \label{lem:qtangdivncan}
Let 
\[
X = X_d \subset \mbP (1, b_1, b_2, b_3, b_4)_{x, y_1, y_2, y_3, y_4}
\] 
be a member of a family $\mcF_{\msi}$ with $\msi \in \msI$.
Let $i \in \{1, 2, 3, 4\}$ be such that $b_i > 1$ and $\msp := \msp_{y_i} \in X$.
If $H_x$ is the quasi-tangent divisor of $X$ at $\msp$, then the pair $(X, H_x)$ is not canonical at $\msp$.
\end{Lem}

\begin{proof}
Note that the point $\msp \in X$ is of type $\frac{1}{b_i} (b_j, b_k, b_l)$, where $\{i, j, k, l\} = \{1, 2, 3, 4\}$, and it is a terminal singularity.
Let $\varphi \colon Y \to X$ be the Kawamata blowup with exceptional divisor $E$.
Since $H_x$ is the quasi-tangent divisor of $X$ at $\msp$, we have
\[
\ord_E (H_x) > \frac{1}{b_i}.
\]
Combining this with 
\[
K_Y = \varphi^*K_X + \frac{1}{b_i} E,
\]
we see that the discrepancy of the pair $(X, H_x)$ along $E$ is negative.
This completes the proof.
\end{proof}

\subsubsection{Some results on singularities of weighted hypersurfaces}

\begin{Lem} \label{lem:normalqhyp}
Let $X$ be a quasi-smooth weighted hypersurface in $\mbP (b_0, \dots, b_4)$.
Assume that $\mbP (b_0, \dots, b_4)$ is well-formed and $X$ has at most isolated singularities.
Then any quasi-hyperplane section on $X$ is a normal surface.
\end{Lem}

\begin{proof}
Let $x_0, \dots, x_4$ be the homogeneous coordinates of $\mbP = \mbP (b_0, \dots, b_4)$ of degree $b_0, \dots, b_4$, respectively.
Let $F = F (x_0,\dots,x_4)$ be the defining polynomial of $X$ and let $S$ be a quasi-hyperplane section on $X$.
After replacing homogeneous coordinates, we may assume $S = (x_4 = 0)_X = (x_4 = F = 0) \subset \mbP$.
It is enough to show that the singular locus $\Sing (S)$ of $S$ is a finite set of points.

We write $F = x_4 G + \bar{F}$, where $G = G (x_0, \dots, x_4)$ and $\bar{F} = \bar{F} (x_0, \dots, x_3)$ are quasi-homogeneous polynomials.
We set $\bar{\mbP} = \mbP (b_0, \dots, b_3)$.
We claim that $\bar{\mbP}$ is well-formed.
Suppose it is not.
Then $\Sing (\mbP)$ contains a $2$-dimensional stratum.
We have $\Sing (X) = \Sing (\mbP) \cap X$ since a quasi-smooth weighted hypersurface is well-formed (\cite[Theorem 6.17]{IF00}).
It follows that $\Sing (X)$ cannot be a finite set of points.
This is a contradiction and the claim is proved.
The surface $S$ is identified with the hypersurface $(\bar{F} = 0) \subset \bar{\mbP}$ and we have 
\[
\Sing (S) = (S \setminus \QSm (S)) \cup (\Sing (\bar{\mbP}) \cap S).
\]

We claim that $\Sing (\bar{\mbP}) \cap S$ is a finite set of  points.
Suppose not.
Then $\Sing (\bar{\mbP}) \cap S$ contains a curve and so does $\Sing (\mbP) \cap S$.
In particular $\Sing (X) = \Sing (\mbP) \cap X$ contains a curve.
This is impossible since $\Sing (X)$ is a finite set of points, and the claim is proved.

It remains to show that the closed subset $\Sigma := S \setminus \QSm (S) \subset \mbP$ is a finite set of points.
Let $\Pi \colon \mbA^5 \setminus \{o\} \to \mbP$ be the natural quotient morphism.
Then $\Sigma = \Pi (\Sing (C^*_S))$, where $C^*_S \subset \mbA^5 \setminus \{o\}$ is the punctured affine quasi-cone of $S$.
We have $\Sing (C_X) \cap C_S = \Sing (C_S) \cap (G = 0)$.
By the quasi-smoothness of $X$, we have $\Sing (C_X) = \{o\} \subset \mbA^5$.
This implies $\Sigma \cap (G = 0) = \emptyset$. 
Since $(G = 0)$ is an ample divisor on $\mbP$, we see that $\Sigma$ is a finite set of points.
This completes the proof.
\end{proof}

\begin{Lem} \label{lem:qsminvhypsec}
Let $S$ be a weighted hypersurface in $\mbP (b_0, b_1, b_2, b_3)$ and let $T \subset \mbP (b_0, b_1, b_2, b_3)$ be a quasi-hyperplane.
If the scheme-theoretic intersection $S \cap T$ is quasi-smooth at a point $\msp$, then $S$ is quasi-smooth at $\msp$.
\end{Lem}

\begin{proof}
Let $x_0, x_1, x_2, x_3$ be the homogeneous coordinates of $\mbP = \mbP (b_0, b_1, b_2, b_3)$ of degree $b_0, b_1, b_2, b_3$, respectively, and let $F= F (x_0, x_1, x_2, x_3)$ be the defining polynomial of $S$.
We may assume $T = H_{x_3} \subset \mbP$ and we write $F = x_3 G + \bar{F}$, where $G = G (x_0, x_1, x_2, x_3)$ and $\bar{F} = \bar{F} (x_0, x_1, x_2)$ are quasi-homogeneous polynomials. 
Then $S \cap T$ is the closed subscheme in $\mbP (b_0,b_1,b_2,b_3)$ defined by $x_3 = \bar{F} = 0$.
By the quasi-smoothness of $S \cap T$ at $\msp$, there exists $i \in \{0, 1, 2\}$ such that
\[
\frac{\prt \bar{F}}{\prt x_i} (\msp) \ne 0.
\]
It follows that
\[
\frac{\prt F}{\prt x_i} (\msp) = \frac{\prt \bar{F}}{\prt x_i} (\msp) \ne 0
\]
since $\msp \in H_{x_3}$.
Thus $S$ is quasi-smooth at $\msp$.
\end{proof}

\begin{Lem} \label{lem:pltsurfpair}
Let $S$ be a normal weighted hypersurface in a well formed weighted projective $3$-space $\mbP (b_0,\dots, b_3)$ and $T \subset \mbP (b_0,\dots, b_3)$ a quasi-hyperplane such that $T \ne S$.
Let $\Gamma$ be an irreducible component of $S \cap T$ and we assume that 
\[
T|_S = \Gamma + \Delta,
\] 
where $\Delta$ is an effective divisor on $S$ such that $\Gamma \not\subset \Supp (\Delta)$.
If $\Gamma$ is a smooth weighted complete intersection curve and $S$ is quasi-smooth at each point of $\Gamma \cap \Supp (\Delta)$, then $S$ is quasi-smooth along $\Gamma$ and the pair $(S, \Gamma)$ is plt along $\Gamma$.
\end{Lem}

\begin{proof}
We set $\Xi = \Gamma \cap \Supp (\Delta)$.
By \cite[Theorem 12.1]{IF00}, $\Gamma$ is quasi-smooth.
We have $(S \cap T) \setminus \Xi = \Gamma \setminus \Xi$.
It follows that $S \cap T$ is quasi-smooth along $\Gamma \setminus \Xi$.
By Lemma \ref{lem:qsminvhypsec}, $S$ is quasi-smooth along $\Gamma \setminus \Xi$.
Therefore $S$ is quasi-smooth along $\Gamma$.

For $i = 0,1,2,3$, let $S_i = (x_i \ne 0) \cap S$ be the standard open set of $S$ and let $\rho_i \colon \breve{S}_i \to S_i$ be the orbifold chart.
Note that $\rho_i$ is a finite surjective morphism of degree $b_i$ which is \'etale in codimension $1$.
By the quasi-smoothness of $S$, the affine varieties $\breve{S}_i$ and $\rho_i^*(\Gamma \cap S_i)$ are smooth.
Hence the pair $(\breve{S}_i, \rho_i^*(\Gamma \cap S_i))$ is plt along $\rho_i^*(\Gamma \cap S_i)$.
By \cite[Corollary 20.4]{FA}, the pair $(S_i, \Gamma \cap S_i)$ is plt along $\Gamma \cap S_i$.
This completes the proof.
\end{proof}

\begin{Rem} \label{rem:compselfint}
Let $S$, $T$, and $\Gamma$ be as in Lemma \ref{lem:pltsurfpair}.
We assume in addition that $\Gamma$ is rational, i.e.\ $\Gamma \cong \mbP^1$.
Let $\Sing_{\Gamma} (S) = \{\msp_1,\dots,\msp_n\}$ be the set of singular points of $S$ along $\Gamma$ and let $m_i$ be the index of the quotient singular point $\msp_i \in S$.
Then, since the pair $(S, \Gamma)$ is plt along $\Gamma$, we can apply \cite[Proposition 16.6]{FA} and we have
\[
(K_S + \Gamma)|_{\Gamma} = K_{\Gamma} + \sum_{i=1}^n \frac{m_i-1}{m_i} \msp_i.
\]
Thus we have
\[
(\Gamma^2)_S = - (K_S \cdot \Gamma)_S -2 + \sum_{i=1}^n \frac{m_i-1}{m_i}.
\]
\end{Rem}

\subsubsection{Isolating set and class}
\label{sec:isol}

We recall the definitions of isolating set and class which are introduced by Corti, Pukhlikov and Reid \cite{CPR00} as well as their basic properties.

Let $V$ be a normal projective variety embedded in a weighted projective space $\mbP = \mbP (a_0, \dots, a_N)$ with homogeneous coordinates $x_0, \dots, x_N$ with $\deg x_i = a_i$, and let $A$ be a Weil divisor on $V$ such that $\mcO_V (A) \cong \mcO_V (1)$.
We do not assume that $a_0 \le \cdots \le a_N$.

\begin{Def}
Let $\msp \in V$ be a point.
We say that a set $\{g_1, \dots, g_m\}$ of quasi-homogeneous polynomials $g_1, \dots, g_m \in \mbC [x_0, \dots, x_N]$ {\it isolates} $\msp$ or is a $\msp$-{\it isolating set} if $\msp$ is an isolated component of the set
\[
(g_1 = \cdots = g_m = 0) \cap V.
\]
\end{Def}

\begin{Def}
Let $\msp \in V$ be a smooth point and let $L$ be a Weil divisor class on $V$.
For positive integers $k$ and $l$, we define $|\mcI_{\msp}^k (l L)|$ to be the linear subsystem of $|l L|$ consisting of divisors vanishing at $\msp$ with multiplicity at least $k$.
We say that $L$ {\it isolates} $\msp$ or is a $\msp$-{\it isolating class} if $\msp$ is an isolated component of the base locus of $|\mcI^k_{\msp} (kL)|$.
\end{Def}

\begin{Lem}[{\cite[Lemma 5.6.4]{CPR00}}]
Let $\msp \in V$ be a smooth point.
If $\{g_1, \dots, g_m\}$ is a $\msp$-isolating class, then $l A$ is a $\msp$-isolating class, where
\[
l = \max \{\, \deg g_i \mid i = 1, 2, \dots, m \,\}.
\]
\end{Lem}

\begin{Lem} \label{lem:isolexT}
Let $\msp \in V$ be a point, $Z_1, \dots, Z_k$ irreducible closed subsets of $V$ such that $\dim Z_i > 0$ for any $i$, and let $g_1, \dots, g_n \in \mbC [x_0, \dots, x_N]$ be quasi-homogeneous polynomials.
Suppose that $V$ is quasi-smooth at $\msp$ and that $\{g_1, \dots, g_n\}$ isolates $\msp$.
We set $G_i = (g_i = 0)_V$ and we set
\[
\mu := \min \left\{\, \frac{\omult_{\msp} (G_i)}{\deg g_i} \; \middle| \; i = 1, \dots, n \,\right\}.
\]
Then there exists an effective $\mbQ$-divisor $T \sim_{\mbQ} A$ such that $\omult_{\msp} (T) \ge \mu$ and $\Supp (T)$ does not contain any $Z_i$.
\end{Lem}

\begin{proof}
Let $d$ be the least common multiple of $\deg g_1, \dots, \deg g_n$ and we set $e_i = d/\deg g_i$.
Consider the linear system $\Lambda \subset |d A|$ on $V$ generated by $g_1^{e_1}, \dots, g_n^{e_n}$.
We see that $\msp$ is an isolating component of $\Bs \Lambda$ since $\{g_1, \dots, g_n\}$ isolates $\msp$.
Hence a general $D \in \Lambda$ does not contain any $Z_i$ in its support.
Moreover, for any $D \in \Lambda$, we have 
\[
\omult_{\msp} (D) \ge \min \{ \, e_i \omult_{\msp} (D_i) \mid i = 1, \dots, n \,\} = d \mu.
\]
Thus the assertion follows by setting $T = \frac{1}{d} D \sim_{\mbQ} A$ for a general $D \in \Lambda$.
\end{proof}

\begin{Rem} \label{rem:isolT}
Lemma \ref{lem:isolexT} will be frequently applied in the following way:
under the same notation and assumptions as in Lemma \ref{lem:isolexT}, there exists an effective $\mbQ$-divisor $T \sim_{\mbQ} e A$, where 
\[
e = \max \{ \, \deg g_i \mid i = 1, \dots, n \,\},
\] 
such that $\omult_{\msp} (T) \ge 1$ and $\Supp (T)$ does not contain any $Z_i$.  
\end{Rem}

\begin{Lem} \label{lem:isolclass}
Let $X = X_d \subset \mbP (1, a_1, \dots, a_4)_{x, y, z, t, w}$ be a member of a family $\mcF_{\msi}$ with $\msi \in \msI$, where we assume that $a_1 \le a_2 \le a_3 \le a_4$, and let $\msp \in H_x \setminus L_{xy}$.
Then $a_1 a_4 A$ isolates $\msp$.
If $w^k$ appears in the defining polynomial of $X$ with nonzero coefficient, then $a_1 a_3 A$ isolates $\msp$.
\end{Lem}

\begin{proof}
We can write $\msp = (0\!:\!1\!:\!\alpha_2\!:\!\alpha_3\!:\!\alpha_4)$ for some $\alpha_2, \alpha_3, \alpha_4 \in \mbC$.
Then it is easy to see that the set
\[
\{ x,  z^{a_1} - \alpha_2^{a_1} y^{a_2}, t^{a_1} - \alpha_3^{a_3} y^{a_3}, w^{a_1} - \alpha_4^{a_4} y^{a_4} \}
\]
isolates $\msp$, and thus $a_1 a_4 A$ isolates $\msp$.

Suppose that $w^k$ appears in the defining polynomial of $X$.
Then the natural projection $\mbP (1, a_1, \dots, a_4) \ratmap \mbP (1, a_1, a_2, a_3)$ restricts to a finite morphism $\pi \colon X \to \mbP (1, a_1, a_2, a_3)$. 
The common zero locus (in $X$) of the sections contained in the set
\[
\{ x,  z^{a_1} - \alpha_2^{a_1} y^{a_2}, t^{a_1} - \alpha_3^{a_3} y^{a_3} \}
\]
coincides with the set $\pi^{-1} (\msq)$, where $\msq = (0\!:\!1\!:\!\alpha_2\!:\!\alpha_3) \in \mbP (1, a_1, a_2, a_3)$.
It follows that the above set isolates $\msp$ since $\pi^{-1} (\msq)$ is a finite set containing $\msp$.
Thus $a_1 a_3 A$ isolates $\msp$.
\end{proof}

\subsection{Methods}

\subsubsection{Computations by intersecting two divisors}

We recall methods of computing LCTs and consider their generalizations for some of them.

\begin{Lem}[{cf.\ \cite[Lemma 2.5]{KOW18}}] \label{lem:exclL}
Let $X$ be a normal projective $\mbQ$-factorial $3$-fold with nef and big anticanonical divisor, and let $\msp \in X$ be either  a smooth point or a terminal quotient singular point of index $r$ (Below we set $r = 1$ when $\msp \in X$ is a smooth point).
Suppose that there are prime divisors $S \sim_{\mbQ} - a K_X$ and $T \sim_{\mbQ} - b K_X$ with $a, b \in \mbQ$ such that $S \cap T$ is irreducible and $q^* S \cdot q^* T = m \check{\Gamma}$, where $q = q_{\msp} \colon \check{U} \to U$ is the quotient morphism of an analytic neighborhood $\msp \in U$ of $\msp \in X$, $\check{\msp}$ is the preimage of $\msp$ via $q$, $m$ is a positive integer and $\check{\Gamma}$ is an irreducible and reduced curve on $\check{U}$. 
Then, we have
\[
\alpha_{\msp} (X) \ge \min \left\{ \lct_{\msp} (X; \tfrac{1}{a} S), \ \frac{b}{m \mult_{\check{\msp}} (\check{\Gamma})}, \ \frac{1}{r a b (-K_X)^3} \right\}.
\] 
\end{Lem}

\begin{proof}
We set
\[
c :=  \min \left\{ \lct_{\msp} (X; \tfrac{1}{a} S), \ \frac{b}{m \mult_{\check{\msp}} (\check{\Gamma})}, \ \frac{1}{r a b (-K_X)^3} \right\}.
\]
We will derive a contradiction assuming $\alpha_{\msp} (X) < c$.
By the assumption, there is an irreducible $\mbQ$-divisor $D \in \left| - K_X \right|_{\mbQ}$ such that $(X, c D)$ is not log canonical at $\msp$.
Then the pair $(\check{U}, c \rho^*D)$ is not log canonical at $\check{\msp}$ and we have
\begin{equation} \label{eq:exclL-1}
\mult_{\check{\msp}} (q^*D) > \frac{1}{c}.
\end{equation}
Since $q^*S \cdot q^*T = m \check{\Gamma}$ and $S \cap T$ is irreducible, we have $S \cdot T = m \Gamma$, where $\Gamma$ is an irreducible and reduced curve such that $\check{\Gamma} = q^*\Gamma$.
We have
\begin{equation} \label{eq:exclL-2}
(-K_X \cdot \Gamma) = \frac{1}{m} (-K_X \cdot S \cdot T) = \frac{a b (-K_X)^3}{m}.
\end{equation}
This in particular implies
\begin{equation} \label{eq:exclL-3}
(T \cdot \Gamma) = b (-K_X \cdot \Gamma) = \frac{a b^2 (-K_X)^3}{m}.
\end{equation}

We have $\Supp (D) \ne S$ since $\lct_{\msp} (X;S) \ge c$, and thus $q^*D \cdot q^*S$ is an effective $1$-cycle on $\check{U}$.
We write $q^*D \cdot q^*S = \gamma \check{\Gamma} + \check{\Delta}$, where $\gamma \ge 0$ and $\check{\Delta}$ is an effective $1$-cycle on $\check{U}$ such that $\check{\Gamma} \not\subset \Supp (\check{\Delta})$.
Then $D \cdot S = \gamma \Gamma + \Delta + \Xi$, where $\Delta = \frac{1}{r} q_*\check{\Delta}$ and $\Xi$ is an effective $1$-cycle such that $\Gamma \not\subset \Supp (\Xi)$.
By \eqref{eq:exclL-2}, we have
\[
a (-K_X)^3 = (-K_X \cdot D \cdot S) \ge \gamma (-K_X \cdot \Gamma) = \frac{ab (-K_X)^3}{m} \gamma,
\]
where the inequality holds since $-K_X$ is nef.
Note that $(-K_X)^3 > 0$ since $-K_X$ is nef and big.
Hence 
\begin{equation} \label{eq:exclL-4}
\gamma \le \frac{m}{b}.
\end{equation}

By \eqref{eq:exclL-1} and \eqref{eq:exclL-3}, we have
\[
\begin{split}
r \left(a b (-K_X)^3 - \frac{a b^2 (-K_X)^3}{m} \gamma\right) 
&= r (T \cdot (D \cdot S - \gamma \Gamma)) \\ 
&= r (T \cdot (\Delta + \Xi)) \ge r (T \cdot \Delta) \\
&\ge r (T \cdot \Delta)_{\msp} = (q^*T \cdot \check{\Delta})_{\check{\msp}} \\
&\ge \mult_{\check{\msp}} (\check{\Delta}) \\ 
&> \frac{1}{c} - \gamma \mult_{\check{\msp}} (\check{\Gamma}),
\end{split}
\]
where $(\ \cdot \ )_{\msp}$ and $(\ \cdot \ )_{\check{\msp}}$ denote the local intersection numbers at $\msp$ and $\check{\msp}$ respectively.
It follows that
\begin{equation} \label{eq:exclL-5}
\left(\mult_{\check{\msp}} (\check{\Gamma}) - \frac{r a b^2 (-K_X)^3}{m}\right) \gamma > \frac{1}{c} - r a b (-K_X)^3.
\end{equation}
We have $\mult_{\check{\msp}} (\check{\Gamma}) - r a b^2 (-K_X)^3/m > 0$ since $1/c - r a b (-K_X)^3 \ge 0$ by the definition of $c$.
Combining \eqref{eq:exclL-4} and \eqref{eq:exclL-5}, we have
\[
c > \frac{b}{m \mult_{\check{\msp}} (\check{\Gamma})}.
\]
This contradicts the definition of $c$ and the proof is completed. 
\end{proof}

Lemma \ref{lem:exclL} is very useful in computing alpha invariants but works only when $S \cap T$ is irreducible.
We consider its generalization that can be applied when $S \cap T$ is reducible.

\begin{Def}
Let $M = (a_{ij})$ be an $n \times n$ matrix with entries in $\mbR$, where $n \ge 2$.
For a non-empty subset $I \subset \{1,2,\dots,n\}$, we denote by $M_I$ the submatrix of $M$ consisting of $i$th rows and columns for $i \in I$.
We say that $M$ satisfies the {\it condition $(\star)$} if the following are satisfied. 
\begin{itemize}
\item $(-1)^{|I|} \det M_I \ge 0$ for any non-empty proper subset $I \subset \{1,2,\dots,n\}$.
\item $(-1)^{n-1} \det M > 0$.
\item $a_{ij} > 0$ for any $i, j$ with $i \ne j$.
\end{itemize}

For $v = {}^t (v_1,\dots,v_n), w = {}^t (w_1,\dots,w_n) \in \mbR^n$, the expression $v \le w$ means $v_i \le w_i$ for any $i$.
\end{Def} 

\begin{Lem} \label{lem:matrix}
Let $M = (a_{i j})$ be an $n \times n$ matrix with entries in $\mbR$ satisfying the condition $(\star)$, and let $v, w \in \mbR^n$.
Then $M v \le M w$ implies $v \le w$.
\end{Lem}

\begin{proof}
It is enough to show that $v \le 0$ assuming $M v \le 0$ for $v \in \mbR^n$.
We prove this assertion by induction on $n \ge 2$.
The case $n = 2$ is easily done and we omit it.

Assume $n \ge 3$.
Suppose that there is a diagonal entry $a_{kk}$ such that $a_{kk} = 0$.
Then, we have $\det M_{\{k,l\}} < 0$ since $a_{k l}, a_{l k} > 0$.
By the condition $(\star)$, this is impossible since $n \ge 3$.

In the following we may assume that $a_{ii} \ne 0$ for any $i$.
By the condition $(\star)$, we have $a_{ii} = \det M_{\{i\}} \le 0$ and hence $a_{ii} < 0$ for any $i$.
Let $M'$ be the matrix obtained by adding the $1$st row multiplied by the positive integer $-a_{i1}/a_{11}$ to the $i$th row, for $i = 2,\dots,n$.
Then we obtain the inequality $M' v \le 0$ and we can write
\[
M' = 
\arraycolsep5pt
\left(
\begin{array}{@{\,}cccc@{\,}}
a_{11} & a_{12} & \cdots &a_{1n}\\
0&&&\\
0&\multicolumn{3}{c}{\raisebox{-5pt}[0pt][0pt]{\Huge $N'$}}\\
0&&&\\
\end{array}
\right),
\]
where $N'$ is an $(n-1) \times (n-1)$ matrix.
It is straightforward to check that $N$ satisfies the condition $(\star)$.
Since $N' {}^t (v_2 \ \dots \ v_n) \le 0$, we have $v_2, \dots, v_n \le 0$ by induction hypothesis.
Next, let $M''$ be the matrix obtained by adding the $n$th row multiplied by the positive integer $- a_{i n}/a_{nn}$ to the $i$th row, for $i = 1, 2, \dots, n-1$.
Then we have $M'' v \le 0$ and, by repeating the similar argument as above, we conclude $v_1, \dots, v_{n-1} \le 0$ by inuction.
This completes the proof.
\end{proof}

\begin{Def}
Let $S$ be a normal projective surface and let $\Gamma_1, \dots, \Gamma_k$ be irreducible and reduced curves on $S$.
Then the $k \times k$ matrix
\[
M (\Gamma_1, \dots, \Gamma_k) = ((\Gamma_i \cdot \Gamma_j)_S)_{1 \le i, j \le k}
\]
is called the {\it intersection matrix} of curves $\Gamma_1, \dots, \Gamma_k$ on $S$.
\end{Def}

\begin{Lem} \label{lem:mtdLred}
Let $X$ be a member of a family $\mcF_{\msi}$ with $\msi \in \msI$.
Let $S \in \left| - a K_X \right|$ be a normal surface on $X$, $T \in \left| - b K_X \right|$ an effective divisor and $\msp \in S$ a point, where $a, b > 0$.
We set $r = 1$ when $\msp \in X$ is a smooth point, and otherwise we denote by $r$ the index of the cyclic quotient singularity $\msp \in X$.
Suppose that 
\[
T|_S =  m_1 \Gamma_1 + m_2 \Gamma_2 + \cdots + m_k \Gamma_k,
\] 
where $\Gamma_1, \dots, \Gamma_k$ are distinct irreducible and reduced curves on $S$ and $m_1, \dots, m_k$ are positive integers, and the following properties are satisfied.
\begin{itemize}
\item $r b \deg \Gamma_1 \le m_1$.
\item $\msp \in \Gamma_1 \setminus (\cup_{i \ge 2} \Gamma_i)$, and $S, \Gamma_1$ are both quasi-smooth at $\msp$.
\item The intersection matrix $M (\Gamma_1, \dots, \Gamma_k)$ satisfies the condition $(\star)$.
\end{itemize}
Then we have
\[
\alpha_{\msp} (X) \ge \min \left\{ a, \ \frac{m_1}{r a b (-K_X)^3 + \frac{m_1^2}{b} - r m_1 \deg \Gamma_1} \right\}.
\]
\end{Lem} 

\begin{proof}
Let $D \in \left| - K_X \right|_{\mbQ}$ be an irreducible $\mbQ$-divisor.
If $\Supp (D) = S$, then $D = \frac{1}{a} S$ and we have $\lct_{\msp} (X, D) \ge a$ since $S$ is quasi-smooth at $\msp$.
We assume $\Supp (D) \ne S$.
It is enough to prove the inequality
\begin{equation}
\label{eq:mtdLred1}
\lct_{\msp} (X; D) \ge \frac{m_1}{r a b (-K_X)^3 + \frac{m_1^2}{b} - r m_1 \deg \Gamma_1}.
\end{equation}

We can write
\[
D|_S = \gamma_1 \Gamma_1 + \cdots + \gamma_k \Gamma_k + \Delta,
\]
where $\gamma_1, \dots, \gamma_k \ge 0$ and $\Delta$ is an effective $\mbQ$-divisor on $S$ such that $\Gamma_i \not\subset \Supp (\Delta)$ for $i = 1, \dots, k$.
We set $\sigma_i = (\Gamma_i^2)_S$ and $\chi_{i,j} = (\Gamma_i \cdot \Gamma_j)_S$.
For $i = 1, \dots, k$, we have
\begin{equation}
\label{eq:mtdLred2}
\begin{split}
b \deg \Gamma_i &= (T|_S \cdot \Gamma_i)_S \\
&= m_1 \chi_{1,i} + \cdots + m_{i-1} \chi_{i - 1, i} + m_i \sigma_i + m_{i+1} \chi_{i+1, i} + \cdots + m_k \chi_{k,i},
\end{split}
\end{equation}
and
\begin{equation}
\label{eq:mtdLred3}
\begin{split}
\deg \Gamma_i &= (D|_S \cdot \Gamma_i)_S \\
&\ge \gamma_1 \chi_{1, i} + \cdots + \gamma_{i-1} \chi_{i-1} + \gamma_i \sigma_i + \gamma_{i+1} \chi_{i+1,i} + \cdots + \gamma_k \chi_{k,i+1}.
\end{split}
\end{equation}
We set $M = M (\Gamma_1, \dots, \Gamma_k)$.
Combining inequalities \eqref{eq:mtdLred2} and \eqref{eq:mtdLred3}, we have
\[
M \begin{pmatrix} b \gamma_1 \\ \vdots \\ b \gamma_k \end{pmatrix} \le M \begin{pmatrix} m_1 \\ \vdots \\ m_k \end{pmatrix}.
\]
By Lemma \ref{lem:matrix}, this implies $\gamma_i \le m_i/b$ for any $i$.

When $\msp$ is a singular point, then we set $\rho = \rho_{\msp} \colon \breve{U}_{\msp} \to U_{\msp}$, which is the orbifold chart of $X$ containing $\msp$.
When $\msp$ is a smooth point of $X$, then we set $\breve{U} = U = X$ and $\rho \colon \breve{U} \to U$ is assumed to be the identity morphism.
Moreover we set $\breve{S} := q^*(S \cap U)$ and $\rho_S = \rho|_{\breve{S}} \colon \breve{S} \to S \cap U$.
We see that $\breve{S}$ is smooth at the preimage $\breve{\msp}$ of $\msp$ since $S$ is quasi-smooth at $\msp$, and 
\[
\rho^*D|_{\breve{S}} = \gamma_1 \rho_S^*\Gamma_1 + \cdots \gamma_k \rho_S^*\Gamma_k + \rho_S^*\Delta.
\]
This implies 
\[
\mult_{\breve{\msp}} (\rho_S^*\Delta) \ge \omult_{\msp} (D) - \gamma_1
\] 
since $\rho_S^*\Gamma_i$ does not pass through $\breve{\msp}$ for $i \ge 2$ and $\rho_S^*\Gamma_1$ is smooth at $\breve{\msp}$ by the quasi-smoothness of $\Gamma_1$ at $\msp$.
We have
\[
\begin{split}
r (ab (-K_X)^3 - b \gamma_1 \deg \Gamma_1) &\ge
r (T|_S \cdot (D|_S - \gamma_1 \Gamma_1 - \cdots - \gamma_k \Gamma_k))_S \\
&= r (T|_S \cdot \Delta)_S \\
&\ge m_1 r (\Gamma_1 \cdot \Delta)_S \\
&\ge m_1 (\rho_S^*\Gamma_1 \cdot \rho_S^*\Delta)_{\breve{\msp}} \\
&\ge m_1 \mult_{\breve{\msp}} (\rho_S^*\Delta) \\
&\ge m_1 (\omult_{\msp} (D) - \gamma_1).
\end{split}
\]
Since $m_1 - r b \deg \Gamma_1 \ge 0$ and $\gamma_1 \le m_1/b$, we have
\[
\begin{split}
\omult_{\msp} (D) &\le \frac{1}{m_1} (r a b (-K_X)^3 + (m_1 - r b \deg \Gamma_1) \gamma_1) \\
& \le \frac{1}{m_1} (r a b (-K_X)^3 + \frac{m_1^2}{b}- r m_1 \deg \Gamma_1).
\end{split}.
\]
This implies \eqref{eq:mtdLred1} and the proof is completed.
\end{proof}

The following is a version of Lemma \ref{lem:mtdLred}, which may be effective when $S$ is singular at $\msp$.

\begin{Lem} \label{lem:mtdLredSsing}
Let $X$ be a normal projective $\mbQ$-factorial $3$-fold.
Let $S \sim_{\mbQ} - a K_X$ be a normal surface on $X$, $T \sim_{\mbQ} - b K_X$ an effective divisor and $\msp \in S$ a point, where $a, b$ are positive rational numbers.
Suppose that 
\[
T|_S =  m_1 \Gamma_1 + m_2 \Gamma_2 + \cdots + m_k \Gamma_k,
\] 
where $\Gamma_1, \dots, \Gamma_k$ are distinct irreducible and reduced curves on $S$ and $m_1, \dots, m_k$ are positive integers, and the following properties are satisfied.
\begin{itemize}
\item $b \deg \Gamma_1 \le \mult_{\msp} (\Gamma_1)$.
\item $\msp \in \Gamma_1 \setminus (\cup_{i \ge 2} \Gamma_i)$, and $X$ is smooth at $\msp$.
\item The intersection matrix $M (\Gamma_1, \dots, \Gamma_k)$  satisfies the condition $(\star)$.
\end{itemize}
Then we have
\[
\alpha_{\msp} (X) \ge \min \left\{ \frac{a}{\mult_{\msp} (S)}, \ \frac{\mult_{\msp} (S)}{a b (-K_X)^3 + \frac{m_1}{b} \mult_{\msp} (\Gamma_1) - m_1 \deg \Gamma_1} \right\}.
\]
\end{Lem} 

\begin{proof}
Let $D \in \left| - K_X \right|_{\mbQ}$ be an irreducible $\mbQ$-divisor.
If $\Supp (D) = S$, then $D = \frac{1}{a} S$ and we have $\lct_{\msp} (X, D) \ge a/\mult_{\msp} (S)$.
We assume $\Supp (D) \ne S$.
It is enough to show that
\begin{equation}
\label{eq:mtdLredSsing1}
\lct_{\msp} (X; D) \ge \frac{\mult_{\msp} (S)}{a b (-K_X)^3 + \frac{m_1}{b} \mult_{\msp} (\Gamma_1) - m_1 \deg \Gamma_1}.
\end{equation}

We write
\[
D|_S = \gamma_1 \Gamma_1 + \cdots + \gamma_k \Gamma_k + \Delta,
\]
where $\gamma_1, \dots, \gamma_k \ge 0$ and $\Delta$ is an effective divisor on $S$ such that $\Gamma_i \not\subset \Supp (\Delta)$ for $i = 1, \dots, k$.
By the same argument as in the proof of Lemma \ref{lem:mtdLred}, we have $\gamma_i \le m_i/b$ for any $i$.
We consider the $1$-cycle $D \cdot S = \gamma_1 \Gamma_1 + \cdots + \gamma_k \Gamma_k$ on $X$ and we have
\[
\begin{split}
ab (-K_X)^3 - b \gamma_1 \deg \Gamma_1 &\ge
(T \cdot (D \cdot S - \gamma_1 \Gamma_1 - \cdots - \gamma_k \Gamma_k))_X \\
& = (T \cdot \Delta)_X \\
&\ge \mult_{\msp} (\Delta) \\
&\ge (\mult_{\msp} (S)) (\mult_{\msp} (D)) - \gamma_1 \mult_{\msp} (\Gamma_1).
\end{split}
\]
Since $\mult_{\msp} (\Gamma_1) - b \deg \Gamma_1 \ge 0$ and $\gamma_1 \le m_1/b$, we have
\[
\begin{split}
\mult_{\msp} (D) &\le \frac{1}{\mult_{\msp} (S)} (ab (-K_X)^3 + (\mult_{\msp} (\Gamma_1) - b \deg \Gamma_1) \gamma_1) \\
& \le \frac{1}{\mult_{\msp} (S)} \left( ab (-K_X)^3 + \frac{m_1}{b} \mult_{\msp} (\Gamma_1) - m_1 \deg \Gamma_1 \right).
\end{split}.
\]
This implies \eqref{eq:mtdLredSsing1} and the proof is completed.
\end{proof}

\begin{Lem} \label{lem:mtdLintpt}
Let $X$ be a normal projective $\mbQ$-factorial $3$-fold.
Let $S \sim_{\mbQ} - a K_X$ be a normal surface on $X$, $T \sim_{\mbQ} - b K_X$ an effective divisor and $\msp \in X$ a point, where $a, b$ be positive rational numbers.
Suppose that 
\[
T|_S = \Gamma_1 + \Gamma_2,
\] 
where $\Gamma_1, \Gamma_2$ are distinct irreducible and reduced curves on $S$, and the following properties are satisfied.
\begin{itemize}
\item $\deg \Gamma_i \le 2/b$ for $i = 1, 2$.
\item $\msp \in \Gamma_1 \cap \Gamma_2$ and all the $X$, $S$, $\Gamma_1$ and $\Gamma_2$ are smooth at $\msp$.
\item The intersection matrix $M (\Gamma_1, \Gamma_2)$ satisfies the condition $(\star)$.
\end{itemize}
Then we have
\[
\alpha_{\msp} (X) \ge \min \left\{ a, \ \frac{b}{2} \right\}.
\]
\end{Lem} 

\begin{proof}
We have $\lct_{\msp} (X, \frac{1}{a} S) \ge a$ since $S$ is smooth at $\msp$ by assumption.
Let $D \in \left| - K_X \right|_{\mbQ}$ be an irreducible $\mbQ$-divisor on $X$ such that $\Supp (D) \ne S$.
It is enough to prove the inequality $\lct_{\msp} (X;D) \ge b/2$.
We write 
\[
D|_S = \gamma_1 \Gamma_1 + \gamma_2 \Gamma_2 + \Delta,
\]
where $\gamma_1, \gamma_2 \ge 0$ and $\Delta$ is an effective divisor on $S$ with $\Gamma_1, \Gamma_2 \not\subset \Supp (\Delta)$.
By the proof of Lemma \ref{lem:mtdLred}, we have $\gamma_1, \gamma_2 \le 1/b$.
We have
\begin{equation} \label{eq:lem:methodLintpt1}
a b (-K_X)^3 = (-K_X|_S \cdot T|_S)_S = \deg \Gamma_1 + \deg \Gamma_2.
\end{equation}
Since $\mult_{\msp} (T|_S) = 2$ and $\mult_{\msp} (\Delta) \ge \mult_{\msp} (D) - \gamma_1 - \gamma_2$, we have
\[
\begin{split}
ab (-K_X)^3 - b \gamma_1 \deg \Gamma_1 - b \gamma_2 \deg \Gamma_2 &=
(T|_S \cdot (D|_S - \gamma_1 \Gamma_1 - \gamma_2 \Gamma_2))_S \\
&= (T|_S \cdot \Delta)_S \\
& \ge 2 (\mult_{\msp} (D) - \gamma_1 - \gamma_2).
\end{split}
\]
By \eqref{eq:lem:methodLintpt1}, the assumption $\deg \Gamma_1, \deg \Gamma_2 \le 2/b$ and $\gamma_1, \gamma_2 \le 1/b$, we have
\[
\begin{split}
\mult_{\msp} (D) &\le\frac{1}{2} (ab (-K_X)^3 + (2 - b \deg \Gamma_1) \gamma_1 + (2- b \deg \Gamma_2) \gamma_2) \\
& \le \frac{2}{b}.
\end{split}
\]
This shows $\lct_{\msp} (X; D) \ge b/2$ and thus $\alpha_{\msp} (X) \ge \min \{a, b/2\}$.
\end{proof}

\subsubsection{Computations by weighted blowups}
\label{sec:compwbl}

We explain methods of computing LCTs via suitable weighted blowups.

Let $\msp \in X$ be a germ of a smooth variety of dimension $n$ with a system of local coordinates $\{x_1, \dots, x_n\}$ at $\msp$, and let $D$ be an effective $\mbQ$-divisor on $X$.
Let $\varphi \colon Y \to X$ be the weighted blowup at $\msp$ with weight $\wt (x_1, \dots, x_n) = (c_1, \dots, c_n)$, where $\underline{c} = (c_1, \dots, c_n)$ is a tuple of positive integers such that $\gcd \{c_1, \dots, c_n\} = 1$.
Let $E \cong \mbP (\underline{c}) = \mbP (c_1, \dots, c_n)$ be the exceptional divisor of $\varphi$.
Note that $Y$ can be singular along a divisor on $E$ (see Remark \ref{rem:wbldiff} below) so that we cannot expect the usual adjunction $(K_Y + E)|_E = K_E$.
In general we need a correction term and we have
\[
(K_Y + E)|_E = K_E + \Diff
\]
where the correction term $\Diff$ is a $\mbQ$-divisor on $E$ which is called the {\it different} (see \cite[Chapter 16]{FA}).

\begin{Rem} \label{rem:wbldiff}
We give a concrete description of $\Diff$.
Let $\mbP (\underline{c})^{\wf}$ be the well-formed model of $\mbP (\underline{c})$ and we identify $E$ with $\mbP (\underline{c})^{\wf}$.
For $i = 0, 1, \dots, n$, let 
\[
H^{\wf}_i = (\tilde{x}_i = 0) \subset E \cong \mbP (\underline{c})^{\wf}
\] 
be the quasi-hyperplane of $\mbP (\underline{c})^{\wf}_{\tilde{x}_1, \dots, \tilde{x}_n}$, and we set $m_i = \gcd \{c_0, \dots, \hat{c}_i, \dots, c_n\}$.
We see that $Y$ is singular at the generic point of $H^{\wf}_i$ if and only if $m_i > 1$, and if this is the case, then the singularity of $Y$ along $H^{\wf}_i$ is a cyclic quotient singularity of index $m_i$.
It follows from \cite[Proposition 16.6]{FA} that    
\[
\Diff = \sum_{i=1}^n \frac{m_i-1}{m_i} H^{\wf}_i
\]
under the identification $E \cong \mbP (\underline{c})^{\wf}$.
\end{Rem}

\begin{Lem} \label{lem:lctwbl}
Let the notation and assumption as above.
Then we have
\begin{equation} \label{eq:lctwbl-1}
\lct_{\msp} (X;D) \ge \min \left\{ \frac{c_1 + \cdots + c_n}{\ord_E (D)}, \  \lct (E, \Diff ; \tilde{D}|_E) \right\},
\end{equation}
where $\tilde{D}$ is the proper transform of $D$.
If in addition the inequality
\begin{equation} \label{eq:lctwbl-2}
\frac{c_1 + \cdots + c_n}{\ord_E (D)} \le \lct (E, \Diff_E; \tilde{D}|_E)
\end{equation}
holds, then we have
\begin{equation} \label{eq:lctwbl-3}
\lct_{\msp} (X; D) = \frac{c_1 + \cdots + c_n}{\ord_E (D)}.
\end{equation}
\end{Lem}

\begin{proof}
We set $c = c_1 + \cdots + c_n$ and let $\lambda$ be any rational number such that
\[
0 < \lambda \le \min \left\{ \frac{c}{\ord_E (D)}, \ \lct (E, \Diff; \tilde{D}|_E) \right\}.
\]
We will show that the pair $(X, \lambda D)$ is log canonical at $\msp$, which will prove the inequality \eqref{eq:lctwbl-1}.

We assume that the pair $(X, \lambda D)$ is not log canonical at $\msp$.
We have
\begin{equation} \label{eq:lctwbl-4}
K_Y + \lambda \tilde{D} + (\lambda \ord_E (D) - c + 1) E = \varphi^* (K_X + \lambda D),
\end{equation}
and the pair $(Y, \lambda \tilde{D} + (\lambda \ord_E (D) - c + 1)E)$ is not log canonical along $E$.
Since $\lambda \le c/\ord_E (D)$, we have
\[
\lambda \ord_E (D) - c + 1 \le 1,
\]
which implies that the pair $(Y, \lambda \tilde{D} + E)$ is not log canonical along $E$.
Thus the pair $(E, \Diff + \lambda \tilde{D}|_E)$ is not log canonical.
This is impossible since $\lambda \le \lct (E, \Diff;\tilde{D}|_E)$.
Therefore the pair $(X, \lambda D)$ is log canonical at $\msp$, and the inequality \eqref{eq:lctwbl-1} is proved.

By considering the coefficient of $E$ in \eqref{eq:lctwbl-4}, it is easy to see that 
\[
\lct_{\msp} (X; D) \le \frac{c}{\ord_E (D)}.
\]
Under the assumption \eqref{eq:lctwbl-2}, this shows the equality \eqref{eq:lctwbl-3}.
\end{proof}

We consider Lemma \ref{lem:lctwbl} in more details in a concrete setting.

\begin{Def}
Let $\underline{c} = (c_1, \dots, c_n)$ be an $n$-tuple of positive integers such that $\gcd \{c_1, \dots, c_n\} = 1$ and we set
\[
m_i = \gcd \{c_1, \dots, \hat{c}_i, \dots, c_n\}
\]
for $i = 1, \dots, n$.
Let $f = f (x_1, \dots, x_n)$ be a polynomial which is quasi-homogeneous with respect to $\wt (x_1, \dots, x_n) = \underline{c}$.

If $f$ is irreducible and $f \ne x_i$ for $i = 1, \dots, n$, then there exists an irreducible polynomial $f^{\wf} = f^{\wf} (\tilde{x}_1, \dots, \tilde{x}_n)$ (in new variables $\tilde{x}_1, \dots, \tilde{x}_n$) such that
\[
f^{\wf} (x_1^{m_1}, \dots, x_n^{m_n}) = f (x_1, \dots, x_n).
\]
We call $f^{\wf}$ the {\it well-formed model} of $f$ (with respect to the weight $\wt (x_1, \dots, x_n) = \underline{c}$).

In general we have a decomposition
\[
f = x_1^{\lambda_1} \cdots x_n^{\lambda_n} f_1^{\mu_1} \cdots f_k^{\mu_k},
\]
where $k, \lambda_1, \dots, \lambda_n, \mu_1, \dots, \mu_k$ are nonnegative integers and $f_1, \dots, f_k$ are irreducible polynomials in variables $x_1, \dots, x_n$ which are quasi-homogeneous with respect to $\wt (x_1, \dots, x_n) = \underline{c}$ and which are not $x_i$ for any $i$.
We define
\[
f^{\wf} := f (\tilde{x}_1^{1/m_1}, \dots, \tilde{x}_n^{1/m_n}) 
= \tilde{x}^{\lambda_1/m_1} \cdots \tilde{x}_n^{\lambda_n/m_n} (f_1^{\wf})^{\mu_1} \cdots (f_k^{\wf})^{\mu_k}
\]
and call it the {\it well-formed model} of $f$.
Note that $f^{\wf}$ is in general not a polynomial since $\lambda_i/m_i$ need not be an integer.
In this case the effective $\mbQ$-divisor
\[
\mcD^{\wf}_f := \sum_{i=1}^n \frac{\lambda_i}{m_i} H^{\wf}_i + \sum_{j=1}^k \mu_j (f_j^{\wf} = 0) 
\]
on the well-formed model $\mbP (\underline{c})^{\wf}_{\tilde{x}_1, \dots, \tilde{x}_n}$ of $\mbP (\underline{c})$ is called the {\it effective $\mbQ$-divisor on $\mbP (\underline{c})^{\wf}$ associated to} $f$, where $H_i^{\wf}$ is the quasi-hyperplane on $\mbP (\underline{c})^{\wf}$ defined by $\tilde{x}_i = 0$.
\end{Def}

\begin{Lem} \label{lem:lctwblwh}
Let $\mbP (\underline{b}) := \mbP (b_0, \dots, b_{n+1})_{x_0, \dots, x_{n+1}}$ be a well-formed weighted projective space and let $X \subset \mbP (\underline{b})$ be a normal weighted hypersurface with defining polynomial $F = F (x_0, \dots, x_{n+1})$.
Let $\underline{c} = (c_1, \dots, c_n)$ be a tuple of positive integers such that $\gcd \{c_1, \dots, c_n\} = 1$.
Assume that
\[
F = x_0^e x_{n+1} + \sum_{i=1}^e x_0^{e-i} f_i,
\]
where $e \in \mbZ_{> 0}$ and $f_i = f_i (x_1, \dots, x_n, x_{n+1})$ is a quasi-homogeneous polynomial of degree $i b_0 + b_{n+1}$.
Let $G = G (x_1, \dots x_n)$ be the lowest weight part of $\bar{F} := F (1, x_1, \dots, x_n,0)$ with respect to $\wt (x_1, \dots, x_n) = \underline{c}$
Then, for the point $\msp = \msp_{x_0} = (1\!:\!0\!:\!\cdots\!:\!0) \in X$, we have
\[
\lct_{\msp} (X; H_{x_{n+1}}) 
\ge \min \left\{ \frac{c_1 + \cdots + c_n}{\wt_{\underline{c}} (\bar{F})}, \ \lct (\mbP (\underline{c})^{\wf}, \Diff; \mcD^{\wf}_G) \right\},
\]
where $\wt_{\underline{c}} (\bar{F})$, $\Diff$ and $\mcD^{\wf}_G$ are as follows.
\begin{itemize}
\item $\wt_{\underline{c}} (\bar{F})$ is the weight of $\bar{F}$ with respect to $\wt (x_1, \dots, x_n) = \underline{c}$.
\item $\Diff = \sum_{i=1}^n \frac{m_i-1}{m_i} H^{\wf}_i$, where $H^{\wf}_i = (\tilde{x}_i = 0)$ is the $i$th coordinate quasi-hyperplane of $\mbP (\underline{c})^{\wf}_{\tilde{x}_1, \dots, \tilde{x}_n}$ and $m_i = \gcd \{c_1, \dots, \hat{c}_i, \dots, c_n\}$ for $i = 1, \dots, n$.
\item $\mcD^{\wf}_G$ is the effective $\mbQ$-divisor on $\mbP (\underline{c})^{\wf}$ associated to $G$.
\end{itemize}
If in addition the inequality
\[
\frac{c_1 + \cdots + c_n}{\wt (\bar{F})} \le \lct (\mbP (\underline{c})^{\wf}, \Diff; \mcD^{\wf}_G)
\]
holds, then we have
\[
\lct_{\msp} (X; H_{x_{n+1}}) = \frac{c_1 + \cdots + c_n}{\wt (\bar{F})}.
\]
\end{Lem}

\begin{proof}
Let $\rho_{\msp} \colon \breve{U}_{\msp} \to U_{\msp} \subset X$, where $U_{\msp} = U_{x_0}$, be the orbifold chart containing $\msp$ and we set $\rho = \rho_{\msp}, \breve{U} = \breve{U}_{\msp}$ and $U = U_{\msp}$.
We set $H = H_{x_{n+1}}$ and $\breve{H} = \rho^*H$.
We have $\lct_{\msp} (X; H) = \lct_{\check{\msp}} (\check{U}; \check{H})$.
The variety $\breve{U}$ is the hypersurface in $\mcU_{x_0} = \mbA^{n+1}_{\breve{x}_1, \dots, \breve{x}_{n+1}}$ defined by the equation
\begin{equation} \label{eq:lctwblwh-1}
F (1, \breve{x}_1, \dots, \breve{x}_{n+1}) = \breve{x}_{n+1} + \sum_{i=1}^e \breve{f}_i = 0,
\end{equation}
where $\breve{f}_i = f_i (\breve{x}_1, \dots, \breve{x}_{n+1})$, and $\breve{\msp}$ corresponds to the origin.
We see that $\{\breve{x}_1, \dots, \breve{x}_n\}$ is a system of local coordinates of $\breve{U}$ at $\breve{\msp}$.
Let $\varphi \colon Y \to \breve{U}$ be the weighted blowup at $\breve{\msp}$ with $\wt (\breve{x}_1, \dots, \breve{x}_n) = (c_1, \dots, c_n)$.
We can identify the $\varphi$-exceptional divisor $E$ with $\mbP (\underline{c})_{\breve{x}_1, \dots, \breve{x}_n}$.
Filtering off terms divisible by $\breve{x}_{n+1}$ in \eqref{eq:lctwblwh-1}, we have
\[
(-1 + \cdots ) \breve{x}_{n+1} = F (1, \breve{x}_1, \dots, \breve{x}_n, 0)
\]
on $\breve{U}$, where the omitted term in the left-hand side is a polynomial vanishing at $\breve{\msp}$.
Since $\breve{H}$ is the divisor on $\breve{U}$ defined by $\breve{x}_{n+1} = 0$, we see that $\ord_E (\breve{H}) = \wt_{\underline{c}} (\bar{F})$ and the divisor $\tilde{H}|_E$ corresponds to the divisor $\mcD^{\wf}_G$ on $E \cong \mbP (\underline{c})^{\wf}$, where $\tilde{H}$ is the proper transform of $\breve{H}$ on $Y$.
Therefore the proof is completed by Lemma \ref{lem:lctwbl} and Remark \ref{rem:wbldiff}.
\end{proof}

\begin{Lem} \label{lem:lcttangcube}
Let $X \subset \mbP (a, b_1, b_2, b_3, r)_{x, y_1, y_2, y_3, z}$ be a member of a family $\mcF_{\msi}$ with $\msi \in \msI$ with defining polynomial $F = F (x, y_1, y_2, y_3, z)$.
Assume that $F$ can be written as
\[
F = z^k x + z^{k-1} f_{r + a} + z^{k-2} f_{2 r + a} + \cdots + f_{k r + a},
\]
where $f_i \in \mbC [x, y_1, y_2, y_3]$ is a  quasi-homogeneous polynomial of degree $i$, and we set
\[
\bar{F} := F (0,y_1, y_2, y_2, 1) \in \mbC [y_1, y_2, y_3].
\]
If either $\bar{F} \in (y_1, y_2, y_3)^2 \setminus (y_1, y_2, y_3)^3$ or $\bar{F} \in (y_1, y_2, y_3)^3$ and the cubic part of $\bar{F}$ is not a cube of a linear form in $y_1, y_2, y_3$, then for the point $\msp := \msp_z \in X$, we have
\[
\lct_{\msp} (X; H_x) \ge \frac{1}{2}.
\]
If in addition $a = 1$, $r > 1$ and $\msp \in X$ is not a maximal center, then 
\[
\alpha_{\msp} (X) = \min \{1, \lct_{\msp} (X;H_x)\} \ge \frac{1}{2}.
\] 
\end{Lem}

\begin{proof}
Let $\rho_{\msp} \colon \breve{U}_{\msp} \to U_{\msp}$ be the orbifold chart of $X$ containing $\msp$.
We see that Let $\breve{U}_{\msp}$ be the hypersurface in $\breve{\mcU}_{\msp} = \mbA^4_{\breve{x}, \breve{y}_1, \breve{y}_2, \breve{y}_3}$ defined by the equation
\[
F (\breve{x}, \breve{y}_1, \breve{y}_2, \breve{y}_3, 1) = 0.
\]
We see that $\breve{U}_{\msp}$ is smooth and the morphism $\rho_{\msp}$ can be identified with the quotient morphism of the singularity $\msp \in X$ over a suitable analytic neighborhood of $\msp$.
We denote by $\breve{\msp} \in \breve{U}$ the origin of $\breve{\mcU}_{\msp} = \mbA^4$ which is the preimage of $\msp$ via $\rho_{\msp}$.
Filtering off terms divisible by $x$ in $F (x, y_1, y_2, y_3, 1)$, we have
\[
(-1 + \cdots) \breve{x} = F (0, \breve{y}_1, \breve{y}_2, \breve{y}_3, 1) = \bar{F} (\breve{y}_1, \breve{y}_2, \breve{y}_3) =: \breve{F}
\]
on $\breve{U}_{\msp}$.
Note that we can choose $\{\breve{y}_1, \breve{y}_2, \breve{y}_3\}$ as a system of local coordinates of $\breve{U}_{\msp}$ at $\breve{\msp}$.
If $\bar{F} \in (y_1, y_2, y_3)^2 \setminus (y_1, y_2, y_3)^3$, then $\omult_{\msp} (H_x) = \mult_{\breve{\msp}} (\rho_{\msp}^*H_x) = 2$ and hence $\lct_{\msp} (X; H_x) \ge 1/2$.

Suppose that $\bar{F} \in (y_1, y_2, y_3)^3$ and the cubic part of $\bar{F}$ is not a cube of a linear form in $y_1, y_2, y_3$.
Let $\varphi \colon V \to \breve{U}$ be the blowup of $\breve{U}$ at $\breve{\msp}$ with exceptional divisor $E \cong \mbP^2$.
We set $D = \rho_{\msp}^*H_x$.
Since $\mult_{\breve{\msp}} (D) = 3$, we have
\[
K_V + \frac{1}{2} \tilde{D} = \varphi^* \left(K_{\check{U}} + \frac{1}{2} D \right) + \frac{1}{2} E,
\]
where $\tilde{D}$ is the proper transform of $D$ on $V$.
The divisor $\tilde{D}|_E$ on $E$ is isomorphic to the hypersurface in $\mbP^2_{\breve{y}_1, \breve{y}_2, \breve{y}_3}$ defined by the cubic part of $\bar{F} (\breve{y}_1, \breve{y}_2, \breve{y}_3)$, and the pair $(E, \frac{1}{2} \tilde{D}|_E)$ is log canonical by Lemma \ref{lem:lctP2cubic}.
It then follows that the pair $(V, \frac{1}{2} \tilde{D})$ is log canonical along $E$.
This shows that the pair $(\breve{U}, \frac{1}{2} D)$ is log canonical at $\breve{\msp}$, and hence $\lct_{\msp} (X;H_x) \ge 1/2$ as desired.

Suppose in addition that $r > 1$ and $\msp \in X$ is not a maximal center.
By Lemma \ref{lem:qtangdivncan}, the pair $(X, H_x)$ is not canonical at $\msp = \msp_z$ and thus we have $\alpha_{\msp} (X) = \min \{1, \lct_{\msp}; (X;H_x)\}$ by Lemma \ref{lem:singnoncanbd}.
This proves the latter assertion.
\end{proof}

\subsubsection{Computations by $2 n$-inequality}

\begin{Lem} \label{lem:complctsingtang}
Let $X \subset \mbP (b_1, b_2, b_3, c, r)_{x_1, x_2, x_3, y, z}$ be a member of a family $\mcF_{\msi}$ with  $\msi \in \msI$ with defining polynomial $F = F (x_1, x_2, x_3, y, z)$ and suppose $\msp :=\msp_z \in X$.
We assume that $b_1 \le b_2 \le b_3$ and that we can choose $y$ as a quasi-tangent coordinate of $X$ at $\msp$.
Then 
\[
\alpha_{\msp} (X) \ge \frac{2}{r b_2 b_3 (A^3)}.
\]
In particular, if $r b_2 b_3 (A^3) \le 4$, then $\alpha_{\msp} (X) \ge 1/2$.
\end{Lem}

\begin{proof}
Let $\rho_{\msp} \colon \breve{U}_{\msp} \to U_{\msp}$ be the orbifold chart of $X$ containing $\msp = \msp_z$.
We set $\rho = \rho_{\msp}$, $\breve{U} = \breve{U}_{\msp}$ and $U = U_{\msp}$.
We see that $\breve{U}$ is the hypersurface in $\breve{\mcU}_{\msp} = \mbA^4_{\breve{x}_1, \breve{x}_2, \breve{x}_3, \breve{y})}$ defined by the equation
\[
F (\breve{x}_1, \breve{x}_2, \breve{x}_3, \breve{y}, 1) = 0.
\]
We see that $\breve{U}$ is smooth and the morphism $\rho$ can be identified with the quotient morphism of $\msp \in X$ over a suitable analytic neighborhood of $\msp$.
We denote by $\breve{\msp} \in \breve{U}$ the origin which is the preimage of $\msp$ via $\rho$.
By the assumption, we can choose $\breve{x}_1, \breve{x}_2, \breve{x}_3$ as a system of local coordinates of $\breve{U}$ at $\breve{\msp}$.

We set
\[
\lambda := \frac{2}{r b_2 b_3 (A^3)}
\]
and assume that $\alpha_{\msp} (X) < \lambda$.
Then there exists an irreducible $\mbQ$-divisor $D \sim_{\mbQ} A$ such that the pair $(X, \lambda D)$ is not log canonical at $\msp$.
In particular the pair $(\breve{U}, \lambda \rho^*D)$ is not log canonical at $\breve{\msp}$.
Let $\varphi \colon V \to \breve{U}$ be the blowup of $\breve{U}$ at $\breve{\msp}$ with exceptional divisor $E \cong \mbP^2$. 
By Lemma \ref{lem:2nineq}, there exists a line $L \subset E$ with the property that for any prime divisor $T$ on $\breve{U}$ such that $T$ is smooth at $\breve{\msp}$ and that its proper transform $\tilde{T}$ contains $L$, we have $\mult_{\breve{\msp}} (D|_T) > 2/\lambda$.
By a slight abuse of notation, we have an isomorphism $E \cong \mbP^2_{\breve{x}_1, \breve{x}_2, \breve{x}_3}$.
The line $L \subset E$ is isomorphic to $(\alpha_1 \breve{x}_1 + \alpha_2 \breve{x}_2 + \alpha_3 \breve{x}_3 = 0) \subset \mbP^2$, for some $\alpha_1, \alpha_2, \alpha_3 \in \mbC$ with $(\alpha_1, \alpha_2, \alpha_3) \ne (0, 0, 0)$.
We set 
\[
\breve{T} := (\alpha_1 \breve{x}_1 + \alpha_2 \breve{x}_2 + \alpha_3 \breve{x}_3 = 0) \subset \breve{U}.
\]
Then $\breve{T}$ is smooth at $\breve{\msp}$ and its proper transform on $V$ contains $L$.
It follows that $\mult_{\breve{\msp}} (\rho^*D|_{\breve{T}}) > 2/\lambda$. 
Set $k := \max \{\,  i \mid \alpha_i \ne 0 \, \}$.
We have $r \breve{T} = \rho^*G$ for some effective Weil divisor $G \sim r b_k A$.
Let $j \in \{1, 2, 3\}$ be such that 
\[
b_j = \max \{\, b_i \mid 1 \le i \le 3, i \ne k \, \}.
\]
Then, since $\{x_1, x_2, x_3\}$ isolates $\msp$, we can take an effective $\mbQ$-divisor $S \sim_{\mbQ} A$ such that $\omult_{\msp} (S) \ge 1/b_j$ and $\rho^*S$ does not contain any component of $\rho^*D|_{\breve{T}}$.
Hence we have
\[
 r b_k (A^3) = (D \cdot G \cdot S) \ge (\rho^*D \cdot \breve{T} \cdot \rho^*S)_{\breve{\msp}} > \frac{2}{b_j \lambda} = \frac{r b_2 b_3 (A^3)}{b_j}.
\]
This is a contradiction since $b_j b_k \le b_2 b_3$, and the proof is completed.
\end{proof}

\subsubsection{Computations by $\bNE$}

Let $X$ be a quasi-smooth Fano 3-fold weighted hypersurface of index $1$.
Let $\msp \in X$ be a singular point and we denote by $\varphi \colon Y \to X$ the Kawamata blowup at $\msp$.
In \cite{CP17}, the assertion $(-K_Y)^2 \notin \Int \bNE (Y)$ is verified in many cases, where $\bNE (Y)$ is the cone of effective curves on $Y$.
Thus the following result is very useful.

\begin{Lem}[{\cite[Lemma 2.8]{KOW18}}] \label{lem:singptNE}
Let $\msp \in X$ be a terminal quotient singular point and $\varphi \colon Y \to X$ the Kawamata blowup at $\msp$.
Suppose that $(-K_Y)^2 \notin \Int \bNE (Y)$ and there exists a prime divisor $S$ on $X$ such that $\tilde{S} \sim_{\mbQ} - m K_Y$ for some $m > 0$, where $\tilde{S}$ is the proper transform of $S$ on $Y$.
Then $\alpha_{\msp} (X) \ge 1$.
\end{Lem}

\section{Smooth points} \label{chap:smpt}

The aim of this section is to prove the following.

\begin{Thm} \label{thm:smpt}
Let $X$ be a member of a family $\mcF_{\msi}$ with $\msi \in \msI \setminus \msI_1$.
Then, 
\[
\alpha_{\msp} (X) \ge \frac{1}{2}
\] 
for any smooth point $\msp \in X$.
\end{Thm}

We explain the organization of this section.
Throughout the present section, let 
\[
X = X_d \subset \mbP (1, a_1, a_2, a_3, a_4)_{x, y, z, t, w}
\] 
be a member of $\mcF_{\msi}$ with $\msi \in \msI \setminus \msI_1$, where we assume $a_1 \le \cdots \le a_4$.
Note that $a_2 \ge 2$ since $\msi \notin \msI_1$.
Recall that we denote by $F = F (x, y, z, t, w)$ the defining polynomial of $X$ with $\deg F = d$ and we set $A := -K_X$.
We set
\[
\begin{split}
U_1 &:= \bigcup_{v \in \{x,y,z,t,w\}, \deg v = 1} (v \ne 0) \cap X, \\
L_{xy} &:= H_x \cap H_y = (x = y = 0) \cap X.
\end{split}
\]
Note that $U_1$ is an open subset of $X$ contained in $\Sm (X)$, and $L_{xy}$ is a $1$-dimensional closed subset of $X$.
The proof of the inequality $\alpha_{\msp} (X) \ge 1/2$ for $\msp \in U_1$ will be done in \S \ref{sec:smptU1}.
The proof for the other smooth points will be done as follows.

\begin{itemize}
\item If $1 < a_1 < a_2$, then $\Sm (X) \subset U_1 \sqcup (H_x \setminus L_{xy}) \sqcup L_{xy}$.
In this case, the proof of $\alpha_{\msp} (X) \ge 1/2$ for smooth point $\msp$ of $X$ contained in $H_x \setminus L_{xy}$ (resp.\ $L_{xy}$) will be done in \S \ref{sec:smptHminusL} (resp.\ \S \ref{sec:smptL1} and \S \ref{sec:smptL2}), respectively. 
\item If $1 = a_1 < a_2$, then $\Sm (X) \subset U_1 \sqcup L_{xy}$.
In this case, the proof of $\alpha_{\msp} (X) \ge 1/2$ for smooth point $\msp$ of $X$ contained in $L_{xy}$ will be done in \S \ref{sec:smptL1} and \S \ref{sec:smptL2}. 
\item If $1 < a_1 = a_2$, then $\Sm (X) \subset U_1 \sqcup H_x$.
In this case, the proof of $\alpha_{\msp} (X) \ge 1/2$ for smooth point $\msp$ of $X$ contained in $H_x$ will be done in \S \ref{sec:smptH}. 
\end{itemize}
Therefore Theorem \ref{prop:smptU1} will follow from Propositions \ref{prop:smptU1}, \ref{prop:smptHminusL}, \ref{prop:smptL1}, \ref{prop:smptL2} and \ref{prop:smptH}, which are the main results of \S \ref{sec:smptU1}, \S \ref{sec:smptHminusL}, \S \ref{sec:smptL1}, \S \ref{sec:smptL2} and \S \ref{sec:smptH}, respectively.

\subsection{Smooth points on $U_1$ for families indexed by $\msI \setminus \msI_1$} \label{sec:smptU1}

\begin{Lem} \label{lem:nsptU1-1}
We have
\[
\alpha_{\msp} (X) \ge \frac{1}{a_2 a_4 (A^3)}
\]
for any point $\msp \in U_1$.
\end{Lem}

\begin{proof}
Let $\msp \in U_1$ be a point.
We may assume $\msp = \msp_x$ by a change of coordinates.
Let $D \in |A|_{\mbQ}$ be an irreducible $\mbQ$-divisor.
The linear system $|\mcI_{\msp} (a_2 A)|$ is movable, and let $S \in |\mcI_{\msp} (a_2 A)|$ be a general member so that $\Supp (S) \ne \Supp (D)$.
The set $\{y, z, t, w\}$ isolates $\msp$, and hence we can take a $\mbQ$-divisor $T \in |a_4 A|_{\mbQ}$ such that $\mult_{\msp} (T) \ge 1$ and $\Supp (T)$ does not contain any component of the effective $1$-cycle $D \cdot S$ (see Remark \ref{rem:isolT}).
Then we have
\[
\mult_{\msp} (D) \le (D \cdot S \cdot T)_{\msp} \le (D \cdot S \cdot T) = a_2 a_4 (A^4).
\]
This shows $\lct_{\msp} (X;D) \ge 1/a_2 a_4 (A^3)$ and the proof is completed.
\end{proof}

\begin{Lem} \label{lem:nsptU1-2}
Suppose that $d$ is divisible by $a_4$.
Then 
\[
\alpha_{\msp} (X) \ge \frac{1}{a_2 a_3 (A^3)}
\] 
for any point $\msp \in U_1$.
\end{Lem}

\begin{proof}
Let $\msp \in U_1$ be a point.
We may assume $\msp = \msp_x$.
We can choose coordinates so that $w^{d/a_4} \in F$.
Indeed, if $a_4 > a_3$, then $w^{d/a_4} \in F$ by the quasi-smoothness of $X$. 
If $a_4 = a_3$, then there is a monomial of degree $d$ consisting of $t, w$ by the quasi-smoothness of $X$ and we can choose coordinates $t, w$ so that $w^{d/a_4} \in F$.
Under the above choice of coordinates, we see that $\{y, z, t\}$ isolates $\msp$.

Let $D \in |A|_{\mbQ}$ be an irreducible $\mbQ$-divisor.
Let $S$ be a general member of the movable linear system $|\mcI_{\msp} (a_2 A)|$, so that $\Supp (S)$ does not contain $\Supp (D)$.
We can take a $\mbQ$-divisor $T \in |a_3 A|_{\mbQ}$ such that $\mult_{\msp} (T) \ge 1$ and $\Supp (T)$ does not contain any component of the effective $1$-cycle $D \cdot S$ since $\{y, z, t\}$ isolates $\msp$.
Then, we have
\[
\mult_{\msp} (D) \le (D \cdot S \cdot T)_{\msp} \le (D \cdot S \cdot T) = a_2 a_3 (A^3).
\] 
This shows $\alpha_{\msp} (D) \ge 1/a_2 a_3 (A^3)$ and the proof is completed.
\end{proof}

\begin{Rem} \label{rem:lemsmpttrue}
The objects of Section~\ref{chap:smpt} are members of families $\mcF_{\msi}$ for $\msi \in \msI \setminus \msI_1$ and the inequality $a_2 \ge 2$ is assumed throughout  the present section.
It is however noted that in Lemmas \ref{lem:nsptU1-1} and \ref{lem:nsptU1-2}, the assumption $a_2 \ge 2$ is not required and the statement holds for members of $\mcF_{\msi}$ for any $\msi \in \msI$.
\end{Rem}

\begin{Lem} \label{lem:nsptU1-3}
Suppose that $d$ is not divisible by $a_4$, and assume $a_1 = 1$.
Then,
\[
\alpha_{\msp} (X) \ge \min \left\{ \lct_{\msp} (X;S_{\msp}), \ \frac{1}{a_4 (A^3)} \right\} \ge \frac{1}{2}
\] 
for any $\msp \in U_1$, where $S_{\msp}$ is the unique member of $|\mcI_{\msp} (A)|$.
\end{Lem}

\begin{proof}
Let $\msp \in U_1$ be a point.
We may assume $\msp = \msp_x$.
Note that we have $a_2 > 1$ and thus the linear system $|\mcI_{\msp} (A)|$ indeed consists of a unique member $S_{\msp}$.
In this case $S_{\msp} = H_y$.

We first prove $\lct_{\msp} (X;S_{\msp}) \ge 1/2$, that is, $(X, \frac{1}{2} H_y)$ is log canonical at $\msp$.
Assume to the contrary that $(X, \frac{1}{2} H_y)$ is not log canonical at $\msp$.
Then $\mult_{\msp} (H_y) \ge 3$.
Suppose $\mult_{\msp} (H_y) = 3$.
Then, by Lemma \ref{lem:lctP2cubic}, the degree $3$ part of $F (1,0,z,t,w)$ with respect to $\deg (z,t,w) = (1,1,1)$ is a cube of a linear form, that is, it can be written as $(\alpha z + \beta t + \gamma w)^3$ for some $\alpha, \beta, \gamma \in \mbC$.
By Lemma \ref{lem:smptHLdegwt}, we have $d < 3 a_4$ since $d$ is not divisible by $a_4$.
From this we deduce $\gamma = 0$.
Then we can write
\[
F = x^{d-1} y + x^{d-3 a_3} (\alpha x^{a_3-a_2} z + \beta t)^3 + g + y h,
\]
where $g = g (x,z,t,w) \in (z,t,w)^4 \subset \mbC [x,z,t,w]$ and $h = h (x,y,z,t,w)$.
By the inequality $d < 3 a_4$, no monomial in $g$ can be divisible by $w^3$, so that $g \in (z,t)^2$.
But then $X$ is not quasi-smooth along the non-empty subset
\[
(y = x^{d-1} + h = z = t = 0) \subset X.
\]
This is impossible and we have $\mult_{\msp} (H_y) \ge 4$.
By the same argument as above, we can write
\[
F = x^{d-1} y + g + y h,
\]
where $g \in (z,t)^2$, which implies that $X$ is not quasi-smooth.
This is a contradiction and thus $\lct_{\msp} (X; S_{\msp}) \ge 1/2$.

Let $D \in |A|_{\mbQ}$ be an irreducible $\mbQ$-divisor other than $S_{\msp} = H_y$.
Note that $D \cdot H_y$ is an effective $1$-cycle.
The set $\{y, z, t, w\}$ isolates $\msp$, and hence we can take a $\mbQ$-divisor $T \in |a_4 A|_{\mbQ}$ such that $\mult_{\msp} (T) \ge 1$ and $\Supp (T)$ does not contain any component of $D \cdot H_y$.
We have
\[
\mult_{\msp} (D) \le (D \cdot H_y \cdot T) \le a_4 (A^3) \le 2,
\]
where the last inequality follows from (5) of Lemma \ref{lem:wtnumerics} since $a_1 = 1$.
Thus $\lct_{\msp} (X;D) \ge 1/a_4 (A^3) \ge 1/2$ and the proof is completed.
\end{proof}

\begin{Lem} \label{lem:nsptU1-4}
Let $X$ be a member of a family $\mcF_{\msi}$ with $\msi \in \{9,17\}$.
Then 
\[
\alpha_{\msp} (X) \ge \frac{1}{2}
\]
for any point $\msp \in U_1$.
\end{Lem}

\begin{proof}
In this case,
\[
X = X_{3 a + 3} \subset \mbP (1, 1, a, a+1, a+1)_{x, y, z, t, w},
\] 
where $a = 2, 3$ if $\msi = 9, 17$, respectively.
Let $\msp \in U_1$ be a point.
We may assume $\msp = \msp_x$.

We show that $(X, \frac{1}{2} H_y)$ is log canonical at $\msp$.
Assume to the contrary that it is not.
Then $\mult_{\msp} (H_y) \ge 3$ and, by Lemma \ref{lem:lctP2cubic}, we can write
\[
F = x^{3 a + 2} y + x^3 (\alpha z x + \beta t + \gamma w)^3 + g + y h,
\]
where $g = g (x,z,t,w) \in (z,t,w)^4 \subset \mbC [x,y,t,w]$ and $h = h (x,y,z,t,w)$.
By degree reason, any monomial in $g \in (z,t,w)^4$ is divisible by $z^3$.
It follows that $X$ is not quasi-smooth along the non-empty subset
\[
(y = x^{3a+2} + h = z = \beta t + \gamma w = 0) \subset X.
\]
This is a contradiction and the pair $(X, \frac{1}{2} H_y)$ is log canonical at $\msp$.

Let $D \in |A|_{\mbQ}$ be an irreducible $\mbQ$-divisor other than $H_y$.
The set $\{y, z, t, w\}$ clearly isolates $\msp$, and hence we can take a $\mbQ$-divisor $T \in |(a+1) A|_{\mbQ}$ such that $\mult_{\msp} (T) \ge 1$ and $\Supp (T)$ does not contain any component of the effective $1$-cycle $D \cdot H_y$.
Then, we have
\[
2 \mult_{\msp} (D) \le (D \cdot H_y \cdot T)_{\msp} \le (D \cdot H_y \cdot T) = \frac{3}{a} \le \frac{3}{2},
\] 
since $\mult_{\msp} (H_y) \ge 2$.
This shows $\lct_{\msp} (X, D) \ge 4/3$ and the proof is completed.
\end{proof}

\begin{Prop} \label{prop:smptU1}
Let $X$ be a member of a family $\mcF_{\msi}$ with $\msi \in \msI \setminus \msI_1$.
Then, 
\[
\alpha_{\msp} (X) \ge 1
\] 
for any point $\msp \in U_1$.
\end{Prop}

\begin{proof}
Let $X = X_d \subset \mbP (1,a_1,a_2,a_3,a_4)$ be a member of $\mcF_{\msi}$ with $\msi \in \msI \setminus \msI_1$, where we assume $a_1 \le \cdots \le a_4$.
\begin{itemize}
\item If $d$ is not divisible by $a_4$ and $a_1 \ge 2$, then $a_2 a_4 (A^3) \le 2$ and the assertion follows from Lemma \ref{lem:nsptU1-1}.
\item If $d$ is not divisible by $a_4$ and $a_1 = 1$, then the assertion follows from Lemma \ref{lem:nsptU1-3}.
\item If $d$ is divisible by $a_4$ and $\msi \notin \{9, 17\}$, then $a_2 a_3 (A^3) \le 2$ by Lemma \ref{lem:wtnumerics} and the assertion follows from Lemma \ref{lem:nsptU1-2}. 
\item If $\msi \in \{9, 17\}$, then the assertion follows from Lemma \ref{lem:nsptU1-4}.
\end{itemize}
This completes the proof.
\end{proof}

\subsection{Smooth points on $H_x \setminus L_{xy}$ for families with $1 < a_1 < a_2$} \label{sec:smptHminusL}

Let 
\[
X = X_d \subset \mbP (1, a_1, \dots, a_4)_{x, y, z, t, w}
\] 
be a member of a family $\mcF_{\msi}$ with $\msi \in \msI \setminus \msI_1$ satisfying $1 < a_1 < a_2 \le a_3 \le a_4$.
In this section, we set $\bar{F} = F (0,y,z,t,w)$.
Then $H_x$ is isomorphic to the weighted hypersurface in $\mbP (a_1, \dots, a_4)$ defined by $\bar{F} = 0$.
We note that if a smooth point $\msp \in X$ contained in $H_x$ satisfies $\mult_{\msp} (H_x) > 2$, then $\msp$ belongs to the subset
\[
\bigcap_{v_1, v_2 \in \{y,z,t,w\}} \left( \frac{\prt^2 \bar{F}}{\prt v_1 \prt v_2} = 0 \right) \cap X.
\]

The following is the main result of this section.

\begin{Prop} \label{prop:smptHminusL}
Let $X = X_d \subset \mbP (1, a_1, \dots, a_4)$, $a_1 \le \dots \le a_4$, be a member of a family $\mcF_{\msi}$ with $\msi \in \msI \setminus \msI_1$ satisfying $1 < a_1 < a_2$.
Then 
\[
\alpha_{\msp} (X) \ge 1
\] 
for any smooth point $\msp$ of $X$ contained in $H_x \setminus L_{xy}$.
\end{Prop}

The rest of this section is entirely devoted to the proof of Proposition \ref{prop:smptHminusL} which will be done by division into several cases.

\subsubsection{Case: $1 < a_1 < a_2$ and $d = 2 a_4$}

We will prove Proposition \ref{prop:smptHminusL} under the assumption of $1 < a_1 < a_2$ and $d = 2 a_4$.

Let $\msp \in X$ be a smooth point contained in $H_x \setminus L_{xy}$. 
We have $\mult_{\msp} (H_x) \le 2$ since $w^2 \in F$, and thus $\lct_{\msp} (X; H_x) \ge 1/2$.
Let $D \in |A|_{\mbQ}$ be an irreducible $\mbQ$-divisor on $X$ other than $H_x$. 
By Lemma \ref{lem:isolclass}, $a_1 a_3 A$ isolates $\msp$, and hence
 we can take a $\mbQ$-divisor $T \in |a_1 a_3 A|_{\mbQ}$ such that $\mult_{\msp} (T) \ge 1$ and $\Supp (T)$ does not contain any component of the effective $1$-cycle $D \cdot H_x$.
Then we have
\[
\mult_{\msp} (D) \le (D \cdot H_x \cdot T)_{\msp} \le (D \cdot H_x \cdot T) = a_1 a_3 (A^3) \le 1,
\]
where the last inequality follows from Lemma \ref{lem:wtnumerics}.
This shows $\lct_{\msp} (X;D) \ge 1$ and the proof is completed.

\subsubsection{Case: $1 < a_1 < a_2$ and $d = 2 a_4 + a_1$}

We will prove Proposition \ref{prop:smptHminusL} under the assumption of $1 < a_1 < a_2$ and $d = 2 a_4 + a_1$.

Let $\msp \in X$ be a smooth point contained in $H_x \setminus L_{xy}$. 
We can write
\[
F = w^2 y + w f + g,
\]
where $f, g \in \mbC [x, y, z, t]$ are quasi-homogeneous polynomials of degree $d-a_4$ and $d$ respectively.
Since $\prt^2 \bar{F}/\prt w^2 = y$ and $y$ does not vanish at $\msp \in H_x \setminus L_{xy}$, we have $\mult_{\msp} (H_x) \le 2$ and thus $\lct_{\msp} (X;H_x) \ge 1/2$. 

Let $D \in |A|_{\mbQ}$ be an irreducible $\mbQ$-divisor on $X$ other than $H_x$.
By Lemma \ref{lem:isolclass}, $a_1 a_4 A$ isolates $\msp$, and hence we can take a $\mbQ$-divisor $T \in |a_1 a_4 A|_{\mbQ}$ such that $\mult_{\msp} (T) \ge 1$ and $\Supp (T)$ does not contain any component of the effective $1$-cycle $D \cdot H_x$.
Then we have
\[
\mult_{\msp} (D) \le (D \cdot H_x \cdot T)_{\msp} \le (D \cdot H_x \cdot T) = a_1 a_4 (A^3) \le 2,
\]
where the last inequality from Lemma \ref{lem:wtnumerics}.
This shows $\lct_{\msp} (X;D) \ge 1/2$ and the proof is completed.

\subsubsection{Case: $1 < a_1 < a_2$ and $d = 2 a_4 + a_2$} \label{sec:smptHminusLc}

We will prove Proposition \ref{prop:smptHminusL} under the assumption of $1 < a_1 < a_2$ and $d = 2 a_4 + a_2$.

Let $\msp \in X$ be a smooth point contained in $H_x \setminus L_{xy}$. 
We can write
\[
F = w^2 z + w (z f + g) + h,
\]
where $f, h \in \mbC [x,y,z,t]$ and $g \in \mbC [x,y,t]$ are quasi-homogeneous polynomials of degrees $a_4, a_2 + a_4$ and $d$, respectively.

\begin{Claim} \label{clm:smptHminusL1}
$\lct_{\msp} (X; H_x) \ge 1/2$.
\end{Claim}

\begin{proof}[Proof of Claim \ref{clm:smptHminusL1}]
We prove $\mult_{\msp} (H_x) \le 2$.
Assume to the contrary that $\mult_{\msp} (H_x) > 2$.
Since 
\[
\frac{\prt^2 \bar{F}}{\prt w^2} = z, \quad \frac{\prt^2 \bar{F}}{\prt w \prt z} = 2 w + f,
\] 
the point $\msp$ is contained in $H_x \cap H_z \cap (2 w + f = 0)$.
Suppose in addition that $\msp \in H_t$.
Note that we have $a_4 = a_1 + a_3$ since $d = a_1 + a_2 + a_3 + a_4$ and $d = 2 a_4 + a_2$.
We see that $a_1$ does not divide $a_4 = a_1 + a_3$ because otherwise $\gcd \{a_1, a_3, a_4\} > 1$ and $X$ has a non-isolated singularity which is clearly worse than terminal, a contradiction.
It follows that $f$ does not contain a power of $y$, i.e.\ $f (0,y,0) = 0$, and
\[
\msp \in H_x \cap H_z \cap H_t \cap (2 w + f = 0) = H_x \cap H_z \cap H_t \cap H_w = \{\msp_y\}.
\]
This is impossible since $\msp_y$ is a singular point of $X$.
Thus $\msp \notin H_t$.
Since $a_4 = a_1 + a_3$, we may assume that $\msp \in H_w$ after replacing $w$ by $w - \xi y t$ for some $\xi \in \mbC$.
We can write $\msp = (0\!:\!1\!:\!0\!:\!\lambda\!:\!0)$ for some non-zero $\lambda \in \mbC$.
The set $\{x, z, w, t^{a_1} - \lambda^{a_1} y^{a_3}\}$ isolates $\msp$, and hence we can take a $\mbQ$-divisor $T \in |a_1 a_3 A|_{\mbQ}$ such that $\mult_{\msp} (T) \ge 1$ and $\Supp (T)$ does not contain any component of $H_x \cdot H_z$.
Then we have
\[
\mult_{\msp} (H_x) \le (H_x \cdot H_z \cdot T)_{\msp} \le (H_x \cdot H_z \cdot T) = a_1 a_2 a_3 (A^3) < 3,
\]
where the last inequality follows from Lemma \ref{lem:wtnumerics}.
This shows $\mult_{\msp} (H_x) \le 2$, and thus $\lct_{\msp} (X; H_x) \ge 1/2$.
\end{proof}

Let $D \in |A|_{\mbQ}$ be an irreducible $\mbQ$-divisor on $X$ other than $H_x$.
By Lemma \ref{lem:isolclass}, $a_1 a_4 A$ isolates $\msp$, and hence we can take a $\mbQ$-divisor $T \in |a_1 a_4 A|_{mbQ}$ such that $\mult_{\msp} (T) \ge 1$ and $\Supp (T)$ does not contain any component of $D \cdot H_x$.
We have
\[
\mult_{\msp} (D) \le (D \cdot H_x \cdot T)_{\msp} \le (D \cdot H_x \cdot T) = a_1 a_4 (A^3) < 2,
\]
where the last inequality follows from Lemma \ref{lem:wtnumerics}.
Thus $\lct_{\msp} (X; D) > 1/2$ and we conclude $\alpha_{\msp} (X) \ge 1/2$.

\subsubsection{Case: $1 < a_1 < a_2$, $d = 2 a_4 + a_3$ and $a_3 \ne a_4$}

The proof of Proposition \ref{prop:smptHminusL} under the assumption of $1 < a_1 < a_2$ and $d = 2 a_4 + a_3$ is completely parallel to the one given in \S \ref{sec:smptHminusLc}.
Indeed, the same proof applies after interchanging the role of $z$ and $t$ (and hence $a_2$ and $a_3$).
Thus we omit the proof.

\subsubsection{Case: $1 < a_1 < a_2$ and $d = 3 a_4$}

We will prove Proposition \ref{prop:smptHminusL} under the assumption of $1 < a_1 < a_2$ and $d = 3 a_4$.

Let $\msp \in X$ be a smooth point contained in $H_x \setminus L_{xy}$. 
We first prove $\alpha_{\msp} (X) \ge 1/2$ assuming that the inequality $\mult_{\msp} (H_x) \le 2$ holds.
Let $D \in |A|_{\mbQ}$ be an irreducible $\mbQ$-divisor on $X$ other than $H_x$. 
By Lemma \ref{lem:isolclass}, $a_1 a_4 A$ isolates $\msp$, and hence we can take a $\mbQ$-divisor $T \in |a_1 a_4 A|_{\mbQ}$ such that $\mult_{\msp} (T) \ge 1$ and $\Supp (T)$ does not contain any component of the effective $1$-cycle $D \cdot H_x$.
Then we have
\[
\mult_{\msp} (D) \le (D \cdot H_x \cdot T)_{\msp} \le (D \cdot H_x \cdot T) = a_1 a_4 (A^3) \le 2,
\]
where the last inequality follows from Lemma \ref{lem:wtnumerics}.
This shows $\mult_{\msp} (D) \le 2$ and we have $\lct_{\msp} (X;D) \ge 1/2$.

Therefore the proof of Proposition \ref{prop:smptHminusL} under the assumption of $1< a_1 < a_2$ and $d = 3 a_4$ is reduced to the following.

\begin{Claim} \label{cl:smHcase3e}
We have $\mult_{\msp} (H_x) \le 2$ for any smooth point $\msp \in X$ contained in $H_x \setminus L_{xy}$.
\end{Claim}

The rest of this subsection is devoted to the proof of Claim \ref{cl:smHcase3e}, which will be done by considering each family individually.
The families satisfying $1 < a_1 < a_2$ and $d = 3 a_4$ are families $\mcF_{\msi}$, where
\[
\msi \in \{19, 27, 39, 49, 59, 66, 84\}.
\]
In the following, for a polynomial $f (x,y,z,\dots)$ in variables $x, y, z, \dots$, we set $\bar{f} = \bar{f} (y,z,\dots) := f (0,y,z,\dots)$.
We first consider the family $\mcF_{27}$, which is the unique family satisfying $d = 3 a_3 = 3 a_4$
We then consider the rest of the families which satisfy $d = 3 a_4 > 3 a_3$.

\paragraph{\it The family $\mcF_{27}$}
We can choose $w$ and $t$ so that
\[
F = w^2 t + w t^2 + w t b_5 + w c_{10} + t d_{10} + e_{15},
\]
where $b_5, c_{10}, d_{10}, e_{15} \in \mbC [x,y,z]$ are quasi-homogeneous polynomials of indicated degrees.  
Let $\msp \in H_x \setminus L_{xy}$ be a smooth point of $X$ and we assume $\mult_{\msp} (H_x) > 2$.
Since
\[
\frac{\prt^2 \bar{F}}{\prt w^2} = 2 t, \quad
\frac{\prt^2 \bar{F}}{\prt t^2} = 2 w,
\]
we have $\msp \in H_t \cap H_w$.
Then we can write $\msp = (0\!:\!1\!:\!\lambda\!:\!0\!:\!0)$ for some non-zero $\lambda \in \mbC$ since $\msp \notin H_y$ and $\msp \ne \msp_y$.
We can write $\bar{e}_{15} = z (z^2 - \lambda^2 y^3)(z^2 - \mu y^3)$ for some $\mu \in \mbC$.
We have
\[
\begin{split}
\frac{\prt^2 \bar{F}}{\prt z^2} (\msp) &= 2 \lambda (7 \lambda^2 - 3 \mu) = 0, \\
\frac{\prt^2 \bar{F}}{\prt z \prt y} (\msp) &= - 3 \lambda^2 (3 \lambda^2 + \mu) = 0,
\end{split}
\]
which implies $\lambda = 0$.
This is a contradiction and Claim \ref{cl:smHcase3e} is proved for the family $\mcF_{27}$.

We consider the rest of the families, which satisfies $d = 3 a_4 > 3 a_3$.
Replacing $w$ if necessary, we can write
\[
F = w^3 + w g_{2 a_4} + h_{3 a_4},
\]
where $g_{2 a_4}, h_{3 a_4} \in \mbC [x,y,z,t]$ are quasi-homogeneous polynomials of degree $2 a_4, 3 a_4$, respectively.
Let $\msp \in H_x \setminus L_{xy}$ be a smooth point of $X$ and we assume $\mult_{\msp} (H_x) > 2$.
Since $\prt^2 \bar{F}/\prt w^2 = 6 w$, we have $\msp \in H_w$, so that $\msp \in H_x \cap H_w$ and $\msp \notin H_y$.
In the following we will derive a contradiction by considering each family separately.

\paragraph{\it The family $\mcF_{19}$}
Replacing $t \mapsto t - \xi z$ for a suitable $\xi \in \mbC$, we may assume $\msp \in H_t$. 
Since $\msp \in H_x \cap H_t \cap H_w$, $\msp \notin H_y$ and $\msp \ne \msp_y$, we have $\msp = (0\!:\!1\!:\!\lambda\!:\!0\!:\!0)$ for a non-zero $\lambda \in \mbC$.
We can write $\bar{h}_{12} = (z^2 - \lambda^2 y^3)(z^2 - \mu y^3) + t e_9$ for some $\mu \in \mbC$ and $e_9 = e_9 (y,z,t)$.
It is then straightforward to check that
\[
\frac{\prt^2 \bar{F}}{\prt z^2} (\msp) = \frac{\prt^2 \bar{F}}{\prt z \prt y} (\msp) = 0
\]
is impossible, and this is a contradiction.

\paragraph{\it The family $\mcF_{39}$}
We have $g_{2 a_4} = g_{12}$, $h_{3 a_4} = h_{18}$ and we can write
\[
g_{12} (0,y,z,t) =  \alpha t z y + \lambda z^3 + \mu y^3,
\]
where $\alpha, \lambda, \mu \in \mbC$.
By the quasi-smoothness of $X$, we have $\lambda \ne 0$ and $\mu \ne 0$. 
We have 
\[
\frac{\prt^2 \bar{F}}{\prt w \prt t} = \alpha z y, \quad
\frac{\prt^2 \bar{F}}{\prt w \prt z} = \alpha t y + 3 \lambda z^2, \quad
\frac{\prt^2 \bar{F}}{\prt w \prt y} = \alpha t z + 3 \mu y^2.
\]
Suppose that $\alpha \ne 0$, then, since both $\alpha z y$ and $\alpha t y + 3 \lambda z^2$ vanish at $\msp$ and $y$ does not vanish at $\msp$, we have $\msp \in H_z \cap H_t$.
It follows that $\msp = \msp_y$.
This is impossible since $\msp_y$ is a singular point of $X$.
Thus $\alpha = 0$.
Then, both $3 \lambda z^2$ and $\alpha t z + 3 \mu y^2$ vanish at $\msp$, which implies that $y$ vanishes at $\msp$.
This is a contradiction.

\paragraph{\it The family $\mcF_{49}$}
We have $g_{2 a_4} = g_{14}, h_{3 a_4} = h_{21}$ and we can write
\[
h_{24} (0,y,z,t) = \lambda t^3 y + \alpha t^2 y^3 + \beta t z^3 + \gamma t y^5 + \delta z^3 y^2 + \varepsilon y^7,
\]
where $\lambda, \alpha, \beta, \dots, \varepsilon \in \mbC$.
By the quasi-smoothness of $X$, we have $\lambda \ne 0$, and by replacing $t$, we can assume that $\alpha = 0$.
Since $\msp \in H_w$, $\msp \notin H_y$ and 
\[
\frac{\prt^2 \bar{F}}{\prt t^2} = w \frac{\prt^2 g_{14} (0, y, z, t)}{\prt t^2} + 6 \lambda t y,
\]
we have $\msp \in H_t$.
Then $\msp \notin H_z$ because otherwise $\msp = \msp_y$ is a singular point and this is impossible.
We have
\[
\frac{\prt^2 \bar{F}}{\prt z \prt y} = w \frac{\prt g_{14} (0,y,z,t)}{\prt z \prt y} + 6 \delta z^2 y,
\]
which implies $\delta = 0$.
Then, by the quasi-smoothness of $X$, we have $\varepsilon \ne 0$.
But the polynomial
\[
\frac{\prt^2 \bar{F}}{\prt y^2} = w \frac{\prt^2 g_{14} (0,y,z,t)}{\prt y^2} + 20 \gamma t y^3 + 42 \varepsilon y^5
\]
does not vanish at $\msp$.
This is a contradiction.

\paragraph{\it The family $\mcF_{59}$}
We have $g_{2 a_4} = g_{16}, h_{3 a_4} = h_{24}$ and we can write
\[
h_{24} (0,y,z,t) = \lambda t^3 y + \mu z^4 + \alpha z^3 y^2 + \beta z^2 y^4 + \gamma z y^6 + \delta y^8,
\]
where $\lambda, \mu, \alpha, \beta, \gamma, \delta \in \mbC$.
By the quasi-smoothness of $X$, we have $\lambda \ne 0$ and $\mu \ne 0$.
Since $\msp \notin H_y$, $\lambda \ne 0$ and
\[
\frac{\prt^2 \bar{F}}{\prt t^2} = w \frac{\prt^2 g_{16} (0,y,z,t)}{\prt t^2} + 6 \lambda t y,
\]
we have $\msp \in H_t$.
This is a contradiction since $\msp \in H_x \cap H_t \cap H_w$ but $H_x \cap H_t \cap H_w$ consists of singular points.

\paragraph{\it The family $\mcF_{66}$}
We have $h_{3 a_4} = h_{27}$ and we can write
\[
h_{27} (0,y,z,t) = \lambda t^3 z + \mu t y^4 + \alpha z^2 y^3,
\]
where $\alpha, \lambda, \mu \in \mbC$.
By the quasi-smoothness of $X$, we have $\lambda \ne 0$ and $\mu \ne 0$.
We have
\[
\frac{\prt^2 \bar{F}}{\prt t \prt y} = w \frac{\prt^2 g_{18} (0,y,z,t)}{\prt t \prt y} + 4 \mu y^3.
\]
This is a contradiction since $\msp \in H_w$, $\msp \notin H_y$ and $\mu \ne 0$.

\paragraph{\it The family $\mcF_{84}$}
We have
\[
g_{24} (0,y,z,t) = \alpha t z y + \lambda z^3, \quad
h_{36} (0,y,z,t) = \mu t^4 + \beta z y^4,
\] 
where $\alpha, \beta, \lambda, \mu \in \mbC$.
By the quasi-smoothness of $X$, we have $\lambda \ne 0$ and $\mu \ne 0$.
We have
\[
\frac{\prt^2 \bar{F}}{\prt t^2} = w \frac{\prt^2 g_{24} (0,y,z,t)}{\prt t^2} + 12 \mu t^2,
\]
which implies $\msp \in H_t$ since $\msp \in H_w$ and $\mu \ne 0$.
We have
\[
\frac{\prt^2 \bar{F}}{\prt w \prt z} = \frac{g_{24} (0,y,z,t)}{\prt z} = \alpha t y + 3 \lambda z^2,
\]
which implies $\msp \in H_z$ since $\msp \in H_t$ and $\lambda \ne 0$.
This shows $\msp = \msp_y$, and this is a contradiction since $\msp_y$ is a singular point. 

Therefore we derive a contradiction for all families and the proof of Claim \ref{cl:smHcase3e} is completed.

\subsection{Smooth points on $L_{xy}$ for families with $a_1 < a_2$, Part 1} \label{sec:smptL1}

In this and next sections, we consider a member 
\[
X = X_d \subset \mbP (1, a_1, a_2, a_3, a_4)_{x, y, , t, w}
\]
of a family $\mcF_{\msi}$ with $\msi \in \msI \setminus \msI_1$ satisfying $a_1 < a_2 \le a_3 \le a_4$ and prove $\alpha_{\msp} (X) \ge 1/2$ for a smooth point $\msp$ of $X$ contained in the $1$-dimensional scheme $L_{xy} := (x = y = 0) \cap X$.
We divide families indexed by $\msI \setminus \msI_1$ into two types:

\begin{itemize}
\item Families $\mcF_{\msi}$ such that $L_{xy} := (x = y = 0) \cap X$ is irreducible and reduced for any member $X$.
These families are treated in the current section \S \ref{sec:smptL1}.
\item Families $\mcF_{\msi}$ such that $L_{xy}$ is either reducible or non-reduced for some member $X$ (See \eqref{eq:Lxyrednonred} for specific families).
These families will be treated in the next section \S \ref{sec:smptL2}.
\end{itemize}

The objects of this section are members $X$ such that the $1$-dimensional scheme $L_{xy} = (x = y = 0) \cap X$ is irreducible and reduced.

\begin{Lem} \label{lem:Lirred}
Let $X$ be a member of a family $\mcF_{\msi}$ with $\msi \in \msI$ which satisfies $a_1 < a_2$ and which is listed in \emph{Table \ref{table:Lsmooth}} $($resp.\ \emph{Table \ref{table:Lsing}}$)$.
Then, $L_{xy}$ is an irreducible smooth curve $($resp.\ irreducible and reduced curve which is smooth along $L_{xy} \cap \Sm (X)$$)$.
\end{Lem}

\begin{proof}
Let $F$ be the defining polynomial of $X$ and set $d = \deg F$.
The scheme $L_{xy}$ is isomorphic to the hypersurface in $\mbP (a_2,a_3,a_4)_{z, t, w}$ defined by the polynomial $f := F (0,0,z,t,w)$.
We explain that $f$ can be transformed into the polynomial given in Tables \ref{table:Lsmooth} and \ref{table:Lsing} by a suitable change of homogeneous coordinates.

Suppose $\msi \notin \{11, 15, 16, 17, 21, 27, 34\}$.
Then, there are only a few monomials of degree $d$ in variables $z, t, w$.
We first simply express $f$ as a linear combination of those monomials, and then consider the following coordinate change.
\begin{itemize}
\item Suppose $d = 2 a_4$.
In this case, $f$ is quadratic with respect to $w$ and we eliminate the term of the form $w g (z, t)$ by replacing $w$ suitably.
Then, we may assume $f = w^2 + h (z, t)$ for some quasi-homogeneous polynomial $h = h (z, t)$ of degree $d$.
\item Suppose $d = 3 a_4$.
In this case, $f$ is cubic with respect to $w$ and we eliminate the term of the form $w^2 g (z, t)$ by replacing $w$ suitably.
Then, we may assume $f = w^3 + w h_1 (z, t) + h_2 (z,t)$ for some quasi-homogeneous polynomials $h_1 = h_1 (z, t), h = h_2 (z, t)$ of degrees $d - a_4 = 2 a_4$ and $d = 3 a_4$, respectively.
\end{itemize}
After the above coordinate change, we observe that $f$ is a linear combination of at most $3$ distinct monomials and it is possible to make those coefficients $1$ by rescaling $z, t, w$.
The resulting polynomial is the one given in Tables \ref{table:Lsmooth} and \ref{table:Lsing}.
Once an explicit form of the defining polynomial is given, it is then easy to show that $L_{xy}$ is entirely smooth or is smooth along $L_{xy} \cap \Sm (X)$. 

For $\msi = \{11, 15, 16, 17, 21, 27, 34\}$, the description of $f$ is explained as follows.

\begin{itemize}
\item Suppose $\msi = 11$.
In this case, $f = w^2 + h (z,t)$, where $h$ is a quintic form in $z, t$.
The solutions of the equation $h = 0$ correspond to the $5$ singular points of type $\frac{1}{2} (1,1,1)$.
Thus $h$ does not have a multiple component and in particular $L_{xy}$ is smooth.
\item Suppose $\msi = 15$.
In this case, $f = w^2 + \alpha t^4 + \beta t^2 z^3 + \gamma z^6$ for some $\alpha, \beta, \gamma \in \mbC$.
We have $\alpha \ne 0$ (resp.\ $\gamma \ne 0$) because otherwise $X$ cannot be quasi-smooth at $\msp_t$ (resp.\ $\msp_z$).
Replacing $t$ and rescaling $z$, we may assume $\alpha = 1$, $\beta = 0$ and $\gamma = 1$, and we obtain the desired form $f = w^2 + t^4 + z^6$.
It is easy to see that $L_{xy}$ is smooth.
\item Suppose $\msi = 16$.
In this case, $f = w^2 z + h (z, t)$, where $h = \alpha t^3 + \beta t^2 z^2 + \gamma t z^4 + \delta z^6$ for some $\alpha, \beta, \gamma, \delta \in \mbC$.
By the quasi-smoothness of $X$, we have $\alpha \ne 0$.
Moreover the solutions of $h (z, t) = 0$ correspond to $3$ singular points of type $\frac{1}{2} (1,1,1)$.
Thus $h (z, t)$ does not have a multiple component and in particular $L_{xy}$ is smooth.
\item Suppose $\msi = 17$.
In this case, $f = c (t, w) + \alpha z^4$, where $\alpha \in \mbC$ and $c (t, w)$ is a cubic form in $t, w$.
By the quasi-smoothness of $X$, we have $\alpha \ne 0$ and we may assume $\alpha = 1$ by rescaling $z$.
Moreover the solutions of $c (t, w) = 0$ correspond to $3$ singular points of type $\frac{1}{4} (1,1,3)$.
Thus $c (t, w)$ does not have a multiple component and in particular $L_{xy}$ is smooth.
\item Suppose $\msi = 21$.
In this case, $f = w^2 + h (z, t)$, where $h = \alpha t^3 z + \beta t^2 z^3 + \gamma t z^5 + \delta z^7$ for some $\alpha, \beta, \gamma, \delta \in \mbC$.
By the quasi-smoothness of $X$ at $\msp_t$, we have $t^3 z \in F$, that is, $\alpha \ne 0$.
Moreover the solutions of $\alpha t^3 + \beta t^2 z^2 + \gamma t z^4 + \delta z^6 = 0$ correspond to the $3$ singular points of type $\frac{1}{2} (1,1,1)$.
Thus $h$ does not have a multiple component and in particular $L_{xy}$ is smooth.
\item Suppose $\msi = 27$.
In this case, $f = c (t, w) + \alpha z^5$, where $\alpha \in \mbC$ and $c (t, w)$ is a cubic form in $t, w$.
By the same arguments as in the case of $\msi = 17$, $c (t, w)$ does not have a multiple component and we may assume $\alpha =1$.
Thus $L_{xy}$ is smooth.
\item Suppose $\msi = 34$.
In this case, $f = w^2 + h (z, t)$, where $h = \alpha t^3 + \beta t^2 z^3 + \gamma t z^6 + \delta z^9$ for some $\alpha, \beta, \gamma, \delta \in \mbC$.
By the quasi-smoothness of $X$, we have $t^3 \in F$, that is, $\alpha \ne 0$.
Moreover the solutions of $h = 0$ correspond to $3$ singular points of type $\frac{1}{2} (1,1,1)$.
Thus $h$ does not have a multiple component and in particular $L_{xy}$ is smooth. 
\end{itemize}
This completes the proof.
\end{proof}

\begin{table}[h]
\renewcommand{\arraystretch}{1.1}
\begin{center}
\caption{$L_{xy}$: Irreducible and smooth case}
\label{table:Lsmooth}
\begin{tabular}{cc|cc}
No. & Equation & No. & Equation \\
\hline
11 & $w^2 + h (z,t)$ & 55 & $w^2 + t^3 z + z^8$ \\
15 & $w^2 + t^4 + z^6$ & 57 & $w^2 + t^4 z + z^6$  \\
16 & $w^2 z + h (z, t)$, $t^3 \in h$ & 66 & $w^3 + w z^3 + t^3 z$ \\
17 & $c (t,w) + z^4$ & 68 & $w^2 + t^4 + z^7$ \\
19 & $w^3 + h (z,t)$ & 70 & $w^2 + t^3 + t z^5$ \\
21 & $w^2 + h (z, t)$ & 71 & $w^2 + t^3 z + z^5$ \\
26 & $w^2 z + z^5 + t^3$ & 72 & $w^2 + t^3 + z^{10}$ \\
27 & $c (t,w) + z^5$ & 74 & $w^2 z + t^3 + t z^5$ \\
34 & $w^2 + h (z, t)$ & 75 & $w^2  + t^5 + z^6$ \\
35 & $w^2 + t^3 z + z^6$ & 76 & $w^2 t + t^3 z + z^5$ \\
36 & $w^2 z + t^3 + t z^3$ & 80 & $w^2 + t^3 z + t z^6$ \\
41 & $w^2 + t^4 + z^5$ & 84 & $w^3 + w z^3 + t^4$ \\
45 & $w^2 z + t^4 + z^5$ & 86 & $w^2 + t^4 z + t z^5$ \\
48 & $w^2 z + t^3 + z^7$ & 88 & $w^2 + t^3 + z^7$ \\
51 & $w^2 + t^3 z + t z^4$ & 90 & $w^2 + t^3 + t z^7$ \\
53 & $w^2 + t^3 + z^8$ & 93 & $w^2 + t^5 + t z^5$ \\
54 & $w^2 z + t^3 + z^4$ & 95 & $w^2 + t^3 + z^{11}$
\end{tabular}
\end{center}
\end{table}

\begin{table}[h]
\begin{center}
\caption{$L_{xy}$: Irreducible and singular case}
\label{table:Lsing}
\begin{tabular}{ccc|ccc}
No. & Equation & Sing. & No. & Equation & Sing. \\
\hline
43 & $t^4 + z^5$ & $\msp_w$ & 77 & $w^2 + t^3 z$ & $\msp_z$ \\
44 & $w^2 t + z^4$ & $\msp_t$ & 78 & $w^2 + t z^5$ & $\msp_t$ \\
46 & $t^3 + z^7$ & $\msp_w$ & 79 & $w^2 z + t^3$ & $\msp_{z}$ \\
47 & $w^2 z + t^3$ & $\msp_z$ & 81 & $w^2 + t^4 z$ & $\msp_z$ \\
56 & $t^3 + z^8$ & $\msp_w$ & 82 & $w^2 + t^3$ & $\msp_z$ \\
59 & $w^3 + z^4$ & $\msp_t$ & 83 & $w^2 + z^9$ & $\msp_t$ \\
61 & $w^2 t + z^5$ & $\msp_t$ & 85 & $w^2 + t^3 z$ & $\msp_z$ \\
62 & $w^2 + t^3 z$ & $\msp_z$ & 87 & $w^2 + t^5$ & $\msp_z$ \\
65 & $w^2 z + t^3$ & $\msp_z$ & 89 & $w^2 + t^3$ & $\msp_z$ \\
67 & $w^2 + z^7$ & $\msp_t$ & 91 & $w^2 + t^3 z$ & $\msp_z$ \\
69 & $w^2 z + t^4$ & $\msp_z$ & 92 & $w^2 + t^3$ & $\msp_z$ \\
73 & $w^2 + z^5$ & $\msp_t$ & 94 & $w^2 + t^3$ & $\msp_z$ 
\end{tabular}
\end{center}
\end{table}

\begin{Prop} \label{prop:smptL1}
Let $X = X_d \subset \mbP (1, a_1, a_2, a_3, a_4)$, $a_1 \le a_2 \le a_3 \le a_4$, be a member of a family $\mcF_{\msi}$ with $\msi \in \msI \setminus \msI_1$ which satisfies $a_1 < a_2$ and which is listed in Tables \ref{table:Lsmooth} and \ref{table:Lsing}.
Then 
\[
\alpha_{\msp} (X) \ge 1
\] 
for any smooth point $\msp$ of $X$ contained in $L_{xy}$.
\end{Prop}

\begin{proof}
Take a point $\msp \in L_{xy} \cap \Sm (X)$.
Let $S \in |A|$ and $T \in |a_1 A|$ be general members.
By Lemma \ref{lem:Lirred}, $L_{xy}$ is an irreducible and reduced curve, and we have $S \cdot T = L_{xy}$. 
Note that $L_{xy}$ is quasi-smooth at $\msp$ and we have $\mult_{\msp} (L_{xy}) = 1$.
By Lemma \ref{lem:qsminvhypsec}, $S$ is quasi-smooth at $\msp$.
It follows that $\lct_{\msp} (X;S) = 1$.
By Lemma \ref{lem:exclL}, we have
\[
\alpha_{\msp} (X) \ge \min \left\{ \lct_{\msp} (X;S), \ \frac{a_1}{\mult_{\msp} (L_{xy})}, \ \frac{1}{a_1 (A^3)} \right\} = 1
\]
since $1/a_1 (A^3) > 1$ by Lemma \ref{lem:wtnumerics}.
\end{proof}

\subsection{Smooth points on $L_{xy}$ for families with $a_1 < a_2$, Part 2} \label{sec:smptL2}

In this section, we consider families $\mcF_{\msi}$ with $\msi \in \msI \setminus \msI_1$ such that $L_{xy}$ is either irreducible or reduced for some member $X$.
Specifically these families consist of families $\mcF_{\msi}$ with
\begin{equation} \label{eq:Lxyrednonred}
\begin{split}
\msi \in \{ & 7, 9, 12, 13, 15, 20, 23, 24, 25, 29, 30, 31, 32, \\
& 33, 37, 38, 39, 40, 42, 49, 50, 52, 58, 60, 63, 64\},
\end{split}
\end{equation}
and the aim of this section is to prove the following.

\begin{Prop} \label{prop:smptL2}
Let $X = X_d \subset \mbP (1,a_1,a_2,a_3,a_4)$, $a_1 \le a_2 \le a_3 \le a_4$, be a member of a family $\mcF_{\msi}$ with $\msi \in \msI \setminus \msI_1$ which satisfies $a_1 < a_2$ and which is not listed in Tables \ref{table:Lsmooth} and \ref{table:Lsing}.
Then 
\[
\alpha_{\msp} (X) \ge \frac{1}{2}
\]
for any smooth point $\msp$ of $X$ contained in $L_{xy}$.
\end{Prop}

The proof of Proposition \ref{prop:smptL2} will be completed in \S \ref{sec:proofsmptL} by considering each family separately and by case-by-case arguments.
Those arguments form several patterns and we describe them in \S \ref{sec:smptLgen}. 

\subsubsection{General arguments} \label{sec:smptLgen}

In this subsection, let 
\[
X = X_d \subset \mbP (1, a, b_1, b_2, b_3)_{x, y, z_1, z_2, z_3}
\] 
be a member of a family $\mcF_{\msi}$ with $\msi \in \msI \setminus \msI_1$.
Throughout this subsection, we assume that $a < b_i$ for $i = 1, 2, 3$.
Note that we do not assume $b_1 \le b_2 \le b_3$.
As before, we denote by $F = F (x, y, z_1, z_2, z_3)$ the defining polynomial of $X$, and we set $A := -K_X$.

The following very elementary lemma will be used several times.

\begin{Lem} \label{lem:numelem}
Let $a, e_1, e_2, e_3$ be positive integers such that $a < e_i$ for $i = 1, 2, 3$ and $\gcd \{e_1, e_2, e_3\} = 1$, and let $\lambda \ge 1$ be a number.
Then the following inequalities hold. 
\begin{enumerate}
\item $\frac{1 + e_2 + e_3}{e_1 e_2 e_3} \le \frac{1}{2}$.
\item $\frac{a + e_2 + e_3}{e_1 e_2 e_3} + \frac{\lambda}{a} \le \frac{1}{2} + \lambda$.
\item $a (a + e_2 + e_3) < e_1 e_2 e_3$.
\end{enumerate}
\end{Lem}

\begin{proof}
In view of the assumption $\gcd \{e_1, e_2, e_3\} = 1$, it is easy to see that 
\[
\frac{1 + e_2 + e_3}{e_1 e_2 e_3} = \frac{1}{e_1 e_2 e_3} + \frac{1}{e_1 e_3} + \frac{1}{e_1 e_2}
\]
attains its maximum when $(e_1, e_2, e_3) = (2, 2, 3)$, which implies (1).

It is also easy to see that 
\[
\frac{a + e_2 + e_3}{e_1 e_2 e_3} + \frac{\lambda}{a},
\]
viewed as a function of $a$, attains its maximum when $a = 1$ since $1 \le a \le \sqrt{\lambda e_1 e_2 e_3}$.
Combining this with (1), the inequality (2) follows.

By the assumption, we have $e_1 e_2 \ge (a+1)^2$.
Hence we have 
\[
\begin{split}
e_1 e_2 e_3 - a (a + e_2 + e_3) 
&= e_3 (e_1 e_2 - a) - a^2 - a e_2 \\
& \ge e_3 (a^2 + a + 1) - a^2 - a e_2 \\
&= a^2 (e_3 - 1) + a (e_3 - e_2) + e_3 \\
&> 0.
\end{split}
\]
This proves (3).
\end{proof}

\begin{Lem} \label{lem:Linteg}
Suppose that $L_{xy} := (x = y = 0)_X$ is an irreducible and reduced curve which is smooth along $L_{xy} \cap \Sm (X)$.
Then, 
\[
\alpha_{\msp} (X) \ge 1
\] 
for any $\msp \in L_{xy} \cap \Sm (X)$.
\end{Lem}

\begin{proof}
Let $S \in |A|$ and $T \in |a A|$ be general members.
Then we have $S \cap T = L_{xy}$.
Take any point $\msp \in L_{xy} \cap \Sm (X)$.
By Lemma \ref{lem:qsminvhypsec}, $S$ is smooth at $\msp$.
It follows that $\mult_{\msp} (L_{xy}) = 1$ and $\lct_{\msp} (X;S) = 1$.
By Lemma \ref{lem:exclL}, we have
\[
\alpha_{\msp} (X) 
\ge \min \left\{ \lct_{\msp} (X;S), \ \frac{1}{\mult_{\msp} (L_{xy})}, \  \frac{1}{a (A^3)} \right\} = 1,
\]
since $1/a (A^3) > 1$ by Lemma \ref{lem:wtnumerics}.
\end{proof}

\begin{Lem} \label{lem:Lredcp1}
Let $S \in |A|$ and $T \in |aA|$ be general members.
Suppose that the following assertions are satisfied.
\begin{enumerate}
\item $T|_S = \Gamma + \Delta$, where $\Gamma = (x = y = z_1 = 0)$ is a quasi-line and $\Delta$ is an irreducible and reduced curve which is quasi-smooth along $\Delta \cap \Sm (X)$.
\item $S$ is quasi-smooth along $\Gamma \cap \Delta$. 
\item $\Sing_{\Gamma} (X) = \{\msp_{z_2}, \msp_{z_3}\}$.
\end{enumerate}
Then 
\[
\alpha_{\msp} (X) \ge \frac{1}{2}
\]
for any point $\msp \in L_{xy} \cap \Sm (X)$.
\end{Lem}

\begin{proof}
By assumptions (1), (2) and Lemma \ref{lem:pltsurfpair}, $S$ is quasi-smooth along $\Gamma$.

\begin{Claim} \label{cl:Lredcp1-1}
The intersection matrix $M = M (\Gamma, \Delta)$ satisfies the condition $(\star)$.
\end{Claim}

\begin{proof}[Proof of Claim \ref{cl:Lredcp1-1}]
By the assumption (3) and the quasi-smoothness of $S$, we see that $\Sing_{\Gamma} (S) = \{\msp_{z_2}, \msp_{z_3}\}$ and $\msp_{z_i} \in S$ is a cyclic quotient singularity of index $b_i$ for $i = 2, 3$.
By Remark \ref{rem:compselfint}, we have
\[
(\Gamma^2)_S = -2 + \frac{b_2-1}{b_2} + \frac{b_3-1}{b_3} = - \frac{b_2 + b_3}{b_2 b_3} < 0.
\] 
By taking the intersection number of $T|_S = \Gamma + \Delta$ and $\Gamma$, we have
\[
(\Gamma \cdot \Delta)_S 
= - (\Gamma^2)_S + (T \cdot \Gamma)
= \frac{a + b_2 + b_3}{b_2 b_3} > 0.
\]
Note that we have
\[
(T \cdot \Delta) 
= (T^2 \cdot S) - (T \cdot \Gamma) 
= a^2 (A^3) - \frac{a}{b_2 b_3}
= \frac{a (a + b_2 + b_3)}{b_1 b_2 b_3},
\]
and then by taking the intersection number of $T|_S = \Gamma + \Delta$ and $\Delta$, we have
\[
(\Delta^2)_S = (T \cdot \Delta) - (\Gamma \cdot \Delta)_S
= - \frac{(b_1-a)(a + b_2 + b_3)}{b_1 b_2 b_3} < 0.
\]
Finally we have
\[
\begin{split}
\det M &= \frac{b_2 + b_3}{b_2 b_3} \cdot \frac{(b_1-a)(a + b_2 + b_3)}{b_1 b_2 b_3} - \frac{(a + b_2 + b_3)^2}{b_2^2 b_3^2} \\
&= - \frac{a (a + b_2 + b_3)(b_1 + b_2 + b_3)}{b_1 b_2^2 b_3^2} < 0.
\end{split}
\]
It follows that $M$ satisfies the condition $(\star)$.
\end{proof}

Let $\msp \in (\Gamma \setminus \Delta) \cap \Sm (X)$ be a point.
By Lemma \ref{lem:normalqhyp}, $S$ is a normal surface.
It is easy to check that $a \deg \Gamma = a/(b_2 b_3) \le 1$, and that $X$, $S$ and $\Gamma$ are smooth at $\msp$.
Thus, we can apply Lemma \ref{lem:mtdLred} and we conclude
\[
\alpha_{\msp} (X) 
\ge \min \left\{ 1, \ \frac{1}{a (A^3) + \frac{1}{a} - \deg \Gamma} \right\} 
= \min \left\{ 1, \ \frac{1}{\frac{a + b_2 + b_3}{b_1 b_2 b_3} + \frac{1}{a}} \right\}
\ge \frac{2}{3},
\]
where the last inequality follows from Lemma \ref{lem:numelem}.

Let $\msp \in (\Delta \setminus \Gamma) \cap \Sm (X)$ be a point.
Note that $\Delta$ is smooth at $\msp$ since it is quasi-smooth at $\msp$ by the assumption (1).
We have
\[
a \deg \Delta = a \left( a (A^3) - \frac{1}{b_2 b_3} \right) = \frac{a (a + b_2 + b_3)}{b_1 b_2 b_3} < 1
\]
by Lemma \ref{lem:numelem}.
Note that we have
\[
a (A^3) + \frac{1}{a} - \deg \Delta = \frac{1}{a} + \frac{1}{b_2 b_3} \le 1 + \frac{1}{4} = \frac{5}{4}
\]
since $1 \le a < b_i$.
Thus, we can apply Lemma \ref{lem:mtdLred} and conclude 
\[
\alpha_{\msp} (X) 
\ge \min \left\{ 1, \ \frac{1}{a (A^3) + \frac{1}{a} - \deg \Delta} \right\} \ge \frac{4}{5}.
\]

Finally let $\msp \in (\Gamma \cap \Delta) \cap \Sm (X)$ be a point.
Note that $S$ is smooth at $\msp$ by the assumption (2), and we have
\[
\deg (\Gamma) = \frac{1}{b_2 b_3} < \frac{2}{a}, \quad
\deg (\Delta) = \frac{a + b_2 + b_3}{b_1 b_2 b_3} < \frac{2}{a},
\]
where the former inequality is obvious and the latter follows from Lemma \ref{lem:numelem}.
Thus we can apply Lemma \ref{lem:mtdLintpt} and conclude that
\[
\alpha_{\msp} (X) \ge \min \left\{1, \frac{a}{2} \right\} \ge \frac{1}{2}.
\]
Therefore the proof is completed.
\end{proof}

\begin{Rem} \label{rem:Lredcp1}
Let the notation and assumption as in Lemma \ref{lem:Lredcp1}.
Assume in addition that $a \ge 2$ and $\Gamma \cap \Delta \subset \Sing (X)$.
Then 
\[
\alpha_{\msp} (X) \ge \frac{43}{54} > \frac{3}{4}
\]
for any $\msp \in L_{xy} \cap \Sm (X)$.

Indeed, since $\Gamma \cap \Delta \subset \Sing (X)$, we have
\[
\alpha_{\msp} (X) \ge \min \left\{1, \ \frac{1}{\frac{a + b_2 + b_3}{b_1 b_2 b_3} + \frac{1}{a}}, \ \frac{4}{5} \right\}
\]
by the proof of Lemma \ref{lem:Lredcp1}.
Since $2 \le a < \sqrt{b_1 b_2 b_3}$ and $a < b_i$ for $i = 1, 2, 3$, we have
\[
\frac{a + b_2 + b_3}{b_1 b_2 b_3} + \frac{1}{a} 
\le \frac{3 a + 1}{(a+1)^3} + \frac{1}{a}
\le \frac{43}{54}.
\]
This proves the desired inequality.
\end{Rem}

\begin{Lem} \label{lem:Lredcp2}
Suppose that $b_1, b_2, b_3$ are mutually coprime and $a \in \{1, 2\}$.
Suppose in addition that $F$ can be written as
\[
F = f_1 (z_1, z_2) x + f_2 (z_1, z_2) y + z_3^m z_2 + g (x, y, z_1, z_2, z_3),  
\]
where $m \in \{2, 3\}$ and $f_1, f_2 \in \mbC [z_1, z_2], g \in [x, y, z_1, z_2, z_3]$ are quasi-homogeneous polynomials satisfying the following properties.
\begin{enumerate}
\item $\deg F = b_1 b_2 + a$.
\item $g$ is contained in the ideal $(x, y) \cap (x, y, z_3)^2 \subset \mbC [x, y, z_1, z_2, z_3]$.
\end{enumerate}
Then, 
\[
\alpha_{\msp} (X) \ge \frac{1}{2}
\]
for any point $\msp \in L_{xy} \cap \Sm (X)$.
\end{Lem}

\begin{proof}
We have $d = m b_3 + b_2$ since $z_3^m z_2 \in F$, and combining this with $d =  a + b_1 + b_2 + b_3$, we have
\begin{equation} \label{eq:Lredcp2-1}
a + b_1 + b_3 = m b_3.
\end{equation}

\begin{Claim} \label{clm:Lredcp2-1}
We can assume
\begin{equation} \label{eq:Lredcp2-2}
F =
\begin{cases}
z_2^{b_1} x - z_1^{b_2} y + z_3^m z_2 + g, & \text{if $a = 1$}, \\
z_1^k z_2^l x + (z_1^{b_2} - z_2^{b_1}) y + z_3^m z_2 + g, & \text{if $a = 2$}, \\
\end{cases}
\end{equation}
after replacing $x$ and $y$ suitably, where $k$ and $l$ are non-negative integers.
\end{Claim}

\begin{proof}[Proof of Claim \ref{clm:Lredcp2-1}]
Suppose $a = 1$.
Then, since $\deg f_1 = \deg f_2 = b_1 b_2$, we can write 
\[
f_1 (z_1,z_2) x + f_2 (z_1,z_2) y = z_2^{b_1} \ell_1 (x, y) + z_1^{b_2} \ell_2 (x, y),
\] 
where $\ell_1, \ell_2$ are linear forms in $x, y$.
We see that $\ell_1, \ell_2$ are linearly independent because otherwise we can write $\ell_2 = \alpha \ell_1$ for some nonzero $\alpha \in \mbC$ and $X$ is not quasi-smooth along
\[
(x = y = w = \ell_1 = z_2^{b_2} + \alpha z_1^{b_1} = 0) \subset X.
\]
This is a contradiction.
Thus $\ell_1, \ell_2$ are linearly independent and we may assume $\ell_1 = x$ and $\ell_2 = y$, as desired.

Suppose $a = 2$.
We have $f_2 = \alpha z_1^{b_2} + \beta z_2^{b_1}$ for some $\alpha, \beta \in \mbC$ since $\deg f_2 = b_1 b_2$.
By the quasi-smooth of $X$ at $\msp_z, \msp_t$, we have $\alpha, \beta \ne 0$, and thus we may assume $\alpha =1$, $\beta = -1$ by rescaling $z_1, z_2$.
We see that the equation $f_1 (z_1, z_2) = f_2 (z_1, z_2) = 0$ on variables $z_1, z_2$ has only trivial solution because otherwise $X$ is not quasi-smooth along the nonempty set
\[
(x = y = w = f_1 (z_1, z_2) = f_2 (z_1, z_2) = 0) \subset X,
\]
which is impossible.
This implies that $f_1 \ne 0$ as a polynomial, and there exists a monomial $z_1^k z_2^l$ of degree $b_1 b_2 + 1$.
Since $b_1$ is coprime to $b_2$, $z_1^k z_2^l$ is the unique monomial of degree $b_1 b_2 + 1$ in variables $z_1, z_2$.
Thus we have $f_2 = \gamma z_1^k z_2^l$ for some nonzero $\gamma \in \mbC$. 
Rescaling $x$, we may assume $\gamma = 1$, and the claim is proved.
\end{proof}

We continue the proof of Lemma \ref{lem:Lredcp2}.
Let $S \in |A|$ and $T \in |a A|$ be general members.
We have 
\[
T|_S = \Gamma + m \Delta,
\]
where 
\[
\Gamma = (x = y = z_2 = 0), \quad
\Delta  = (x = y = z_3 = 0),
\]
since $F (0,0,z,t,w) = z_3^m z_2$.
We see that $\Gamma$ and $\Delta$ are quasi-lines of degree $1/b_1 b_3$ and $1/b_1 b_2$, respectively, and $\Gamma \cap \Delta = \{\msp_{z_1}\} \subset \Sing (X)$.
We see that $S$ is quasi-smooth at $\msp_{z_1}$ since $z_1^{b_2} y \in F$.
By Lemma \ref{lem:pltsurfpair}, $S$ is quasi-smooth along $\Gamma$ and the pair $(S, \Gamma)$ is plt.

\begin{Claim} \label{clm:Lredcp2-2}
The intersection matrix $M = M (\Gamma, \Delta)$ satisfies the condition $(\star)$.
\end{Claim}

\begin{proof}[Proof of Claim \ref{clm:Lredcp2-2}]
We see that $d$ is not divisible by $b_1$ or $b_3$ since $d = b_1 b_2 + a = m b_3 + b_2$, $a < b_1$ and $b_2$ is coprime to $b_3$.
It follows that $\Sing_{\Gamma} (S) = \{\msp_{z_1}, \msp_{z_3}\}$, and $\msp_{z_i} \in S$ is a cyclic quotient singularity of index $b_i$ for $i = 1, 3$.
By Remark \ref{rem:compselfint}, we have
\[
(\Gamma^2)_S 
= -2 + \frac{b_1 - 1}{b_1} + \frac{b_3 - 1}{b_3} 
= - \frac{b_1 + b_3}{b_1 b_3} < 0.
\]
By taking intersection number of $T|_S = \Gamma + m \Delta$ and $\Gamma$, we obtain
\[
(\Gamma \cdot \Delta)_S 
= \frac{1}{m} (a \deg \Gamma - (\Gamma^2)_S)
= \frac{a + b_1 + b_3}{m b_1 b_3} = \frac{1}{b_1} > 0.
\]
Similarly, by taking intersection number of $T|_S$ and $\Delta$, we obtain
\[
(\Delta^2)_S 
= \frac{1}{m} (a \deg \Delta - (\Gamma \cdot \Delta)_S)
= - \frac{b_2 - a}{m b_1 b_2} < 0,
\]
where the second equality follows from \eqref{eq:Lredcp2-1}.
Finally, we have
\[
\det M = \frac{b_1 + b_3}{b_1 b_3} \cdot \frac{b_2 - a}{m b_1 b_2} - \frac{1}{b_1^2} = - \frac{a (b_1 + b_2 + b_3)}{b_1^2 b_2 b_3} < 0,
\]
where the second equality follows from \eqref{eq:Lredcp2-1}.
It follows that $M$ satisfies the condition $(\star)$.
\end{proof}

Let $\msp \in (\Gamma \setminus \Delta) \cap \Sm (X)$.
We see that $X, S$ and $\Gamma$ are smooth at $\msp$, and it is easy to see $a \deg \Gamma = a/(b_1 b_3) < 1$.
Hence we can apply Lemma \ref{lem:mtdLred} and we have
\[
\alpha_{\msp} (X) 
\ge \min \left\{ 1, \ \frac{1}{a (A^3) + \frac{1}{a} - \deg \Gamma} \right\} 
= \min \left\{ 1, \ \frac{1}{\frac{a + b_1 + b_3}{b_1 b_2 b_3} + \frac{1}{a}} \right\}
\ge \frac{2}{3},
\]
where the last inequality follows from Lemma \ref{lem:numelem}.

It remains to consider $\msp \in (\Delta \setminus \Gamma) \cap \Sm (X)$ since $\Gamma \cap \Delta = \{\msp_{z_1}\} \subset \Sing (X)$.
We first consider the case when $a = 2$.
In this case $S = H_x$ is quasi-smooth along $\Delta \setminus \{\msq\}$, where $\msq = (0\!:\!0\!:\!1\!:\!1\!:\!0) \in (\Delta \setminus \Gamma) \cap \Sm (X)$, and $S$ has a double point at $\msq$.
We have $\mult_{\msp} (\Delta) = 1$, and $a \deg \Delta = a/(b_1 b_2) < 1$.
Thus we can apply Lemma \ref{lem:mtdLredSsing} and conclude
\[
\begin{split}
\alpha_{\msp} (X) 
&\ge \min \left\{ \frac{2}{\mult_{\msp} (S)}, \ \frac{\mult_{\msp} (S)}{2 (A^3) + \frac{m}{2} - m \deg \Delta} \right\} \\
&= \min \left\{ \frac{2}{\mult_{\msp} (S)}, \ \frac{\mult_{\msp} (S)}{\frac{1}{b_1 b_3} + \frac{m}{2}} \right\} \\
&\ge \min \left\{ 1, \ \frac{1}{\frac{1}{12} + \frac{3}{2}} \right\} \\
&= \frac{12}{19},
\end{split}
\]
since $1/(b_1 b_3) \le 1/12$, $m \in \{2, 3\}$ and $\mult_{\msp} (S) \in \{1, 2\}$.

Suppose $a = 1$.
We set $S' = (z_3 = 0) \cap X \in |b_3 A|$.
For $\lambda \in \mbC$, we set $T'_{\lambda} = (y - \lambda x = 0) \cap X \in |A|$. 
We can write $g (x, \lambda x, z_1, z_2, 0) = x^2 h_{\lambda}$ for some $h_{\lambda} = h_{\lambda} (x,z_1,z_2)$ since $g \in (x,y,z_3)^2$.
In view of \eqref{eq:Lredcp2-2}, we have
\[
F (x, \lambda x, z_1, z_2, 0) = x \phi_{\lambda} (x, z_1, z_2),
\]
where
\[
\phi_{\lambda} (x, z_1, z_2) =
z_1^{b_2} - \lambda z_2^{b_1} + x h_{\lambda}
\]
The polynomial $\phi_{\lambda}$ is irreducible for any nonzero $\lambda \in \mbC$.
We have 
\[
T'_{\lambda}|_{S'} = \Delta + \Xi_{\lambda},
\]
where
\[
\Xi_{\lambda} = (y - \lambda x = z_3 = \phi_{\lambda} = 0)
\]
is an irreducible and reduced curve.
We have $\Delta \cap \Xi_{\lambda} = \{\msq_{\lambda}\} \subset \Sm (X)$, where $\msq_{\lambda} = (0\!:\!0\!:\!\sqrt[b_2]{\lambda}\!:\!1\!:\!0)$.
It is easy to see that $S'$ is quasi-smooth at $\msq_{\lambda}$.
Hence, $S'$ is quasi-smooth along $\Delta$ by Lemma \ref{lem:pltsurfpair}.

\begin{Claim} \label{clm:Lredcp2-3}
The intersection matrix $M' = M (\Delta, \Xi_{\lambda})$ satisfies the condition $(\star)$.
\end{Claim}

\begin{proof}[Proof of Claim \ref{clm:Lredcp2-3}]
By Remark \ref{rem:compselfint}, we have
\[
(\Delta^2)_S = - \frac{b_3 -1}{b_1 b_2} - 2 + \frac{b_1 - 1}{b_1} + \frac{b_2 - 1}{b_2} = - \frac{b_1 + b_2 + b_3 - 1}{b_1 b_2} < 0.
\]
By taking intersection number of $T'_{\lambda}|_{S'} = \Delta + \Xi_{\lambda}$ and $\Delta$, we obtain
\[
(\Delta \cdot \Xi_{\lambda})_S = \frac{b_1 + b_2 + b_3}{b_1 b_2} > 0.
\]
By taking intersection number of $T'_{\lambda}|_{S'}$ and $\Xi_{\lambda}$, we obtain
\[
(\Xi_{\lambda}^2)_S = 0.
\]
It is then obvious that $\det M' < 0$ and the proof is completed.
\end{proof}

Now we take any point $\msp \in (\Delta \setminus \Gamma) \cap \Sm (X)$, and then we can choose a nonzero $\lambda \in \mbC$ so that $\msp \ne \msq_{\lambda}$.
By Lemma \ref{lem:qsminvhypsec}, $S'$ is smooth at $\msp$ since $S' \cap T'_{\lambda}$ is smooth at $\msp$.
It is easy to see that $\deg \Delta = 1/(b_1 b_2) < 1$.
Thus we can apply Lemma \ref{lem:mtdLred} and conclude 
\[
\alpha_{\msp} (X) 
\ge \min \left\{ b_3, \ \frac{1}{b_3 (A^3) + 1 - \deg \Delta} \right\}
= \min \left\{ b_3, \ \frac{1}{2} \right\} = \frac{1}{2},
\]
where the first equality follows since
\[
b_3 (A^3) + 1 - \deg \Delta = \frac{d - 1}{b_1 b_2} + 1 = 2.
\] 
This completes the proof.
\end{proof}

\begin{Lem} \label{lem:Lnonredu2}
Suppose that $b_1, b_2, b_3$ are mutually coprime and $a \in \{1, 2, 3\}$.
Suppose in addition that $F$ can be written as
\[
F = z_3^m + z_1^{e_1} y - z_2^{e_2} x + g (x, y, z_1, z_2, z_3),
\]
where $m \ge 2, e_1, e_2$ are positive integers and $g \in \mbC [x, y, z_1 z_2, z_3]$ is a homogeneous polynomial satisfying the following properties.
\begin{enumerate}
\item If $a \ge 2$, then $m \le 2 a$.
\item If $a = 1$, then $e_1 \le b_2$.
\item $g$ is a homogeneous polynomial contained in the ideal $(x, y) \cap (x, y, z_3)^2 \subset \mbC [x, y, z_1, z_2, z_3]$.
\end{enumerate}
Then, 
\[
\alpha_{\msp} (X) \ge \frac{1}{2}
\]
for any $\msp \in L_{xy} \cap \Sm (X)$.
\end{Lem}

\begin{proof}
We first consider the case where $a \ge 2$.
Let $S \in |A|$ and $T \in |a A|$ be general members.
We have 
\[
S \cdot T = m \Gamma,
\]
where
\[
\Gamma = (x = y = z_3 = 0)
\]
is a quasi-line of degree $1/(b_1 b_2)$.
Let $\msp \in L_{xy} \cap \Sm (X)$.
It is straightforward to check that $S$ is smooth at $\msp$, which implies $\lct_{\msp} (X; \frac{1}{a} S) = a$.
By Lemma \ref{lem:exclL}, we have
\[
\alpha_{\msp} (X) \ge \min \left\{ a, \ \frac{a}{m}, \ \frac{1}{a (A^3)} \right\} \ge \frac{1}{2},
\]
since $a/m \ge 1/2$ and $1/(a (A^3)) > 1$ by the assumption (1) and Lemma \ref{lem:wtnumerics}, respectively.

In the following we assume $a = 1$.
We set $S' = (z_3 = 0) \cap X \in |b_3 A|$ and $\Gamma = (x = y = z_3 = 0) \subset S'$.
We have $L_{xy} = \Gamma$ set-theoretically.
For $\lambda \in \mbC$, we set $T'_{\lambda} = (y - \lambda x = 0) \cap X \in |aA|$.
We can write 
\[
g (x, \lambda x, z_1, z_2, 0) = x^2 h_{\lambda} (x, z_1, z_2),
\] 
where $h_{\lambda}$ is a quasi-homogeneous polynomial in variables $x, z_1, z_2$ since $g \in (x, y, z_3)^2$.
We have
\[
F (x, \lambda x, z_1, z_2, 0) = x (\lambda z_1^{e_1} - z_2^{e_2} + x h_{\lambda}).
\]

\begin{Claim} \label{clm:Lnonredu2-1}
The quasi-homogeneous polynomial
\[
\phi_{\lambda} := \lambda z_1^{e_1} - z_2^{e_2} + x h_{\lambda} \in \mbC [x, z_1, z_2]
\] 
is irreducible for any $\lambda \in \mbC \setminus \{0\}$.
\end{Claim}

\begin{proof}[Proof of Claim \ref{clm:Lnonredu2-1}]
We assume $\lambda \ne 0$.
If $\phi_{\lambda}$ is a reducible polynomial, then we can write 
\[
\phi_{\lambda} = - (z_2^{c_2} + \cdots + \alpha  z_1^{c_1} + \cdots)(z_2^{e_2-c_2} + \cdots + \beta z_1^{e_1-c_1} + \cdots)
\]
for some $c_1, c_2 \in \mbZ_{\ge 0}$ with $c_1 \le e_1$ and $0 < c_2 < e_2$, and nonzero $\alpha, \beta \in \mbC$ such that $\alpha \beta = \lambda$.
We have $c_2 b_2 = c_1 b_1$.
Since $b_1$ is coprime to $b_2$, we see that $c_1$ is divisible by $b_2$.
This implies $c_1 = e_1 = b_2$ since $c_1 \le e_1 \le b_2$.
By the equality $e_2 b_2 = e_1 b_1$, we have $c_2 = e_2 = b_1$.
This is a contradiction since $c_2 < e_2$.
Therefore $\phi_{\lambda}$ is irreducible for $\lambda \ne 0$.
\end{proof}

We continue the proof of Lemma \ref{lem:Lnonredu2}.
By Claim \ref{clm:Lnonredu2-1}, we have 
\[
T'_{\lambda}|_{S'} = \Gamma + \Delta_{\lambda},
\] 
where
\[
\Delta_{\lambda} = (y - \lambda x = z_3 = \phi_{\lambda} = 0)
\]
is an irreducible and reduced curve for any $\lambda \in \mbC \setminus \{0\}$.
We have $\Gamma \cap \Delta_{\lambda} = \{\msq_{\lambda}\}$, where 
\[
\msq_{\lambda} =
(0\!:\!0\!:\!1\!:\!\sqrt[e_2]{\lambda}\!:\!0).
\]
It is easy to check that $S'$ is quasi-smooth at $\msq_{\lambda}$.
By Lemma \ref{lem:pltsurfpair}, $S'$ is quasi-smooth along $\Gamma$ and the pair $(S', \Gamma)$ is plt.

\begin{Claim} \label{clm:Lnonredu2-2}
The intersection matrix $M' = M (\Gamma, \Xi'_{\lambda})$ satisfies the condition $(\star)$.
\end{Claim}

\begin{proof}[Proof of Claim \ref{clm:Lnonredu2-2}]
We see that $\Sing_{\Gamma} (S') = \{\msp_{z_1}, \msp_{z_2}\}$ and $\msp_{z_i} \in S'$ is a cyclic quotient singular point of index $b_i$ for $i = 1, 2$.
By the same computation as in Claim \ref{clm:Lredcp2-3}, we have
\[
\begin{split}
(\Gamma^2)_{S'} &= - \frac{b_1 + b_2 + b_3 -1}{b_1 b_2} < 0, \\
(\Gamma \cdot \Delta_{\lambda})_{S'} &= \frac{b_1 + b_2 + b_3}{b_1 b_2} > 0, \\
(\Delta_{\lambda}^2)_{S'} &= 0.
\end{split}
\] 
It is then easy to see that $\det M' < 0$, which shows that $M'$ satisfies $(\star)$.
\end{proof}

Now take a point $\msp \in L_{xy} \cap \Sm (X) = \Gamma \cap \Sm (X)$.
We choose and fix a general $\lambda \in \mbC$ so that $\Delta_{\lambda}$ is irreducible and $\msq_{\lambda} \ne \msp$.
Then $\msp \in (\Gamma \setminus \Xi_{\lambda}) \cap \Sm (X)$.
We see that $X$, $S'$ and $\Gamma$ are smooth at $\msp$, and $\deg \Gamma = 1/(b_1 b_2) < 1$.
Thus we can apply Lemma \ref{lem:mtdLred} and conclude
\[
\alpha_{\msp} (X) 
\ge \min \left\{ b_3, \ \frac{1}{b_3 (A^3) + 1 - \deg \Gamma} \right\} 
= \min \left\{ b_3, \ \frac{1}{\frac{e_1}{b_2} + 1} \right\}
\ge \frac{1}{2},
\]
where the last inequality follows from the assumption (2).
This completes the proof.
\end{proof}

\begin{Lem} \label{lem:nsptLpuredouble}
Let $S \in |A|$ and $T \in |a A|$ be general members.
Suppose that 
\[
S \cdot T = 2 \Gamma,
\] 
where $\Gamma = (x = y = z_3 = 0)$.
Then 
\[
\alpha_{\msp} (X) 
\ge \min \left\{ \lct_{\msp} (X; S), \ \frac{a}{2}, \ \frac{1}{a (A^3)} \right\}
\ge \frac{1}{2}
\]
for any point $\msp \in L_{xy} \cap \Sm (X)$.
\end{Lem}

\begin{proof}
Let $\msp \in L_{xy} \cap \Sm (X)$.
We have $\mult_{\msp} (\Gamma) = 1$ and the first inequality follows from Lemma \ref{lem:exclL}.
We have $\mult_{\msp} (S) \le \mult_{\msp} (S \cdot T) = 2$, which implies $\lct_{\msp} (X; S) \ge 1/2$.
Thus the second inequality in the statement follows since $1/(a (A^3)) > 1$ by Lemma \ref{lem:wtnumerics}.
\end{proof}

\subsubsection{Proof of Proposition \ref{prop:smptL2}} \label{sec:proofsmptL}

This subsection is entirely devoted to the proof of Proposition \ref{prop:smptL2}.

Let $X = X_d \subset \mbP (1, a_1, a_2, a_3, a_4)$, $a_1 \le \cdots \le a_4$, be a member of a family $\mcF_{\msi}$ with $\msi \in \msI \setminus \msI_1$ satisfying $a_1 < a_2$.
Let $S \in |A|$ and $T \in |a_1 A|$ are general members so that their scheme-theoretic intersection $S \cap T$ coincides with $L_{xy}$.
Note that $S$ is a normal surface by Lemma \ref{lem:normalqhyp} and $T$ is a quasi-hyperplane section on $X$. 
We set
\[
f := F (0, 0, z, t, w),
\]
so that $L_{xy}$ is isomorphic to the hypersurface in $\mbP (a_2, a_3, a_4)_{z, t, w}$ defined by $f = 0$.

\paragraph{\it The family $\mcF_7$}
\label{sec:smptL2-7}

We have 
\[
f = w^2 \ell (z,t) + h (z,t), 
\]
where $\ell, h$ are linear and quadratic forms in $z, t$, respectively.
By the quasi-smoothness of $X$, we have $\ell (z, t) \ne 0$, and $h (z,t)$ does not have a multiple component. 

\begin{itemize}
\item Case (i): $h$ is not divisible by $\ell$.
In this case $S \cdot T = L_{xy}$ is irreducible and smooth. 
By Lemma \ref{lem:Linteg}, we have $\alpha_{\msp} (X) \ge 1$ for any $\msp \in L_{xy} \cap \Sm (X)$.

\item Case (ii): $h$ is divisible by $\ell$.
Replacing $z$ and $t$, we may assume $\ell (z,t) = z$.
We can write $h = z c (z,t)$, where $c (z,t)$ is a cubic form in $z, t$.
Note that $c (z,t)$ is not divisible by $z$ since $h (z,t) = z c (z,t)$ does not have a multiple component, and we can assume $c (0,t) = - t^3$ by re-scaling $t$.
In this case $T|_S = \Gamma + \Delta$, where 
\[
\Gamma = (x = y = z = 0), \quad
\Delta = (x = y = w^2 + c (z,t) = 0).
\]
We see that $\Gamma$ is a quasi-line and $\Delta$ is an irreducible quasi-smooth curve since $c (z, t)$ does not have a multiple component.
We have $\Gamma \cap \Delta = \{\msq\}$, where $\msq = (0\!:\!0\!:\!0\!:\!1\!:\!1) \in \Sm (X)$.
We claim that $S$ is quasi-smooth  (and hence smooth) at $\msq$.
We have $(\prt F/\prt z) (\msq) = (\prt F/\prt t)(\msq) = (\prt F/\prt w) (\msq) = 0$.
Hence at least one of $(\prt F/\prt x)(\msq)$ and $(\prt F/\prt y)(\msq)$ is nonzero by the quasi-smoothness of $X$.
By choosing $x$ and $y$, we may assume that $S = H_x$ and $(\prt F/\prt y)(\msq) \ne 0$.
It then follows that $S$ is quasi-smooth at $\msq$.
Finally we have $\Sing_{\Gamma} (X) = \{\msp_t, \msp_w\}$.
Thus the assumptions of Lemma \ref{lem:Lredcp1} are satisfied and we have $\alpha_{\msp} (X) \ge 1/2$ for any $L_{xy} \cap \Sm (X)$.
\end{itemize}

\paragraph{\it The family $\mcF_9$}
\label{sec:smptL2-9}

We have
\[
f = c (t,w) + z^3 \ell (t,w), 
\]
where $\ell = \ell (t,w)$ and $c = c (t,w)$ are linear and cubic forms in $t, w$ respectively.
By the quasi-smoothness of $X$, $c (t, w)$ does not have a multiple component.

\begin{itemize}
\item Case (i): $\ell (t,w) \ne 0$ and $c (t,w)$ is not divisible by $\ell (t,w)$.
In this case $S \cdot T = L_{xy}$ is irreducible and smooth.
By Lemma \ref{lem:Linteg}, we have $\alpha_{\msp} (X) \ge 1$ for any $\msp \in L_{xy} \cap \Sm (X)$.

\item Case (ii): $\ell (t, w) \ne 0$ and $c (t, w)$ is divisible by $\ell (t,w)$.
We write $c (t,w) = \ell (t, w) q (t, w)$, where $q (t,w)$ is a quadratic form in $t, w$ which is not divisible by $\ell (t,w)$. 
Replacing $t$ and $w$, we may assume $\ell = t$, that is, $f = w (q (t,w) + z^3)$.
We may also assume $q (0,w) = - w^2$ since $c (t,w)$ does not have a multiple component.
In this case $T|_S = \Gamma + \Delta$, where
\[
\Gamma = (x = y = t = 0), \quad
\Delta = (x = y = q + z^3 = 0).
\]
We see that $\Gamma$ is a quasi-line and $\Delta$ is an irreducible quasi-smooth curve since $q (t,w)$ is not a square of a linear form.
We have $\Gamma \cap \Delta = \{\msq\}$, where $\msq = (0\!:\!0\!:\!1\!:\!0\!:\!1) \in \Sm (X)$.
By the similar argument as in Case (ii) of \S \ref{sec:smptL2-7}, we can conclude that $S$ is quasi-smooth at $\msq$.
Finally we have $\Sing_{\Gamma} (X) = \{\msp_z, \msp_w\}$.
Thus, by Lemma \ref{lem:Lredcp1}, we have $\alpha_{\msp} (X) \ge 1/2$ for any $L_{xy} \cap \Sm (X)$.

\item Case (iii): $\ell (t,w) = 0$.
In this case $f = \ell_1 \ell_2 \ell_3$, where $\ell_1, \ell_2, \ell_3$ are linear forms in $t, w$ which are not mutually proportional, and $T|_S = \Gamma_1 + \Gamma_2 + \Gamma_3$, where $\Gamma_1, \Gamma_2, \Gamma_3$ are as follows.
\begin{itemize}
\item For $i = 1, 2, 3$, $\Gamma_i = (x = y = \ell_i = 0)$ is a quasi-line and $\Sing_{\Gamma_i} = \{1 \times \frac{1}{2} (1, 1), 1 \times \frac{1}{3} (1, 2)\}$.
\item $\Gamma_i \cap \Gamma_j = \{\msp_z\} \subset \Sing (X)$ for $i \ne j$.
Moreover, $S$ is quasi-smooth at $\msp_z$ since $S \in |A|$ is general.
\end{itemize}
We can compute $(\Gamma_i^2)_S = -5/6$ by the method explained in Remark \ref{rem:compselfint} and then we have $(\Gamma_i \cdot \Gamma_j)_S = 1/2$ for $i \ne j$ by considering $(\Gamma_l \cdot T|_S)_S$ for $l = 1, 2, 3$.
Thus the intersection matrix of $\Gamma_1, \Gamma_2, \Gamma_3$ is given by
\[
((\Gamma_i \cdot \Gamma_j)_S) =
\begin{pmatrix}
- \frac{5}{6} & \frac{1}{2} & \frac{1}{2} \\[1mm]
\frac{1}{2} & - \frac{5}{6} & \frac{1}{2} \\[1mm]
\frac{1}{2} & \frac{1}{2} & - \frac{5}{6}
\end{pmatrix}
\]
and it satisfies the condition $(\star)$.
By Lemma \ref{lem:mtdLred}, we have 
\[
\alpha_{\msp} (X) \ge \min \left\{1, \frac{3}{4} \right\} =  \frac{3}{4}
\] 
for any $\msp \in L_{xy} \cap \Sm (X)$.
\end{itemize}

\paragraph{\it The family $\mcF_{12}$}
We have $w^2 z \in F$ by the quasi-smoothness of $X$ at $\msp_w$.
Hence rescaling $w$, we can write $f = w^2 z + \alpha w t^2 + \lambda w z^3 + \beta t^2 z^2 + \mu z^5$, where $\alpha, \beta, \lambda, \mu \in \mbC$.
We can eliminate the monomial $z^5$ by replacing $w$ and hence we assume $\mu = 0$.
Then, by the quasi-smoothness of $X$ at $\msp_z$, we have $\lambda \ne 0$ and we may assume $\lambda = -1$ by rescaling $z$.
Thus we can write
\[
f = w^2 z + \alpha w t^2 - w z^3 + \beta t^2 z^2.
\]

\begin{itemize}
\item Case (i): $\alpha \ne 0$ and $\beta \ne 0$.
In this case $S \cdot T = L_{xy}$ is irreducible and smooth.
By Lemma \ref{lem:Linteg}, we have $\alpha_{\msp} (X) \ge 1$ for any $\msp \in L_{xy} \cap \Sm (X)$.

\item Case (ii): $\alpha \ne 0$ and either $\beta = 0$ or $\beta = \alpha$.
When $\beta = \alpha$, we replace $w \mapsto w - z^2$ and $z \mapsto - z$.
After this replacement, we may assume $\beta = 0$.
Moreover we may assume $\alpha = 1$ by re-scaling $t$.
Then we have $f = w (w z + t^2 - z^3)$ and $T|_S = \Gamma + \Delta$, where
\[
\Gamma = (x = y = w = 0), \quad
\Delta = (x = y =  w z + t^2 - z^3 = 0).
\]
We see that $\Gamma$ is a quasi-line and $\Delta$ is an irreducible quasi-smooth curve.
We have $\Gamma \cap \Delta = \{\msq\}$, where $\msq = (0\!:\!0\!:\!1\!:\!1\!:\!0) \in \Sm (X)$.
By a similar argument as in Case (ii) of \S \ref{sec:smptL2-7}, we conclude that $S$ is quasi-smooth at $\msq$.
Finally we have $\Sing_{\Gamma} (X) = \{\msp_z, \msp_t\}$.
Thus, by Lemma \ref{lem:Lredcp1}, we have $\alpha_{\msp} (X) \ge 1/2$ for any $\msp \in L_{xy} \cap \Sm (X)$.

\item Case (iii): $\alpha = 0$ and $\beta \ne 0$.
Re-scaling $t$, we may assume $\beta = 1$.
Then we have $f = z (w^2 + w z^2 + t^2 z)$ and $T|_S = \Gamma + \Delta$, where 
\[
\Gamma = (x = y = z = 0), \quad
\Delta = (x = y = w^2 + w z^2 + t^2 z = 0).
\]
We see that $\Gamma$ is a quasi-line and $\Delta$ is an irreducible quasi-smooth curve.
We have $\Gamma \cap \Delta = \{\msp_t\} \subset \Sing (X)$.
By the similar argument as in Case (ii) of \S \ref{sec:smptL2-7}, we conclude that $S$ is quasi-smooth at $\msp_t$.
Finally we have $\Sing_{\Gamma} (X) = \{\msp_t,\msp_w\}$.
Thus, by Lemma \ref{lem:Lredcp1}, we have $\alpha_{\msp} (X) \ge 1/2$ for any $\msp \in L_{xy} \cap \Sm (X)$.

\item Case (iv): $\alpha = \beta = 0$.
In this case $f = z w (w + z^2)$ and $T|_S = \Gamma_1 + \Gamma_2 + \Gamma_3$, where $\Gamma_1, \Gamma_2, \Gamma_3$ are as follows.
\begin{itemize}
\item $\Gamma_1 = (x = y = z = 0)$ is a quasi-line of degree $1/12$ and $\Sing_{\Gamma_1} (S) = \{1 \times \frac{1}{3} (1,3), 1 \times \frac{1}{4} (1,3)\}$.
\item $\Gamma_2 = (x = y = w = 0)$ and $\Gamma_3 = (x = y = w + z^2 = 0)$ are quasi-lines of degree $1/6$ and $\Sing_{\Gamma_i} (S) = \{1 \times \frac{1}{2} (1,1), \frac{1}{3} (1,2)\}$ for $i = 2, 3$.
\item For $1 \le i < j \le 3$, we have $\Gamma_i \cap \Gamma_j = \{\msp_t\} \subset \Sing (X)$.
Moreover, $S$ is quasi-smooth at $\msp_t$ since $S \in |A|$ is general.
\end{itemize}
By the similar computation as in Case (iii) of \S \ref{sec:smptL2-9}, the intersection matrix of $\Gamma_1, \Gamma_2, \Gamma_3$ is given by
\[
\begin{pmatrix}
- \frac{7}{12} & \frac{1}{3} & \frac{1}{3} \\[1mm]
\frac{1}{3} & - \frac{5}{6} & \frac{2}{3} \\[1mm]
\frac{1}{3} & \frac{2}{3} & - \frac{5}{6}
\end{pmatrix}
\]
and it satisfies the condition $(\star)$.
By Lemma \ref{lem:mtdLred}, we have 
\[
\alpha_{\msp} (X) \ge 
\begin{cases}
3/4, & \text{if $\msp \in \Gamma_1 \cap \Sm (X)$}, \\
4/5, & \text{if $\msp \in \Gamma_i \cap \Sm (X)$ for $i = 2, 3$}
\end{cases}
\] 
Thus, $\alpha_{\msp} (X) \ge 3/4$ for any $\msp \in L_{xy} \cap \Sm (X)$.
\end{itemize}

\paragraph{\it The family $\mcF_{13}$}

We have
\[
f = \alpha w t^2 + \beta w z^3 + \gamma t^3 z + \delta t z^4, 
\]
where $\alpha, \beta, \gamma, \delta \in \mbC$.
Note that $(\alpha,\gamma) \ne (0,0)$ since $X$ is quasi-smooth at $\msp_t$.

\begin{itemize}
\item Case (i): $\alpha \ne 0$, $(\beta, \delta) \ne (0,0)$ and $(\alpha,\gamma)$ is not proportional to $(\beta,\delta)$.
In this case $S \cdot T = L_{xy}$ is irreducible and is smooth outside $\msp_w \in \Sing (X)$.
By Lemma \ref{lem:Linteg}, we have $\alpha_{\msp} (X) \ge 1$ for any $\msp \in L_{xy} \cap \Sm (X)$.

\item Case (ii): $\alpha \ne 0$, $(\beta, \delta) \ne (0,0)$ and $(\alpha,\gamma)$ is proportional to $(\beta,\delta)$.
In this case $f = (\alpha w + \gamma t z)(t^2 + \varepsilon z^3)$, where $\varepsilon := \beta/\alpha \in \mbC$ is non-zero.  
Replacing $w$ and $z$, we may assume $f = w (t^2 - z^3)$ and $T|_S = \Gamma + \Delta$, where 
\[
\Gamma = (x = y = w = 0), \quad \Delta = (x = y = t^2 - z^3 = 0).
\]
We see that $\Gamma$ is a quasi-line and $\Delta$ is an irreducible curve which is quasi-smooth along $\Delta \setminus \{\msp_w\} = \Delta \cap \Sm (X)$.
We have $\Gamma \cap \Delta = \{\msq\}$, where $\msq = (0\!:\!0\!:\!1\!:\!1\!:\!0) \in \Sm (X)$.
By a similar argument as in Case (ii) of \S \ref{sec:smptL2-7}, we conclude that $S$ is quasi-smooth at $\msq$.
Finally we have $\Sing_{\Gamma} (X) = \{\msp_z,\msp_t\}$.
Thus, by Lemma \ref{lem:Lredcp1}, we have $\alpha_{\msp} (X) \ge 1/2$ for any $\msp \in L_{xy} \cap \Sm (X)$.

\item Case (iii): $\alpha \ne 0$ and $(\beta,\delta) = (0,0)$.
In this case $f = t^2 (\alpha w + \gamma t z)$ and we may assume $f = t^2 w$ by replacing $w$.
We can write
\[
F = f_1 (z,w) x + f_2 (z,w) y + t^2 w + g (x,y,z,t,w),
\]
where $f_1, f_2 \in \mbC [z,w]$ and $g \in \mbC [x,y,z,t,w]$ are homogeneous polynomials such that $g \in (x, y) \cap (x,y,t)^2$.
By Lemma \ref{lem:Lredcp2}, we have $\alpha_{\msp} (X) \ge 1/2$ for any $\msp \in L_{xy} \cap \Sm (X)$.

\item Case (iv): $\alpha = 0$ and $\beta \ne 0$.
Note that $\gamma \ne 0$.
In this case $f = z (\beta w z^2 + \gamma t^3 + \delta t z^3)$.
Then $T|_S = \Gamma + \Delta$, where
\[
\Gamma = (x = y = z = 0), \quad
\Delta = (x = y = \beta w z^2 + \gamma t^3 + \delta t z^3 = 0).
\]
We see that $\Gamma$ is a quasi-line and $\Delta$ is an irreducible curve which is quasi-smooth along $\Delta \setminus \{\msp_w\} \supset \Delta \cap \Sm (X)$.
We have $\Gamma \cap \Delta = \{\msp_w\} \subset \Sing (X)$.
By a similar argument as in Case (ii) of \S \ref{sec:smptL2-7}, we conclude that $S$ is quasi-smooth at $\msp_w$.
Finally we have $\Sing_{\Gamma} (X) = \{\msp_t, \msp_w\}$.
Then, by Lemma \ref{lem:Lredcp1}, we have $\alpha_{\msp} (X) \ge 1/2$ for any $\msp \in L_{xy} \cap \Sm (X)$.

\item Case (v): $\alpha = \beta = 0$ and $\delta \ne 0$.
Note that $\gamma \ne 0$.
In this case $f = z t (\gamma t^2 + \delta z^3)$ and $T|_S = \Gamma_1 + \Gamma_2 + \Gamma_3$, where $\Gamma_1, \Gamma_2, \Gamma_3$ are as follows.
\begin{itemize}
\item $\Gamma_1 = (x = y = z = 0)$ is a quasi-line of degree $1/15$ and $\Sing_{\Gamma_1} (S) = \{1 \times \frac{1}{3} (1,2), 1 \times \frac{1}{5} (2,3)\}$.
\item $\Gamma_2 = (x = y = t = 0)$ is a quasi-line of degree $1/10$ and $\Sing_{\Gamma_2} (S) = \{1 \times \frac{1}{2} (1,1), 1 \times \frac{1}{5} (2,3)\}$.
\item $\Gamma_3 = (x = y = \gamma t^2 + \delta z^3 = 0)$ is an irreducible smooth curve of degree $1/5$.
\item For $ 1 \le i < j \le 3$, we have $\Gamma_i \cap \Gamma_j = \{\msp_w\} \subset \Sing (X)$.
Moreover $S$ is quasi-smooth at $\msp_w$.
\end{itemize}
We compute $(\Gamma_1^2)_S = - 8/15$ and $(\Gamma_2^2)_S = - 15/19$ by the method explained in Remark \ref{rem:compselfint}.
We can choose $z, t$ as orbifold coordinates of $S$ at $\msp_w$.
It follows that $\Gamma_1, \Gamma_2$ intersect transversally at the point over $\msp_w$ on the orbifold chart of $S$ at $\msp_w$ and we have $(\Gamma_1 \cdot \Gamma_2)_S = 1/5$.
Then, by taking intersections with $T|_S = \Gamma_1 + \Gamma_2 + \Gamma_3$ with $\Gamma_i$ for $i = 1,2,3$, we see that the intersection matrix of $\Gamma_1, \Gamma_2, \Gamma_3$ is given by
\[
\begin{pmatrix}
- \frac{8}{15} & \frac{1}{5} & \frac{2}{5} \\[1mm]
\frac{1}{5} & - \frac{7}{10} & \frac{3}{5} \\[1mm]
\frac{2}{5} & \frac{3}{5} & - 1
\end{pmatrix}
\]
and it satisfies the condition $(\star)$.
By Lemma \ref{lem:mtdLred},
\[
\alpha_{\msp} (X) \ge
\begin{cases}
10/13, & \text{if $\msp \in \Gamma_1 \cap \Sm (X)$}, \\
15/19, & \text{if $\msp \in \Gamma_2 \cap \Sm (X)$}, \\
6/7, & \text{if $\msp \in \Gamma_3 \cap \Sm(X)$}.
\end{cases}
\] 
Thus, we have $\alpha_{\msp} (X) \ge 10/13$ for any $\msp \in L_{xy} \cap \Sm (X)$.

\item Case (vi): $\alpha = \beta = \delta = 0$.
In this case $\gamma \ne 0$ and we may assume that $f = t^3 z$ by re-scaling $z$.
We can write
\[
F = f_1 (z,w) x + f_2 (z,w) y + t^3 z + g (x,y,z,t,w),
\]
where $f_1, f_2 \in \mbC [z,w]$ and $g \in \mbC [x,y,z,t,w]$ are homogeneous polynomials such that $g \in (x,y) \cap (x,y,t)^2$.
By Lemma \ref{lem:Lredcp2}, we have $\alpha_{\msp} (X) \ge 1/2$ for any $\msp \in L_{xy} \cap \Sm (X)$.
\end{itemize}

\paragraph{\it The family $\mcF_{15}$}

We have $w^2 \in f$.
Replacing $w$, we may assume that the coefficients of $z^6$ and $t^4$ are both $0$.
Then, by the quasi-smoothness of $X$ at $\msp_z, \msp_t \in X$, we have $t^2 w, z^3 w \in F$.
Hence, by rescaling $w, t$ and $z$, we can write
\[
f = w^2 + (t^2 - z^3)w + \alpha t^2 z^3
\] 
for some $\alpha \in \mbC$.

\begin{itemize}
\item Case (i): $\alpha \ne 0$.
In this case $S \cdot T = L_{xy}$ is irreducible and smooth.
By Lemma \ref{lem:Linteg}, we have $\alpha_{\msp} (X) = 1$ for any $\msp \in L_{xy} \cap \Sm (X)$.

\item Case (ii): $\alpha = 0$.
Replacing $w$ and rescaling $z$, we may assume $f = w(w + t^2 - z^3)$ and $T|_S = \Gamma + \Delta$, where
\[
\Gamma = (x = y = w = 0), \quad
\Delta = (x = y = w + t^2 - z^3 = 0).
\]
We see that $\Gamma$ and $\Delta$ are both quasi-lines.
We have $\Gamma \cap \Delta = \{\msq\}$, where $\msq = (0\!:\!0\!:\!1\!:\!1\!:\!0) \in \Sm (X)$.
By a similar argument as in Case (ii) of \S \ref{sec:smptL2-7}, we conclude that $S$ is quasi-smooth at $\msq$.
Finally we have $\Sing_{\Gamma} (X) = \{\msp_z, \msp_t\}$.
Thus, by Lemma \ref{lem:Lredcp1}, we have $\alpha_{\msp} (X) \ge 1/2$ for any $\msp \in L_{xy} \cap \Sm (X)$.
\end{itemize}

\paragraph{\it The family $\mcF_{20}$}

We have $w^2 z \in F$ by the quasi-smoothness of $X$ at $\msp_w$.
Hence we can write
\[
f = w^2 z + \alpha w t^2 + \beta t z^3,
\] 
where $\alpha, \beta \in \mbC$.

\begin{itemize}
\item Case (i): $\alpha \ne 0$ and $\beta \ne 0$.
In this case $S \cdot T = L_{xy}$ is irreducible and smooth.
By Lemma \ref{lem:Linteg}, we have $\alpha_{\msp} (X) = 1$ for any $\msp \in L_{xy} \cap \Sm (X)$.

\item Case (ii): $\alpha \ne 0$ and $\beta = 0$.
We have $f = w (w z + \alpha t^2)$ and thus $T|_S = \Gamma + \Delta$, where
\[
\Gamma = (x = y = w = 0), \quad
\Delta = (x = y = w z + \alpha t^2 = 0).
\]
We see that $\Gamma$ is a quasi-line and $\Delta$ is an irreducible quasi-smooth curve.
We have $\Gamma \cap \Delta = \{\msp_z\} \subset \Sing (X)$.
By a similar argument as in Case (ii) of \S \ref{sec:smptL2-7}, we conclude that $S$ is quasi-smooth at $\msp_z$.
Finally we have $\Sing_{\Gamma} (X) = \{\msp_z,\msp_t\}$.
Thus, by Lemma \ref{lem:Lredcp1}, we have $\alpha_{\msp} (X) \ge 1/2$ for any $\msp \in L_{xy} \cap \Sm (X)$.

\item Case (iii): $\alpha = 0$ and $\beta \ne 0$.
We have $f = z (w^2 + \beta t z^2)$ and thus $T|_S = \Gamma + \Delta$, where
\[
\Gamma = (x = y = z = 0), \quad
\Delta = (x = y = w^2 + \beta t z^2 = 0).
\]
We see that $\Gamma$ is a quasi-line and $\Delta$ is an irreducible curve which is quasi-smooth along $\Delta \setminus \{\msp_t\} \supset \Delta \cap \Sm (X)$.
We have $\Gamma \cap \Delta = \{\msp_t\}$.
By a similar argument as in Case (ii) of \S \ref{sec:smptL2-7}, we conclude that $S$ is quasi-smooth at $\msp_t$.
Finally we have $\Sing_{\Gamma} (X) = \{\msp_t, \msp_w\}$.
Thus, by Lemma \ref{lem:Lredcp1}, we have $\alpha_{\msp} (X) \ge 1/2$ for any $\msp \in L_{xy} \cap \Sm (X)$.

\item Case (iv): $\alpha = \beta = 0$.
In this case $f = w^2 z$ and we can write
\[
F = f_1 (z,t) x + f_2 (z,t) y + w^2 z + g (x,y,z,t,w),
\]
where $f_1, f_2 \in \mbC [z,t]$ and $g \in \mbC [x,y,z,t,w]$ are homogeneous polynomials such that $g \in (x,y) \cap (x,y,w)^2$.
By Lemma \ref{lem:Lredcp2}, we have $\alpha_{\msp} (X) \ge 1/2$ for any $\msp \in L_{xy} \cap \Sm (X)$.
\end{itemize}

\paragraph{\it The family $\mcF_{23}$}

We have $w^2 t \in F$ by the quasi-smoothness of $X$ at $\msp_w$.
Hence we can write 
\[
f = w^2 t + \alpha w z^3 + \beta t^2 z^2,
\] 
where $\alpha, \beta \in \mbC$.

\begin{itemize}
\item Case (i): $\alpha \ne 0$ and $\beta \ne 0$.
In this case $S \cdot T = L_{xy}$ is irreducible and is smooth outside $\msp_t \in \Sing (X)$.
By Lemma \ref{lem:Linteg}, we have $\alpha_{\msp} (X) = 1$ for any $\msp \in L_{xy} \cap \Sm (X)$.

\item Case (ii): $\alpha \ne 0$ and $\beta = 0$.
We have $f = w (w t + \alpha z^3)$ and thus $T|_S = \Gamma + \Delta$, where
\[
\Gamma = (x = y = w = 0), \quad
\Delta = (x = y = w t + \alpha z^3 = 0).
\]
We see that $\Gamma$ is a quasi-line and $\Delta$ is an irreducible quasi-smooth curve.
We have $\Gamma \cap \Delta = \{\msp_t\} \subset \Sing (X)$, and $S$ is quasi-smooth at $\msp_t$ since $t^3 y \in F$ and $S = H_x$.
Finally we have $\Sing_{\Gamma} (X) = \{\msp_z,\msp_t\}$.
Thus, by Lemma \ref{lem:Lredcp1}, we have $\alpha_{\msp} (X) \ge 1/2$ for any $\msp \in L_{xy} \cap \Sm (X)$.

\item Case (iii): $\alpha = 0$ and $\beta \ne 0$.
We have $f = t (w^2 + \beta t z^2)$ and thus $T|_S = \Gamma + \Delta$, where
\[
\Gamma = (x = y = t = 0), \quad
\Delta = (x = y = w^2 + \beta t z^2 = 0).
\]
We see that $\Gamma$ is a quasi-line and $\Delta$ is an irreducible curve which is quasi-smooth along $\Delta \setminus \{\msp_t\} \supset \Delta \cap \Sm (X)$.
We have $\Gamma \cap \Delta = \{\msp_z\}$, and $S$ is quasi-smooth at $\msp_z$ since $z^4 y \in F$ and $S = H_x$.
Finally we have $\Sing_{\Gamma} (X) = \{\msp_z,\msp_w\}$.
Thus, by Lemma \ref{lem:Lredcp1}, we have $\alpha_{\msp} (X) \ge 1/2$ for any $\msp \in L_{xy} \cap \Sm (X)$.

\item Case (iv): $\alpha = \beta = 0$.
In this case $f = w^2 t$ and we can write
\[
F = f_1 (z,t) x + f_2 (z,t) + w^2 t + g (x,y,z,t,w),
\]
where $f_1, f_2 \in \mbC [z,t]$ and $g \in \mbC [x,y,z,t,w]$ are homogeneous polynomials such that $g \in (x,y) \cap (x,y,w)^2$.
By Lemma \ref{lem:Lredcp2}, we have $\alpha_{\msp} (X) \ge 1/2$ for any $\msp \in L_{xy} \cap \Sm (X)$.
\end{itemize}

\paragraph{\it The family $\mcF_{24}$}

We have $t^3 \in F$ and, by rescaling $t$, we can write
\[
f = \alpha w z^4 + t^3 + \beta t z^5,
\] 
where $\alpha, \beta \in \mbC$.

\begin{itemize}
\item Case (i): $\alpha \ne 0$.
In this case $S \cdot T = L_{xy}$ is irreducible and is smooth outside $\msp_w \in \Sing (X)$.
By Lemma \ref{lem:Linteg}, we have $\alpha_{\msp} (X) = 1$ for any $\msp \in L_{xy} \cap \Sm (X)$.

\item Case (ii): $\alpha = 0$ and $\beta \ne 0$.
By rescaling $z$, we may assume $f = t (t^2 + z^5)$.
Then $T|_S = \Gamma + \Delta$, where
\[
\Gamma = (x = y = t = 0), \quad
\Delta = (x = y = t^2 + z^5 = 0).
\]
We see that $\Gamma$ is a quasi-line and $\Delta$ is an irreducible curve which is quasi-smooth along $\Delta \setminus \{\msp_w\} = \Delta \cap \Sm (X)$.
We have $\Gamma \cap \Delta = \{\msp_w\}$, and $S$ is quasi-smooth at $\msp_w$ by a similar argument as in Case (ii) of \S \ref{sec:smptL2-7}.
Finally we have $\Sing_{\Gamma} (X) = \{\msp_z, \msp_w\}$.
Thus, by Lemma \ref{lem:Lredcp1}, we have $\alpha_{\msp} (X) \ge 1/2$ for any $\msp \in L_{xy} \cap \Sm (X)$.

\item Case (iii): $\alpha = \beta = 0$.
In this case, $f = t^3$ and the defining polynomial $F$ of $X$ can be written as
\[
F = w^2 \ell_1 (x,y) + t^3 + z^7 \ell_2 (x,y) + w h_8 (x,y,z,t) + h_{15} (x,y,z,t),
\]
where $h_8, h_{15} \in \mbC [x,y,z,t]$ are homogeneous polynomials of degrees $8, 15$, respectively, such that $z^4 \notin h_8$ and $t^3, z^7 x, z^7 y \notin h_{15}$, and $\ell_1, \ell_2$ are linear forms in $x, y$.
Note that $h_8, h_{15} \in (x,y) \cap (x,y,t)^2$.
By the quasi-smoothness of $X$, we see that $\ell_1$ and $\ell_2$ are linearly independent.
Replacing $x, y$, we can assume that
\[
F = w^2 x + t^3 - z^7 y + g,
\]
where $h = w h_8 + h_{15} \in (x, y) \cap (x, y, t)^2$.
Thus, by Lemma \ref{lem:Lnonredu2}, we have $\alpha_{\msp} (X) \ge 1/2$ for any $\msp \in L_{xy} \cap \Sm (X)$.
\end{itemize}

\paragraph{\it The family $\mcF_{25}$}

We have $z^5 \in F$ and, by rescaling $z$, we can write
\[
f = \alpha w t^2 + \beta t^3 z + z^5,
\] 
where $\alpha, \beta \in \mbC$.

\begin{itemize}
\item Case (i): $\alpha \ne 0$.
In this case $S \cdot T = L_{xy}$ is irreducible and is smooth outside $\msp_w \in \Sing (X)$.
By Lemma \ref{lem:Linteg}, we have $\alpha_{\msp} (X) \ge 1$ for any $\msp \in L_{xy} \cap \Sm (X)$.

\item Case (ii): $\alpha = 0$.
By the quasi-smoothness of $X$ at $\msp_t$, we have $\beta \ne 0$, and hence we may assume $\beta = 1$.
Then $f = z ( t^3 + z^5)$ and we have $T|_S = \Gamma + \Delta$, where
\[
\Gamma = (x = y = z = 0), \quad
\Delta = (x = y = t^3 + z^5 = 0).
\]
We see that $\Gamma$ is a quasi-line and $\Delta$ is an irreducible curve which is quasi-smooth along $\Delta \setminus \{\msp_w\} = \Delta \cap \Sm (X)$.
We have $\Gamma \cap \Delta = \{\msp_w\}$, and $S$ is quasi-smooth at $\msp_w$ by a similar argument as in Case (ii) of \S \ref{sec:smptL2-7}.
Finally we have $\Sing_{\Gamma} (X) = \{\msp_t,\msp_w\}$.
Thus, by Lemma \ref{lem:Lredcp1}, we have $\alpha_{\msp} (X) \ge 1/2$ for any $\msp \in L_{xy} \cap \Sm (X)$.
\end{itemize}

\paragraph{\it The family $\mcF_{29}$}

We have $w^2 \in F$ and, by rescaling $w$, we can write $f = w^2 + \lambda w z^4 + \alpha t^2 z^3 + \mu z^8$, where $\alpha, \lambda, \mu \in \mbC$.
By replacing $w$, we can eliminate the term $\mu z^8$, i.e.\ we may assume $\mu = 0$.
Then, by the quasi-smoothness of $X$ at $\msp_z$, we have $w z^4 \in F$, i.e.\ $\lambda \ne 0$.
Thus we can write 
\[
f = w^2 + w z^4 + \alpha t^2 z^3.
\]

\begin{itemize}
\item Case (i): $\alpha \ne 0$.
In this case $S \cdot T = L_{xy}$ is irreducible and is smooth outside $\msp_t \in \Sing (X)$.
By Lemma \ref{lem:Linteg}, we have $\alpha_{\msp} (X) \ge 1$ for any $\msp \in L_{xy} \cap \Sm (X)$. 

\item Case (ii): $\alpha = 0$.
Then we have $f = w(w+z^4)$ and $T|_S = \Gamma + \Delta$, where
\[
\Gamma = (x = y = w = 0), \quad
\Delta = (x = y = w + z^4).
\]
We see that $\Gamma$ and $\Delta$ are both quasi-lines.
We have $\Gamma \cap \Delta = \{\msp_t\}$, and $S$ is quasi-smooth at $\msp_t$ by a similar argument as in Case (ii) of \S \ref{sec:smptL2-7}.
Finally we have $\Sing_{\Gamma} (X) = \{\msp_z, \msp_t\}$.
Thus, by Lemma \ref{lem:Lredcp1}, we have $\alpha_{\msp} (X) \ge 1/2$ for any $\msp \in L_{xy} \cap \Sm (X)$.
\end{itemize}

\paragraph{\it The family $\mcF_{30}$}

We have $w^2 \in F$ and, by rescaling $w$, we can write $f = w^2 + \lambda w t^2 + \mu t^4 + \alpha t z^4$, where $\alpha, \lambda, \mu \in \mbC$.
We may assume $\mu = 0$ by replacing $w$, and then we have $w t^2 \in F$ by the quasi-smoothness of $X$ at $\msp_t$.
Thus we can write 
\[
f = w^2 + w t^2 + \alpha t z^4.
\]

\begin{itemize}
\item Case (i): $\alpha \ne 0$.
In this case $S \cdot T = L_{xy}$ is irreducible and smooth.
By Lemma \ref{lem:Linteg}, we have $\alpha_{\msp} (X) \ge 1$ for any $\msp \in L_{xy} \cap \Sm (X)$. 

\item Case (ii): $\alpha = 0$.
Then we have $f = w(w+t^2)$ and $T|_S = \Gamma + \Delta$, where
\[
\Gamma = (x = y = w = 0), \quad
\Delta = (x = y = w + t^2 = 0).
\]
We see that $\Gamma$ and $\Delta$ are both quasi-lines.
We have $\Gamma \cap \Delta = \{\msp_z\}$, and $S$ is quasi-smooth at $\msp_z$ by a similar argument as in Case (ii) of \S \ref{sec:smptL2-7}.
Finally we have $\Sing_{\Gamma} (X) = \{\msp_z,\msp_t\}$.
Thus, by Lemma \ref{lem:Lredcp1}, we have $\alpha_{\msp} (X) \ge 1/2$ for any $\msp \in L_{xy} \cap \Sm (X)$.
\end{itemize}

\paragraph{\it The family $\mcF_{31}$}

We have $w^2 z \in F$ by the quasi-smoothness of $X$ at $\msp_w$, and we have $z^4 \in F$.
Rescaling $w$ and $z$, we can write 
\[
f = w^2 z + \alpha w t^2 - z^4,
\] 
where $\alpha \in \mbC$.

\begin{itemize}
\item Case (i): $\alpha \ne 0$.
In this case $S \cdot T = L_{xy}$ is irreducible and smooth.
By Lemma \ref{lem:Linteg}, we have $\alpha_{\msp} (X) \ge 1$ for any $\msp \in L_{xy} \cap \Sm (X)$.

\item Case (ii): $\alpha = 0$.
In this case $f = z (w^2 - z^3)$ and $T|_S = \Gamma + \Delta$, where 
\[
\Gamma = (x = y = z = 0), \quad
\Delta = (x = y = w^2 - z^3 = 0).
\]
We see that $\Gamma$ is a quasi-line and $\Delta$ is an irreducible curve which is quasi-smooth along $\Delta \setminus \{\msp_t\} \supset \Delta \cap \Sm (X)$.
We have $\Gamma \cap \Delta = \{\msp_t\} \subset \Sing (X)$, and $S$ is quasi-smooth at $\msp_t$ by a similar argument as in Case (ii) of \S \ref{sec:smptL2-7}.
Thus, by Lemma \ref{lem:Lredcp1}, we have $\alpha_{\msp} (X) \ge 1/2$ for any $\msp \in L_{xy} \cap \Sm (X)$.
\end{itemize}

\paragraph{\it The family $\mcF_{32}$}
We have $t^4 \in F$, and we can write
\[
f = \alpha w z^3 + t^4 + \beta t z^4,
\] 
where $\alpha, \beta \in \mbC$.

\begin{itemize}
\item Case (i): $\alpha \ne 0$.
In this case $S \cdot T = L_{xy}$ is irreducible and is smooth outside $\msp_w \in \Sing (X)$.
By Lemma \ref{lem:Linteg}, we have $\alpha_{\msp} (X) \ge 1$ for any $\msp \in L_{xy} \cap \Sm (X)$.

\item Case (ii): $\alpha = 0$ and $\beta \ne 0$.
Re-scaling $z$, we may assume $\beta = 1$ and $f = t (t^3 + z^4)$.
Then $T|_S = \Gamma + \Delta$, where
\[
\Gamma = (x = y = t = 0), \quad
\Delta = (x = y = t^3 + z^4 = 0).
\]
We see that $\Gamma$ is a quasi-line and $\Delta$ is an irreducible curve which is quasi-smooth along $\Delta \setminus \{\msp_w\} = \Delta \cap \Sm (X)$.
We have $\Gamma \cap \Delta = \{\msp_w\}$, and $S$ is quasi-smooth at $\msp_w$ since $w^2 y \in F$ and $S = H_x$.
Finally we have $\Sing_{\Gamma} = \{\msp_z, \msp_w\}$.
Thus, by Lemma \ref{lem:Lredcp1}, we have $\alpha_{\msp} (X) \ge 1/2$ for any $\msp \in L_{xy} \cap \Sm (X)$.

\item Case (iii): $\alpha = \beta = 0$.
In this case $f = t^4$.
Since $w z^3, t z^4 \notin F$, we have $z^5 x \in F$, and we can write
\[
F = w^2 y + t^4 + z^5 x + w h_9 (x,y,z,t) + h_{16} (x,y,z,t),
\]
where $g_i \in \mbC [x,y,z,t]$ is a homogeneous polynomial of degree $i$ such that $z^3 \notin h_9$ and $t^4, t z^4, z^5 x \notin h_{16}$.
Note that $h_9, h_{16} \in (x, y) \cap (x, y, t)^2$.
Thus, by Lemma \ref{lem:Lnonredu2}, we have $\alpha_{\msp} (X) \ge 1/2$ for any $\msp \in L_{xy} \cap \Sm (X)$.
\end{itemize}

\paragraph{\it The family $\mcF_{33}$}
We have $w^2 z \in F$ by the quasi-smoothness of $X$ at $\msp_w$, and we can write 
\[
f = w^2 z + \alpha w t^2 + \beta t z^4,
\]
where $\alpha, \beta \in \mbC$.

\begin{itemize}
\item Case (i): $\alpha \ne 0$ and $\beta \ne 0$.
In this case $S \cdot T = L_{xy}$ is irreducible and smooth.
By Lemma \ref{lem:Linteg}, we have $\alpha_{\msp} (X) \ge 1$ for any $\msp \in L_{xy} \cap \Sm (X)$.

\item Case (ii): $\alpha \ne 0$ and $\beta = 0$.
In this case $f = w (w z + \alpha z^2)$ and $T|_S = \Gamma + \Delta$, where
\[
\Gamma = (x = y = w = 0), \quad
\Delta = (x = y = w z + \alpha z^2 = 0).
\]
We see that $\Gamma$ is a quasi-line and $\Delta$ is an irreducible quasi-smooth curve.
We have $\Gamma \cap \Delta = \{\msp_z\} \subset \Sing (X)$, and $S$ is quasi-smooth at $\msp_z$ since $z^5 t \in F$ and $S = H_x$.
Finally we have $\Sing_{\Gamma} (X) = \{\msp_z, \msp_t\}$.
Thus, by Lemma \ref{lem:Lredcp1}, we have $\alpha_{\msp} (X) \ge 1/2$ for any $\msp \in L_{xy} \cap \Sm (X)$.

\item Case (iii): $\alpha = 0$ and $\beta \ne 0$.
In this case $f = z (w^2 + \beta t z^3)$ and $T|_S = \Gamma + \Delta$, where
\[
\Gamma = (x = y = z = 0), \quad
\Delta = (x = y = w^2 + \beta t z^3 = 0).
\]
We see that $\Gamma$ is a quasi-line and $\Delta$ is an irreducible curve which is quasi-smooth along $\Delta \setminus \{\msp_t\} \supset \Delta \cap \Sm (X)$.
We have $\Gamma \cap \Delta = \{\msp_t\} \subset \Sing (X)$, and $S$ is quasi-smooth at $\msp_t$ since $t^3 y \in F$ and $S = H_x$.
Finally we have $\Sing_{\Gamma} (X) = \{\msp_t, \msp_w\}$.
Thus, by Lemma \ref{lem:Lredcp1}, we have $\alpha_{\msp} (X) \ge 1/2$ for any $\msp \in L_{xy} \cap \Sm (X)$.

\item Case (iv): $\alpha = \beta = 0$.
In this case $f = w^2 z$ and we can write
\[
F = f_1 (z,t) x + f_2 (z,t) y + w^2 z + g (x,y,z,t,w),
\]
where $f_1, f_2 \in ^mbC [z,t]$ and $g \in \mbC [x,y,z,t,w]$ are homogeneous polynomials such that $g \in (x,y) \cap (x,y,w)^2$.
By Lemma \ref{lem:Lredcp2}, we have $\alpha_{\msp} (X) \ge 1/2$ for any $\msp \in L_{xy} \cap \Sm (X)$.
\end{itemize}

\paragraph{\it The family $\mcF_{37}$}
\label{sec:smptL2-37}

We have $w^2 \in F$ and we can write $f = w^2 + \lambda w z^3 + \alpha t^3 z^2 + \mu z^6$, where $\alpha, \lambda, \mu \in \mbC$.
Replacing $w$, we may assume $\mu = 0$.
Then, by the quasi-smoothness of $X$ at $\msp_z$, we have $\lambda \ne 0$.
Re-scaling $z$, we can write 
\[
f = w^2 + w z^3 + \alpha t^3 z^2.
\]

\begin{itemize}
\item Case (i): $\alpha \ne 0$.
In this case $S \cdot T = L_{xy}$ is irreducible and is smooth outside $\msp_t \in \Sing (X)$. 
By Lemma \ref{lem:Linteg}, we have $\alpha_{\msp} (X) \ge 1$ for any $\msp \in L_{xy} \cap \Sm (X)$.

\item Case (ii): $\alpha = 0$.
In this case $f = w(w+z^3)$ and $T|_S = \Gamma + \Delta$, where
\[
\Gamma = (x = y = w = 0), \quad
\Delta = (x = y = w + z^3 = 0),
\]
are both quasi-lines.
We have $\Gamma \cap \Delta = \{\msp_t\}$, and $S$ is quasi-smooth at $\msp_t$ since $t^4 y \in F$ and $S = H_x$.
Finally we have $\Sing_{\Gamma} (X) = \{\msp_z, \msp_t\}$.
Thus, by Lemma \ref{lem:Lredcp1}, we have $\alpha_{\msp} (X) \ge 1/2$ for any $\msp \in L_{xy} \cap \Sm (X)$.
\end{itemize}

\paragraph{\it The family $\mcF_{38}$}

We have $z^6 \in F$, and we can write
\[
f = \alpha w t^2 + \beta t^3 z + z^6,
\]
where $\alpha, \beta \in \mbC$.
Note that we have $(\alpha,\beta) \ne (0,0)$ by the quasi-smoothness of $X$.

\begin{itemize}
\item Case (i): $\alpha \ne 0$.
In this case $S \cdot T = L_{xy}$ is irreducible and is smooth outside $\msp_w \in \Sing (X)$.
By Lemma \ref{lem:Linteg}, we have $\alpha_{\msp} (X) \ge 1$ for any $\msp \in L_{xy} \cap X_{\sm}$.

\item Case (ii): $\alpha = 0$.
Note that $\beta \ne 0$.
In this case $f = z (\beta t^3 + z^5)$ and $T|_S = \Gamma + \Delta$, where 
\[
\Gamma = (x = y = z = 0), \quad
\Delta = (x = y = \beta t^3 + z^5 = 0).
\]
We see that $\Gamma$ is a quasi-line and $\Delta$ is an irreducible curve which is quasi-smooth along $\Delta \setminus \{\msp_w\} = \Delta \cap \Sm (X)$.
We have $\Gamma \cap \Delta = \{\msp_w\} \subset \Sing (X)$, and $S$ is quasi-smooth at $\msp_w$ since $w^2 y \in F$ and $S = H_x$.
Finally we have $\Sing_{\Gamma} (X) = \{\msp_t,\msp_w\}$.
Thus, by Lemma \ref{lem:Lredcp1}, we have $\alpha_{\msp} (X) \ge 1/2$ for any $\msp \in L_{xy} \cap \Sm (X)$.
\end{itemize}

\paragraph{\it The family $\mcF_{39}$}
\label{sec:smptL2-39}

We have $w^3 \in F$ and $w z^3 \in F$ by the quasi-smoothness of $X$.
Rescaling $w$ and $z$, we can write 
\[
f = w^3 + w z^3 + \alpha t^2 z^2,
\]
where $\alpha \in \mbC$.

\begin{itemize}
\item Case (i): $\alpha \ne 0$.
In this case $S \cdot T = L_{xy}$ is irreducible and is smooth outside $\msp_t \in \Sing (X)$.
By Lemma \ref{lem:Linteg}, we have $\alpha_{\msp} (X) = 1$ for any $\msp \in L_{xy} \cap \Sm (X)$.

\item Case (ii): $\alpha = 0$.
In this case $f = z (w^2 + z^3)$ and we have $T|_S = \Gamma + \Delta$, where
\[
\Gamma = (x = y = z = 0), \quad
\Delta = (x = y = w^2 + z^3 = 0).
\]
We see that $\Gamma$ is a quasi-line and $\Delta$ is an irreducible curve which is quasi-smooth along $\Delta \setminus \{\msp_t\} = \Delta \cap \Sm (X)$.
We have $\Gamma \cap \Delta = \{\msp_t\}$, and $S$ is quasi-smooth at $\msp_t$ since $t^3 y \in F$ and $S = H_x$.
Finally we have $\Sing_{\Gamma} (X) = \{\msp_z,\msp_t\}$.
Thus, by Lemma \ref{lem:Lredcp1}, we have $\alpha_{\msp} (X) \ge 1/2$ for any $\msp \in L_{xy} \cap \Sm (X)$.
\end{itemize}

\paragraph{\it The family $\mcF_{40}$}
\label{sec:smptL2-40}

We have $w^2 t \in F$ and $t^3 z \in F$ by the quasi-smoothness of $X$ at $\msp_w$ and $\msp_t$.
Rescaling $w$ and $z$, we can write 
\[
f = w^2 t + \alpha w z^3 + t^3 z,
\]
where $\alpha \in \mbC$.

\begin{itemize}
\item Case (i): $\alpha \ne 0$.
In this case $S \cdot T = L_{xy}$ is irreducible and smooth.
By Lemma \ref{lem:Linteg}, we have $\alpha_{\msp} (X) = 1$ for any $\msp \in L_{xy} \cap \Sm (X)$.

\item Case (ii): $\alpha = 0$.
In this case $f = t (w^2 + t^2 z)$ and $T|_S = \Gamma + \Delta$, where
\[
\Gamma = (x = y = t = 0), \quad
\Delta = (x = y = w^2 + t^2 z = 0).
\]
We see that $\Gamma$ is a quasi-line and $\Delta$ is an irreducible curve which is quasi-smooth along $\Delta \setminus \{\msp_z\} \supset \Delta \cap \Sm (X)$.
We have $\Gamma \cap \Delta = \{\msp_z\}$, and $S$bis quasi-smooth at $\msp_z$ since $z^4 y \in F$ and $S = H_x$.
Finally we have $\Sing_{\Gamma} (X) = \{\msp_z, \msp_w\}$.
Thus, by Lemma \ref{lem:Lredcp1}, we have $\alpha_{\msp} (X) \ge 1/2$ for any $\msp \in L_{xy} \cap \Sm (X)$.
\end{itemize}

\paragraph{\it The family $\mcF_{42}$}
\label{sec:smptL2-42}

We have $w^2 \in F$ and we can write $f = w^2 + \lambda w t^2 + \mu t^4 + \alpha t z^5$, where $\alpha, \lambda, \mu \in \mbC$.
Replacing $w$, we may assume $\mu = 0$.
Then, by the quasi-smoothness of $X$ at $\msp_t$, we have $\lambda \ne 0$.
Rescaling $t$, we can write 
\[
f = w^2 + w t^2 + \alpha t z^5.
\]

\begin{itemize}
\item Case (i): $\alpha \ne 0$.
In this case $S \cdot T = L_{xy}$ is irreducible and smooth.
By Lemma \ref{lem:Linteg}, we have $\alpha_{\msp} (X) = 1$ for any $\msp \in L_{xy} \cap \Sm (X)$.

\item Case (ii): $\alpha = 0$.
In this case $f = w(w+t^2)$ and $T|_S = \Gamma + \Delta$, where
\[
\Gamma = (x = y = w = 0), \quad
\Delta = (x = y = w + t^2 = 0).
\]
We see that $\Gamma$ and $\Delta$ are both quasi-lines.
We have $\Gamma \cap \Delta = \{\msp_z\} \subset \Sing (X)$, and $S$ is quasi-smooth at $\msp_z$ since $z^6 y \in F$ and $S = H_x$.
Finally we have $\Sing_{\Gamma} (X) = \{\msp_z, \msp_t\}$.
Thus, by Lemma \ref{lem:Lredcp1}, we have $\alpha_{\msp} (X) \ge 1/2$ for any $\msp \in L_{xy} \cap \Sm (X)$.
\end{itemize}

\paragraph{\it The family $\mcF_{49}$}

We have $w^3 \in F$, and we can write
\[
f = w^3 + \alpha t z^3,
\]
where $\alpha \in \mbC$.

\begin{itemize}
\item Case (i): $\alpha \ne 0$.
In this case $S \cdot T = L_{xy}$ is irreducible and is smooth outside $\msp_t \in \Sing (X)$.
By Lemma \ref{lem:Linteg}, we have $\alpha_{\msp} (X) = 1$ for any $\msp \in L_{xy} \cap \Sm (X)$.

\item Case (ii): $\alpha = 0$.
In this case $f = w^3$ and, after replacing $x, y$ suitably, the defining polynomial $F$ of $X$ can be written as
\[
F = w^3 + t^3 y - z^3 x + g_{21} (x,y,z,t,w),
\]
where $g_{21} \in \mbC [x,y,z,t,w]$ is a homogeneous polynomial of degree $21$ such that $g \in (x,y) \cap (x,y,w)^2$.
By Lemma \ref{lem:Lnonredu2}, we have $\alpha_{\msp} (X) \ge 1/2$ for any $\msp \in L_{xy} \cap \Sm (X)$.
\end{itemize}

\paragraph{\it The family $\mcF_{50}$}

We have $w^2 \in F$, and we can write
\[
f = w^2 + \alpha t z^5,
\]
where $\alpha \in \mbC$.

\begin{itemize}
\item Case (i): $\alpha \ne 0$.
In this case $S \cdot T = L_{xy}$ is irreducible and is smooth outside $\msp_t \in \Sing (X)$.
By Lemma \ref{lem:Linteg}, we have $\alpha_{\msp} (X) = 1$ for any $\msp \in L_{xy} \cap \Sm (X)$.

\item Case (ii): $\alpha = 0$.
We have $S \cdot T = 2 \Gamma$, where $\Gamma = (x = y = w = 0)$.
By Lemma \ref{lem:nsptLpuredouble}, we have $\alpha_{\msp} (X) \ge 1/2$ for any $\msp \in L_{xy} \cap \Sm (X)$.
\end{itemize}

\paragraph{\it The family $\mcF_{52}$}

We have $w^2 \in F$, and we can write 
\[
f = w^2 + \alpha t^2 z^3,
\] 
where $\alpha \in \mbC$.

\begin{itemize}
\item Case (i): $\alpha \ne 0$.
In this case $S \cdot T = L_{xy}$ is irreducible and is smooth outside $\{\msp_z, \msp_t\} \subset \Sing (X)$.
By Lemma \ref{lem:Linteg}, we have $\alpha_{\msp} (X) \ge 1$ for any $\msp \in L_{xy} \cap \Sm (X)$.

\item Case (ii): $\alpha = 0$.
We have $S \cdot T = 2 \Gamma$, where $\Gamma = (x = y = w = 0)$.
By Lemma \ref{lem:nsptLpuredouble}, we have $\alpha_{\msp} (X) \ge 1/2$ for any $\msp \in L_{xy} \cap \Sm (X)$.
\end{itemize}

\paragraph{\it The family $\mcF_{58}$}

We have $w^2 z \in F$ by the quasi-smoothness of $X$ at $\msp_w$.
Also we have $z^6 \in F$.
Rescaling $w$ and $z$, we can write 
\[
f = w^2 z + \alpha w t^2 + z^6,
\]
where $\alpha \in \mbC$.

\begin{itemize}
\item Case (i): $\alpha \ne 0$.
In this case $S \cdot T = L_{xy}$ is irreducible and smooth.
By Lemma \ref{lem:Linteg}, we have $\alpha_{\msp} (X) = 1$ for any $\msp \in L_{xy} \cap \Sm (X)$.

\item Case (ii): $\alpha = 0$.
In this case $f = z (w^2 + z^5)$ and $T|_S = \Gamma + \Delta$, where
\[
\Gamma = (x = y = z = 0), \quad
\Delta = (x = y = w^2 + z^5 = 0).
\]
We see that $\Gamma$ is a quasi-line and $\Delta$ is an irreducible curve which is quasi-smooth along $\Delta \setminus \{\msp_t\} \supset \Delta \cap \Sm (X)$.
We have $\Gamma \cap \Delta = \{\msp_t\}$, and $S$ is quasi-smooth at $\msp_t$ since $t^3 y \in F$ and $S = H_x$.
Finally we have $\Sing_{\Gamma} (X) = \{\msp_t,\msp_w\}$.
Thus, by Lemma \ref{lem:Lredcp1}, we have $\alpha_{\msp} (X) \ge 1/2$ for any $\msp \in L_{xy} \cap \Sm (X)$.
\end{itemize}

\paragraph{\it The family $\mcF_{60}$}

We have $w^2 t \in F$ by the quasi-smoothness of $X$ at $\msp_w$.
Also we have $t^4 \in F$.
Rescaling $w$ and $t$, we can write
\[
f = w^2 t + \alpha w z^3 + t^4,
\] 
where $\alpha \in \mbC$.

\begin{itemize}
\item Case (i): $\alpha \ne 0$.
In this case $S \cdot T = L_{xy}$ is irreducible and smooth.
By Lemma \ref{lem:Linteg}, we have $\alpha_{\msp} (X) = 1$ for any $\msp \in L_{xy} \cap \Sm (X)$.

\item Case (ii): $\alpha = 0$.
In this case $f = t (w^2 + t^3)$ and $T|_S = \Gamma + \Delta$, where
\[
\Gamma = (x = y = t = 0), \quad
\Delta = (x = y = w^2 + t^3 = 0).
\]
We see that $\Gamma$ is a quasi-line and $\Delta$ is an irreducible curve which is quasi-smooth along $\Delta \setminus \{\msp_z\} \supset \Delta \cap \Sm (X)$.
We have $\Gamma \cap \Delta = \{\msp_z\}$, and $S$ is quasi-smooth at $\msp_z$ since $z^4 y \in F$ and $S = H_x$.
Finally we have $\Sing_{\Gamma} (X) = \{\msp_z,\msp_w\}$.
Thus, by Lemma \ref{lem:Lredcp1}, we have $\alpha_{\msp} (X) \ge 1/2$ for any $\msp \in L_{xy} \cap \Sm (X)$.
\end{itemize}

\paragraph{\it The family $\mcF_{63}$}
\label{sec:smptL2-63}

We have $w^2 \in F$, and we can write
\[
f = w^2 + \alpha t z^6,
\] 
where $\alpha \in \mbC$.

\begin{itemize}
\item Case (i): $\alpha \ne 0$.
In this case $S \cdot T = L_{xy}$ is irreducible and is smooth outside $\msp_t \in \Sing (X)$.
By Lemma \ref{lem:Linteg}, we have $\alpha_{\msp} (X) = 1$ for any $
\msp \in L_{xy} \cap \Sm (X)$.

\item Case (ii): $\alpha = 0$.
We have $S \cdot T = 2 \Gamma$, where $\Gamma = (x = y = w = 0)$.
By Lemma \ref{lem:nsptLpuredouble}, we have $\alpha_{\msp} (X) \ge 1/2$ for any $\msp \in L_{xy} \cap \Sm (X)$.
\end{itemize}

\paragraph{\it The family $\mcF_{64}$}

We have $w^2 \in F$, and we can write
\[
f = w^2 + \alpha t z^4,
\] 
where $\alpha \in \mbC$.

\begin{itemize}
\item Case (i): $\alpha \ne 0$.
In this case $S \cdot T = L_{xy}$ is irreducible and is smooth outside $\msp_t \in \Sing (X)$. 
By Lemma \ref{lem:Linteg}, we have $\alpha_{\msp} (X) = 1$ for any $\msp \in L_{xy} \cap \Sm (X)$.

\item Case (ii): $\alpha = 0$.
We have $S \cdot T = 2 \Gamma$, where $\Gamma = (x = y = w = 0)$.
By Lemma \ref{lem:nsptLpuredouble}, we have $\alpha_{\msp} (X) \ge 1/2$ for any $\msp \in L_{xy} \cap \Sm (X)$.
\end{itemize}

\subsection{Smooth points on $H_x$ for families with $1 < a_1 = a_2$} \label{sec:smptH}

The aim of this section is to prove the following.

\begin{Prop} \label{prop:smptH}
Let $X = X_d \subset \mbP (1, a_1, \dots, a_4)$, $a_1 \le \dots \le a_4$, be a member of a family $\mcF_{\msi}$ with $\msi \in \msI \setminus \msI_1$ such that $1 < a_1 = a_2$.
Then, 
\[
\alpha_{\msp} (X) \ge \frac{1}{2}
\]
for any smooth point $\msp \in X$ contained in $H_x$.
\end{Prop}

Note that a family $\mcF_{\msi}$ with $\msi \in \msI \setminus \msI_1$ satisfies the assumption of Proposition \ref{prop:smptH} if and only if
\[
\msi \in \{18, 22, 28\}.
\]

\subsubsection{The family $\mcF_{18}$}

This subsection is devoted to the proof of Proposition \ref{prop:smptH} for the family $\mcF_{18}$.
Let $X = X_{12} \subset \mbP (1, 2, 2, 3, 5)$ be a member of $\mcF_{18}$.

By the quasi-smoothness of $X$, We have $t^4 \in F$ and we may assume $\coeff_F (t^4) = 1$ by rescaling $t$.
We have $(x = y = z = 0) \cap X = \{\msp_w\} \subset \Sing (X)$.
Hence we may assume $\msp \in H_y$ and $\msp \notin H_z$ after possibly replacing $y$ and $z$, and we can write $\msp = (0\!:\!0\!:\!1\!:\!\lambda\!:\!\mu)$ for some $\lambda, \mu \in \mbC$.
We can write
\[
F (0, 0, z, t, w) = \alpha w^2 z + \beta w t z^2 + t^4 + \gamma t^2 z^3 + \delta z^6,
\]
where $\alpha, \beta, \gamma, \delta \in \mbC$.
We will derive a contradiction by assuming $\alpha_{\msp} (X) < 1/2$.
By the assumption, there exists an irreducible $\mbQ$-divisor $D \in |A|_{\mbQ}$ such that $\lct_{\msp} (X;D) < 1/2$.

Suppose $\lambda \ne 0$.
Then, by replacing $w$ by $\lambda w - \mu z t$, we may assume $\msp = (0\!:\!0\!:\!1\!:\!\lambda\!:\!0)$.
Let $S$ be a general member of the pencil $|\mcI_{\msp} (2 A)|$ so that $\frac{1}{2} S \ne D$. 
We can take a $\mbQ$-divisor $T \in |5 A|_{\mbQ}$ such that $\mult_{\msp} (T) \ge 1$ and $\Supp (T)$ does not contain any component of the effective $1$-cycle $D \cdot S$ since $\{x, y, w\}$ isolates $\msp$.
Then, we have
\[
\mult_{\msp} (D) \le (D \cdot S \cdot T)_{\msp} \le (D \cdot S \cdot T) = 2 \cdot 5 \cdot (A^3) = 2,
\]
which implies $\lct_{\msp} (X;D) \ge 1/2$.
This is impossible and we have $\lambda = 0$.

By rescaling $w$, we may assume $\msp = (0\!:\!0\!:\!1\!:\!0\!:\!1)$.
Suppose $\alpha \ne 0$.
Then we have $\delta = - \alpha$ since $F (\msp) = 0$.
In this case $H_x$ is smooth at $\msp$ since $(\prt F/\prt z) (\msp) = - 5 \alpha \ne 0$.
In particular $H_x \ne D$.
We can take a $\mbQ$-divisor $T \in |3 A|_{\mbQ}$ such that $\mult_{\msp} (T) \ge 1$ and $\Supp (T)$ does not contain any component of the effective $1$-cycle $D \cdot H_x$ since $\{x, y, t\}$ isolates $\msp$.
Then we have
\[
\mult_{\msp} (D) \le (D \cdot H_x \cdot T)_{\msp} \le (D \cdot H_x \cdot T) = 3 (A^3) = \frac{3}{5},
\]
which implies $\lct_{\msp} (X; D) \ge 5/3$.
This is impossible and we have $\alpha = 0$.
Note that $\delta = 0$ since $F (\msp) = 0$, and we have
\[
F (0,0,z,t,w) = t (\beta w z^2 + t^3 + \gamma t z^3).
\]

We claim $\mult_{\msp} (H_x) \le 2$.
We set $\zeta := \coeff_F (w^2 y)$ and $\eta := \coeff_F (z^5 y)$.
By the quasi-smoothness of $X$, we see $\zeta, \eta \ne 0$ since $w^2 y, z^6 \notin F$.
We have 
\[
\frac{\prt F}{\prt y} (\msp) = \zeta + \eta, \quad 
\frac{\prt F}{\prt t} (\msp) = \beta.
\]
If either $\zeta + \eta \ne 0$ or $\beta \ne 0$, then we have $\mult_{\msp} (H_x) = 1$.
If $\zeta + \eta = \beta = 0$, then we have $\mult_{\msp} (H_x) = 2$ since the term $\zeta y (w^2 - z^5)$ appears in $F$.
Thus the claim is proved.

By the claim, we have $\lct_{\msp} (X;H_x) \ge 1/2$ and in particular $D \ne H_x$.
We can take a $\mbQ$-divisor $T \in |10 A|_{\mbQ}$ such that $\mult_{\msp} (T) \ge 1$ and $\Supp (T)$ does not contain any component of $D \cdot H_x$ since $\{x, y, t, w^2 - z^5\}$ isolates $\msp$.
Then we have
\[
\mult_{\msp} (D) \le (D \cdot H_x \cdot T)_{\msp} \le (D \cdot H_x \cdot T) = 10 (A^3) = 2,
\]
which implies $\lct_{\msp} (X;D) \ge 1/2$.
This is a contradiction and the proof is completed.

\subsubsection{The family $\mcF_{22}$}

This subsection is devoted to the proof of Proposition \ref{prop:smptH} for the family $\mcF_{22}$.
Let $X = X_{14} \subset \mbP (1, 2, 2, 3, 7)$ be a member of $\mcF_{22}$.

By the quasi-smoothness of $X$, we have $w^2 \in F$ and we may assume $\coeff_F (w^2) = 1$ by rescaling $w$.
We see that $(x = y =  z = 0) \cap X = \{\msp_t\} \subset \Sing (X)$.
Hence we may assume $\msp = (0\!:\!0\!:\!1\!:\!\lambda\!:\!\mu)$ for some $\lambda, \mu \in \mbC$ after possibly replacing $y$ and $z$.
We can write
\[
F (0,0,z,t,w) = w^2 + \alpha w t z^2 + \beta t^4 z + \gamma t^2 z^4 + \delta z^7,
\]
where $\alpha, \beta, \gamma, \delta \in \mbC$. 
We will derive a contradiction by assuming $\alpha_{\msp} (X) < 1/2$.
By the assumption, there exists an irreducible $\mbQ$-divisor $D \in |A|_{\mbQ}$ such that $\lct_{\msp} (X;D) < 1/2$.
Let $S$ be a general member of the pencil $|\mcI_{\msp} (2 A)|$ so that $\frac{1}{2} S \ne D$. 

Suppose $\lambda = 0$.
In this case, we can take a $\mbQ$-divisor $T \in |3 A|_{\mbQ}$ such that $\mult_{\msp} (T) \ge 1$ and $\Supp (T)$ does not contain any component of the effective $1$-cycle $D \cdot S$ since $\{x, y, t\}$ isolates $\msp$.
Then, we have
\[
\mult_{\msp} (D) \le (D \cdot S \cdot T)_{\msp} \le (D \cdot S \cdot T) = 2 \cdot 3 \cdot (A^3) = 1,
\]
which implies $\lct_{\msp} (X;D) \ge 1$.
This is impossible and we have $\lambda \ne 0$.

Replacing $w$ by $\lambda w - \mu t z^2$, we may assume $\mu = 0$, that is, $\msp = (0\!:\!0\!:\!1\!:\!\lambda\!:\!0)$.
We see that the set $\{x, y, t^2 - \lambda^2 z^3\}$ isolates $\msp$.
It follows that we can take a $\mbQ$-divisor $T \in |6 A|_{\mbQ}$ such that $\mult_{\msp} (T) \ge 1$ and $\Supp (T)$ which does not contain any component of $D \cdot S$.
Then, we have
\[
\mult_{\msp} (D) \le (D \cdot S \cdot T)_{\msp} \le (D \cdot S \cdot T) = 2 \cdot 6 \cdot (A^3) = 2,
\]
which implies $\alpha_{\msp} (X) \ge 1/2$.
This is a contradiction and the proof is completed.

\subsubsection{The family $\mcF_{28}$}

This subsection is devoted to the proof of Proposition \ref{prop:smptH} for the family $\mcF_{28}$.
Let $X = X_{15} \subset \mbP (1, 3, 3, 4, 5)$ be a member of $\mcF_{28}$.

By the quasi-smoothness of $X$, we have $w^3 \in F$ and we may assume $\coeff_F (w^3) = 1$ by rescaling $w$.
We see that $(x = y = z = 0) \cap X = \{\msp_t\} \subset \Sing (X)$.
Hence we may assume $\msp = (0\!:\!0\!:\!1\!:\!\lambda\!:\!\mu)$ for some $\lambda, \mu \in \mbC$ after possibly replacing $y$ and $z$. 
We can write
\[
F (0,0,z,t,w) = w^3 + \alpha w t z^2 + \beta t^3 z + \gamma z^5,
\]
where $\alpha, \beta, \gamma \in \mbC$.
We will derive a contradiction by assuming $\alpha_{\msp} (X) < 1/2$.
By the assumption, there exists an irreducible $\mbQ$-divisor $D \in |A|_{\mbQ}$ such that $\lct_{\msp} (X;D) < 1/2$.
Let $S$ be a general member of the pencil $|\mcI_{\msp} (3 A)|$ so that $\Supp (S) \ne \Supp (D)$.

Suppose $\lambda \ne 0$ and $\mu \ne 0$.
In this case the set $\{x, y, \mu t^2 - \lambda^2 w z\}$ isolates $\msp$, and we can take a $\mbQ$-divisor $T \in |8 A|_{\mbQ}$ such that $\mult_{\msp} (T) \ge 1$ and $\Supp (T)$ does not contain any component of the effective $1$-cycle $D \cdot S$.
Then, we have
\[
\mult_{\msp} (D) \le (D \cdot S \cdot T)_{\msp} \le (D \cdot S \cdot T) = 3 \cdot 8 \cdot (A^3) = 2,
\]
which implies $\lct_{\msp} (X;D) \ge 1/2$.
This is impossible and we have either $\lambda = 0$ or $\mu = 0$.

Suppose $\lambda = 0$.
In this case we can take a $\mbQ$-divisor $T \in |4 A|_{\mbQ}$ such that $\mult_{\msp} (T) \ge 1$ and $\Supp (T)$ does not contain any component of the effective $1$-cycle $D \cdot S$ since $\{x, y, t\}$ isolates $\msp$.
Then, we have
\[
\mult_{\msp} (D) \le (D \cdot S \cdot T)_{\msp} \le (D \cdot S \cdot T) = 3 \cdot 4 \cdot (A^3) = 1,
\]
which implies $\lct_{\msp} (X;D) \ge 1$.
This is impossible.
We have $\lambda \ne 0$ and $\mu = 0$.
In this case we may assume $\lambda = 1$ by rescaling $t$, that is, we may assume $\msp = (0\!:\!0\!:\!1\!:\!1\!:\!0)$.

We claim $\mult_{\msp} (H_x) \le 2$.
We set $\zeta := \coeff_F (t^3 y)$ and $\eta := \coeff_F (z^4 y)$.
We have $\beta + \gamma = 0$ since $F (\msp) = 0$.
Then,
\[
\frac{\prt F}{\prt z} (\msp) = \beta + 5 \gamma = 4 \gamma, \quad
\frac{\prt F}{\prt y} (\msp) = \zeta + \eta.
\]
If either $\gamma \ne 0$ or $\zeta + \eta \ne 0$, then we have $\mult_{\msp} (H_x) = 1$.
It remains to consider the case where $\gamma = \zeta + \eta = 0$.
Note that we have $\beta = 0$ since $\beta + \gamma = 0$.
By the quasi-smoothness of $X$ at $\msp_t$, we have $\zeta \ne 0$.
Then, we see that $\mult_{\msp} (H_x) = 2$ since the term $\zeta y (t^3 - z^4)$ appears in $F$.
Thus the claim is proved.

By the claim, we have $\lct_{\msp} (X;H_x) \ge 1/2$ and in particular $D \ne H_x$.
We can take a $\mbQ$-divisor $T \in |12 A|_{\mbQ}$ such that $\mult_{\msp} (T) \ge 1$ and $\Supp (T)$ does not contain any component of the effective $1$-cycle $D \cdot H_x$ since $\{x, y, w, t^3 - z^4\}$ isolates $\msp$.
Then, we have
\[
\mult_{\msp} (D) \le (D \cdot H_x \cdot T)_{\msp} \le (D \cdot H_x \cdot T) = 12 (A^3) = 1,
\]
which implies $\alpha_{\msp} (X) \ge 1$.
This is a contradiction and the proof is completed.

\section{Singular points}
\label{chap:singpt}

The aim of this section is to prove the following.

\begin{Thm} \label{thm:singpt}
Let $X$ be a member of a family $\mcF_{\msi}$ with $\msi \in \msI \setminus \msI_1$.
Then,  
\[
\alpha_{\msp} (X) \ge \frac{1}{2}
\] 
for any singular point $\msp \in X$.
\end{Thm}

Let $X$ be a member of a family $\mcF_{\msi}$ with $\msi \in \msI \setminus \msI_1$.
Then the inequality $\alpha_{\msp} (X) \ge 1/2$ will follow from Propositions \ref{prop:singptCP}, \ref{prop:lctsingptL}, \ref{prop:singptrem} for singular points $\msp \in X$ which are not BI centers, from Proposition  \ref{prop:alphaEI} for EI centers, and from Propositions  \ref{prop:lctexcQI}, \ref{prop:lctdegQI} and \ref{prop:ndQIcent} for QI centers.
This will complete the proof of Theorem \ref{thm:singpt}.

\subsection{Non-BI centers}

Throughout the present section, let $X$ be a member of a family $\mcF_{\msi}$ with $\msi \in \msI$.

\subsubsection{Computation by $\bNE (Y)$}

\begin{Prop} \label{prop:singptCP}
Let $\msp \in X$ be a singular point with subscript $\heartsuit$ in the $5$th column of Table \ref{table:main}, and let $\varphi \colon Y \to X$ be the Kawamata blowup at $\msp$.
Then $(-K_Y)^2 \notin \Int \bNE (Y)$ and $\tilde{D} \sim -K_Y$ for the proper transform of a general member $D \in |A|$.
In particular, we have 
\[
\alpha_{\msp} (X) \ge 1.
\]
\end{Prop}

\begin{proof}
Let $r > 1$ be the index of the quotient singular point $\msp \in X$.
For every instance, we have either $a_1 = 1$ or $d - 1$ is not divisible by $r$.
This means that we can take $x$ as a part of local orbifold coordinates of $X$ at $\msp$, and hence $\tilde{D} \sim - K_Y$ for a general $D \in |A|$.
The point $\msp$ is excluded as a maximal center by either Lemma 3.2.2 or 3.2.4 of \cite{CP17}.

We set $S := \tilde{D} \sim - K_Y$, where $D \in |A|$ is a general member.
If $\msp$ is excluded by \cite[Lemma 3.2.2]{CP17}, then it follows from its proof that $(- K_Y)^2 = (-K_Y) \cdot S \notin \Int \bNE (Y)$.
If $\msp$ is excluded by \cite[Lemma 3.2.4]{CP17}, then there exists a nef divisor $T$ on $Y$ such that $(T \cdot S \cdot - K_Y) \le 0$, which implies $(-K_Y)^2 = (-K_Y) \cdot S \notin \Int \bNE (Y)$.
The latter assertion follows from Lemma \ref{lem:singptNE}.
\end{proof}

\subsubsection{Computation by $L_{xy}$}

\begin{Prop} \label{prop:lctsingptL}
Let $\msp \in X$ be a singular point with the subscript $\diamondsuit$ or $\diamondsuit'$ in the $5$th column of Table \ref{table:main}, and let $q = q_{\msp}$ be the quotient morphism of $\msp \in X$.
We denote by $r$ the index of the cyclic quotient singularity $\msp \in X$.
Let $S \in |A|$ and $T \in |a_1 A|$ be general members.
Then the following assertions hold.
\begin{enumerate}
\item The pair $(X, S)$ is log canonical at $\msp$.
\item The intersection $S \cap T$ is irreducible and we have $q^*S \cdot q^*T = \check{\Gamma}$, where $\check{\Gamma}$ is an irreducible and reduced curve such that
\[
0 < \mult_{\check{\msp}} (\check{\Gamma}) \le a_1.
\]
\item We have
\[
r a_1 (A^3) \le 
\begin{cases}
1, & \text{if the subscript of $\msp$ is $\diamondsuit$}, \\
\frac{3}{2}, & \text{if the subscript of $\msp$ is $\diamondsuit'$}.
\end{cases}
\]
\end{enumerate}
In particular, 
\[
\alpha_{\msp} (X) \ge 
\begin{cases}
1, & \text{if the subscript of $\msp$ is $\diamondsuit$}, \\
\frac{3}{2}, & \text{if the subscript of $\msp$ is $\diamondsuit'$}.
\end{cases}
\]
\end{Prop}

\begin{proof}
Let $\msp \in X$ be as in the statement.
The assertion (3) can be checked individually and it remains to consider (1) and (2).

It is straightforward to check that $X$ is a member of a family $\mcF_{\msi}$ which is listed in one of the Tables \ref{table:Lsmooth} and \ref{table:Lsing}.
It follows that $S \cap T = L_{xy}$ is irreducible.
It is easy to check that we may assume $\msp = \msp_{v}$ for some $v \in \{z, t, w\}$ after replacing coordinates.
Let $\rho = \rho_v \colon \breve{U}_v \to U_v$ be the orbifold chart.
We set $\breve{\Gamma} = (\breve{x} = \breve{y} = 0) \subset U_v$.
We see that $\breve{\Gamma}$ is an irreducible and reduced curve since so is $L_{xy}$, and that $\rho^*S \cdot \rho^*T = \breve{\Gamma}$.
Note that $q$ can be identified with $\rho$ over a suitable analytic neighborhood of $\msp \in U_v$, and hence it is enough to prove the inequality $\mult_{\breve{\msp}} (\breve{\Gamma}) \le a_1$ for the proof of (2).

If $X$ is listed in Table \ref{table:Lsmooth}, then $\breve{\Gamma}$ is irreducible and smooth by Lemma \ref{lem:Lirred}.
In this case $S$ is quasi-smooth at $\msp$ and thus both (1) and (2) are clearly satisfied.

Suppose that $X$ is listed in Table \ref{table:Lsing}.
Then $X$ is a member of a family $\mcF_{\msi}$, where
\[
\msi \in \{44, 47, 61, 62, 65, 69, 77, 79, 83, 85\}.
\]
If $\msp$ is not the unique singular point of $L_{xy}$ which is described in Table \ref{table:Lsing}, then (1) and (2) follow immediately.
Suppose that $\msp$ is the unique singular point of $L_{xy}$.
Then we have $\msi \in \{44, 61, 83\}$ and $\msp = \msp_t$.
By the equation given in Table \ref{table:Lsing}, we compute $\mult_{\breve{\msp}} (\breve{\Gamma}) = 2$.
This shows (2) since $a_1 \ge 2$.
We see that $r = a_3$ does not divide $d-1$, which implies that $S$ is quasi-smooth at $\msp$ and hence (1) follows.
Therefore (1), (2) and (3) are verified and the assertion on $\alpha_{\msp} (X)$ follows from Lemma \ref{lem:exclL}.
\end{proof}

\subsubsection{Computation by isolating class}

\begin{Prop} \label{prop:singptic}
Let $\msp \in X$ be a singular point with subscript $\clubsuit$ in the 5th column of Table \ref{table:main} which is also listed in Table \ref{table:isolsingpt}.
Then the set of coordinates given in the 5th column of Table \ref{table:isolsingpt} isolates $\msp$.
In particular, we have
\[
\alpha_{\msp} (X) \ge \min \{1, c\} \ge \frac{1}{2},
\]
where $c$ is the number given in the 7th column of Table \ref{table:isolsingpt}.
\end{Prop}

\begin{proof}
Let $\mcC$ be the set of homogeneous coordinates given in the 5th column of Table \ref{table:isolsingpt}.
It is straightforward to check that 
\[
\bigcap_{v \in \mcC} (v = 0) \cap X
\]
is a finite set of points including $\msp$, which shows that $\mcC$ isolates $\msp$.

Let $c$ be the number listed in the 7th column of Table \ref{table:isolsingpt} and assume that $\alpha_{\msp} (X) < \min \{1, c\}$.
Then there exists an irreducible $\mbQ$-divisor $D \sim_{\mbQ} A$ such that $(X, c D)$ is not log canonical at $\msp$.
In particular we have $\omult_{\msp} (D) > 1/c$.
If $H_x$ (resp.\ $|n A|$ for some $n > 0$) is given in the 4th column of Table \ref{table:isolsingpt}, then we set $S := H_x$ (resp. we define $S$ to be a general member of $|n A|$).
We set $n = 1$ if $S = H_x$ so that $S \sim n A$ in any case.
Let $r$ be the index of the cyclic quotient singularity $\msp \in X$.
We claim that $\Supp (D)$ is not contained in $S$.
This is clear when $S \in |nA|$ is a general member.
Suppose that $S = H_x$.
Then we see that $d-1$ is not divisible by $r$, which implies that $S = H_x$ is quasi-smooth at $\msp$.
Hence $(X, S)$ is log canonical at $\msp$ and we have $D \ne H_x$ as desired.
By the claim, $D \cdot S$ is an effective $1$-cycle on $X$.
Let $e$ be the integer given in the 6th column of Table \ref{table:isolsingpt}.
Note that $e = \max \{\, \deg v \mid v \in \mcC \,\}$ and
\[
r n e_{\max} (A^3) = \frac{1}{c}.
\]
There exists an irreducible $\mbQ$-divisor $T \sim_{\mbQ} e A$ such that $\mult_{\msp} (T) \ge 1$ and $\Supp (T)$ does not contain any component of $D \cdot S$ since $\mcC$ isolates $\msp$.
It follows that
\[
\frac{1}{c} < \omult_{\msp} (D) 
\le (q_{\msp}^*D \cdot q_{\msp}^*S \cdot q_{\msp}^*T)_{\check{\msp}} 
\le r (D \cdot S \cdot T) 
= r n e (A^3) 
= \frac{1}{c},
\]
where $q = q_{\msp}$ is the quotient morphism of $\msp \in X$ and $\check{\msp}$ is the preimage of $\msp$ via $q$.
This is a contradiction and the inequality $\alpha_{\msp} (X) \ge \min \{1, c\}$ is proved.
\end{proof}

\begin{table}[h]
\renewcommand{\arraystretch}{1.15}
\begin{center}
\caption{Isolating set}
\label{table:isolsingpt}
\begin{tabular}{ccccccc}
No. & Pt. & Type & $S$ & Isol.\ set & $e_{\max}$ & $c$\\
\hline
10 & $\msp_t$ & $\frac{1}{3} (1,1,2)$ & $|A|$ & $\{x,y,z\}$ & $1$ & $1/2$  \\
23 & $\msp_{yt}$ & $\frac{1}{2} (1,1,1)$ & $H_x$ & $\{x,z,w\}$ & $5$ & $6/7$  \\
23 & $\msp_z$ & $\frac{1}{3} (1,1,2)$ & $H_x$ & $\{x,y,t,w\}$ & $5$ & $4/7$  \\
23 & $\msp_t$ & $\frac{1}{4} (1,1,3)$ & $H_x$ & $\{x,y,z\}$ & $3$ & $5/7$  \\
29 & $\msp_{zw}$ & $\frac{1}{2} (1,1,1)$ & $|A|$ & $\{x,y,t\}$ & $5$ & $1/2$  \\
29 & $\msp_t$ & $\frac{1}{5} (1,2,3)$ & $|A|$ & $\{x,y,z\}$ & $2$ & $1/2$  \\
31 & $\msp_{zw}$ & $\frac{1}{2} (1,1,1)$ & $|A|$ & $\{x,y,t\}$ & $5$ & $3/4$ \\
33 & $\msp_z$ & $\frac{1}{3} (1,1,2)$ & $H_x$ & $\{x,y,t,w\}$ & $7$ & $10/17$ \\
37 & $\msp_{zw}$ & $\frac{1}{3} (1,1,2)$ & $H_x$ & $\{x,y,t\}$ & $4$ & $1$ \\
39 & $\msp_{yw}$ & $\frac{1}{3} (1,1,2)$ & $H_x$ & $\{x,z,t\}$ & $5$ & $1$ \\
39 & $\msp_z$ & $\frac{1}{4} (1,1,3)$ & $H_x$ & $\{x,y,t\}$ & $5$ & $1$ \\
40 & $\msp_z$ & $\frac{1}{4} (1,1,3)$ & $H_x$ & $\{x,y,t,w\}$ & $7$ & $15/19$ \\
40 & $\msp_t$ & $\frac{1}{5} (1,2,3)$ & $H_x$ & $\{x,y,z,w\}$ & $4$ & $1$ \\
50 & $\msp_t$ & $\frac{1}{7} (1,3,4)$ & $|A|$ & $\{x,y,z\}$ & $3$ & $1/2$ \\
61 & $\msp_y$ & $\frac{1}{4} (1, 1, 3)$ & $H_x$ & $\{x, z, t, w\}$ & $9$ & $7/5$ \\
63 & $\msp_t$ & $\frac{1}{8} (1,3,5)$ & $H_x$ & $\{x, y, z\}$ & $3$ & $1$ \\
64 & $\msp_z$ & $\frac{1}{5} (1,2,3)$ & $|2 A|$ & $\{x, y, t\}$ & $6$ & $1/2$ \\
66 & $\msp_y$ & $\frac{1}{5} (1,1,4)$ & $H_x$ & $\{x,z,t\}$ & $7$ & $1$ \\
68 & $\msp_y$ & $\frac{1}{3} (1,1,2)$ & $|4 A|$ & $\{x,z,t\}$ & $7$ & $1/2$ \\
80 & $\msp_y$ & $\frac{1}{3} (1,1,2)$ & $|4A|$ & $\{x,z,t\}$ & $10$ & $1/2$ \\
93 & $\msp_y$ & $\frac{1}{7} (1,3,4)$ & $|8A|$ & $\{x,z,t\}$ & $10$ & $1/2$ \\
95 & $\msp_y$ & $\frac{1}{5} (1,2,3)$ & $|6A|$ & $\{x,z,t\}$ & $22$ & $1/2$
\end{tabular}
\end{center}
\end{table}

\subsubsection{Remaining non-BI centers}

\begin{Prop} \label{prop:singptrem}
Let $X$ be a member of a family $\mcF_{\msi}$ with $\msi \in \msI$ and let $\msp \in X$ be a singular point with subscript $\spadesuit$ in the 5th column of Table \ref{table:main} which is also listed below.
\begin{itemize}
\item $\msi = 12$ and singular points of type $\frac{1}{2} (1,1,1)$.
\item $\msi = 13$ and the singular point of type $\frac{1}{2} (1,1,1)$.
\item $\msi = 24$ and the singular point of type $\frac{1}{2} (1,1,1)$.
\item $\msi = 27$ and the singular point of type $\frac{1}{2} (1,1,1)$.
\item $\msi = 32$ and the singular point of type $\frac{1}{3} (1,1,2)$.
\item $\msi = 33$ and the singular point of type $\frac{1}{2} (1,1,1)$.
\item $\msi = 40$ and the singular point of type $\frac{1}{3} (1,1,2)$.
\item $\msi = 47$ and the singular point of type $\frac{1}{5} (1, 2, 3)$.
\item $\msi = 48$ and the singular point of type $\frac{1}{2} (1,1,1)$.
\item $\msi = 49$ and the singular point of type $\frac{1}{5} (1, 2, 3)$.
\item $\msi = 62$ and the singular point of type $\frac{1}{5} (1, 2, 3)$.
\item $\msi = 65$ and the singular point of type $\frac{1}{2} (1, 1, 1)$.
\item $\msi = 67$ and the singular point of type $\frac{1}{9} (1, 4, 5)$.
\item $\msi = 82$ and the singular point of type $\frac{1}{5} (1, 2, 3)$.
\item $\msi = 84$ and the singular point of type $\frac{1}{7} (1, 2, 5)$.
\end{itemize}
Then, we have 
\[
\alpha_{\msp} (X) \ge \frac{1}{2}.
\]
\end{Prop}

The rest of this subsection is to prove Proposition \ref{prop:singptrem} which will be separately for each family.

\paragraph{\it The family $\mcF_{12}$, points of type $\frac{1}{2} (1,1,1)$}
Let $X = X_{10} \subset \mbP (1, 1, 2, 3, 4)$ be a member of $\mcF_{12}$ and $\msp$ a singular point of type $\frac{1}{2} (1,1,1)$.
We may assume $\msp = \msp_z$ after replacing $w$.
Then we have $z^3 w \in F$ by the quasi-smoothness of $X$ at $\msp$.
By Lemma \ref{lem:complctsingtang}, we have
\[
\alpha_{\msp} (X) \ge \frac{2}{2 \cdot 1 \cdot 3 \cdot (A^3)} = \frac{4}{5},
\]
and the proof is completed in this case.

\paragraph{\it The family $\mcF_{13}$, the singular point of type $\frac{1}{2} (1,1,1)$}
Let $X = X_{11} \subset \mbP (1, 1, 2, 3, 5)$ be a member of $\mcF_{13}$ and $\msp = \msp_z$ the singular point of type $\frac{1}{2} (1,1,1)$.
By Lemma \ref{lem:complctsingtang}, we have
\[
\alpha_{\msp} (X) \ge
\begin{cases}
\frac{2}{2 \cdot 1 \cdot 3 \cdot (A^3)} = \frac{10}{11}, & \text{if $z^3 w \in F$}, \\
\frac{2}{2 \cdot 1 \cdot 5 \cdot (A^3)} = \frac{6}{11}, & \text{if $z^3 w \notin F$ and $z^4 t \in F$}.
\end{cases}
\]

It remains to consider the case where $z^3 w, z^4 t \notin F$.
Then, by choosing $x, y$ suitably, we can write
\[
F = z^5 x + z^4 f_3 + z^3 f_5 + z^2 f_7 + z f_9 + f_{11},
\]
where $f_i = f_i (x,y,t,w)$ is a quasi-homogeneous polynomial of degree $i$ with $t \notin f_3$ and $w \notin f_5$.

We claim $w^2 y \in F$.
Assume $w^2 y \notin F$.
Then, by the quasi-smoothness of $X$ at $\msp_w$, we have $w^2 x \in F$ and we may assume $\coeff_F (w^2 x) = -1$.
We can write $F = (z^5 - w^2) x + f'$, where $f' = z^4 f_3 + z^3 f_5 + z^2 f_7 + z f_9 + f_{11} + w^2 x$.
It is straightforward to check that $f' \in (x, y, t)^2$ and thus $X$ is not quasi-smooth at the point $(0\!:\!0\!:\!1\!:\!0\!:\!1) \in X$, which is a contradiction.
Thus $w^2 y \in F$.

We see that $\bar{F} := F (0, y, 1, t, w) \in (y, t, w)^3$ and the cubic part of $\bar{F}$ is not a cube of a linear form since $w^2 y \in \bar{F}$ and $w^3 \notin \bar{F}$.
Thus we have $\alpha_{\msp} (X) \ge 1/2$ by Lemma \ref{lem:lcttangcube}.

\paragraph{\it The family $\mcF_{24}$, the point of type $\frac{1}{2} (1,1,1)$}
Let $X = X_{15} \subset \mbP (1, 1, 2, 5, 7)$ be a member of $\mcF_{24}$ and $\msp = \msp_z$ the singular point of type $\frac{1}{2} (1,1,1)$.
By Lemma \ref{lem:complctsingtang}, we have
\[
\alpha_{\msp} (X) \ge 
\begin{cases} 
\frac{2}{2 \cdot 1 \cdot 5 \cdot (A^3)} = \frac{14}{15}, & \text{if $z^4 w \in F$}, \\
\frac{2}{2 \cdot 1 \cdot 7 \cdot (A^3)} = \frac{2}{3}, & \text{if $z^4 w \notin F$ and $z^5 t \in F$}.
\end{cases}
\]
Suppose $z^4 w, z^5 t \notin F$.
Then we can write
\[
F = z^7 x + z^6 f_3 + z^5 f_5 + z^4 f_7 + z^3 f_9 + z^2 f_{11} + z f_{13} + f_{15},
\]
where $f_i = f_i (x, y, t, w)$ is a quasi-homogeneous polynomial of degree $i$ with $t \notin f_5$ and $w \notin f_7$.

We claim $w^2 y \in F$.
Assume to the contrary $w^2 y \notin F$.
Then we can write $F = (z^7 + g) x + h$, where $g \in \mbC [x,y,z,t,w]$ and $h \in \mbC [y,z,t,w]$ are quasi-homogeneous polynomials such that $h \in (y,t)^2$.
But then $X$ is not quasi-smooth along the non-empty subset
\[
(x = y = t = z^7 + g = 0) \subset X.
\]
This is a contradiction and the claim is proved.

We see that $\bar{F} := F (0, y, 1, t, w) \in (y,t,w)^3$, $w^2 y \in \bar{F}$ and $w^3 \notin \bar{F}$.
In particular, $\bar{F}$ cannot be a cube of a linear form and thus $\alpha_{\msp} (X) \ge 1/2$ by Lemma \ref{lem:lcttangcube}.

\paragraph{\it The family $\mcF_{27}$, the point of type $\frac{1}{2} (1,1,1)$}
Let $X = X_{15} \subset \mbP (1, 2, 3, 5, 5)$ be a member of $\mcF_{27}$ and $\msp = \msp_y$ the singular point of type $\frac{1}{2} (1,1,1)$.
If either $y^5 w \in F$ or $y^5 t \in F$, then we have
\[
\alpha_{\msp} (X) \ge \frac{2}{2 \cdot 3 \cdot 5 \cdot (A^3)} = \frac{2}{3}
\]
by Lemma \ref{lem:complctsingtang}.

Suppose $y^5 w, y^5 t \notin F$ and $y^6 z \in F$.
Then we can write
\[
F = y^6 z + y^5 f_5 + y^4 f_7 + y^3 f_9 + y^2 f_{11} + y f_{13} + f_{15},
\]
where $f_i = f_i (x,z,t,w)$ is a quasi-homogeneous polynomial of degree $i$ with $w, t \notin f_5$.
We see that $\omult_{\msp} (H_z) = 3$ and thus $\lct_{\msp} (X; \frac{1}{3} H_z) \ge 1$.
Let $D \in |A|_{\mbQ}$ be an irreducible $\mbQ$-divisor other than $\frac{1}{3} H_z$.
Then, we can take a $\mbQ$-divisor $T \in |5 A|_{\mbQ}$ such that $\mult_{\msp} (T) \ge 1$ and $\Supp (T)$ does not contain any component of $D \cdot H_x$ since $\{x, z, t, w\}$ isolates $\msp$.
Since $\omult_{\msp} (H_z) = 3$, we have
\[
3 \omult_{\msp} (D) \le 3 (q^*D \cdot q^*H_z \cdot q^*T)_{\check{\msp}} \le 2 (D \cdot H_z \cdot T) = 3,
\] 
where $q = q_{\msp}$ is the quotient morphism of $\msp \in X$ and $\check{\msp}$ is the preimage of $\msp$ via $q$.
This shows $\lct_{\msp} (X; D) \ge 1$ and thus $\alpha_{\msp} (X) \ge 1$.

Suppose $y^5 w, y^5 t, y^6 z \notin F$.
Then we can write
\[
F = y^7 x + y^6 f_3 + y^5 f_5 + y^4 f_7 + y^3 f_9 + y^2 f_{11} + y f_{13} + f_{15},
\]
where $f_i = f_i (x,z,t,w)$ is a quasi-homogeneous polynomial of degree $i$ with $z \notin f_3$ and $w, t \notin f_5$.
We see that $\bar{F} = F (0,1,z,t,w) \in (z,t,w)^3$ and the cubic part of $\bar{F}$ cannot be cube of a linear form since $F (0,1,0,t,w) = F (0,0,0,t,w)$ is a product of three linearly independent linear forms in $t, w$ by the quasi-smoothness of $X$.
By Lemma \ref{lem:lcttangcube}, we have $\lct_{\msp} (X; H_x) \ge 1/2$.

\paragraph{\it The family $\mcF_{32}$, the point of type $\frac{1}{3} (1,1,2)$}
Let $X = X_{16} \subset \mbP (1, 2, 3, 4, 7)$ be a member of $\mcF_{32}$ and $\msp = \msp_z$ the singular point of type $\frac{1}{3} (1,1,2)$.
By Lemma \ref{lem:complctsingtang}, we have
\[
\alpha_{\msp} (X) \ge
\begin{cases}
\frac{2}{3 \cdot 2 \cdot 4 \cdot (A^3)} = \frac{7}{8}, & \text{if $z^3 w \in F$}, \\
\frac{2}{3 \cdot 2 \cdot 7 \cdot (A^3)} = \frac{1}{2}, & \text{if $z^3 w \notin F$ and $z^4 t \in F$}.
\end{cases}
\]

Suppose $z^3 w, z^4 t \notin F$.
Then $z^5 x \in F$ by the quasi-smoothness of $X$ at $\msp$ and we can write
\[
F = z^5 x + z^4 f_4 + z^3 f_7 + z^2 f_{10} + z f_{13} + f_{16},
\]
where $f_i = f_i (x,y,t,w)$ is a quasi-homogeneous polynomial of degree $i$.
We have $w^2 y \in F$ by the quasi-smoothness of $X$ at $\msp_w$.
It follows that either $\bar{F} = F (0,y,1,t,w) \in (y,t,w)^2 \setminus (y,t,w)^3$ or $\bar{F} \in (y,t,w)^3$ and the cubic part of $\bar{F}$ is not a cube of a linear form since $w^3 \notin \bar{F}$.
By Lemma \ref{lem:lcttangcube}, we have $\alpha_{\msp} (X) \ge 1/2$.

\paragraph{\it The family $\mcF_{33}$, the point of type $\frac{1}{2} (1, 1, 1)$}
Let $X = X_{17} \subset \mbP (1, 2, 3, 5, 7)$ be a member of $\mcF_{33}$ and $\msp = \msp_y$ the singular point of type $\frac{1}{2} (1, 1, 1)$.

Suppose that at least one of $y^5 w$, $y^6 t$ and $y^7 z$ appear in $F$ with nonzero coefficient.
In this case, $H_x$ is quasi-smooth at $\msp$ and we have $\lct_{\msp} (X; H_x) = 1$.
Let $D \in |A|_{\mbQ}$ be an irreducible $\mbQ$-divisor other than $H_x$.
We can take a $\mbQ$-divisor $T \in |7 A|_{\mbQ}$ such that $\omult_{\msp} (T) \ge 1$ and $\Supp (T)$ does not contain any component of of the effective $1$-cycle $D \cdot H_x$ since the set $\{x, z, t, w\}$ isolates $\msp$.
It follows that
\[
\omult_{\msp} (D) \le (q^*_{\msp} D \cdot q^*_{\msp} H_x \cdot q^*_{\msp} T)_{\check{\msp}} \le 2 (D \cdot H_x \cdot T) = \frac{17}{15}.
\]
Thus $\lct_{\msp} (X; D) \ge 15/17$ and we have $\alpha_{\msp} (X) \ge 15/17$.

Suppose $y^5 w, y^6 t, y^7 z \notin F$.
Then $y^8 x \in F$ by the quasi-smoothness of $X$ at $\msp$ and we can write
\[
F = y^8 x + y^7 f_3 + y^6 f_5 + y^5 f_7 + y^4 f_9 + y^3 f_{11} + y^2 f_{13} + y f_{15} + f_{17},
\]
where $f_i = f_i (x, z, t, w)$ is a quasi-homogeneous polynomial of degree $i$.
We see that $\bar{F} := F (0, 1, z, t, w) \in (z, t, w)^3$, $w^3 \notin \bar{F}$ and $w^2 z \in \bar{F}$.
It follows that $\bar{F} \in (z, t, w)^3$ and it cannot be a cube of a linear form.
By Lemma \ref{lem:lcttangcube}, we have $\alpha_{\msp} (X) \ge 1/2$.

\paragraph{\it The family $\mcF_{40}$, the point of type $\frac{1}{3} (1, 1, 2)$}
Let $X = X_{19} \subset \mbP (1, 3, 4, 5, 7)$ be a member of $\mcF_{40}$ and $\msp = \msp_y$ the singular point of type $\frac{1}{3} (1, 1, 2)$.

Suppose that either $y^4 w \in F$ or $y^5 z \in F$.
In this case $\lct_{\msp} (X; H_x) = 1$ since $H_x$ is quasi-smooth at $\msp$. 
Let $D \in |A|_{\mbQ}$ be an irreducible $\mbQ$-divisor other than $H_x$.
We can take a $\mbQ$-divisor $T \in |7 A|_{\mbQ}$ such that $\omult_{\msp} (T) \ge 1$ and $\Supp (T)$ does not contain any component of of the effective $1$-cycle $D \cdot H_x$ since the set $\{x, z, t, w\}$ isolates $\msp$.
It follows that
\[
\omult_{\msp} (D) \le (q^*_{\msp} D \cdot q^*_{\msp} H_x \cdot q^*_{\msp} T)_{\check{\msp}} \le 3 (D \cdot H_x \cdot T) = \frac{19}{20}.
\]
Thus $\lct_{\msp} (X; D) \ge 20/19$ and we have $\alpha_{\msp} (X) \ge 1$.

Suppose $y^4 w, y^5 z \notin F$.
Then $y^6 x \in F$ by the quasi-smoothness of $X$ at $\msp$ and we can write
\[
F = y^6 x + y^5 f_4 + y^4 f_7 + y^3 f_{10} + y^2 f_{13} + y f_{16} + f_{19},
\]
where $f_i = f_i (x, z, t, w)$ is a quasi-homogeneous polynomial of degree $i$.
We set $\bar{F} := F (0, 1, z, t, w) \in (z, t, w)^3$.
It is easy to see that $w^3 \notin \bar{F}$ and $w^2 z \in \bar{F}$.
It follows that either $\bar{F} \in (z, t, w)^2$ or $\bar{F} \in (z, t, w)^3$ and it cannot be a cube of a linear form.
By Lemma \ref{lem:lcttangcube}, we have $\alpha_{\msp} (X) \ge 1/2$.

\paragraph{\it The family $\mcF_{47}$, the point of type $\frac{1}{5} (1, 2, 3)$}
Let $X = X_{21} \subset \mbP (1, 1, 5, 7, 8)$ be a member of $\mcF_{47}$ and $\msp = \msp_z$ the singular point of type $\frac{1}{5} (1,2,3)$.
We can write
\[
F = z^4 x + z^3 f_6 + z^2 f_{11} + z f_{16} + f_{21},
\]
where $f_i = f_i (x,y,t,w)$ is a quasi-homogeneous polynomial of degree $i$.
By the quasi-smoothness of $X$ at $\msp_w$, we have $w^2 \in f_{16}$, which implies $\bar{F} = F (0, y, 1, t, w) \in (y, t, w)^2 \setminus (y, t, w)^3$.
Thus, by Lemma \ref{lem:lcttangcube}, we have $\alpha_{\msp} (X) \ge  1/2$.

\paragraph{\it The family $\mcF_{48}$, the point of type $\frac{1}{2} (1,1,1)$}
Let $X = X_{21} \subset \mbP (1, 2, 3, 7, 9)$ be a member of $\mcF_{48}$ and $\msp = \msp_y$ the singular point of type $\frac{1}{2} (1,1,1)$.
By Lemma \ref{lem:complctsingtang}, we have
\[
\alpha_{\msp} (X) \ge
\begin{cases}
\frac{2}{2 \cdot 3 \cdot 7 \cdot (A^3)} = \frac{6}{7}, & \text{if $y^6 w \in F$}, \\
\frac{2}{2 \cdot 3 \cdot 9 \cdot (A^3)} = \frac{2}{3}, & \text{if $y^6 w \notin F$ and $y^7 t \in F$}.
\end{cases}
\]

Suppose $y^6 w, y^7 t \notin F$ and $y^9 z \in F$.
We can write
\[
F = y^9 z + y^8 f_5 + y^7 f_7 + y^6 f_9 + \cdots + f_{21},
\]
where $f_i = f_i (x,z,t,w)$ is a quasi-homogeneous polynomial of degree $i$ with $t \notin f_7$ and $w \notin f_9$.
We have $\omult_{\msp} (H_x) = 1$ and thus $\lct_{\msp} (X;H_x) = 1$.
Let $D \in |A|_{\mbQ}$ be an irreducible $\mbQ$-divisor on $X$ other than $H_x$.
We can take a $\mbQ$-divisor $T \in |9 A|_{\mbQ}$ with such that $\mult_{\msp} (T) \ge 1$ and $\Supp (T)$ does not contain any component of $D \cdot H_x$ since $\{x, z, t, w\}$ isolates $\msp$.
Then we have
\[
\omult_{\msp} (D) \le (q^*D \cdot q^*H_x \cdot q^*T)_{\check{\msp}} \le 2 (D \cdot H_x \cdot T) = 2 \cdot 1 \cdot 9 \cdot (A^3) = 1,
\]
where $q = q_{\msp}$ is the quotient morphism of $\msp \in X$ and $\check{\msp}$ is the preimage of $\msp$ via $q$.
This shows $\lct_{\msp} (X;D) \ge 1$ and thus $\alpha_{\msp} (X) \ge 1$.

Suppose $y^6 w, y^7 t, y^9 z \notin F$.
Then $y^{10} x \in F$ and we can write
\[
y^{10} x + y^9 f_3 + y^8 f_5 + y^7 f_7 + y^6 f_9 + \cdots + f_{21},
\]
where $f_i = f_i (x,z,t,w)$ is a quasi-homogeneous polynomial of degree $i$ with $z \notin f_3$, $t \notin f_7$ and $w \notin f_9$.
We see that $\bar{F} := F (0,1,z,t,w) \in (z,t,w)^3$ and the cubic part of $\bar{F}$ is not a cube of a linear form since $w^2 z \in \bar{F}$ and $w^3 \notin \bar{F}$. 
Thus, by Lemma \ref{lem:lcttangcube}, we have $\alpha_{\msp} (X) \ge 1/2$.

\paragraph{\it The family $\mcF_{49}$, the point of type $\frac{1}{5} (1,2,3)$}
Let $X = X_{21} \subset \mbP (1, 3, 5, 6, 7)$ be a member of $\mcF_{49}$ and $\msp = \msp_z$ the singular point of type $\frac{1}{5} (1,2,3)$.
If $z^3 t \in F$, then
\[
\alpha_{\msp} (X) \ge \frac{2}{5 \cdot 3 \cdot 7 \cdot (A^3)} = \frac{4}{7}
\]
by Lemma \ref{lem:complctsingtang}.

Suppose $z^3 t \notin F$.
Then $z^4 x \in F$ and we can write
\[
F = z^4 x + z^3 f_6 + z^2 f_{11} + z f_{16} + f_{21},
\]
where $f_i = f_i (x,y,t,w)$ is a quasi-homogeneous polynomial of degree $i$ with $t \notin f_6$.
We have $w^3, t^3 y \in F$ and we may assume $\coeff_F (w^3) = \coeff_F (t^3 y) = 1$.

We claim $\lct_{\msp} (X;H_x) \ge 1/2$.
If $y^2 \in f_6$, then $\omult_{\msp} (H_x) = 2$ and thus $\lct_{\msp} (X;H_x) \ge 1/2$.
We assume $y^2 \notin f_6$.
Then we can write
\[
F (0,y,z,t,w) = z (\alpha w t y + \beta w y^3) + w^3 + y (t^3 + \gamma t^2 y^2 + \delta t y^4 + \varepsilon y^6),
\]
where $\alpha, \beta, \gamma, \delta, \varepsilon \in \mbC$.
We set $\bar{F} := F (0, y, 1, t, w) \in \mbC [y, t, w]$.
\begin{itemize}
\item Suppose $\alpha \ne 0$.
Then, $\bar{F} \in (y,t,w)^3$ and its cubic part is $\alpha w t y + w^3$.
By Lemma \ref{lem:lcttangcube}, we have $\lct_{\msp} (X;H_x) \ge 1/2$ in this case.
\item Suppose $\alpha = 0$ and $\beta \ne 0$.
Then the lowest weight part of $\bar{F}$ with respect to $\wt (y,t,w) = (6,7,9)$ is $\beta w y^3 + w^3 + t^3 y$.
By Lemma \ref{lem:lctwblwh}, we have
\[
\lct_{\msp} (X; H_x) \ge \min \left\{ \frac{22}{27}, \lct (\tilde{\mbP}, \Diff; \mcD ) \right\},
\]
where 
\begin{itemize}
\item $\tilde{\mbP} = \mbP (2, 7, 9)_{\tilde{y}, \tilde{t}, \tilde{w}} = \mbP (6, 7, 9)^{\wf}$,
\item $\Diff = \frac{2}{3} H_{\tilde{t}}$ with $H_{\tilde{t}} = (\tilde{t} = 0) \subset \tilde{\mbP}$, and
\item $\mcD$ is the prime divisor  $(\beta \tilde{w} \tilde{y}^3 + \tilde{w}^3 + \tilde{t} \tilde{y} = 0)$ on $\tilde{\mbP}$.
\end{itemize}
We see that $\mcD$ is quasi-smooth and it intersects $H_{\tilde{t}}$ transversally.
It follows that $\lct (\tilde{\mbP}, \Diff; \mcD) = 1$ and thus $\lct_{\msp} (X;H_x) \ge 22/27$.
\item Suppose $\alpha = \beta = 0$.
Then the lowest weight part of $\bar{F}$ with respect to $\wt (y,z,t) = (3,6,7)$ is $w^3 + t^3 y + \gamma t^2 y^3 + \delta t y^5 + \varepsilon y^7$.
By Lemma \ref{lem:lctwblwh}, we have 
\[
\lct_{\msp} (X; H_x) \ge \min \left\{ \frac{16}{21}, \lct (\tilde{\mbP}, \Diff; \mcD ) \right\},
\]
where 
\begin{itemize}
\item $\tilde{\mbP} = \mbP (1, 2, 9)_{\tilde{y}, \tilde{t}, \tilde{w}} = \mbP (3, 6, 7)^{\wf}$,
\item $\Diff = \frac{2}{3} H_{\tilde{w}}$ with $H_{\tilde{w}} = (\tilde{w} = 0) \subset \tilde{\mbP}$, and
\item $\mcD$ is the prime divisor  $(\tilde{w} + \tilde{t}^3 \tilde{y} + \gamma \tilde{t}^2 \tilde{y}^3 + \delta \tilde{t} \tilde{y}^5 + \varepsilon \tilde{y}^7 = 0)$ on $\tilde{\mbP}$.
\end{itemize}
We see that $\mcD$ is quasi-smooth.
The solutions of the equation $\tilde{t}^3 \tilde{y} + \gamma \tilde{t}^2 \tilde{y}^3 + \delta \tilde{t} \tilde{y}^5 + \varepsilon \tilde{y}^7 = 0$ corresponds to the 3 points of type $\frac{1}{3} (1, 1, 2)$ on $X$.
In particular the equation has 3 distinct solutions.
It follows that $\mcD$ intersects $H_{\tilde{w}}$ transversally and we have $\lct (\tilde{\mbP}, \Diff; \mcD) = 1$.
Thus $\lct_{\msp} (X; H_x) \ge 16/21$.
\end{itemize}
Thus the claim is proved.

The point $\msp$ is not a maximal center and the pair $(X, H_x)$ is not canonical by Lemma \ref{lem:qtangdivncan}.
Thus, 
\[
\alpha_{\msp} (X) \ge \min \{1, \lct_{\msp} (X;H_x) \} \ge \frac{1}{2}
\] 
by Lemma \ref{lem:singnoncanbd}.

\paragraph{\it The family $\mcF_{62}$, the point of $\frac{1}{5} (1,2,3)$}
Let $X = X_{26} \subset \mbP (1, 1, 5, 7, 13)$ be a member of $\mcF_{62}$ and $\msp = \msp_z$ the singular point of type $\frac{1}{5} (1,2,3)$.
Replacing $x$ and $y$, we can write
\[
F = z^5 x + z^4 f_6 + z^3 f_{11} + z^2 f_{16} + z f_{21} + f_{26},
\]
where $f_i = f_i (x,y,t,w)$ is a quasi-homogeneous polynomial of degree $i$. 
We have $\omult_{\msp} (H_x) = 2$ since $w^2 \in F$.
Hence $\lct_{\msp} (X;H_x) \ge 1/2$.
The point $\msp$ is not a maximal center and the pair $(X, H_x)$ is not canonical at $\msp$ by Lemma \ref{lem:qtangdivncan}.
Thus 
\[
\alpha_{\msp} (X) \ge \min \{1, \lct_{\msp} (X;H_x)\} \ge \frac{1}{2}
\] 
by Lemma \ref{lem:singnoncanbd}.

\paragraph{\it The family $\mcF_{65}$, the point of type $\frac{1}{2} (1,1,1)$}
Let $X = X_{27} \subset \mbP (1, 2, 5, 9, 11)$ be a member of $\mcF_{65}$ and $\msp = \msp_y$ the singular point of type $\frac{1}{2} (1,1,1)$.
By Lemma \ref{lem:complctsingtang}, we have
\[
\alpha_{\msp} (X) \ge 
\begin{cases}
\frac{2}{2 \cdot 5 \cdot 9 \cdot (A^3)} = \frac{22}{27}, & \text{if $y^8 w \in F$}, \\
\frac{2}{2 \cdot 5 \cdot 11 \cdot (A^3)} = \frac{2}{3}, & \text{if $y^8 w \notin F$ and $y^9 t \in F$}.
\end{cases}
\]

Suppose $y^8 w, y^9 t \notin F$ and $y^{11} z \in F$.
Then we can write
\[
F = y^{11} z + y^{10} f_7 + \cdots + y f_{25} + f_{27},
\]
where $f_i = f_i (x,z,t,w)$ is a quasi-homogeneous polynomial of degree $i$ with $t \notin f_9$ and $w \notin f_{11}$.
We have $\omult_{\msp} (H_x) = 1$ and thus $\lct_{\msp} (X;H_x) = 1$.
Let $D \in |A|_{\mbQ}$ be an irreducible $\mbQ$-divisor on $X$ other than $H_x$.
We see that $\{x, z, t, w\}$ isolates $\msp$, hence we can take a $\mbQ$-divisor $T \in |11 A|_{\mbQ}$ such that $\omult_{\msp} (T) \ge 1$ and $\Supp (T)$ does not contain any component of $D \cdot H_x$.
Then we have
\[
\omult_{\msp} (D) 
\le (q_{\msp}^*D \cdot q_{\msp}^*H_x \cdot q_{\msp}
^*T)_{\check{\msp}}
\le 2 (D \cdot H_x \cdot T)
= 2 \cdot 1 \cdot 11 \cdot (A^3) = \frac{3}{5},
\]
where $q_{\msp}$ is the quotient morphism of $\msp \in X$ and $\check{\msp}$ is the preimage of $\msp$ via $q_{\msp}$.
This shows $\lct_{\msp} (X;D) \ge 1$ and thus $\alpha_{\msp} (X) \ge 1$.

Suppose that $y^8 w, y^9 t, y^{11} z \notin F$.
Then $y^{13} x \in F$ and we can write
\[
F = y^{13} x + y^{12} f_3 + \cdots + y f_{25} + f_{27},
\]
where $f_i = f_i (x,z,t,w)$ is a quasi-homogeneous polynomial of degree $i$ with $z \notin f_5$, $t \notin f_9$ and $w \notin f_{11}$.
We see that $\bar{F} := F (0,1,z,t,w) \in (z,t,w)^3$ and the cubic part of $\bar{F}$ is not a cube of a linear form since $w^2 z \in \bar{F}$ and $w^3 \notin \bar{F}$. 
By Lemma \ref{lem:lcttangcube}, we have $\alpha_{\msp} (X) \ge 1/2$.

\paragraph{\it The family $\mcF_{67}$, the point of $\frac{1}{9} (1,4,5)$}
Let $X = X_{28} \subset \mbP (1, 1, 4, 9, 14)$ be a member of $\mcF_{67}$ and $\msp = \msp_t$ the singular point of type $\frac{1}{9} (1,4,5)$.
Replacing $x$ and $y$, we can write
\[
F = t^3 x + t^2 f_{10} + t f_{19} + f_{28},
\]
where $f_i = f_i (x,y,z,w)$ is a quasi-homogeneous polynomial of degree $i$.
We have $\omult_{\msp} (H_x) = 2$ since $w^2 \in F$.
Hence $\lct_{\msp} (X;H_x) \ge 1/2$ by Lemma \ref{lem:qtangdivncan}.
Thus, 
\[
\alpha_{\msp} (X) \ge \min \{1, \lct_{\msp} (X;H_x)\} \ge \frac{1}{2}
\] 
by Lemma \ref{lem:singnoncanbd}.

\paragraph{\it The family $\mcF_{82}$, the point of $\frac{1}{5} (1,2,3)$}
Let $X = X_{36} \subset \mbP (1, 1, 5, 12, 18)$ be a member of $\mcF_{82}$ and $\msp = \msp_z$ the singular point of type $\frac{1}{5} (1,2,3)$.
Replacing $x$ and $y$, we can write
\[
F = z^7 x + z^6 f_6 + z^5 f_{11} + \cdots + f_{36},
\]
where $f_i = f_i (x, y, t, w)$ is a quasi-homogeneous polynomial of degree $i$.
We have $\omult_{\msp} (H_x) = 2$ since $w^2 \in F$.
Hence $\lct_{\msp} (X;H_x) \ge 1/2$ by Lemma \ref{lem:qtangdivncan}.
Thus, 
\[
\alpha_{\msp} (X) \ge \min \{1, \lct_{\msp} (X;H_x)\} \ge \frac{1}{2}
\] 
by Lemma \ref{lem:singnoncanbd}.

\paragraph{\it The family $\mcF_{84}$, the point of type $\frac{1}{7} (1,2,5)$}
Let $X = X_{36} \subset \mbP (1, 7, 8, 9, 12)$ be a member of $\mcF_{84}$ and $\msp = \msp_y$ the singular point of type $\frac{1}{7} (1,2,5)$.
By the quasi-smoothness of $X$, either $y^4 z \in F$ or $y^5 x \in F$.
Moreover, we have $w^3, t^4, z^3 w \in F$ and we assume $\coeff_F (w^3) = \coeff_F (t^4) = \coeff_F (z^3 w) = 1$ by rescaling $z, t, w$.

Suppose $y^4 z \in F$.
Let $\rho_{\msp} \colon \breve{U}_{\msp} \to U_{\msp}$ the orbifold chart of $X$ containing $\msp$.
Then we have $\rho_{\msp}^* H_x \cdot \rho^*_{\msp} H_z = \breve{\Gamma}$, where 
\[
\breve{\Gamma} = (\breve{x} = \breve{z} = \breve{w}^3 + \breve{t}^4 = 0) \subset \breve{U}_{\msp}
\]
is an irreducible and reduced curve with $\mult_{\breve{\msp}} (\breve{\Gamma}) = 3$.
We see that $(X, H_x)$ is log canonical at $\msp$ since $H_x$ is quasi-smooth at $\msp$.
Thus, by Lemma \ref{lem:exclL}, we have $\alpha_{\msp} (X) \ge 1$.

Suppose $y^4 z \notin F$.
Then $y^5 x \in F$ and we can write
\[
F = y^5 x + y^4 f_8 + y^3 f_{15} + y^2 f_{22} + y f_{29} + f_{36},
\]
where $f_i = f_i (x, z, t, w)$ is a quasi-homogeneous polynomial of degree $i$ with $z \notin f_8$.
By setting $\alpha := \coeff_F (y w t z)$, we have
\[
\bar{F} := F (0,1,z,t,w) = \alpha w t z + w^3 + w z^3 + t^4.
\]
\begin{itemize}
\item If $\alpha \ne 0$, then $\bar{F} \in (z,t,w)^3$ and the cubic part of $\bar{F}$ is not a cube of a linear form.
Hence $\lct_{\msp} (X; H_x) \ge 1/2$ by Lemma \ref{lem:lcttangcube}.
\item If $\alpha = 0$, then the lowest weight part of $\bar{F}$ with respect to $\wt (z, t, w) = (8,9,12)$ is $\bar{F} = w^3 + w z^3 + t^4$.
By Lemma \ref{lem:lctwblwh}, we have
\[
\lct_{\msp} (X; H_x) \ge \min \left\{ \frac{29}{36}, \lct (\tilde{\mbP}, \Diff; \mcD ) \right\},
\]
where 
\begin{itemize}
\item $\tilde{\mbP} = \mbP (2, 3, 1)_{\tilde{z}, \tilde{t}, \tilde{w}} = \mbP (8, 9, 12)^{\wf}$,
\item $\Diff = \frac{2}{3} H_{\tilde{z}} + \frac{3}{4} H_{\tilde{t}}$ with $H_{\tilde{z}} = (\tilde{z} = 0) \subset \tilde{\mbP}$, $H_{\tilde{t}} = (\tilde{t} = 0) \subset \tilde{\mbP}$, and
\item $\mcD$ is the prime divisor  $(\tilde{w}^3 + \tilde{w} \tilde{z} + \tilde{t} = 0)$ on $\tilde{\mbP}$.
\end{itemize}
We see that $\mcD$ is quasi-smooth, $\mcD \cap H_{\tilde{z}} \cap H_{\tilde{t}} = \emptyset$ and any two of $\mcD, H_{\tilde{z}}, H_{\tilde{t}}$ intersect transversally.
It follows that $\lct (\tilde{\mbP}, \Diff; \mcD) = 1$ and thus $\lct_{\msp} (X;H_x) \ge 29/36$.
\end{itemize}
Note that $\msp \in X$ is not a maximal center and the pair $(X, H_x)$ is not canonical at $\msp$ by Lemma \ref{lem:qtangdivncan}.
Thus 
\[
\alpha_{\msp} (X) \ge \min \{1, \lct_{\msp} (X;H_x)\} \ge \frac{1}{2}
\] 
by Lemma \ref{lem:singnoncanbd}.

\subsection{EI centers}

\begin{Prop} \label{prop:alphaEI}
Let $X$ be a member of a family $\mcF_{\msi}$ with $\msi \in \msI$ and $\msp \in X$ an EI center.
Then, 
\[
\alpha_{\msp} (X) \ge \frac{1}{2}.
\]
\end{Prop}

\begin{proof}
We have $\alpha_{\msp} (X) \ge 1$ by Proposition \ref{prop:lctsingptL} for a member $X$ of $\mcF_{\msi}$ and $\msp \in X$, where
\begin{itemize}
\item $\msi = 36$ and $\msp$ is of type $\frac{1}{4} (1, 1, 3)$.
\item $\msi = 44$ and $\msp$ is of type $\frac{1}{6} (1, 1, 5)$.
\item $\msi = 61$ and $\msp$ is of type $\frac{1}{7} (1, 2, 5)$.
\item $\msi = 76$ and $\msp$ is of type $\frac{1}{8} (1, 3, 5)$.
\end{itemize}
We have $\alpha_{\msp} (X) \ge 1/2$ by Proposition \ref{prop:singptic} for a member $X$ of $\mcF_{\msi}$ and $\msp \in X$, where
\begin{itemize}
\item $\msi = 23$ and $\msp$ is of type $\frac{1}{4} (1, 1, 3)$.
\item $\msi = 40$ and $\msp$ is of type $\frac{1}{5} (1, 2, 3)$.
\end{itemize} 
It remains to consider members of families $\mcF_7$ and $\mcF_{20}$, and singular points of types $\frac{1}{2} (1, 1, 1)$ and $\frac{1}{3} (1, 1, 2)$, respectively.

Let $X = X_8 \subset \mbP (1, 1, 2, 2, 3)$ be a member of $\mcF_7$ and $\msp$ a singular point of type $\frac{1}{2} (1, 1, 1)$.  
Replacing homogeneous coordinates, we may assume $\msp = \msp_t$ and we can write 
\[
F = t^3 z + t^2 f_4 + t f_6 + f_8,
\]
where $f_i = f_i (x, y, z, w)$ is a quasi-homogeneous polynomial of degree $i$.
Hence, by Lemma \ref{lem:complctsingtang}, we have
\[
\alpha_{\msp} (X) \ge \frac{2}{2 \cdot 1 \cdot 3 \cdot (A^3)} = \frac{1}{2}.
\]

Let $X = X_{13} \subset \mbP (1, 1, 3, 4, 5)$ be a member of $\mcF_{20}$ and $\msp = \msp_z$ be the singular point of type $\frac{1}{3} (1, 1, 2)$.
If $z^3 t \in F$, then we have
\[
\alpha_{\msp} (X) \ge \frac{2}{3 \cdot 1 \cdot 5 \cdot (A^3)} = \frac{8}{13}
\]
by Lemma \ref{lem:complctsingtang}.
Suppose $z^3 t \notin F$.
Then we can write that 
\[
F = z^4 x + z^3 f_4 + z^2 f_7 + z f_{10} + f_{13},
\]
where $f_i = f_i (x,y,t,w)$ is a quasi-homogeneous polynomial of degree $i$ with $t \notin f_4$.
We have $\omult_{\msp} (H_x) = 2$ since $w^2 z \in F$ by the quasi-smoothness of $X$ at $\msp_w$.
This shows $\lct_{\msp} (X;H_x) \ge 1/2$.
The point $\msp$ is not a maximal singularity and the pair $(X, H_x)$ is not canonical at $\msp$ by Lemma \ref{lem:qtangdivncan}.
Thus, 
\[
\alpha_{\msp} (X) \ge \min \{1, \lct_{\msp} (X;H_x)\} \ge \frac{1}{2}
\] 
by Lemma \ref{lem:singnoncanbd}.
This completes the proof.
\end{proof}

\subsection{Equations for QI centers}
\label{sec:eqQI}

Let 
\[
X = X_d \subset \mbP (1,a_1,\dots, a_4)_{x_0, \dots, x_4} =: \mbP
\] 
be a member of a family $\mcF_{\msi}$ with $\msi \in \msI$.
We set $a_0 = 1$ and let $F = F (x_0, \dots, x_4)$ be the defining polynomial of $X$.

\begin{Def}
Let $\msp \in X$ be a QI center and let $j, k$ be such that $j \ne k$, $d = 2 a_k + a_j$ and the index of $\msp \in X$ coincides with $a_k$.
Then we can choose coordinates so that $\msp = \msp_{x_k}$.
We say that $\msp$ is {\it an exceptional QI center} if $x_k^2 x_l \notin F$ for any $l \in \{0, \dots, 4\}$.
\end{Def}

\begin{Lem} \label{lem:QIcoord}
Let $\msp \in X$ be a non-exceptional QI center.
Then we can choose homogeneous coordinates $x_{i_1}, x_{i_2}, x_{i_3}, x_j, x_k$ of $\mbP$, where $\{i_1, i_2, i_3, j, k\} = \{0, 1, 2, 3, 4\}$, such that $a_{i_1}, a_{i_2}, a_{i_3} < a_k$, $\msp = \msp_{x_k}$ and
\begin{equation} \label{eq:QIcoord}
F = x_k^2 x_j + x_k f (x_{i_1}, x_{i_2}, x_{i_3},x_j) + g (x_{i_1},x_{i_2},x_{i_3},x_j)
\end{equation}
for some quasi-homogeneous polynomials $f, g \in \mbC [x_{i_1}, x_{i_2}, x_{i_3}, x_j]$ of degree $d - a_k$, $d$, respectively.
\end{Lem}

\begin{proof}
Basically this follows by looking at Table \ref{table:main}.
See also \cite[Theorem 4.9]{CPR00}.
\end{proof}

Let $\msp \in X$ be a non-exceptional QI center and we choose and fix homogeneous coordinates $x_{i_1}, x_{i_2}, x_{i_3}, x_j, x_k$ of $\mbP$ as in Lemma \ref{lem:QIcoord}.

\begin{Def}
We say that $\msp$ is a {\it degenerate QI center} if $f (x_{i_1}, x_{i_2}, x_{i_3},0) = 0$ as a polynomial, otherwise we call $\msp$ a {\it non-degenerate QI center}.
\end{Def}

\begin{Rem} \label{rem:QImaxcent}
It is proved in \cite[Section 4.1]{CP17} that a QI center $\msp \in X$ is a maximal center if and only if it is non-degenerate.
\end{Rem}

\begin{Lem} \label{lem:eqdegenQI}
Let $\msp$ be a degenerate QI center.
Then we can choose homogeneous coordinates $x_{i_1}, x_{i_2}, x_{i_3}, x_j, x_k$ of $\mbP$ such that $a_{i_1}, a_{i_2}, a_{i_3} < a_k$, $\msp = \msp_{x_k}$ and
\begin{equation} \label{eq:eqdegenQI}
F = x_k^2 x_j + g (x_{i_1},x_{i_2},x_{i_3},x_j)
\end{equation}
for some quasi-homogeneous polynomial $g \in \mbC [x_{i_1},x_{i_2},x_{i_3},x_j]$ of degree $d$.
Moreover the hypersurface
\[
(g (x_{i_1}, x_{i_2}, x_{i_3}, 0) = 0) \subset \mbP (a_{i_1}, a_{i_2}, a_{i_3})_{x_{i_1}, x_{i_2}, x_{i_3}}
\]
is quasi-smooth.
\end{Lem}

\begin{proof}
We have $f = x_j f'$ for some $f' \in \mbC [x_{i_1}, x_{i_2}, x_{i_3}, x_j]$ since $\msp$ is degenerate.
Filtering off terms divisible by $x_j$ in \eqref{eq:QIcoord}, we have
\[
F = x_j (x_k^2 + x_k f') + g.
\]
We can eliminate the term $x_k x_j f'$ by replacing $x_k \mapsto x_k - f'/2$.
This shows the first assertion.

We choose and fix homogeneous coordinates so that $F$ is of the form \eqref{eq:eqdegenQI}.
We set $\bar{g} = g (x_{i_1},x_{i_2},x_{i_3},0)$.
Then we can write $g = \bar{g} + x_j h$, where $h = h (x_{i_1},x_{i_2},x_{i_3},x_j)$.
Suppose to the contrary that $(\bar{g} = 0) \subset \mbP (a_{i_1},a_{i_2},a_{i_3})$ is not quasi-smooth at a point $(\alpha_1\!:\!\alpha_2\!:\!\alpha_3)$.
We choose and fix $\beta \in \mbC$ such that $\beta^2 + h (\alpha_1,\alpha_2,\alpha_3,0) = 0$, and set 
\[
\msq := (\alpha_1\!:\!\alpha_2\!:\!\alpha_3\!:\!0\!:\!\beta) \in \mbP (a_{i_1},a_{i_2},a_{i_3},a_j,a_k) = \mbP.
\]
It is easy to see that $(\prt F/\prt v) (\msq) = 0$ for any $v \in \{x_{i_1},x_{i_2},x_{i_3},x_j,x_k\}$.
This is impossible since $X$ is quasi-smooth.
Therefore $(\bar{g} = 0) \subset \mbP (a_{i_1},a_{i_2},a_{i_3})$ is quasi-smooth.
\end{proof}

\begin{Lem} \label{lem:QIeqtypes}
Let $X$ be a member of a family $\mcF_{\msi}$ with $\msi \in \msI \setminus \{2, 8\}$.
Suppose that $X$ has a QI center.
Then one of the following holds.
\begin{enumerate}
\item $X$ has a unique QI center.
In this case, by a choice of homogeneous coordinates, we have
\[
X = X_{2 r + c} \subset \mbP (1, a, b, c, r)_{x, s, u, v, w},
\]
where $a$ is coprime to $b$, $a < b$, $a + b = r$, $c < r$, and the unique QI center is the point $\msp = \msp_w$, which is of type $\frac{1}{r} (1, a, b)$.
\item $X$ has exactly $3$ distinct QI centers.
In this case, by a choice of homogeneous coordinates, we have
\[
X = X_{3 r} \subset \mbP (1, a, b, r, r)_{x, y, z, t, w},
\]
where $a$ is coprime to $b$, $a \le b$ and $a + b = r$.
The $3$ QI centers are the $3$ points in $(x = y = z = 0) \cap X$ and they are all of type $\frac{1}{r} (1, a, b)$.
\item $X$ has exactly $2$ distinct QI centers and their singularity types are equal.
In this case, by a choice of homogeneous coordinates, we have
\[
X = X_{4 r} \subset \mbP (1, a, b, r, 2 r)_{x, y, z, t, w},
\]
where $a$ is coprime to $b$, $a \le b$ and $a + b = r$.
The QI centers are the 2 points in $(x = y = z = 0) \cap X$ and they are both of type $\frac{1}{r} (1, a, b)$.
\item $X$ has exactly $2$ distinct QI centers and their singularity types are distinct. 
In this case, by a choice of homogeneous coordinates, we have
\[
X = X_{4 a + 3 b} \subset \mbP (1, a, b, r_1, r_2)_{x, u, v, t, w},
\]
where $a$ is coprime to $b$, $a + b = r_1$ and $2 a + b = r_2$.
The QI centers are $\msp_t$ and $\msp_w$ which are of types $\frac{1}{r_1} (1, a, b)$ and $\frac{1}{r_2} (1, a, a + b)$, respectively.
\end{enumerate}
\end{Lem}

\begin{proof}
Let 
\[
X = X_d \subset \mbP (1, a_1, a_2, a_3, a_4)_{x_0, x_1, x_2, x_3, x_4}
\]
be a member of $\mcF_{\msi}$ with$\msi \in \msI \setminus \{2, 8\}$.
We set $a_0 = 1$ and assume $a_1 \le \dots \le a_4$.
We assume that $X$ has at least one QI center.

Suppose $d = 3 a_4$.
Let $\msp \in X$ be a QI center.
Then, after replacing homogeneous coordinates, we may assume $\msp = \msp_{x_i}$ and $x_i^2 x_j \in F$ for some $i \in \{0, 1, 2, 3, 4\}$ and $j \in \{0, 1, 2, 3, 4\} \setminus \{i\}$.
In particular we have $d = 3 a_4 = 2 a_i + a_j$, which is possible if and only if $a_i = a_j = a_4$.
Thus we have $a_3 = a_4$ and we may assume $i = 4, j = 3$.
We see that $\msp \in X$ is of type $\frac{1}{a_4} (1, a_1, a_2)$ and $\msp \in X$ is terminal.
It follows that $a_1 + a_2 = a_4$ and $a_l$ is coprime to $a_4$ for $l = 1, 2$.
By setting $a := a_1$, $b := a_2$ and $r := a_3 = a_4$, this case corresponds to (2).
In the following we assume $d < 3 a_4$.

Suppose $d = 2 a_4$.
We have $a_4 = a_1 + a_2 + a_3$ since $d = a_1 + a_2 + a_3 + a_4 = 2 a_4$.
Let $\msp \in X$ be a QI center.
Then we may assume $\msp = \msp_{x_i}$ and $x_i^2 x_j \in F$ for some $i \in \{0, 1, 2, 3\}$ and $j \in \{0, 1, 2, 3, 4\} \setminus \{i\}$.
In particular we have $d = 2 a_i + a_j$, and hence 
\[
2 a_i + a_j = a_1 + a_2 + a_3 + a_4 = 2 (a_1 + a_2 + a_3),
\] 
which is only possible when $j = 4$ and $i = 3$.
Hence $i = 3$, $j = 4$, and we have $a_4 = 2 a_3$ since $d = 2 a_3 + a_4 = 2 a_4$.
We see that $\msp \in X$ is of type $\frac{1}{a_3} (1, a_1, a_2)$ and $\msp \in X$ is terminal.
It follows that $a_3 = a_1 + a_2$ and $a_l$ is coprime to $a_3$ for $l = 1, 2$.
By setting $a:= a_1$, $b := a_2$, $r := a_3$, this case corresponds to (3).

Suppose $d = 2 a_4 + a_3$.
We have $a_4 = a_1 + a_2$ since $d = a_1 + a_2 + a_3 + a_4$.
We see that $\msp_4 \in X$ is of type $\frac{1}{a_4} (1, a_1, a_2)$ and it is a QI center.
It follows that $a_4$ is coprime to $a_l$ for $l = 1, 2$ since $\msp_4 \in X$ is a terminal singularity.
If $X$ admits a QI center other than $\msp_4$, then we have $d = 2 a_i + a_j$, where $i \in \{1, 2, 3\}$ and $j \in \{0, 1, 2, 3, 4\} \setminus \{
i\}$ which is impossible.
Thus $\msp_4 \in X$ is a unique QI center, and we are in case (1) by setting $a := a_1$, $b := a_2$, $c := a_3$ and $r := a_4$. 
Note that we have $a < b$ because otherwise we have $a_1 = a_2 = 1$ and $a_4 = 2$ and $X$ belongs to a family $\mcF_{\msi}$ with $\msi \in \{2, 8\}$ which is impossible.

Suppose $d = 2 a_4 + a_2$.
Then $a_4 = a_1 + a_3$.
We see that $\msp_4 \in X$ is of type $\frac{1}{a_4} (1, a_1, a_3)$ and it is a QI center.
If $\msp_4$ is a unique QI center, then we are in case (1) by setting $a := a_1$, $b := a_3$, $c := a_2$ and $r := a_4$.
We assume that $X$ admits a QI center $\msp \in X$ other than $\msp_4$.
We may assume $\msp = \msp_i$ after replacing homogeneous coordinates and we have $d = 2 a_i + a_j$ for some $i \in \{1, 2, 3\}$ and $j \in \{0, 1, 2, 3, 4\} \setminus \{i\}$.
Then we have $a_j = a_4$ and $a_i = a_3$.
Thus $i = 3$, $j = 4$ and we have $a_3 = a_1 + a_2$.
The singularity of $\msp = \msp_i \in X$ is of type $\frac{1}{a_3} (1, a_1, a_2)$ and it is terminal.
It follows that $a_1$ is coprime to $a_2$.
Thus we are in case (4) by setting $a := a_1$, $b := a_2$, $r_1 := a_3$ and $r_2 = a_4$.

Suppose $d = 2 a_4 + a_1$.
Then, by interchanging the role of $a_1$ and $a_2$ in the privious arguments, we conclude that this case corresponds to either (1) or (4).
This completes the proof.
\end{proof}

\begin{Lem} \label{lem:uniqQItypes}
Let 
\[
X = X_{2 r + c} \subset \mbP (1, a, b, c, r)_{x, s, u, v, w}
\]
be a member of a family $\mcF_{\msi}$ with $\msi \in \msI \setminus \{2, 8\}$ with a unique QI center, where $a$ is coprime to $b$, $a < b$, $r = a + b$ and $c < r$.
Then the following assertions hold.
\begin{enumerate}
\item If $c = 1$, then $X$ belongs to a family $\mcF_{\msi}$ with $\msi \in \{24, 46\}$.
\item If $2 r + c$ is not divisible by $b$, then $b = a + 1$, $c = a + 2$, $r = 2 a + 1$ and $a \in \{2, 3, 4\}$.
\end{enumerate}
\end{Lem}

\begin{proof}
This follows from Table \ref{table:main}.
\end{proof}

\begin{Lem} \label{lem:QIexceq}
Let $X$ be a member of a family $\mcF_{\msi}$ with $\msi \in \msI \setminus \msI_1$ and let $\msp \in X$ be an exceptional QI center.
Then we are in Case (4) of Lemma \ref{lem:QIeqtypes} and $\msp = \msp_t$.
Moreover we can write
\begin{equation} \label{eq:QIexceq}
F = t^3 u + t^2 f_{2 a + b} + t f_{3 a + 2 b} + f_{4 a + 3 b},
\end{equation}
where $f_i \in \mbC [x, u, v, w]$ is a quasi-homogeneous polynomial of degree $i$ with $w \notin f_{2 a + b}$.
\end{Lem}

\begin{proof}
We are in (1), (2), (3) and (4) of Lemma \ref{lem:QIeqtypes}.
Suppose that we are in (1).
Then $\msp = \msp_w$.
Since $\msp \in X$ is exceptional and $X$ is quasi-smooth at $\msp$, we have $w^m q \in F$ for some $m \ge 3$ and a homogeneous coordinate $q \in \{x, s, u, v\}$.
This implies 
\[
2 r + c = d = m r + \deg{q} \ge 3 r + 1,
\]
which is impossible since $c < r$.
By similar arguments we can show that (2) and (3) are both impossible.

It follows that we are in Case (4). 
In this case either $\msp = \msp_t$ or $\msp = \msp_w$.
The latter is impossible since $d = 4 a + 3 b < 3 r_2$.
Hence $\msp = \msp_t$.
We have $t^m q \in F$ for some integer $m \ge 3$ and a homogeneous coordinate $q \in \{x, u, v, w\}$.
It is easy to see that this is possible if and only if $m = 3$ and $\deg q = a$.
Possibly replacing coordinates we may assume $q = u$.
Then it is straightforward to see that $F$ can be written as \eqref{eq:QIexceq}.
\end{proof}

\subsection{QI centers: exceptional case}

The aim of this section is to prove the following.

\begin{Prop} \label{prop:lctexcQI}
Let $X$ be a member of a family $\mcF_{\msi}$ with $\msi \in \msI \setminus \msI_1$ and let $\msp \in X$ be an exceptional QI center. 
Then,
\[
\alpha_{\msp} (X) \ge \frac{1}{2}.
\]
\end{Prop}

Let $X$ be a member of $\mcF_{\msi}$ with $\msi \in \msI \setminus \{2, 8\}$ which admits an exceptional QI center,
Then, by Lemma \ref{lem:QIexceq} and Table \ref{table:main}, we have
\[
\msi \in \{12, 13, 20, 25, 31, 33, 38, 58\}.
\]
The rest of this section is devoted to the proof of Proposition \ref{prop:lctexcQI} which will be done by division into cases.
By Lemma \ref{lem:QIexceq}, we can choose coordinates $x, u, v, t, w$ of $\mbP = \mbP (1, a, b, r_1, r_2)$ as in Case (4) of Lemma \ref{lem:QIeqtypes} with $\msp = \msp_t$ and the defining polynomial $F$ is as in \eqref{eq:QIexceq}.

\subsubsection{Case: $a \ge 2$ and $4 a \le 3 b$}
This case corresponds to families $\mcF_{33}$ and $\mcF_{58}$.
We have $w^2 v \in f_{4 a + 3 b}$ since $X$ is quasi-smooth at $\msp_w$.
Moreover, we see that no quadratic monomial in variables $x, v, w$ appear in $f_{2 a + b}, f_{3 a + 2 b}, f _{4 a + 3 b}$.
This implies $\omult_{\msp} (H_u) = 3$, and we have
\[
\alpha_{\msp} \left( X; \frac{1}{a} H_u \right) \ge \frac{a}{3} \ge \frac{2}{3}.
\]
Let $D \in |A|_{\mbQ}$ be an effective $\mbQ$-divisor other than $\frac{1}{a} H_u$.
We can take a $\mbQ$-divisor $T \in |r_2 A|_{\mbQ}$ such that $\omult_{\msp} (T) \ge 1$ and $\Supp (T)$ does not contain any component of the effective $1$-cycle $D \cdot H_u$ since $\{x, u, v, w\}$ isolates $\msp$.
Let $q = q_{\msp}$ be the quotient morphism of $\msp \in X$ and let $\check{\msp}$ be the preimage of $\msp$ via $q$.
Then we have
\[
3 \omult_{\msp} (D) 
\le (q^*D \cdot q^*H_u \cdot q^*T)_{\check{\msp}}
\le r_1 (D \cdot H_u \cdot T)
= \frac{4 a + 3 b}{b} \le 6
\]
since $4 a \le 3 b$.
Thus $\lct_{\msp} (X; D) \ge 1/2$ and the inequality $\alpha_{\msp} (X) \ge 1/2$ is proved.

\subsubsection{Case: $a = 1$}
This case corresponds to families $\mcF_{12}$, $\mcF_{20}$ and $\mcF_{31}$.
We have either $\bar{F} = F (x, 0, v, 1, w) \in (x, v, w)^2 \setminus (x, v, w)^3$ or $\bar{F} \in (x, v, w) \in (x, v, w)^3$ and its cubic part is not a cube of a linear form since $w^2 v \in F$ and $w^3 \notin F$.
By Lemma \ref{lem:lcttangcube}, we have $\alpha_{\msp} (X) \ge 1/2$ since $a = 1$.

\subsubsection{Case: $X$ is a member of the family $\mcF_{13}$}
Let
\[
X = X_{11} \subset \mbP (1, 1, 2, 3, 5)_{x, y, z, t, w}
\]
be a member of $\mcF_{13}$ and $\msp \in X$ an exceptional QI center.
Then we have
\[
F = t^3 z + t^2 f_5 + t f_8 + f_{11},
\]
where $f_i \in \mbC [x, y, z, w]$ is a quasi-homogeneous polynomial of degree $i$ with $w \notin f_5$, and $\msp = \msp_t$.
Let $S, T \in |A|$ be general members.
We have
\[
F (0, 0, z, t, w) = t^3 z + \alpha t z^4 + \beta w z^3 = z (t^3 + \alpha t z^3 + \beta w z^2),
\]
where $\alpha, \beta \in \mbC$.
We set $\Gamma = (x = y = z = 0)$, which is a quasi-line of degree $1/15$.
If $\beta \ne 0$, then we set 
\[
\Delta = (x = y = t^3 + \alpha t z^3 + \beta w z^2 = 0),
\]
which is clearly an irreducible and reduced curve of degree $3/10$ and does not pass through $\msp$.
Moreover, we have
\[
T|_S = \Gamma + \Delta.
\]

\begin{Claim} \label{clm:lctexcQINo13-1}
If $\beta \ne 0$, then the intersection matrix $M (\Gamma, \Delta)$ satisfies the condition $(\star)$.
\end{Claim}

\begin{proof}[Proof of Claim \ref{clm:lctexcQINo13-1}]
We have $\Gamma \cap \Delta = \{\msp_w\}$ and it is easy to see that $S$ is quasi-smooth at $\msp_w$ since $S \in |A|$ is general.
By Lemma \ref{lem:pltsurfpair}, $S$ is quasi-smooth along $\Gamma$ and we have $\Sing_{\Gamma} (S) = \{\msp_t, \msp_w\}$, where $\msp_t, \msp_w \in S$ are of types $\frac{1}{3} (1, 2)$, $\frac{1}{5} (2, 3)$, respectively.
By Remark \ref{rem:compselfint}, we have
\[
(\Gamma^2)_S = -2 + \frac{2}{3} + \frac{4}{5} = - \frac{8}{15}.
\]
By taking intersection numbers of $T|_S = \Gamma + \Delta$ and $\Gamma, \Delta$, we have
\[
(\Gamma \cdot \Delta)_S = \frac{3}{5}, \quad
(\Delta^2)_S = - \frac{3}{10}.
\]
Thus $M (\Gamma, \Delta)$ satisfies the condition $(\star)$.
\end{proof}

Suppose $\beta = 0$ and $\alpha \ne 0$.
We set
\[
\Xi = (x = y = t = 0), \quad
\Theta = (x = y = t^2 + \alpha z^3 = 0),
\]
which are irreducible and reduced curves of degrees $1/10$, $1/5$, respectively, which do not pass through $\msp$.
We have
\[
T|_S = \Gamma + \Xi + \Theta.
\]

\begin{Claim} \label{clm:lctexcQINo13-2}
If $\beta = 0$ and $\alpha \ne 0$, then the intersection matrix $M (\Gamma, \Xi, \Theta)$ satisfies the condition $(\star)$.
\end{Claim}

\begin{proof}[Proof of Claim \ref{clm:lctexcQINo13-2}]
We have $(\Gamma^2)_S = - 8/15$ by the proof of Claim \ref{clm:lctexcQINo13-1}.
We see that $\Xi \cap (\Gamma \cup \Theta) = \{\msp_w\}$ and $S$ is quasi-smooth at $\msp_w$.
By Lemma \ref{lem:pltsurfpair}, $S$ is quasi-smooth along $\Xi$ and we have $\Sing_{\Xi} (S) = \{\msp_z, \msp_w\}$, where $\msp_z, \msp_w \in S$ are of types $\frac{1}{2} (1, 1)$, $\frac{1}{5} (2, 3)$, respectively.
By Remark \ref{rem:compselfint}, we have
\[
(\Xi^2)_S = -2 + \frac{1}{2} + \frac{4}{5} = - \frac{7}{10}.
\]
We compute the intersection number $(\Gamma \cdot \Xi)_S$.
We have $\Gamma \cap \Xi = \{\msp_w\}$ and the germ $\msp_w \in S$ is analytically isomorphic to $\bar{o} \in \mbA^2_{z, t}/\bmu_5$, where the $\bmu_5$-action on $\mbA^2_{z, t}$ is given by
\[
(z, t) \mapsto (\zeta^2 z, \zeta^3 t),
\]
and $\bar{o}$ is the image of the origin $o \in \mbA^2_{z, t}$.
Under the isomorphism, $\Gamma$ and $\Xi$ corresponds to the quotient of $(z = 0)$ and $(t = 0)$.
It follows that
\[
(\Gamma \cdot \Xi)_S = (\Gamma \cdot \Xi)_{\msp_w} = \frac{1}{5}.
\]
Then, by taking intersection numbers of $T|_S = \Gamma + \Xi + \Theta$ and $\Gamma, \Xi, \Theta$, we have
\[
(\Gamma \cdot \Theta)_S = \frac{2}{5}, \quad
(\Xi \cdot \Theta)_S = \frac{3}{5}, \quad
(\Theta^2)_S = - \frac{4}{5}.
\]
Thus $M (\Gamma, \Xi, \Theta)$ satisfies the condition $(\star)$.
\end{proof}

Suppose $\beta = \alpha = 0$.
Then 
\[
T|_S = \Gamma + 3 \Xi,
\]
where $\Xi = (x = y = t = 0)$.

\begin{Claim} \label{clm:lctexcQINo13-3}
If $\beta = \alpha = 0$, then the intersection matrix $M (\Gamma, \Xi)$ satisfies the condition $(\star)$.
\end{Claim}

\begin{proof}[Proof of Claim \ref{clm:lctexcQINo13-3}]
We have $(\Gamma^2)_S = - 8/15$ by the proof of Claim \ref{clm:lctexcQINo13-1}.
By taking intersection numbers of $T|_S = \Gamma + 3 \Xi$ and $\Gamma, \Xi$, we have
\[
(\Gamma \cdot \Xi)_S = \frac{1}{5}, \quad
(\Xi^2)_S = - \frac{1}{30}.
\]
Thus $M (\Gamma, \Xi)$ satisfies the condition $(\star)$.
\end{proof}

By Claims \ref{clm:lctexcQINo13-1}, \ref{clm:lctexcQINo13-2}, \ref{clm:lctexcQINo13-3} and Lemma \ref{lem:mtdLred}, we have
\[
\alpha_{\msp} (X) \ge \min \left\{ 1, \ \frac{1}{3 (A^3) + 1 - 3 \deg \Gamma} \right\} = \frac{10}{19}.
\]

\subsubsection{Case: $X$ is a member of the family $\mcF_{25}$}
Let 
\[
X = X_{15} \subset \mbP (1, 1, 3, 4, 7)_{x, y, z, t, w}
\] 
be a member of $\mcF_{25}$ and let $\msp$ be an exceptional QI center.
Then we have
\[
F = t^3 z + t^2 f_7 + t f_{11} + f_{15},
\]
where $f_i \in \mbC [x, y, z, w]$ is a quasi-homogeneous polynomial of degree $i$ with $w \notin f_7$, and we have $\msp = \msp_t$.
By the quasi-smoothness we have $z^5 \in f_{15}$ and we may assume $\coeff_{f_{15}} (z^5) = 1$.
Then we have
\[
F (0, 0, z, t, w) = t^3 z + z^5 = z (t^3 + z^4).
\]
Let $S, T \in |A|$ be general members.
Then we have
\[
T|_S = \Gamma + \Delta,
\]
where $\Gamma = (x = y = z = 0)$ is a quasi-line of degree $1/28$ and $\Delta = (x = y = t^3 + z^4)$ is an irreducible and reduced curve of degree $1/7$ that does not pass through $\msp$.

\begin{Claim} \label{clm:lctexcQINo25-1}
The intersection matrix $M (\Gamma, \Delta)$ satisfies the condition $(\star)$.
\end{Claim}

\begin{proof}[Proof of Claim \ref{clm:lctexcQINo25-1}]
We have $\Gamma \cap \Delta = \{\msp_w\}$ and $S$ is quasi-smooth at $\msp_w$.
Hence $S$ is quasi-smooth along $\Gamma$ by Lemma \ref{lem:pltsurfpair}, and we have $\Sing_{\Gamma} (S) = \{ \msp_t, \msp_w\}$, where $\msp_t, \msp_w \in S$ are of types $\frac{1}{4} (1, 3), \frac{1}{7} (3, 4)$, respectively.
By Remark \ref{rem:compselfint}, we have
\[
(\Gamma^2)_S = -2 + \frac{3}{4} + \frac{6}{7} = - \frac{11}{28}.
\]
By taking intersection numbers of $T|_S = \Gamma + \Delta$ and $\Gamma, \Delta$, we have
\[
(\Gamma \cdot \Delta)_S = \frac{3}{7}, \quad
(\Delta^2)_S = - \frac{2}{7}.
\]
It follows that $M (\Gamma, \Delta)$ satisfies the condition $(\star)$.
\end{proof}

By Claim \ref{clm:lctexcQINo25-1} and Lemma \ref{lem:mtdLred}, we have
\[
\alpha_{\msp} (X) \ge \min \left\{ 1, \ \frac{1}{4 (A^3) + 1 - 4 \deg \Gamma} \right\} = \frac{7}{11}.
\]

\subsubsection{Case: $X$ is a member of the family $\mcF_{38}$}
Let 
\[
X = X_{18} \subset \mbP (1, 2, 3, 5, 8)_{x, y, z, t, w}
\]
be a member of $\mcF_{38}$ and $\msp \in X$ an exceptional QI center.
Then we have
\[
F = t^3 z + t^2 f_8 + t f_{13} + f_{18},
\]
where $f_i \in \mbC [x, y, z, w]$ is a quasi-homogeneous polynomial of degree $i$ with $w \notin f_8$, and $\msp = \msp_t$.
By the quasi-smoothness of $X$, we have $z^6 \in f_{18}$ and we may assume $\coeff_{f_{18}} (z^6) = 1$.
Then we have
\[
F (0, 0, z, t, w) = t^3 z + z^6 = z (t^3 + z^5).
\]
We set $S = H_x$ and $T = H_y$.
We have
\[
T|_S = \Gamma + \Delta,
\]
where $\Gamma = (x = y = z = 0)$ is a quasi-line of degree $1/40$ and $\Delta = (x = y = t^3 + z^5 = 0)$ is an irreducible and reduced curve of degree $1/8$ that does not pass through $\msp$.

\begin{Claim} \label{clm:lctexcQINo38-1}
The intersection matrix $M (\Gamma, \Delta)$ satisfies the condition $(\star)$.
\end{Claim}

\begin{proof}[Proof of Claim \ref{clm:lctexcQINo38-1}]
By similar arguments as in the proof of Claim \ref{clm:lctexcQINo25-1}, we have
\[
(\Gamma^2)_S = -2 + \frac{4}{5} + \frac{7}{8} = - \frac{13}{40},
\]
and
\[
(\Gamma \cdot \Delta)_S = \frac{3}{8}, \quad
(\Delta^2)_S = - \frac{1}{8}.
\]
Thus $M (\Gamma, \Delta)$ satisfies the condition $(\star)$.
\end{proof}

By Claim \ref{clm:lctexcQINo38-1} and Lemma \ref{lem:mtdLred}, we have
\[
\alpha_{\msp} (X) \ge \min \left\{ 1, \ \frac{1}{5 \cdot 2 \cdot (A^3) + \frac{1}{2} - 5 \deg \Gamma} \right\} = \frac{8}{9}.
\]
This completes the proof of Proposition \ref{prop:lctexcQI}.

\subsection{QI centers: degenerate case}

The aim of this section is to prove the following, which gives the exact value of $\alpha_{\msp} (X)$ for a degenerate QI center $\msp \in X$.

Let 
\[
X = X_d \subset \mbP (a_0, a_1, \dots, a_4)_{x_0, x_1, x_2, x_3, x_4}
\]
be a member of a family $\mcF_{\msi}$ with $\msi \in \msI$, where $1 = a_0 \le a_1 \le \cdots \le a_4$, and let $\msp \in X$ be a degenerate QI center.
We choose homogeneous coordinates as in Lemma \ref{lem:eqdegenQI}.

\begin{Prop} \label{prop:lctdegQI}
Let the notation as above and let $\msp = \msp_{x_k} \in X$ be  a degenerate QI center.
Then, 
\[
\alpha_{\msp} (X) =
\begin{cases}
\frac{a_k + 1}{2 a_k + 1}, & \text{if $a_j = 1$}, \\
1, & \text{otherwise}.
\end{cases}
\] 
In particular, we have $\alpha_{\msp} (X) > \frac{1}{2}$.
\end{Prop}

\begin{proof}
Let $\varphi \colon Y \to X$ be the Kawamata blowup at $\msp$ with exceptional divisor $E$.
Note that we can choose $x_{i_1}, x_{i_2}, x_{i_3}$ as a system of orbifold coordinates at $\msp$ and $\varphi$ is the weighted blowup with weight $\wt (x_{i_1}, x_{i_2}, x_{i_3}) = \frac{1}{a_k} (a_{i_1}, a_{i_2}, a_{i_3})$.
Filtering off terms divisible by $x_j$ in \eqref{eq:eqdegenQI}, we have
\[
x_j (x_k^2 + \cdots) = g (x_{i_1}, x_{i_2}, x_{i_3}, 0) =: \bar{g}.
\]
Since the polynomial $x_k^2 + \cdots$ does not vanish at $\msp$, the vanishing order of $x_j$ along $E$ coincides with that of $\bar{g}$, which is clearly $d/a_k$.  
Hence we have
\begin{equation} \label{eq:complctdegQI}
K_Y + \frac{1}{a_j} \tilde{H}_{x_j} + \frac{2}{a_j} E
= \varphi^* \left( K_X + \frac{1}{a_j} H_{x_j} \right),
\end{equation}
where $\tilde{H}_{x_j}$ is the proper transform of $H_{x_j}$ on $Y$.
In particular $(X, \frac{1}{a_j} H_{x_j})$ is not canonical at $\msp$.
By Lemma \ref{lem:singnoncanbd}, we have $\alpha_{\msp} (X) \ge 1$ if $(X, \frac{1}{a_j} H_{x_j})$ is log canonical at $\msp$, and otherwise $\alpha_{\msp} (X) = \lct_{\msp} (X;\frac{1}{a_j} H_{x_j})$. 

Suppose $a_j > 1$.
The pair $(E, \frac{1}{a_j} \tilde{H}_{x_j}|_E)$ is log canonical since $\tilde{H}_{x_j}|_E$ is isomorphic to 
\[
(g (x_{i_1},x_{i_2},x_{i_3},0) = 0) \subset \mbP (a_{i_1},a_{i_2},a_{i_3}) \cong E,
\]
which is quasi-smooth by Lemma \ref{lem:eqdegenQI}.
By the inversion of adjunction, the pair $(Y, \frac{1}{a_j} \tilde{H}_{x_j} + E)$ is log canonical along $E$, and so is the pair $(Y, \frac{1}{a_j} \tilde{H}_{x_j} + \frac{2}{a_j} E)$ since $2/a_j \le 1$.
By \eqref{eq:complctdegQI}, the pair $(X, \frac{1}{a_j} H_{x_j})$ is log canonical at $\msp$.
Thus $\alpha_{\msp} (X) \ge 1$.
The existence of the prime divisor $H_{x_0} \in A$ passing through $\msp$ shows $\alpha_{\msp} (X) \le 1$, and we conclude $\alpha_{\msp} (X) = 1$ in this case.

Suppose $a_j = 1$.
We set
\[
\theta = \frac{a_k+1}{2 a_k + 1},
\]
and prove $\lct_{\msp} (X;\frac{1}{a_j} H_{x_j}) = \theta$.
For a rational number $c \ge 0$, it is easy to see that the discrepancy of the pair $(X, \frac{c}{a_j} H_{x_j})$ along $E$ is 
\[
\frac{1}{a_k} - \frac{c d}{a_j a_k}
\]
and it is at least $-1$ if and only if $c \le \theta$.
This shows $\lct_{\msp} (X;\frac{1}{a_j} H_{x_j}) \le \theta$.
Moreover, since
\[
K_Y + \frac{\theta}{a_j} \tilde{H}_{x_j} + E = \varphi^* \left( K_X + \frac{\theta}{a_j} H_{x_j} \right)
\]
and the pair $(Y, \frac{\theta}{a_j} \tilde{H}_{x_j} + E)$ is log canonical along $E$, the pair $(X, \frac{\theta}{a_j} H_{x_j})$ is log canonical at $\msp$.
This shows $\alpha_{\msp} (X) = \theta$ and the proof is completed.
\end{proof}

\begin{Ex} \label{ex:No46degQI}
Let $X = X_{21} \subset \mbP (1,1,3,7,10)$ be a member of the family $\mcF_{46}$ and $\msp = \msp_w$ the $\frac{1}{10} (1,3,7)$ point, which is the center of a quadratic involution.
Assume that $\msp$ is degenerate, which is equivalent to $X$ being birationally superrigid.
Then, by Proposition \ref{prop:lctdegQI}, we have
\[
\alpha (X) \le \alpha_{\msp} (X) = \frac{11}{21}.
\]
\end{Ex}

\subsection{QI centers: non-degenerate case}

The aim of this section is to prove the following.

\begin{Prop} \label{prop:ndQIcent}
Let $X$ be a member of a family $\mcF_{\msi}$ with $\msi \in \msI \setminus \{2, 5, 8\}$ and $\msp \in X$ be a nondegenerate QI center.
Then, 
\[
\alpha_{\msp} (X) \ge \frac{1}{2}.
\]
\end{Prop}

The rest of this section is entirely devoted to the proof of Proposition \ref{prop:ndQIcent}, which will be done by dividing into several cases.

\subsubsection{Case: $X$ has a unique QI center}

By Lemma \ref{lem:QIeqtypes}, we can choose homogeneous coordinates so that
\[
X = X_{2 r + c} \subset \mbP (1, a, b, c, r)_{x, s, u, v, w},
\]
where $a$ is corprime to $b$, $a < b$, $r = a + b$ and $c < r$.
Let $\msp \in X$ be the QI center.
Then $\msp = \msp_w$ and the defining polynomial $F$ of $X$ can be written as
\[
F = w^2 v + w f_{r + c} + f _{2 r + c},
\]
where $f_{r + c} = f_{r + c} (x, s, u)$ and $f_{2 r + c} = f_{2 r + c} (x, s, u, v)$ are quasi-homogeneous polynomials of the indicated degree.
We will show that $\lct_{\msp} (X; \frac{1}{c} H_v) \ge 1/2$.

\begin{Claim} \label{clm:ndQIcent-1}
Suppose that $c \ge 2$ and $2 r + c$ is divisible by $b$.
Then 
\[
\lct_{\msp} \left(X; \frac{1}{c} H_v \right) \ge \frac{1}{2}.
\]
\end{Claim}

\begin{proof}[Proof of Claim \ref{clm:ndQIcent-1}]
We first show that $\msp_u \notin X$ unless $X$ belongs to the family $\mcF_7$.
Suppose $\msp_u \in X$.
By the quasi-smoothness of $X$ at $\msp_u$, we have $d = m b + e$, where $m \in \mbZ_{> 0}$ and $e \in \{1, a, c, r\}$.
Since $d$ is divisible by $b$, we see that $e$ is divisible by $b$.
This is possible only when $e = c$ since $r = a + b$ and $a$ are both coprime to $b$.
Thus we can write $c = k b$ for some $k \in \mbZ_{> 0}$.
Take any point $\msq \in (x = s = w = 0) \cap X$.
The singularity $\msq \in X$ is of type $\frac{1}{b} (1, a, r) = \frac{1}{b} (1, a, a)$.
It follows that $a = 1$ and $b = 2$ since $\msq \in X$ is terminal.
We have $r = a + b = 3$ and $c = 2$ since $c = k b = 2 k < r = 3$. 
Thus $X = X_8 \subset \mbP (1, 1, 2, 2, 3)$ and this belongs to $\mcF_7$.

We first consider the case where $\msp_u \notin X$.
This means that $u^m \in F$ for some $m \in \mbZ_{> 0}$.
We have $\omult_{\msp} (H_v) \le m = (2 r + c)/b$ and hence
\[
\lct_{\msp} \left( X; \frac{1}{c} H_v \right) \ge \frac{b c}{2 r + c}.
\]
It remains to prove the inequality $bc/(2r + c) \ge 1/2$, which is equivalent to $(2 b - 1) c \ge 2 r$.
We have
\[
(2 b - 1) c \ge 2 (2 b - 1) = 2 (b-1) + 2 b \ge 2 a + 2 b = 2 r,
\]
since $c \ge 2$ and $b > a$.
This shows $\lct_{\msp} (X; \frac{1}{c} H_v) \ge 1/2$.

We next consider the case where $\msp_u \in X$.
In this case $X$ belongs to the family $\mcF_7$ and $X = X_8 \subset \mbP (1, 1, 2, 2, 3)_{x, s, u, v, w}$ with defining polynomial
\[
F = w^2 v + w f_5 (x, s, u) + f_8 (x, s, u, v).
\]
We have $u^4 \notin F$ since $\msp_u \in X$.
We show that $f_5 (x, s, u)$ contains a monomial involving $u$.
Suppose to the contrary that $f_5 = f_5 (x, s)$ is a polynomial in variables $x$ and $s$.
We can write $f_8 = u^3 g_2 + u^2 g_4 + u g_6 + g_8 + v h_6$, where $g_i = g_i (x, s)$ and $h_6 = h_6 (x, s, u, v)$ are quasi-homogeneous polynomials of indicated degree.
Then we have $F = v (w^2 + h_6) + g$, where $g = w f_5 + u^3 g_2 + u^2 g_4 + u g_6 + g_8 \in (x, s)^2$, and we see that $X$ is not quasi-smooth at any point in the nonempty set
\[
(t = w^2 + h_6 = x = s = 0) \subset \mbP (1, 1, 2, 2, 3).
\]
This is a contradiction.
It follows that there is a monomial involving $u$ which appears in $f_5$ with nonzero coefficient.
This implies $\omult_{\msp} (H_v) \le 4$ and we have $\lct_{\msp} (X; \frac{1}{2} H_v) \ge 1/2$ as desired.
\end{proof}

\begin{Claim} \label{clm:ndQIcent-2}
Suppose that $c \ge 2$ and $2 r + c$ is not divisible by $b$.
Then 
\[
\lct_{\msp} \left(X; \frac{1}{c} H_v \right) \ge \frac{c}{5} \ge \frac{4}{5}.
\]
\end{Claim}

\begin{proof}[Proof of Claim \ref{clm:ndQIcent-2}]
By Lemma \ref{lem:uniqQItypes}, we have 
\[
X = X_{5 a + 4} \subset \mbP (1, a, a + 1, a + 2, 2 a + 1)_{x, s, u, v, w}
\] 
with $a \in \{2, 3, 4\}$.
Moreover $\msp = \msp_w$ and the defining polynomial of $X$ can be written as
\[
F = w^2 v + w f_{3 a + 3} (x, s, u) + f_{5 a + 4} (x, s, u, v).
\]
By the quasi-smoothness of $X$ at $\msp_z$, we have either $u^3 \in f_{3 a + 3}$ or $u^4 s \in f_{5 a + 4}$.
This implies $\omult_{\msp} (H_v) \le 5$ and we have
\[
\lct_{\msp} \left(X; \frac{1}{c} H_v \right) \ge \frac{c}{5} = \frac{a+2}{5} \ge \frac{4}{5}.
\]
This proves the claim.
\end{proof}

It remains to consider the case where $c = 1$.
By Lemma \ref{lem:uniqQItypes}, $X$ belongs to a family $\mcF_{\msi}$ with $\msi \in \{24, 46\}$.

\begin{Claim} \label{clm:ndQIcent-3}
Suppose $X$ is a member of the family $\mcF_{24}$.
Then $\lct_{\msp} (X;H_v) \ge 1/2$.
\end{Claim}

\begin{proof}[Proof of Claim \ref{clm:ndQIcent-3}]
We have
\[
X = X_{15} \subset \mbP (1, 2, 5, 1, 7)_{x, s, u, v, w}
\]
and $\msp = \msp_w$ is of type $\frac{1}{7} (1, 2, 5)$.
We can write
\[
F = w^2 v + w f_8 (x, s, u) + f_{15} (x, s, u, v),
\]
where $f_8 = f_8 (x, s, u) \ne 0$ and $f_{15} = f_{15} (x, s, u, v)$ are quasi-homogeneous polynomials of degree $8$ and $15$, respectively.
We have $u^3 \in f_{15}$ and we may assume $\coeff_{f_{15}} (u^3) = 1$.
We set $\bar{F} := F (x, s, u, 0, 1)$.
For a given $\underline{c} = (c_1, c_2, c_3) \in (\mbZ_{ > 0})^3$, we denote by $G_{\underline{c}} \in \mbC [x, s, u]$ the lowest weight part of $\bar{F}$ with respect to the weight $\wt (x, s, u) = \underline{c}$ and let 
\[
\mcD_{\underline{c}} = \mcD^{\wf}_{G_{\underline{c}}}
\]
be the effective $\mbQ$-divisor on $\mbP (\underline{c})^{\wf}$ associated to $G_{\underline{c}}$.
By Lemma \ref{lem:lctwblwh}, we have
\begin{equation} \label{eq:ndBIuniclaim}
\lct_{\msp} (X; H_v) \ge \min \left\{ \frac{c_1 + c_2 + c_3}{\deg G_{\underline{c}}}, \ \lct (\mbP (\underline{c})^{\wf}, \Diff; \mcD_{\underline{c}}) \right\},
\end{equation}
where $\deg G_{\underline{c}}$ is the degree with respect to the weight $\wt (x, s, u) = \underline{c}$.

Suppose $u s x \in f_8$.
In this case, we may assume $\coeff_{f_8} (u s x) = 1$ and we have $G_{\underline{c}} = u s x + u^3$ for $\underline{c} = (1, 1, 1)$.
In this case $\mbP (\underline{c})^{\wf} = \mbP^2$, $\Diff = 0$ and $\mcD_{\underline{c}}$ is the sum of a line and a conic intersection at distinct 2 points.
It is straightforward to check 
\[
\lct (\mbP (\underline{c})^{\wf}, \Diff; \mcD_{\underline{c}}) 
= \lct (\mbP^2; \mcD_{\underline{c}})
= 1
\] and we have $\lct_{\msp} (X; H_v) \ge 1$ in this case.

Suppose that $u s x \notin f_8$ and $s^4 \in f_8$.
In this case, we may assume $\coeff_{f_8} (s^4) = 1$ and $\coeff_{f_{15}} (u s^5) = 0$ by replacing $s$ and $w$.
Hence we have
\[
F (0, s, u, 0, 1) = s^4 + u^3,
\]
and thus $\lct_{\msp} (X;H_v) \ge \lct_{\msp} (H_x, H_v|_{H_x}) = 7/12$, where the equality follows from \cite[8.21 Proposition]{Kol} (or by Lemma \ref{lem:lctwblwh} with $\wt (s, u) = (3, 4)$).

Suppose $u x^3 \in f_8$.
In this case we may assume $\coeff_{f_8} (u s x) = 1$ by rescaling $x$.
We consider a weight $\underline{c} = (2, e, 3)$, where $e$ is a sufficiently large integer which is corprime to $6$.
Then $G_{\underline{c}} = u (x^3 + u^2)$, $\mbP (\underline{c})^{\wf} = \mbP (2, e, 3)$, $\Diff = 0$ and $\mcD_{\underline{c}}$ is the union of $2$ quasi-smooth curves $(u = 0)$ and $(x^3 + u^2 = 0)$.
We have
\[
\lct (\mbP (\underline{c})^{\wf}, \Diff; \mcD_{\underline{c}}) = \lct (\mbP (2, e, 3); \mcD_{\underline{c}}) 
= \lct_{(0:1:0)} (\mbP (2, e, 3); \mcD_{\underline{c}}) = \frac{5}{9},
\] 
and thus $\lct_{\msp} (X; H_v) \ge 5/9$ in this case.

In the following we assume $u s x, s^4, u x^3 \notin f_8$.
We can write
\[
\bar{F} = (\alpha_1 s^3 x^2 + \alpha_2 s^2 x^4 + \alpha_3 s x^6 + \alpha_4 x^8) + (u^3 + \beta u s^5 + \gamma s^7 x + g_{15}),
\]
where $\alpha_1, \dots, \alpha_4, \beta, \gamma \in \mbC$ and $g_{15} = g_{15} (x, s, u) \not\ni u^3$ is a quasi-homogeneous polynomial of degree $15$ which is contained in the ideal $(x, u)^2 \subset \mbC [x, s, u]$.
Note that at least one of $\alpha, \beta, \gamma, \delta$ and $\varepsilon$ is nonzero since $f_8 (x, s, u) \ne 0$.
Note also that $\lambda u s^5$ and $\mu s^7 x$ are the only terms in $\bar{F}$ which is not contained in $(x, u)^2$.
It follows from the quasi-smoothness of $X$ that $(\lambda, \mu) \ne (0,0)$.

Suppose $\beta \ne 0$.
Replacing $u$, we may assume $\beta = 1$ and $\gamma = 0$.
There exists $j \in \{1, 2, 3, 4\}$ such that $\alpha_j \ne 0$ since $f_8 \ne 0$ as a polynomial, and thus we set $i = \min \{\, j \mid \alpha_j \ne 0 \,\} \in \{1, 2, 3, 4\}$.
We may assume $\alpha_i = 1$ by rescaling $x$.
We set $\underline{c} = (2 i + 7, 4 i, 10 i)$.
We have $G_{\underline{c}} = s^{4-i} x^{2 i} + u^3 + u s^5$ for $1 \le i \le 4$.
Moreover we see that
\[
\mbP (\underline{c})^{\wf} =
\begin{cases}
\mbP (2 i + 7, 2, 5)_{\tilde{x}, \tilde{s}, \tilde{u}}, & \text{if $1 \le i \le 3$}, \\
\mbP (3, 2, 1)_{\tilde{x}, \tilde{s}, \tilde{u}}, & \text{if $i = 4$},
\end{cases}
\]
and
\[
(\Diff, \mcD_{\underline{c}}) = 
\begin{cases}
(\frac{2 i -1}{2 i} H_{\tilde{x}}, D_i), & \text{if $1 \le i \le 3$}, \\
(\frac{7}{8} H_{\tilde{x}} + \frac{4}{5} H_{\tilde{s}}, D'), & \text{if $i = 4$},
\end{cases}
\]
where 
\[
D_i = (\tilde{s}^{4-i} \tilde{x} + \tilde{u}^3 + \tilde{u} \tilde{s}^5 = 0), \quad
D' = (\tilde{x} + \tilde{u}^3 + \tilde{u} \tilde{s} = 0)
\]
are prime divisors on $\mbP (\underline{c})^{\wf}$.
We first consider the case where $1 \le i \le 3$.
We see that $H_{\tilde{x}}$ is quasi-smooth, and $D_i$ is quasi-smooth outside $\{\msq\}$, where $\msq = (1\!:\!0\!:\!0) \in \mbP (\underline{c})^{\wf}$.
Moreover they intersect at two points $(0\!:\!1\!:\!0)$ and $(0\!:\!-1\!:\!1)$ transversally.
It follows that $\lct (\mbP (\underline{c})^{\wf}, \Diff; \mcD_{\underline{c}}) = \min \{1, \ \lct_{\msq} (\mbP (\underline{c})^{\wf}, \mcD_{\underline{c}})\}$ since $H_{\tilde{x}}$ does not pass through $\msq$. 
If $i = 3$, then $D_3$ is also quasi-smooth at $\msq$, which implies $\lct (\mbP (\underline{c}), \Diff; \mcD_{\underline{c}}) = 1$.
If $i = 1, 2$, then we have
\[
\begin{split}
\lct_{\msq} (\mbP (\underline{c})^{\wf}, \mcD_{\underline{c}})
&= \lct_{(0,0)} (\mbA^2_{\tilde{s}, \tilde{u}}, (\tilde{s}^{4 - i} + \tilde{u}^3 + \tilde{u} \tilde{s}^5 = 0)) \\
&= \lct_{(0,0)} (\mbA^2_{\tilde{s}, \tilde{u}}, (\tilde{s}^{4-i} + \tilde{u}^3 = 0)) \\
&= \frac{3 + (4-i)}{3 (4-i)} = \frac{7 - i}{3 (4-i)}
\end{split}
\]
since $(\tilde{s}^{4 - i} + \tilde{u}^3 + \tilde{u} \tilde{s}^5 = 0)$ is analytically equivalent to $(\tilde{s}^{4-i} + \tilde{u}^3 = 0)$.
Thus, by Lemma \ref{lem:lctwblwh}, we have
\[
\begin{split}
\lct_{\msp} (X; H_v) 
&\ge \min \left\{ \frac{(2 i + 7) + 4 i + 10 i}{30 i}, \ \lct (\mbP (\underline{c})^{\wf}, \Diff; \mcD_{\underline{c}}) \right\} \\
&= \begin{cases}
\frac{2}{3}, & \text{if $i = 1$}, \\
\frac{13}{20}, & \text{if $i = 2$}, \\
\frac{11}{18}, & \text{if $i = 3$}.
\end{cases}
\end{split}
\]

Suppose $\beta = 0$.
In this case we have $\gamma \ne 0$.
We set $i = \min \{\, j \mid \alpha_j \ne 0 \, \} \in \{1, 2, 3, 4\}$.
We may assume $\gamma = \alpha_i = 1$ by rescaling $x$ and $s$ appropriately.
We set 
\[
\underline{c} =
\begin{cases}
(3 (i + 3), 3 (2 i - 1), 15 i - 4), & \text{if $1 \le i \le 3$}, \\
(3, 3, 8), & \text{if $i = 4$}.
\end{cases}
\]
We have $G_{\underline{c}} = s^{4-i} x^{2 i} + u^3 + s^7 x$ for $1 \le i \le 4$.
Moreover we see that
\[
\mbP (\underline{c})^{\wf} = 
\begin{cases}
\mbP (i+3, 2 i - 1, 15 i - 4)_{\tilde{x}, \tilde{s}, \tilde{u}}, & \text{if $1 \le i \le 3$}, \\
\mbP (1, 1, 8)_{\tilde{x}, \tilde{s}, \tilde{u}}, & \text{if $i = 4$},
\end{cases}
\]
and
\[
\Diff = \frac{2}{3} H_{\tilde{u}}, \quad
\mcD_{\underline{c}} = (\tilde{s}^{4-i} \tilde{x}^{2 i} + \tilde{u} + \tilde{s}^7 \tilde{x} = 0).
\]
We see that $H_{\tilde{u}}$ and $\mcD_{\underline{c}}$ are both quasi-smooth.
If $i = 3, 4$, then $H_{\tilde{u}}$ and $\mcD_{\underline{c}}$ intersect transversally and thus we have $\lct (\mbP (\underline{c})^{\wf}, \Diff; \mcD_{\underline{c}}) = 1$.
Suppose $i = 1, 2$.
Then $H_{\tilde{u}}$ and $\mcD_{\underline{c}}$ intersect transversally except at $\msp_{\tilde{x}} = (1\!:\!0\!:\!0) \in \mbP (\underline{c})^{\wf}$, and we have
\[
\begin{split}
\lct (\mbP (\underline{c})^{\wf}, \Diff; \mcD_{\underline{c}}) 
&= \lct_{\msp_{\tilde{x}}} (\mbP (\underline{c})^{\wf}, \Diff; \mcD_{\underline{c}}) \\
&= \lct (\mbA^2_{\tilde{s}, \tilde{u}}, \tfrac{2}{3} (\tilde{u} = 0); (\tilde{s}^{4-i} + \tilde{u} + \tilde{s}^7 = 0)) \\
&= \lct (\mbA^2_{\tilde{s}, \tilde{u}}, \tfrac{2}{3} (\tilde{u} = 0); (\tilde{s}^{4-i} + \tilde{u} = 0)) \\
&=
\begin{cases}
\frac{2}{3}, & \text{if $i = 1$}, \\
1, & \text{if $i = 2$}.
\end{cases}
\end{split}
\]
Thus, by Lemma \ref{lem:lctwblwh}, we have
\[
\lct_{\msp} (X; H_v) \ge
\begin{cases}
\frac{2}{3}, & \text{if $i = 1$}, \\
\frac{25}{39}, & \text{if $i = 2$}, \\
\frac{74}{123}, & \text{if $i = 3$}, \\
\frac{7}{12}, & \text{if $i = 4$}.
\end{cases}
\]
This proves the claim.
\end{proof}

\begin{Claim} \label{clm:ndQIcent-4}
Suppose $X$ is a member of the family $\mcF_{46}$.
Then $\lct_{\msp} (X;H_v) \ge 1/2$.
\end{Claim}

\begin{proof}[Proof of Claim \ref{clm:ndQIcent-4}]
We have 
\[
X = X_{21} \subset \mbP (1, 3, 7, 1, 10)_{x, s, u, v, w}
\]
and $\msp = \msp_w$ is of type $\frac{1}{10} (1, 3, 7)$.
We can write
\[
F = w^2 v + w f_{11} (x, s, u) + f_{21} (x, s, u, v),
\]
where $f_{11} = f_{11} (x, s, u) \ne 0$ and $f_{21} = f_{21} (x, s, u, v)$ are quasi-homogeneous polynomials of degree $11$ and $21$, respectively.
We have $u^3, s^7 \in F$ by the quasi-smoothness of $X$, and we may assume $\coeff_F (u^3) = \coeff_F (s^7) = 1$.
We set $\bar{F} = F (x, s, u, 0, 1)$, which can be written as
\[
\bar{F} = (\alpha u s x + \beta s^3 x^2 + \gamma u x^4 + \delta s^2 x^5 + \varepsilon s x^8 + \zeta x^{11}) + (u^3 + s^7),
\]
where $\alpha, \beta, \dots, \zeta \in \mbC$.
We introduce various $3$-tuples $(\underline{c} = (c_1, c_2, c_3)$ of positive integers according to the following division into cases.
We denote by $G_{\underline{c}}$ the lowest weight part of $\bar{F}$ with respect to $\wt (x, s, u) = \underline{c}$.
\begin{itemize}
\item[(i)] $\alpha \ne 0$.
In this case we may assume $\alpha = 1$.
We set $\underline{c} = (1, 1, 1)$.
Then we have $G_{\underline{c}} = u s x + u^3$.
\item[(ii)] $\alpha = 0$ and $\beta \ne 0$.
In this case we may assume $\beta = 1$.
We set $\underline{c} = (6, 3, 7)$.
Then we have $G_{\underline{c}} = s^3 x^2 + u^3 + s^7$.
\item[(iii)] $\alpha = \beta = 0$ and $\gamma \ne 0$.
In this case we may assume $\gamma = 1$.
We set $\underline{c} = (7, 6, 14)$.
Then we have $G_{\underline{c}} = u x^4 + u^3 + s^7$.
\item[(iv)] $\alpha = \beta = \gamma = 0$ and $\delta \ne 0$.
In this case we may assume $\delta = 1$.
We set $\underline{c} = (3, 3, 7)$.
In this case we have $G_{\underline{c}} = s^2 x^5 + u^3 + s^7$.
\item[(v)] $\alpha = \beta = \gamma = \delta = 0$ and $\varepsilon \ne 0$.
In this case we may assume $\varepsilon = 1$.
We set $\underline{c} = (9, 12, 28)$.
In this case we have $G_{\underline{c}} = s x^8 + u^3 + s^7$.
\item[(vi)] $\alpha = \beta = \gamma = \delta = \varepsilon = 0$.
In this case we may assume $\delta = 1$.
We set $\underline{c} = (21, 33, 77)$.
In this case we have $x^{11} + u^3 + s^7$.
\end{itemize}
The descriptions of $\mbP (\underline{c})^{\wf}$, $\Diff$ and $G_{\underline{c}} ^{\wf}$ are given in Table \ref{table:No46lwp},  where we choose $\tilde{x}, \tilde{s}, \tilde{u}$ as homogeneous coordinates of $\mbP (\underline{c})^{\wf}$.

\begin{table}[h]
\renewcommand{\arraystretch}{1.15}
\begin{center}
\caption{Family $\mcF_{46}$: Weights and LCT}
\label{table:No46lwp}
\begin{tabular}{ccccccc}
Case & $\mbP (\underline{c})^{\wf}$ & $\Diff$ & $G_{\underline{c}}^{\wf}$ & $\eta$ & $\theta$ \\
\hline
(i) & $\mbP (1, 1, 1)$ & $0$ & $\tilde{u} (\tilde{s} \tilde{x} + \tilde{u}^2)$ & $1$ & $1$  \\
(ii) & $\mbP (2, 1, 7)$ & $\frac{2}{3} H_{\tilde{u}}$ & $\tilde{s}^3 \tilde{x}^2 + \tilde{u} + \tilde{s}^7$ & $2/3$ & $2/3$ \\
(iii) & $\mbP (1, 3, 1)$ & $\frac{1}{2} H_{\tilde{x}} + \frac{6}{7} H_{\tilde{s}}$ & $\tilde{u} \tilde{x}^2 + \tilde{u}^3 + \tilde{s}$ & $1$ & $9/14$ \\
(iv) & $\mbP (1, 1, 7)$ & $\frac{2}{3} H_{\tilde{u}}$ & $\tilde{s}^2 \tilde{x}^5 + \tilde{u} + \tilde{s}^7$ & $5/6$ & $13/21$ \\
(v) & $\mbP (3, 1, 7)$ & $\frac{3}{4} H_{\tilde{x}} + \frac{2}{3} H_{\tilde{u}}$ & $\tilde{s} \tilde{x}^2 + \tilde{u} + \tilde{s}^7$ & $1$ & $7/12$ \\
(vi) & $\mbP (1, 1, 1)$ & $\frac{10}{11} H_{\tilde{x}} + \frac{6}{7} H_{\tilde{s}} + \frac{2}{3} H_{\tilde{u}}$ & $\tilde{x} + \tilde{u} + \tilde{s}$ & $1$ & $131/231$
\end{tabular}
\end{center}
\end{table}

We set $\mcD_{\underline{c}} = \mcD^{\wf}_{G_{\underline{c}}}$.
We explain the computation of $\eta := \lct (\mbP (\underline{c})^{\wf}, \Diff; \mcD_{\underline{c}})$ whose value is given in the 5th column of Table \ref{table:No46lwp}.
The computation $\eta = 1$ is straightforward when we are in case (i) since $\mcD_{\underline{c}}$ is the union of of a line and a conic on $\mbP^2$ intersecting at $2$ points.
In the other cases, $\mcD_{\underline{c}}$ is the divisor defined by $G_{\underline{c}}^{\wf} = 0$ which is a quasi-line in $\mbP (\underline{c})^{\wf}$.
If we are in one of the cases (iii), (v) and (vi), then any 2 of the components of $\Diff + \mcD_{\underline{c}}$ intersect transversally, which implies $\eta = 1$. 
If we are in case (ii) or (iv), then $H_{\tilde{u}}$ and $\mcD_{\underline{c}}$ intersect transversally except at $\msq = (1\!:\!0\!:\!0) \in \mbP (\underline{c})^{\wf}$.
We set $e = 3, 2$ if we are in case (ii), (iv), respectively.
Then we have
\[
\begin{split}
\lct (\mbP (\underline{c})^{\wf}, \Diff; \mcD_{\underline{c}}) 
&= \lct_{\msq} (\mbP (\underline{c})^{\wf}, \Diff; \mcD_{\underline{c}}) \\
&= \lct_{(0,0)} (\mbA^2_{\tilde{s}, \tilde{u}}, \tfrac{2}{3} (\tilde{u} = 0); (\tilde{s}^e + \tilde{u} + \tilde{s}^7 = 0)) \\
&= \lct_{(0,0)} (\mbA^2_{\tilde{s}, \tilde{u}}, \tfrac{2}{3} (\tilde{u} = 0); (\tilde{s}^e + \tilde{u} = 0)).
\end{split}
\]
This completes the explanations of the computations of $\eta$.
We set 
\[
\theta = \left\{\frac{c_1 + c_2 + c_3}{\deg_{\underline{c}} (G_{\underline{c}}^{\wf})}, \ \eta \right\}
\]
which is described in the 6th column of Table \ref{table:No46lwp}.
By Lemma \ref{lem:lctwblwh}, we have $\lct_{\msp} (X; H_v) \ge \theta \ge 1/2$ and the claim is proved.  
\end{proof}

By Claims \ref{clm:ndQIcent-1}, \ref{clm:ndQIcent-2}, \ref{clm:ndQIcent-3} and \ref{clm:ndQIcent-4}, we have $\lct_{\msp} (X; \frac{1}{c} H_v) \ge 1/2$.
Let $D \in |A|_{\mbQ}$ be an irreducible $\mbQ$-divisor other than $\frac{1}{c} H_v$.
We set $\lambda = (r + c)/(2 r + c)$ and we will show that $\lct_{\msp} (X; D) \ge \lambda$.
Suppose not, that is, $(X, \lambda D)$ is not log canonical at $\msp$.
Let $\varphi \colon Y \to X$ be the Kawamata blowup of $\msp \in X$.
Then, for the proper transforms $\tilde{H}_t$ and $\tilde{D}$ of $H_t$ and $D$, respectively, we have
\[
\begin{split}
\tilde{H}_t &\sim c \varphi^*A - \frac{r + c}{r} E, \\
\tilde{D} & \sim_{\mbQ} \varphi^*A - \frac{e}{r} E,
\end{split}
\]
where $e \in \mbQ_{\ge 0}$.
By \cite{Kawamata}, the discrepancy of the pair $(X, \lambda D)$ along $E$ is negative, and thus we have $e > 1/\lambda$.
By \cite[Theorem 4.9]{CPR00}, $-K_Y \sim_{\mbQ} \varphi^* A - \frac{1}{r} E$ is nef (more precisely, $- m K_Y$ defines the flopping contraction for a sufficiently divisible $m > 0$).
Hence $(-K_Y \cdot \tilde{H}_t \cdot \tilde{D}) \ge 0$ and we have
\[
\begin{split}
0 \le (-K_Y \cdot \tilde{H}_t \cdot \tilde{D}) &= c (A^3) - \frac{e (r+c)}{r^3} (E^3) \\
&= \frac{2 r + c}{a b r} - \frac{e (r+c)}{a b r} < \frac{2 r + c}{a b r} - \frac{r+c}{\lambda a b r} = 0.
\end{split}
\]
This is a contradiction.
Therefore $\lct_{\msp} (X;D) \ge \lambda$ and thus
\[
\alpha_{\msp} (X) \ge \min \left\{ \lct_{\msp} \left(X; \frac{1}{c} H_v \right), \ \frac{r + c}{2 r + c} \right\} \ge \frac{1}{2}.
\]
This completes the proof of Proposition \ref{prop:ndQIcent} when $X$ has a unique QI center.

\subsubsection{Case: $X$ has exactly $3$ distinct QI centers}
\label{sec:ndQI3ditinct}

By Lemma \ref{lem:QIeqtypes}, we can choose homogeneous coordinates so that
\[
X = X_{3 r} \subset \mbP (1, a, b, r, r)_{x, y, z, t, w},
\]
where $a$ is coprime to $b$, $a \le b$ and $a + b = r$.
Let $\msp \in X$ be a QI center.
Then we may assume $\msp = \msp_w$ by replacing $t$ and $w$ suitably.
Then the defining polynomial $F$ of $X$ can be written as
\[
F = w^2 t + w f_{2 r} + f_{3r},
\]
where $f_{2 r} (x, y, z)$ and $f_{3 r} (x, y, z, t)$ are quasi-homogeneous polynomials of degrees $2 r$ and $3 r$, respectively.
We have $(A^3) = 3 r/a b r^2 = 3 / a b r$.
By Lemma \ref{lem:complctsingtang}, we have
\[
\alpha_{\msp} (X) \ge \frac{2}{r a b (A^3)} = \frac{2}{3},
\]
and Proposition \ref{prop:ndQIcent} is proved when $X$ has exactly $3$ distinct QI centers.

\subsubsection{Case: $X$ has exactly $2$ distinct QI centers and their singularity types are equal}
\label{sec:ndQI2eqtypes}

By Lemma \ref{lem:QIeqtypes}, we can choose homogeneous coordinates so that 
\[
X = X_{4 r} \subset \mbP (1, a, b, r, 2 r)_{x, y, z, t, w},
\]
where $a$ is coprime to $b$, $a \le b$ and $a + b = r$.
Let $\msp \in X$ be a QI center.
We may assume $\msp = \msp_t$ by replacing $w$ suitably.
Then the defining polynomial $F$ of $X$ can be written as
\[
F = t^2 w + t f_{3 r} + f_{4 r},
\]
where $f_{3 r} (x, y, z)$ and $f_{4 r} (x, y, z, w)$ are quasi-homogeneous polynomials of degrees $3 r$ and $4 r$, respectively.
Note that $(A^3) = 4 r/2 a b r^2 = 2/a b r$.
By Lemma \ref{lem:complctsingtang}, we have
\[
\alpha_{\msp} (X) \ge \frac{2}{r a b (A^3)} = 1,
\]
and Proposition \ref{prop:ndQIcent} is proved in this case.

\subsubsection{Case: $X$ has exactly $2$ distinct QI centers and their singularity types are distinct}

By Lemma \ref{lem:QIeqtypes}, we have
\[
X = X_{4 a + 3 b} \subset \mbP (1, a, b, r_1, r_2)_{x, u, v, t, w},
\]
where $a$ is coprime to $b$, $r_1 = a + b$ and $r_2 = 2 a + b$. 

We first consider the QI center $\msp = \msp_t \in X$ of type $\frac{1}{r_1} (1, a, b)$.
The defining polynomial $F$ of $X$ can be written as
\[
F = t^2 w + t f_{3 a + 2 b} + f_{4 a + 3 b},
\]
where $f_{3 a + 2 b} = f_{3 a + 2 b} (x, u, v)$ and $f_{4 a + 3 b} = f_{4 a + 3 b} (x, u, v, w)$ are quasi-homogeneous polynomials of the indicated degrees.
Note that $(A^3) = (4 a + 3 b)/a b r_1 r_2$.
By Lemma \ref{lem:QIeqtypes}, we have
\[
\alpha_{\msp} (X) \ge \frac{2}{r_1 a b (A^3)} = \frac{2 r_2}{4 a + 3b} = \frac{4 a + 2 b}{4 a + 3 b} > \frac{2}{3}.
\]

We next consider the QI center $\msp = \msp_w \in X$ of type $\frac{1}{r_2} (1, a, a + b)$.
Then the defining polynomial $F$ of $X$ can be written as
\[
F = w^2 v + w f_{2 a + 2 b} + f_{4 a + 3 b},
\]
where $f_{2 a + 2 b} = f_{2 a + 2 b} (x, u, t)$ and $f_{4 a + 3 b} = f_{4 a + 3 b} (x, u, v, t)$ are quasi-homogeneous polynomials of the indicated degree.

Suppose $t^2 w \in F$, that is, $t^2 \in f_{2 a + 2 b}$.
Then $\omult_{\msp} (H_v) = 2$ and we have $\lct_{\msp} (X; \frac{1}{b} H_v) \ge b/2 \ge 1/2$.
Let $D \in |A|_{\mbQ}$ be an irreducible $\mbQ$-divisor other than $\frac{1}{b} H_v$.
We see that the set $\{x, u, v\}$ isolates $\msp$ since $t^2 w \in F$.
In particular, a general member $T \in |a A|$ does not contain any component of the effective $1$-cycle $D \cdot H_z$.
Then we have
\[
\begin{split}
2 \omult_{\msp} (D) 
&\le (\rho_{\msp}^*D \cdot \rho_{\msp}^*H_v \cdot \rho_{\msp}^*T)_{\check{\msp}} \le r_2 (D \cdot H_v \cdot T) \\
&= r_2 b a (A^3) = \frac{4 a + 3 b}{a + b}.
\end{split}
\]
This implies 
\[
\lct_{\msp} (X;D) \ge \frac{2 a + 2 b}{4 a + 3 b} > \frac{1}{2}.
\]
Therefore we have $\alpha_{\msp} (X) \ge 1/2$.

In the following, we consider the case where $t^2 w \notin F$.

\begin{Claim} \label{clm:ndQIcent-5}
If $b \ge 2$, then $\lct_{\msp} (X; \frac{1}{b} H_v) \ge 1/2$.
\end{Claim}

\begin{proof}[Proof of Claim \ref{clm:ndQIcent-5}]
By the quasi-smoothness of $X$ at $\msp_t$, we have $t^3 u \in f_{4 a + 3 b}$ since $t^2 w \notin F$ by assumption.
Hence we have $\omult_{\msp} (H_v) \le 4$ and this shows $\lct_{\msp} (X; \frac{1}{b} H_v) \ge 1/2$.
\end{proof}

If $b = 1$, then $X$ is a member of a family $\mcF_{\msi}$ with $\msi \in \{13, 25\}$.

\begin{Claim} \label{clm:ndQIcent-6}
If $X$ is a member of the family $\mcF_{13}$, then $\lct_{\msp} (X;H_v) \ge 1/2$.
\end{Claim}

\begin{proof}[Proof of Claim \ref{clm:ndQIcent-6}]
We have
\[
F = w^2 v + w f_6 (x, u, v, t) + f_{11} (x, u, v, t).
\]
Note that $f_6 (x, u, 0, t) \ne 0$ as a polynomial since $\msp \in X$ is non-degenerate.
We set $\bar{F} = F (x, u, 0, t, 1) \in \mbC [x, u, t]$.
We have $t^3 u \in F$ and we may assume $\coeff_F (t^3 u) = 1$.
If $t u x \in f_6$, then the cubic part of $\bar{F}$ is not a cube of a linear form and thus we have $\lct_{\msp} (X; H_v) \ge 1/2$ by Lemma \ref{lem:lcttangcube}.
In the following, we assume $t u x \notin f_6$.
Then we can write
\[
\bar{F} = (\alpha u^3 + \beta u^2 x^2 + \gamma t x^3 + \delta u x^4 + \varepsilon x^6) + (t^3 u + \lambda t u^4 + x g_{10}),
\]
where $\alpha, \beta, \dots, \varepsilon, \lambda \in \mbC$ and $g_{10} = g_{10} (x, u, t)$ is a quasi-homogeneous polynomial of degree $10$.
We introduce $3$-tuples $\underline{c} = (c_1, c_2, c_3)$ of positive integers according to the following division into cases.
We denote by $G_{\underline{c}}$ the lowest weight part of $\bar{F}$ with respect to $\wt (x, u, t) = \underline{c}$.

\begin{enumerate}
\item[(i)] $\alpha \ne 0$.
In this case we may assume $\alpha = 1$.
We choose and fix a sufficiently large integer $e$ which is coprime to $2$ and $3$, and we set $\underline{c} = (e, 3, 2)$.
Then we have $G_{\underline{c}} = u^3 + t^3 u$.
\item[(ii)] $\alpha = 0$ and $\beta \ne 0$.
In this case we may assume $\beta = 1$.
We set $\underline{c} = (7, 4, 6)$.
Then we have $G_{\underline{c}} = u^2 x^2 + t^3 u + \lambda t u^4$.
\item[(iii)] $\alpha = \beta = 0$ and $\gamma \ne 0$.
In this case we may assume $\gamma = 1$.
We set $\underline{c} = (8, 6, 9)$.
Then $G_{\underline{c}} = t x^3 + t^3 u + \lambda t u^4$.
\item[(iv)] $\alpha = \beta = \gamma = 0$ and $\delta \ne 0$.
In this case we may assume $\delta = 1$.
We set $\underline{c} = (9, 8, 12)$.
Then $G_{\underline{c}} = u x^4 + t^3 u + \lambda t u^4$.
\item[(v)] $\alpha = \beta = \gamma = \delta = 0$.
In this case we may assume $\varepsilon = 1$.
We set $\underline{c} = (11, 12, 18)$.
Then we have $x^6 + t^3 u + \lambda t u^4$.
\end{enumerate}

\begin{table}[h]
\renewcommand{\arraystretch}{1.15}
\begin{center}
\caption{Family $\mcF_{13}$: Weights and LCT}
\label{table:No13lwp}
\begin{tabular}{ccccccc}
Case & $\mbP (\underline{c})^{\wf}$ & $\Diff$ & $G_{\underline{c}}^{\wf}$ & $\eta$ & $\theta$ \\
\hline
(i) & $\mbP (e, 3, 2)$ & $0$ & $\tilde{u} (\tilde{u}^2 + \tilde{t}^3)$ & $5/9$  & $5/9$  \\
(ii) & $\mbP (7, 2, 3)$ & $\frac{1}{2} H_{\tilde{x}}$ & $\tilde{u} (\tilde{u} \tilde{x} + \tilde{t}^3 + \lambda \tilde{t} \tilde{u}^3)$ & $2/3$ & $2/3$ \\
(iii) & $\mbP (4, 1, 3)$ & $\frac{2}{3} H_{\tilde{x}} + \frac{1}{2} H_{\tilde{t}}$ & $\tilde{t}^{1/2} (\tilde{x} + \tilde{t} \tilde{u} + \lambda \tilde{u}^4)$ & $\ge 5/9$ & $\ge 5/9$  \\
(iv) & $\mbP (3, 2, 1)$ & $\frac{3}{4} H_{\tilde{x}} + \frac{2}{3} H_{\tilde{u}}$ & $\tilde{u}^{1/3} (\tilde{x} + \tilde{t}^3 + \lambda \tilde{t} \tilde{u})$ & $\ge 7/12$  & $\ge 7/12$  \\
(v) & $\mbP (11, 2, 3)$ & $\frac{5}{6} H_{\tilde{x}}$ & $\tilde{x} + \tilde{t}^3 \tilde{u} + \lambda \tilde{t} \tilde{u}^4$ & $\ge 1/2$ & $\ge 1/2$ 
\end{tabular}
\end{center}
\end{table}
The descriptions of $\mbP (\underline{c})^{\wf}$, $\Diff$ and $G_{\underline{c}}^{\wf}$ are given in Table \ref{table:No13lwp}, where 
we choose $\tilde{x}, \tilde{u}, \tilde{t}$ as homogeneous coordinates of $\mbP (\underline{c})^{\wf}$.

We set $\mcD_{\underline{c}} = \mcD_{G_{\underline{c}}}^{\wf}$.
We explain the computation of $\eta := \lct (\mbP (\underline{c})^{\wf}, \Diff; \mcD_{\underline{c}})$ whose value (or lower bound) is given in the 5th column of Table \ref{table:No13lwp}.
\end{proof}

\begin{Claim} \label{clm:ndQIcent-7}
If $X$ is a member of the family $\mcF_{25}$, then $\lct_{\msp} (X;H_v) \ge 1/2$.
\end{Claim}

\begin{proof}[Proof of Claim \ref{clm:ndQIcent-7}]
We have
\[
F = w^2 v + w f_8 (x, u, v, t) + f_{15} (x, u, v, t).
\]
Note that $f_8 (x, u, 0, t) \ne 0$ as a polynomial since $\msp \in X$ is non-degenerate.
We set $\bar{F} = F (x, u, 0, t, 1) \in \mbC [x, u, t]$.
We have $t^3 u, u^5 \in F$ and we may assume $\coeff_F (t^3 u) = \coeff_F (u^5) = 1$.
If $t u x \in f_8$, then the cubic part of $\bar{F}$ is not a cube of a linear form and thus we have $\lct_{\msp} (X; H_v) \ge 1/2$ by Lemma \ref{lem:lcttangcube}.
In the following, we assume $t u x \notin f_8$.
Then we can write
\[
\bar{F} = (\alpha u^2 x^2 + \beta t x^4 + \gamma u x^5 + \delta x^8) + (t^3 u + u^5 + x g_{14}),
\]
where $\alpha, \beta, \gamma, \delta \in \mbC$ and $g_{14} = g_{14} (x, u, t)$ is a quasi-homogeneous polynomial of degree $14$.
We introduce $3$-tuples $\underline{c} = (c_1, c_2, c_3)$ of positive integers according to the following division into cases.
We denote by $G_{\underline{c}}$ the lowest weight part of $\bar{F}$ with respect to $\wt (x, u, t) = \underline{c}$.

\begin{enumerate}
\item[(i)] $\alpha \ne 0$.
In this case we may assume $\alpha = 1$.
We set $\underline{c} = (9, 6, 8)$.
Then we have $G_{\underline{c}} = u^2 x^2 + t^3 u + u^5$.
\item[(ii)] $\alpha = 0$ and $\beta \ne 0$.
In this case we may assume $\beta = 1$.
We set $\underline{c} = (11, 12, 16)$.
Then we have $G_{\underline{c}} = t x^4 + t^3 u + u^5$.
\item[(iii)] $\alpha = \beta = 0$ and $\gamma \ne 0$.
In this case we may assume $\gamma = 1$.
We set $\underline{c} = (12, 15, 20)$.
Then $G_{\underline{c}} = u x^5 + t^3 u + u^5$.
\item[(iv)] $\alpha = \beta = \gamma = 0$ and $\delta \ne 0$.
In this case we may assume $\delta = 1$.
We set $\underline{c} = (15, 24, 32)$.
Then $G_{\underline{c}} = x^8 + t^3 u + u^5$.
\end{enumerate}

\begin{table}[h]
\renewcommand{\arraystretch}{1.15}
\begin{center}
\caption{Family $\mcF_{25}$: Weights and LCT}
\label{table:No25lwp}
\begin{tabular}{ccccccc}
Case & $\mbP (\underline{c})^{\wf}$ & $\Diff$ & $G_{\underline{c}}^{\wf}$ & $\eta$ & $\theta$ \\
\hline
(i) & $\mbP (3, 1, 4)$ & $\frac{1}{2} H_{\tilde{x}} + \frac{2}{3} H_{\tilde{t}}$ & $\tilde{u} (\tilde{u} \tilde{x} + \tilde{t} + \tilde{u}^4)$ & $1$ & $23/30$  \\
(ii) & $\mbP (11, 3, 4)$ & $\frac{3}{4} H_{\tilde{x}}$ & $\tilde{t} \tilde{x} + \tilde{t}^3 \tilde{u} + \tilde{u}^5$ & $1$ & $13/20$ \\
(iii) & $\mbP (1, 1, 1)$ & $\frac{4}{5} H_{\tilde{x}} + \frac{3}{4} H_{\tilde{u}} + \frac{2}{3} H_{\tilde{t}}$ & $\tilde{u}^{1/4} (\tilde{x} + \tilde{t} + \tilde{u})$ & $1$ & $47/75$  \\
(iv) & $\mbP (5, 1, 4)$ & $\frac{7}{8} H_{\tilde{x}} + \frac{3}{4} H_{\tilde{t}}$ & $\tilde{x} + \tilde{t} \tilde{u} + \tilde{u}^5$ & $1$ & $71/120$  \\
\end{tabular}
\end{center}
\end{table}
The descriptions of $\mbP (\underline{c})^{\wf}$, $\Diff$ and $G_{\underline{c}}^{\wf}$ are given in Table \ref{table:No25lwp}, where 
we choose $\tilde{x}, \tilde{u}, \tilde{t}$ as homogeneous coordinates of $\mbP (\underline{c})^{\wf}$.

We set $\mcD_{\underline{c}} = \mcD_{G_{\underline{c}}}^{\wf}$.
We explain the computation of $\eta := \lct (\mbP (\underline{c})^{\wf}, \Diff; \mcD_{\underline{c}})$ whose value is given in the 5th column of Table \ref{table:No25lwp}.
Suppose that we are in case (ii) or (iv).
Then $\mcD_{\underline{c}}$ is a prime divisor which is quasi-smooth and intersects any component of $\Diff$ transversally.
This shows $\eta = 1$.
Suppose that we are in case (i).
Then $\mcD_{\underline{c}} = H_{\tilde{u}} + \Gamma$, where $\Gamma = (\tilde{u} \tilde{x} + \tilde{t} + \tilde{u}^4 = 0)$ is a quasi-line. 
We see that any two of $H_{\tilde{x}}, H_{\tilde{u}}, H_{\tilde{t}}, \Gamma$ intersect transversally, and thus $\eta = 1$.
Suppose that we are in case (iii).
Then $\mcD_{\underline{c}} = \frac{1}{4} H_{\tilde{u}} + \Gamma$, where $\Gamma = (\tilde{x} + \tilde{t} \tilde{u} + \tilde{u}^5 = 0)$ is a quasi-line.
We see that any two of $H_{\tilde{x}}, H_{\tilde{u}}, H_{\tilde{t}}$ and $\Gamma$ intersect transversally, and thus $\eta = 1$.

We set
\[
\theta := \min \left\{\ \frac{c_1 + c_2 + c_3}{\wt_{\underline{c}} (\bar{F})}, \eta \right\},
\]
which is listed in the 6th column of Table \ref{table:No25lwp}.
By Lemma \ref{lem:lctwblwh}, we have $\lct_{\msp} (X; H_v) \ge \theta \ge 1/2$ and the claim is proved.
\end{proof}

By Claims \ref{clm:ndQIcent-5}, \ref{clm:ndQIcent-6} and \ref{clm:ndQIcent-7}, we have $\lct_{\msp} (X; \frac{1}{b} H_v) \ge 1/2$.
Suppose $\alpha_{\msp} (X) < 1/2$.
Then there exists an irreducible $\mbQ$-divisor $D \in |A|_{\mbQ}$ other than $\frac{1}{b} H_v$ such that $(X, \frac{1}{2} D)$ is not log canonical at $\msp$.
Let $\varphi \colon Y \to X$ be the Kawamata blowup at $\msp$ with exceptional divisor $E$.
We set $\lambda = \ord_E (D)$.
Since the pair $(X, \frac{1}{2} D)$ is not canonical at $\msp$, the discrepancy of $(X, \frac{1}{2} D)$ along $E$ is negative, which implies
\[
\lambda > \frac{2}{r_2}.
\]
By \cite[Theorem 4.9]{CPR00}, the divisor $- K_Y \sim_{\mbQ} \varphi^*A - \frac{1}{r_2} E$ is nef.
We see that $\tilde{D} \cdot \tilde{H}_v$ is an effective $1$-cycle on $Y$, where $\tilde{D}$ and $\tilde{H}_v$ are proper transforms of $D$ and $H_v$, respectively.
It follows that
\[
\begin{split}
0 &\le (- K_Y \cdot \tilde{D} \cdot \tilde{H}_v) 
= b (A^3) - \frac{(2 a + 2 b) \lambda}{r_2^2} (E^3) \\
&= \frac{(4 a + 3 b) - (2 a + 2 b) r_2 \lambda}{a r_1 r_2} 
< - \frac{b}{a r_1 r_2} < 0.
\end{split}
\]
This is a contradiction and we have $\alpha_{\msp} (X) \ge 1/2$.
Therefore, the proof of Proposition \ref{prop:ndQIcent} is completed.

\section{Families $\mcF_2$, $\mcF_{4}$, $\mcF_5$, $\mcF_6$, $\mcF_8$, $\mcF_{10}$ and $\mcF_{14}$}
\label{chap:exc}

This section is devoted to the proof of the following theorem.

\begin{Thm} \label{thm:famI1}
Let $X$ be a member of a family $\mcF_{\msi}$ with $\msi \in \msI_1$.
Then,  
\[
\alpha (X) \ge \frac{1}{2}.
\]
\end{Thm}

\subsection{Families $\mcF_6$, $\mcF_{10}$ and $\mcF_{14}$}

In this section, we prove Theorem \ref{thm:famI1} for families $\mcF_6$, $\mcF_{10}$ and $\mcF_{14}$ whose member is a weighted hypersurface 
\[
X = X_{2 (a+2)} \subset \mbP (1, 1, 1, a, a+ 2)_{x, y, z, t, w},
\] 
where $a = 2, 3, 4$, respectively.
Let $X$ be a member of a family $\mcF_{\msi}$ with $\msi \in \{6, 10, 14\}$.

Let $\msp \in X$ be a smooth point.
We may assume $\msp = \msp_x$ by a suitable choice of coordinates.
By Lemma \ref{lem:nsptU1-2} (see also Remark \ref{rem:lemsmpttrue}), we have
\[
\alpha_{\msp} (X) \ge \frac{1}{1 \cdot a \cdot (A^3)} = \frac{1}{2}.
\]

Let $\msp \in X$ be a singular point.
If $\msi = 14$, then $\msp \in X$ is of type $\frac{1}{2} (1, 1, 1)$ and we have $\alpha_{\msp} (X) \ge 1$ by Proposition \ref{prop:singptCP}.
If $\msi = 6, 10$, then $\msp \in X$ is of type $\frac{1}{2} (1, 1, 1)$, $\frac{1}{3} (1, 1, 2)$, respectively, and in both cases we have $\alpha_{\msp} (X) \ge 1/2$ by Proposition \ref{prop:singptic}.
Thus the proof of Theorem \ref{thm:famI1} for families $\mcF_6$, $\mcF_{10}$ and $\mcF_{14}$ is completed.

\subsection{The family $\mcF_2$}

This section is devoted to the proof Theorem \ref{thm:famI1} for the family $\mcF_2$.
In the following, let 
\[
X = X_5 \subset \mbP (1,1,1,1,2)_{x, y, z, t, w}
\] 
be a member of $\mcF_2$ with defining polynomial $F = F (x,y,z,t,w)$.

\subsubsection{Smooth points}

Let $\msp \in X$ be a smooth point.
In this subsection, we will prove $\alpha_{\msp} (X) \ge 1/2$.
We may assume $\msp = \msp_x$ by a choice of coordinates.
The proof will be done by division into cases.

\paragraph{\it Case: $x^3 w \in F$} 
In this case, we can write
\[
F = x^3 w + x^2 f_3 + x f_4 + f_5,
\]
where $f_i = f_i (y,z,t,w)$ is a quasi-homogeneous polynomial of degree $i$.
We have $\mult_{\msp} (H_w) \ge 3$.

\begin{Claim} \label{cl:No2-1}
$\lct_{\msp} (X; \frac{1}{2} H_w) \ge 1/2$.
\end{Claim}

\begin{proof}[Proof of Claim \ref{cl:No2-1}]
This is obvious when $\mult_{\msp} (H_w) \le 4$, hence we assume $\mult_{\msp} (H_w) \ge 5$.
Then we can write
\[
F = x^3 w + x^2 w a_1 + x (\alpha w^2 + w b_2) + w^2 c_1 + w d_3 + e_5,
\]
where $\alpha \in \mbC$ and $a_1, b_2, c_1, d_3, e_5 \in \mbC [y,z,t]$ are quasi-homogeneous polynomials of indicated degrees.
We show that $(e_5 = 0) \subset \mbP (1,1,1)_{y, z, t}$ is smooth.
Indeed, if it has a singular point at $(y\!:\!z\!:\!t) = (\lambda\!:\!\mu\!:\!\nu)$, then, by setting $\theta \in \mbC$ to be a solution of the equation
\[
x^3 + x^2 a_1 (\lambda,\mu,\nu) + x b_2 (\lambda,\mu,\nu) + d_3 (\lambda,\mu,\nu) = 0,
\]
we see that $X$ is not quasi-smooth at the point $(\theta\!:\!\lambda\!:\!\mu\!:\!\nu\!:\!0)$ and this is a contradiction.
The lowest weight part of $F (1, y, z, t, 0) = e_5$ with respect to $\wt (y, z, t) = (1, 1, 1)$ is $e_5$ which defines a smooth hypersurface in $\mbP^2$.
By Lemma \ref{lem:lctwblwh}, we have $\lct_{\msp} (X, H_w) \ge 3/5$.
Thus $\lct_{\msp} (X;\frac{1}{2} H_w) \ge 6/5$ in this case and the claim is proved.
\end{proof}

Let $D \in |A|_{\mbQ}$ be an irreducible $\mbQ$-divisor other than $\frac{1}{2} H_w$.
We can take a $\mbQ$-divisor $T \in |A|_{\mbQ}$ such that $\mult_{\msp} (T) \ge 1$ and $\Supp (T)$ does not contain any component of the effective $1$-cycle $D \cdot H_w$ since $\{y, z, t\}$ isolates $\msp$.
It follows that
\[
3 \mult_{\msp} (D) \le (D \cdot H_w \cdot T)_{\msp} \le (D \cdot H_w \cdot T) = 5.
\]
This shows $\lct_{\msp} (X;D) \ge 3/5$ and thus $\alpha_{\msp} (X) \ge 1/2$.

\paragraph{\it Case: $x^3 w \notin F$}
By a choice of coordinates, we can write
\[
F = x^4 t + x^3 f_2 + x^2 f_3 + x f_4 + f_5,
\]
where $f_i = f_i (y,z,t,w)$ is a quasi-homogeneous polynomial of degree $i$ with $w \notin f_2$.

Suppose $w^2 \in f_4$.
In this case, $\mult_{\msp} (H_t) = 2$ and hence $\lct_{\msp} (X;H_t) \ge 1/2$.
Let $D \in |A|_{\mbQ}$ be an irreducible $\mbQ$-divisor other than $H_t$.
We can take a $\mbQ$-divisor $T \in |A|_{\mbQ}$ such that $\mult_{\msp} (T) \ge 1$ and $\Supp (T)$ does not contain any component of the effective $1$-cycle $D \cdot H_t$ since $\{y, z, t\}$ isolates $\msp$, so that
\[
2 \mult_{\msp} (D) \le (D \cdot H_t \cdot T)_{\msp} \le (D \cdot H_t \cdot T) = \frac{5}{2}.
\]
This shows $\lct_{\msp} (X;D) \ge 4/5$ and thus $\alpha_{\msp} (X) \ge 1/2$ in this case.

Suppose $w^2 \notin f_4$.
We have $\Bs |\mcI_{\msp} (A)| = \Gamma$, where $\Gamma = (y = z = t = 0) \subset X$ is a quasi-line.
We assume $\alpha_{\msp} (X) < 1/2$.
Then there exists an irreducible $\mbQ$-divisor $D \in |A|_{\mbQ}$ such that $(X, \frac{1}{2} D)$ is not log canonical at $\msp$.
Let $S \in |\mcI_{\msp} (A)|$ be a general member so that $S \ne \Supp (D)$.
Then $S$ is a normal surface by Lemma \ref{lem:normalqhyp} and it is quasi-smooth along $\Gamma$.
Moreover, for another general $T \in |\mcI_{\msp} (A)|$, the multiplicity of $T|_S$ along $\Gamma$ is $1$, that is, we can write 
\[
T|_S = \Gamma + \Delta,
\] 
where $\Delta$ is an effective divisor on $S$ such that $\Gamma \not\subset \Supp (\Delta)$.
We see that $\Gamma$ is a quasi-line, $S$ is quasi-smooth at $\msp_w$, $\Gamma$ passes through the $\frac{1}{2} (1,1)$ point $\msp_w$ of $S$ and $(K_S \cdot \Gamma) = 0$.
It follows that
\[
(\Gamma^2)_S = - 2 + \frac{1}{2} = - \frac{3}{2},
\]
by Remark \ref{rem:compselfint}.
Hence
\[
(\Delta \cdot \Gamma)_S = (T|_S \cdot \Gamma)_S - (\Gamma^2)_S = 2.
\]
The divisor $D|_S$ on $S$ is effective and we write $\frac{1}{2} D|_S = \gamma \Gamma + \Xi$, where $\gamma \ge 0$ and $\Xi$ is an effective divisor on $S$ such that $\Gamma \not\subset \Supp (\Xi)$.
Since $\Bs |\mcI_{\msp} (A)| = \Gamma$ and $S$ is general, we may assume that $\Supp (\Xi)$ does not contain any component of $\Supp (\Delta)$.
In particular $(\Xi \cdot \Delta)_S \ge 0$.
Note also that 
\[
(D|_S \cdot \Delta)_S = (T|_S \cdot \Delta)_S = ((A^3) - (T \cdot \Gamma)_S) = 2.
\]
It follows that
\[
2 = (D|_S \cdot \Delta)_S \ge 2 \gamma (\Gamma \cdot \Delta)_S = 4 \gamma,
\]
which implies $\gamma \le \frac{1}{2}$.
We see that $(X, \frac{1}{2} D|_S)$ is not log canonical at $\msp$, and hence $(S, \Gamma + \Xi) = (S, \frac{1}{2} D|_S + (1-\gamma) \Gamma)$ is not log canonical at $\msp$.
By the inversion of adjunction, we have
\[
1 \ge \frac{1}{4} + \frac{3}{2} \gamma = ((\frac{1}{2} D|_S - \gamma \Gamma) \cdot \Gamma)_S = (\Delta \cdot \Gamma)_S \ge \mult_{\msp} (\Delta|_{\Gamma}) > 1.
\]
This is a contradiction and the inequality $\alpha_{\msp} (X) \ge 1/2$ is proved.

\subsubsection{The singular point of type $\frac{1}{2} (1,1,1)$}

Let $\msp = \msp_w$ be the singular point of type $\frac{1}{2} (1,1,1)$.
Note that the point $\msp \in X$ is a QI center.

\paragraph{\it Case: $\msp$ is non-degenerate}
By a choice of coordinates, we can write
\[
F = w^2 t + w f_3 (x,y,z) + g_5 (x,y,z,t),
\]
where $f_3, g_5$ are non-zero homogeneous polynomials such that $f_3 \ne 0$ as a polynomial.
Let $\varphi \colon Y \to X$ be the Kawamata blowup at $\msp$ with exceptional divisor $E$.

\begin{Claim} \label{cl:No2-2}
$\lct_{\msp} (X, H_t) \ge \frac{1}{2}$.
\end{Claim}

\begin{proof}[Proof of Claim \ref{cl:No2-2}]
The lowest weight part of $F (x,y,z,0,1)$ with respect to $\wt (x,y,z) = (1,1,1)$ is $f_3$.
By Lemma \ref{lem:lcttangcube}, we have $\alpha_{\msp} (X) \ge 1/2$ unless $f_3$ is a cube of a linear form.
Hence it remains to prove the claim assuming that $f_3$ is a cube of a linear form.
By a choice of coordinates, we may assume $f_3 = z^3$. 
Let $S$ be the divisor on $X$ defined by $x - \lambda y = 0$ for a general $\lambda \in \mbC$.
By the quasi-smoothness of $X$, the polynomial $F$ cannot be contained in the ideal $(z, t) \subset \mbC [x, y, z, t, w]$.
This implies $g_5 (x, y, 0, 0) \ne 0$, and hence $g_5 (\lambda y, y,0,0) \ne 0$.
By eliminating $x$, the surface $S$ is isomorphic to the hypersurface in $\mbP (1, 1, 1, 2)_{y, z, t, w}$ defined by
\[
G := w^2 t+ w z^3 + \alpha y^5 + z a_4 + t b_4 = 0,
\]
where $a_4 = a_4 (y, z), b_4 = b_4 (y, z, t)$ are homogeneous polynomials of degree $4$ and $\alpha \ne 0$ is a constant.
The lowest weight part of $G (x,z,0,1)$ with respect to $\wt (y, z) = (3, 5)$ is $z^3 + \alpha y^5$ which defines a smooth point of $\mbP (3, 5)_{y, z}$.
By Lemma \ref{lem:lctwblwh}, $\lct_{\msp} (S; H_t|_S) \ge 8/15$, and hence $\lct_{\msp} (X;H_t) \ge 8/15$.
Thus the claim is proved.
\end{proof}

Let $D \in |A|_{\mbQ}$ be an irreducible $\mbQ$-divisor on $X$ other than $H_t$.
We can take $T \in |A|_{\mbQ}$ such that $\mult_{\msp} (T) \ge 1$ and $\Supp (T)$ does not contain any component of the effective $1$-cycle $D \cdot H_t$ since $\{x, y, z, t\}$ isolates $\msp$.
Then
\[
3 \omult_{\msp} (D) < 2 (D \cdot H_t \cdot T) = 5
\] 
since $\omult_{\msp} (H_t) = 3$.
This shows $\lct_{\msp} (X; D) \ge \frac{3}{5}$ and thus $\alpha_{\msp} (X) \ge \frac{1}{2}$.

\paragraph{\it Case: $\msp$ is degenerate}

In this case we have $\alpha_{\msp} (X) = 3/5$ by Proposition \ref{prop:lctdegQI}.
Therefore the proof of Theorem \ref{thm:famI1} for the family $\mcF_2$ is completed.

\subsection{The family $\mcF_4$}

This subsection is devoted to the proof of Theorem \ref{thm:famI1} for the family $\mcF_4$.
In the following, let 
\[
X = X_6 \subset \mbP (1,1,1,2,2)_{x, y, z, t, w}
\] 
be a member of $\mcF_4$ with defining polynomial $F = F (x, y, z, t, w)$.

\subsubsection{Smooth points}

Let $\msp$ be a smooth point of $X$.
We will prove $\alpha_{\msp} (X) \ge 1/2$.
We may assume $\msp = \msp_x$ by a choice of coordinates.
The proof will be done by division into cases.

\paragraph{\it Case: Either $x^4 w \in F$ or $x^4 t \in F$}
In this case we have
\[
\alpha_{\msp} (X) \ge \frac{2}{1 \cdot 1 \cdot 2 \cdot (A^3)} = \frac{2}{3}.
\]
by Lemma \ref{lem:complctsingtang}.

\paragraph{\it Case: $x^4 w, x^4 t \notin F$}
We can write
\[
F = x^5 y + x^4 f_2 + x^3 f_3 + x^2 f_4 + x f_5 + f_6,
\]
where $f_i = f_i (y,z,t,w)$ is a quasi-homogeneous polynomial of degree $i$ with $t, w \notin f_2$.

We claim $\lct_{\msp} (X;H_y) \ge 1/2$.
This is obvious when $\mult_{\msp} (H_y) \le 2$ and hence we assume $\mult_{\msp} (H_y) \ge 3$.
Then we can write
\[
\bar{F} := F (1, 0, z, t, w) 
= \sum_{i=2}^6 f_i (0, z, t, w) 
= \alpha z^3 + \beta t z^2 + \gamma w z^2 + c (t, w) + h,
\]
where $c (t, w) = f_6 (0, 0, z, t)$ and $h = h (y, t, w)$ is in the ideal $(y, t, w)^4$.
By the quasi-smoothness of $X$, $c$ cannot be a cube of a linear form.
This implies that the cubic part of $\bar{F}$ is not a cube of a linear form.
Thus $\lct_{\msp} (X;H_y) \ge 1/2$ by Lemma \ref{lem:lcttangcube} and the claim is proved.

Let $D \in |A|_{\mbQ}$ be an irreducible $\mbQ$-divisor other than $H_y$.
We can take $T \in |2 A|_{\mbQ}$ such that $\mult_{\msp} (T) \ge 1$ and $\Supp (T)$ does not contain any component of the effective $1$-cycle $D \cdot H_y$ since $\{y, z, t, w\}$ isolates $\msp$.
Then we have
\[
2 \mult_{\msp} (D) \le (D \cdot H_y \cdot T) = 2 (A^3) = 3
\]
since $\mult_{\msp} (H_y) \ge 2$.
This implies $\lct_{\msp} (X;D) \ge 2/3$ and thus $\alpha_{\msp} (X) \ge 1/2$.

\subsubsection{Singular points of type $\frac{1}{2} (1,1,1)$}

Let $\msp$ be a singular point of type $\frac{1}{2} (1,1,1)$.
Then we have $\alpha_{\msp} (X) \ge 1/2$ by Proposition \ref{prop:ndQIcent} (actually we have $\alpha_{\msp} (X) \ge 2/3$ by the argument in \S \ref{sec:ndQI3ditinct}).
Therefore the proof of Theorem \ref{thm:famI1} for the family $\mcF_4$ is completed.

\subsection{The family $\mcF_5$}

This subsection is devoted to the proof Theorem \ref{thm:famI1} for the family $\mcF_5$.
In the following, let 
\[
X = X_7 \subset \mbP (1,1,1,2,3)_{x, y, z, t, w}
\] 
be a member of family $\mcF_5$ with defining polynomial $F = F (x, y, z, t, w)$.

\subsubsection{Smooth points}

Let $\msp$ be a smooth point of $X$.
We will prove $\alpha_{\msp} (X) \ge 1/2$.
The proof will be done by division into cases.

\paragraph{\it Case: $\msp \in U_x \cup U_y \cup U_z$}

By a choice of coordinates $x, y, z$, we may assume $\msp = \msp_x$.
By Lemma \ref{lem:complctsingtang}, we have
\[
\alpha_{\msp} (X) \ge
\begin{cases} \frac{2}{1 \cdot 1 \cdot 2 \cdot (A^3)} = \frac{6}{7}, & \text{if $x^4w \in F$}, \\
\frac{2}{1 \cdot 1 \cdot 3 \cdot (A^3)} = \frac{4}{7}, & \text{if $x^4 w \notin F$ and $x^5 t \in F$}.
\end{cases}
\]
It remains to consider the case where $x^4 w, x^5 t \notin F$.
In this case we can write
\[
F = x^6 y + x^5 f_2 + x^4 f_3 + x^3 f_4 + x^2 f_5 + x f_6 + f_7,
\]
where $f_i = f_i (y,z,t,w)$ is a quasi-homogeneous polynomial of degree $i$ with $t \notin f_2$ and $w \notin f_3$.

\begin{Claim} \label{cl:No5-1}
$\lct_{\msp} (X;H_y) \ge 1/2$.
\end{Claim}

\begin{proof}[Proof of Claim \ref{cl:No5-1}]
This is obvious when $\mult_{\msp} (H_y) = 2$, and we assume $\mult_{\msp} (H_y) \ge 3$.
It follows that each monomial appearing in $F$ is contained in $(y) \cup (z, t, w)^3$.
A monomial of degree $d \in \{2, 3, 4, 5, 6, 7\}$ in variables $y, z, t, w$ which is contained in $(y) \cup (z, t, w)^3$ is contained in $(y) \cup (z, t)^2$ except for the monomial $w^2 z$ of degree $7$.
Hence we can write
\[
F = x^6 y + y g + h + \alpha w^2 z,
\]
where $g = g (x, y, z, t, w) \in \mbC [x, y, z, t, w]$ and $h = h (x, z, t, w) \in (z, t)^2$.
If $\alpha = 0$, then $X$ is not quasi-smooth at any point of the nonempty set
\[
(y = x^6 + g = z = t = 0) \subset \mbP (1,1,1,2,3).
\]
Thus $w^2 z \in F$ and we see that $\bar{F} = F (1,0,z,t,w) \in (z,t,w)^3$ and the cubic part of $\bar{F}$ is not a cube of a linear form since $w^2 z \in \bar{F}$ and $w^3 \notin \bar{F}$. 
By Lemma \ref{lem:lcttangcube}, we have $\lct_{\msp} (X;H_y) \ge 1/2$ and the claim is proved.
\end{proof}

Let $D \in |A|_{\mbQ}$ be an irreducible $\mbQ$-divisor other than $H_y$.
We can take a $\mbQ$-divisor $T \in |3 A|_{\mbQ}$ such that $\mult_{\msp} (T) \ge 1$ and $\Supp (T)$ does not contain any component of the effective $1$-cycle $D \cdot H_y$ since $\{y, z, t, w\}$ isolates $\msp$.
Then
\[
2 \mult_{\msp} (D) \le (D \cdot H_y \cdot T)_{\msp} \le (D \cdot H_y \cdot T) = 3 (A^3) = \frac{7}{2}
\]
since $\mult_{\msp} (H_y) \ge 2$.
This shows $\lct_{\msp} (X;D) \ge 4/7$ and thus $\alpha_{\msp} (X) \ge 1/2$.

\paragraph{\it Case $\msp \notin U_x \cup U_y \cup U_z$}
If $w t^2 \in F$, then $X \setminus (U_x \cup U_y \cup U_z)$ consists of singular points.
Hence we have $w t^2 \notin F$ in this case, and $\msp$ is contained in the quasi-line $\Gamma := (x = y = z = 0) \subset X$.
We will show $\alpha_{\msp} (X) \ge 1$.
Assume to the contrary that $\alpha_{\msp} (X) < 1$.
Then there exists an irreducible $\mbQ$-divisor $D \in |A|_{\mbQ}$ such that the pair $(X, D)$ is not log canonical at $\msp$.
Let $S \in |A|$ be a general member and write $D|_S = \gamma \Gamma + \Delta$, where $\gamma \ge 0$ is a rational number and $\Delta$ is an effective $1$-cycle on $S$ such that $\Gamma \not\subset \Supp (\Delta)$.

\begin{Claim} \label{cl:No5-2}
$(\Gamma^2)_S = - 5/6$ and $\gamma \le 1$.
\end{Claim}

\begin{proof}[Proof of Claim \ref{cl:No5-2}]
We see that $S$ has singular points of type $\frac{1}{2} (1,1)$ and $\frac{1}{3} (1,2)$ at $\msp_t$ and $\msp_w$, respectively, and smooth elsewhere since $S \in |A|$ is general.
Since $\Gamma$ is a quasi-line on $S$ passing through $\msp_t, \msp_w$ and $K_S = (K_X + S)|_S \sim 0$ by adjunction, we have
\[
(\Gamma^2)_S = -2 + \frac{1}{2} + \frac{2}{3} = - \frac{5}{6}.
\]
We choose a general member $T \in |A|$ which does not contain any component of $\Delta$.
This is possible since $\Bs |A| = \Gamma$.
We write $T|_S = \Gamma + \Xi$, where $\Xi$ is an effective divisor on $S$ such that $\Gamma \not\subset \Supp (\Xi)$.
We have
\[
\begin{split}
(D|_S \cdot \Xi)_S &= (D|_S \cdot (T|_S - \Gamma))_S = \frac{7}{6} - \frac{1}{6} = 1, \\
(\Gamma \cdot \Xi)_S &= (\Gamma \cdot (T|_S - \Gamma))_S = \frac{1}{6} + \frac{5}{6} = 1.
\end{split}
\]
Note that $\Xi$ does not contain any component of $\Delta$ by our choice of $T$, and hence
\[
1 = (D|_S \cdot \Xi)_S = ((\gamma \Gamma + \Delta) \cdot \Xi)_S \ge \gamma (\Gamma \cdot \Xi)_S = \gamma,
\]
as desired.
\end{proof}

The pair $(S, D|_S) = (S, \gamma \Gamma + \Delta)$ is not log canonical at $\msp$.
Hence the pair $(S, \Gamma + \Delta)$ is not log canonical at $\msp$ since $\gamma \le 1$.
By the inversion of adjunction, we have $\mult_{\msp} (\Delta|_{\Gamma}) > 1$ and thus
\[
1 < \mult_{\msp} (\Delta|_{\Gamma}) 
\le (\Delta \cdot \Gamma)_S
= ((D|_S - \gamma \Gamma) \cdot \Gamma)_S
= \frac{1}{6} + \frac{5}{6} \gamma 
\le 1.
\]
This is a contradiction and we have $\alpha_{\msp} (X) \ge 1$.

\subsubsection{The singular point of type $\frac{1}{2} (1,1,1)$}

Let $\msp = \msp_t$ be the singular point of type $\frac{1}{2} (1,1,1)$.

\paragraph{\it Case: $t^2 w \in F$}

In this case,  we have
\[
\alpha_{\msp} (X) \ge \frac{2}{2 \cdot 1 \cdot 1 \cdot (A^3)} = \frac{6}{7}
\]
by Lemma \ref{lem:complctsingtang}.

\paragraph{\it Case: $t^2 w \notin F$}
\label{sec:No5singpt2Case2}

Replacing $x, y, z$, we can write
\[
F = t^3 x + t^2 f_3 + t f_5 + f_7,
\]
where $f_i = f_i (x, y, z, w)$ is a quasi-homogeneous polynomial of degree $i$ with $w \notin f_3$.

\begin{Claim} \label{cl:No5-3}
If $\mult_{\msp} (H_x) \ge 3$, then either $w^2 y \in F$ or $w^2 z \in F$.
\end{Claim}

\begin{proof}[Proof of Claim \ref{cl:No5-3}]
Suppose $w^2 y, w^2 z \notin F$.
Then $h := F (0, y, z, t, w)$ is contained in the ideal $(y, z)^2 \subset \mbC [y, z, t, w]$, and we can write $F = x g + h$, where $g = g (x, y, z, t, w)$.
We see that $X$ is not quasi-smooth at any point in the nonempty subset
\[
(x = y = z = g = 0) \subset \mbP (1, 1, 1, 2, 3).
\]
This is a contradiction and the claim is proved.
\end{proof}

We set $\bar{F} := F (0, y, z, 1, w)$.
By Claim \ref{cl:No5-3}, either $\bar{F} \in (y, z, w)^2 \setminus (y, z, w)^3$ or $\bar{F} \in (y, z, w)^3$ and the cubic part of $\bar{F}$ is not a cube of a linear form since $w^3 \notin \bar{F}$.
By Lemma \ref{lem:lcttangcube}, we have $\alpha_{\msp} (X) \ge 1/2$ since $\msp \in X$ is not a maximal center.

\subsubsection{Singular point of type $\frac{1}{3} (1,1,2)$}

Let $\msp = \msp_w$ be the singular point of type $\frac{1}{3} (1,1,2)$.
We can write
\[
F = w^2 x + w (\alpha t^2 + t a_2 (y, z) + b_4 (y, z)) + f_7 (x, y, z, t),
\]
where $\alpha \in \mbC$ and $a_2 = a_2 (x, y), b_4 = b_4 (y, z), f_7 = f_7 (x, y, z, t)$ are quasi-homogeneous polynomials of degree $2, 4, 7$, respectively.
Let $q = q_{\msp}$ be the quotient morphism of $\msp \in X$ and $\check{\msp}$ be the preimage of $\msp$.

\paragraph{\it Case: $\alpha \ne 0$}
We have $\mult_{\msp} (H_x) = 2$ and $\lct_{\msp} (X;H_x) \ge 1/2$.
Let $D \in |A|_{\mbQ}$ be an irreducible $\mbQ$-divisor other than $H_x$.
We can take a $\mbQ$-divisor $T \in |A|_{\mbQ}$ such that $\mult_{\msp} (T) \ge 1$ and $\Supp (T)$ does not contain any component of the effective $1$-cycle $D \cdot H_x$ since $\{x, y, z\}$ isolates $\msp$.
Then
\[
2 \omult_{\msp} (D) \le (q^*D \cdot q^*H_x \cdot q^*T)_{\check{\msp}} \le 3 (D \cdot H_x \cdot T) = \frac{7}{2}.
\]
This shows $\lct_{\msp} (X;D) \ge 4/7$ and thus $\alpha_{\msp} (X) \ge 1/2$.

\paragraph{\it Case $\alpha = 0$ and $a_2 \ne 0$}

The cubic part of $F (0, y, z, t, 1)$ is $t a_2$ and, by Lemma \ref{lem:lcttangcube}, we have $\lct_{\msp} (X;H_x) \ge 1/2$.
Let $D \sim_{\mbQ} A$ be an irreducible $\mbQ$-divisor on $X$ other than $H_x$.
Then we can take a general $T \in |\mcI_{\msp} (2 A)| = |2 A|$ which does not contain any component of $D \cap H_x$ since $\Bs |2 A| =\msp$.
We see that $T$ is defined by $t - q (x, y, z) = 0$ on $X$, where $q \in \mbC [x,y,z]$ is a general quadratic form.
Let $\rho = \rho_{\msp} \colon \breve{U}_{\msp} \to U_{\msp}$ be the orbifold chart of $X$ containing $\msp$ and let $\breve{\msp}$ be the preimage of $\msp$.
It is then easy to see that the effective $1$-cycle $\rho^*H_x \cdot \rho^*T$ on $\breve{U}_{\msp}$ has multiplicity $4$ at $\breve{\msp}$.
Then we have
\[
4 \omult_{\msp} (D) \le
(\rho^*D \cdot \rho^*H_x \cdot \rho^*T)_{\breve{\msp}}
\le 3 (D \cdot H_x \cdot T) = 7
\]
This shows $\lct_{\msp} (X;D) \ge 4/7$ and thus $\alpha_{\msp} (X) \ge 4/7$. 

\paragraph{\it Case: $\alpha = a_2 = 0$ and $b_4 \ne 0$}

By similar arguments as in the proof of Claim \ref{cl:No5-3}, we see that either $t^3 y \in f_7$ or $t^3 z \in f_7$. 
We choose $z$ and $t$ so that $b_4 (0,z) = z^4$ and $\coeff_{f_7} (t^3 z) = 1$.
Then we have
\[
F (0,0,z,t,1) = z^4 + t^3 z + \beta t^2 z^3 + \gamma t z^5 + \delta z^7,
\]
where $\beta, \gamma, \delta \in \mbC$.
The lowest weight part of $F (0,0,z,t,1)$ with respect to $\wt (z,t) = (1,1)$ is $z^4 + t^3 z$ which defines $4$ distinct points of $\mbP^1_{z, t}$.
Hence we have 
\[
\lct_{\msp} (X;H_x) \ge \lct_{\msp} (H_y;H_x|_{H_y}) \ge \frac{1}{2}
\] 
by Lemma \ref{lem:lctwblwh}.
Let $D \in |A|_{\mbQ}$ be an irreducible $\mbQ$-divisor other than $H_x$.
We can take a $\mbQ$-divisor $T \in |2 A|_{\mbQ}$ such that $\omult_{\msp} (T) \ge 1$ and $\Supp (T)$ does not contain any component of the effective $1$-cycle $D \cdot H_x$ since $\{x, y, z, t\}$ isolates $\msp$.
Then
\[
4 \omult_{\msp} (D) \le
(q^*D \cdot q^*H_x \cdot q^*T)_{\check{\msp}}
\le 3 (D \cdot H_x \cdot T) = 7.
\]
since $\omult_{\msp} (H_x) = 4$.
This shows $\lct_{\msp} (X;D) \ge 4/7$ and thus $\alpha_{\msp} (X) \ge 1/2$.

\paragraph{\it Case: $\alpha = a_2 = b_4 = 0$}
In this case the point $\msp \in X$ is a degenerate QI center and we have $\alpha_{\msp} (X) = 4/7$ by Proposition \ref{prop:lctdegQI}.

\subsection{The family $\mcF_8$} 

This subsection is devoted to the proof of Theorem \ref{thm:famI1} for the family $\mcF_8$.
In the following, let 
\[
X = X_9 \subset \mbP (1,1,1,3,4)_{x, y, z, t, w}
\] 
be a member of $\mcF_8$ with defining polynomial $F = F (x, y, z, t, w)$.

\subsubsection{Smooth points}

Let $\msp \in X$ be a smooth point.
We will prove $\alpha_{\msp} (X) \ge 1/2$.
We may assume $\msp = \msp_x$.
The proof will be done by division into cases.

\paragraph{\it Case: $x^5 w \in F$}
We can write
\[
F = x^5 w + x^4 f_5 + x^3 f_6 + x^2 f_7 + x f_8 + f_9,
\]
where $f_i = f_i (y,z,t,w)$ is a quasi-homogeneous polynomial of degree $i$.
We have $\mult_{\msp} (H_w) = 3$ since $t^3 \in f_9$.
Let $D \in |A|_{\mbQ}$ be an irreducible $\mbQ$-divisor on $X$.
Let $S \in |\mcI_{\msp} (A)|$ be a general member so that $\Supp (D) \ne S$.
Since $\{y, z, w\}$ isolates $\msp$, we can take a $\mbQ$-divisor $T \in |A|_{\mbQ}$ such that $\Supp (T)$ does not contain any component of the effective $1$-cycle $D \cdot S$ and $\mult_{\msp} (T) \ge 3/4$ (Note that $T$ is one of $H_y, H_z$ and $\frac{1}{4} H_w$). 
Then we have
\[
\frac{3}{4} \mult_{\msp} (D) \le (D \cdot S \cdot T)_{\msp} \le (D \cdot S \cdot T) = \frac{3}{4}.
\]
This shows $\lct_{\msp} (X;D) \ge 1$ and thus $\alpha_{\msp} (X) \ge 1$.

\paragraph{\it Case: $x^5 w \notin F$ and $x^6 t \in F$}
We can write
\[
F = x^6 t + x^5 f_4 + x^4 f_5 + x^3 f_6 + x^2 f_7 + x f_8 + f_9,
\]
where $f_i = f_i (y,z,t,w)$ is a quasi-homogeneous polynomial of degree $i$ with $w \notin f_4$.
Let $S, T \in |\mcI_{\msp} (A)|$ be general members.
Note that $S$ is smooth at $\msp$.
The intersection $S \cap T$ is isomorphic to the subscheme in $\mbP (1_x, 3_t, 4_w)$ defined by the equation $F (x,0,0,t,w) = 0$ and we can write
\[
F (x,0,0,t,w) = x^6 t + \alpha x^3 t^2 + \beta x^2 w t + \gamma x w^2 + t^3,
\]
where $\alpha, \beta, \gamma \in \mbC$.

\begin{Claim} \label{clm:No8smpt2-1}
If $\gamma \ne 0$, then $S \cdot T = \Gamma$, where $\Gamma$ is an irreducible and reduced curve of degree $3/4$ that is smooth at $\msp$.
\end{Claim}

\begin{proof}[Proof of Claim \ref{clm:No8smpt2-1}]
Suppose $\gamma \ne 0$.
Then it is easy to see that the polynomial $F (x, 0, 0, t, w)$ is irreducible.
Hence the curve
\[
\Gamma = (y = z = F (x,0,0,t,w) = 0) \subset \mbP (1,1,1,3,4).
\]
is irreducible and reduced.
It is also obvious that $\deg \Gamma = 3/4$ and $\Gamma$ is smooth at $\msp$.
\end{proof}

If $\gamma \ne 0$, then we have $\alpha_{\msp} (X) \ge 1$ by Claim \ref{clm:No8smpt2-1} and Lemma \ref{lem:exclL}.

In the following we consider the case where $\gamma = 0$.
We set 
\[
\Delta = (y = z = t = 0) \subset \mbP (1, 1, 1, 3, 4),
\]
which is a quasi-line of degree $1/4$ passing through $\msp$.
Note that $\Delta$ is smooth at $\msp$.

\begin{Claim} \label{clm:No8smpt2-2}
If $\gamma = 0$ and $\beta \ne 0$, then $T|_S = \Delta + \Xi$, where $\Xi$ is an irreducible and reduced curve which does not pass through $\msp$.
Moreover the intersection matrix $M (\Delta, \Xi)$ satisfies the condition $(\star)$.
\end{Claim}

\begin{proof}[Proof of Claim \ref{clm:No8smpt2-2}]
We have
\[
F (x, 0, 0, t, w) = t (x^6 + \alpha x^3 t + \beta x^2 w + t^2),
\]
and the polynomial $x^6 + \alpha x^3 t + \beta x^2 w + t^2$ is irreducible since $\beta \ne 0$.
It follows that $T|_S = \Delta + \Xi$, where 
\[
\Xi = (y = z = x^6 + \alpha x^3 t + \beta x^2 w + t^2 = 0) \subset \mbP (1, 1, 1, 3, 4) 
\]
is an irreducible and reduced curve of degree $1/2$ that does not pass through $\msp$.
We have $\Delta \cap \Xi = \{\msp_w, \msq\}$, where $\msq = (1\!:\!0\!:\!0\!:\!0\!:\!-1/\beta)$.
It is easy to see that $S$ is quasi-smooth at $\msp_w$ and $\msq$, hence $S$ is quasi-smooth along $\Delta$ by Lemma \ref{lem:pltsurfpair}.
We have $\Sing_{\Gamma} (S) = \{\msp_w\}$ and $\msp_w \in S$ is of type $\frac{1}{4} (1, 3)$.
By Remark \ref{rem:compselfint}, we have
\[
(\Delta^2)_S = -2 + \frac{3}{4} = - \frac{5}{4}.
\]
By taking intersection number of $T|_S = \Delta + \Xi$ and $\Delta$, and then $T|_S$ and $\Xi$, we have
\[
(\Delta \cdot \Xi) = \frac{3}{2}, \quad
(\Xi^2)_S = -1.
\]
It follows that the intersection matrix $M (\Delta, \Xi)$ satisfies the condition $(\star)$.
\end{proof}

\begin{Claim} \label{clm:No8smpt2-3}
If $\gamma = \beta = 0$ and $\alpha \ne \pm 2$, then $T|_S = \Delta + \Theta_1 + \Theta_2$, where $\Theta_1$ and $\Theta_2$ are distinct quasi-lines which does not pass through $\msp$.
Moreover the intersection matrix $M (\Delta, \Theta_1, \Theta_2)$ satisfies the condition $(\star)$.
\end{Claim}

\begin{proof}[Proof of Claim \ref{clm:No8smpt2-3}]
We have
\[
F (x, 0, 0, t, w) = t (x^6 + \alpha x^3 t + t^2) = t (t - \lambda x^3)(t - \lambda^{-1} x^3),
\]
where $\lambda \ne 0, 1$ is a complex number such that $\alpha = \lambda + \lambda^{-1}$.
Hence we have
\[
T|_S = \Delta + \Theta_1 + \Theta_2,
\]
where 
\[
\Xi_1 = (y = z = t - \lambda x^3 = 0), \quad
\Xi_2 = (y = z = t - \lambda^{-1} x^3 = 0)
\]
are both quasi-lines of degree $1/4$ that do not pass through $\msp$.
We have $\Delta \cap (\Theta_1 \cup \Theta_2) = \{\msp_w\}$ and $S$ is clearly quasi-smooth at $\msp_w$.
It follows that $S$ is quasi-smooth along $\Gamma$ by Lemma \ref{lem:pltsurfpair}, and $\Sing_{\Delta} (S) = \{\msp_w\}$, where $\msp_w \in S$ is of type $\frac{1}{4} (1, 3)$.
Thus we have 
\[
(\Delta^2)_S = -\frac{5}{4}.
\]
By similar arguments, we see that $S$ is quasi-smooth along $\Theta_i$ and $\Sing_{\Theta} (S) = \{\msp_w\}$ for $i = 1, 2$, and hence 
\[
(\Theta_i^2)_S = - \frac{5}{4}.
\]
By taking intersection number of $T|_S = \Delta + \Theta_1 + \Theta_2$ and $\Delta, \Theta_1, \Theta_2$, we conclude
\[
(\Delta \cdot \Theta_1)_S = (\Delta \cdot \Theta_2)_S = (\Theta_1 \cdot \Theta_2)_S = \frac{3}{4}.
\]
It is then straightforward to see that $M (\Delta, \Theta_1, \Theta_2)$ satisfies the condition $(\star)$.
\end{proof}

\begin{Claim} \label{clm:No8smpt2-4}
If $\gamma = \beta = 0$ and $\alpha = \pm 2$, then $T|_S = \Delta + 2 \Theta$, where $\Theta$ is an irreducible and reduced curve which does not pass through $\msp$.
Moreover the intersection matrix $M (\Delta, \Theta)$ satisfies the condition $(\star)$.
\end{Claim}

\begin{proof}[Proof of Claim \ref{clm:No8smpt2-4}]
Without loss of generality, we may assume $\alpha = -2$.
We have
\[
F (x, 0, 0, t, w) = t (t - x^3)^2,
\]
and hence
\[
T|_S = \Delta + 2 \Theta,
\]
where 
\[
\Theta = (y = z = t - x^3 = 0) \subset \mbP (1, 1, 1, 3, 4)
\]
is a quasi-line of degree $1/4$ that does not pass through $\msp$.
By the same arguments as in Claim \ref{clm:No8smpt2-3}, we have
\[
(\Delta^2)_S = - \frac{5}{4}.
\]
Then, by taking intersection number of $T|_S = \Delta + 2 \Theta$ and $\Delta, \Theta$, we have
\[
(\Delta \cdot \Theta)_S = \frac{3}{4}, \quad
(\Theta^2)_S = - \frac{1}{4}.
\]
Thus the matrix $M (\Delta, \Theta)$ satisfies the condition $(\star)$.
\end{proof}

By Claims \ref{clm:No8smpt2-2}, \ref{clm:No8smpt2-3},  \ref{clm:No8smpt2-4} and Lemma \ref{lem:mtdLred}, we conclude
\[
\alpha_{\msp} (X) \ge \min \left\{ 1, \ \frac{1}{(A^3) + 1 - \deg \Delta} \right\} = \frac{2}{3}.
\]

\paragraph{\it Case: $x^5 w, x^6 t \notin F$}
Replacing $y$ and $z$, we can write
\[
F = x^8 y + x^7 f_2 + x^6 f_3 + x^5 f_4 + x^4 f_5 + x^3 f_6 + x^2 f_7 + x f_8 + f_9,
\]
where $f_i = f_i (y,z,t,w)$ is a homogeneous polynomial of degree $i$ with $w \notin f_4$ and $t \notin f_3$.
Note that we have $2 \le \mult_{\msp} (H_y) \le 3$ since $t^3 \in F$.

Let $D \in |A|_{\mbQ}$ be an irreducible $\mbQ$-divisor other than $H_y$.
We can take a $\mbQ$-divisor $T \in |4 A|_{\mbQ}$ such that $\mult_{\msp} (T) \ge 1$ and $\Supp (T)$ does not contain any component of the effective $1$-cycle $D \cdot H_y$ since $\{y, z, t, w\}$ isolates $\msp$.
Then
\[
2 \mult_{\msp} (D) \le (D \cdot H_y \cdot T)_{\msp} \le (D \cdot H_y \cdot T) = 4 (A^3) = 3.
\]
This shows $\lct_{\msp} (X;D) \ge 2/3$ and thus it remains to show that $\lct (X; H_y) \ge 1/2$.

Suppose that either $\mult_{\msp} (H_y) = 2$ or $\mult_{\msp} (H_y) = 3$ and the cubic part of $\bar{F} := F (1,0,z,t,w)$ is a cube of a linear form.
Then $\lct_{\msp} (X;H_y) \ge 1/2$ by Lemma \ref{lem:lcttangcube}, and we are done.

In the following, we assume that $\mult_{\msp} (H_y) = 3$ and the cubic part of $\bar{F}$ is a cube of a linear form.
Since $t^3 \in \bar{F}$ and $w^3 \notin \bar{F}$, we may assume that the cubic part of $\bar{F}$ is $t^3$ after replacing $t$.

We claim $w^2 z \in F$.
We see that a monomial other than $w^2 z$ which appears in $\bar{F}$ with nonzero coefficient is contained in the ideal $(z, t)^2 \subset \mbC [z, t, w]$.
We can write $F = y G + F (x, 0, z, t, w)$ for some homogeneous polynomial $G (x, y, z, t, w)$.
If $w^2 z \notin F$, then $F (x, 0, z, t, w) \in (z, t)^2$ and $X$ is not quasi-smooth at any point contained in the nonempty set 
\[
(y = z = t = G = 0) \subset \mbP (1,1,1,3,4).
\]
This is a contradiction and the claim is proved.

Then, we may assume $\coeff_F (w^2 z) = 1$ and, by replacing $t$ and $w$, we can write
\[
\begin{split}
\bar{F} = \alpha_4 z^4 & + \alpha_5 z^5 + (\beta t z^3 + \alpha_6 z^6) + (\gamma w z^3 + \delta t z^4 + \alpha_7 z^7) + \\ 
+ & (\varepsilon w z^4 + \zeta t^2 z^2 + \eta t z^5 + \alpha_8 z^8)
+ (w^2 z + t^3 + \theta t^2 z^3 + \lambda t z^6 + \alpha_9 z^9),
\end{split}
\]
where $\alpha_4, \dots, \alpha_9, \beta, \gamma, \dots, \lambda \in \mbC$.
The lowest weight part of $\bar{F}$ with respect to $\wt (z,t,w) = (6,8,9)$ is $G := \alpha_4 z^4 + w^2 z + t^3$.
We set $\mbP = \mbP (6, 8, 9)$.
Then $\mbP^{\wf} = \mbP (1, 4, 3)_{\tilde{z}, \tilde{t}, \tilde{w}}$ and, by Lemma \ref{lem:lctwblwh}, we have
\[
\lct_{\msp} (X; H_y) \ge \min \left\{ \frac{23}{24}, \ \lct (\mbP^{\wf}, \Diff; \Gamma) \right\},
\]
where 
\[
\begin{split}
\Diff &= \frac{2}{3} H^{\wf}_{\tilde{t}} + \frac{1}{2} H^{\wf}_{\tilde{w}}, \\
\Gamma &= \mcD^{\wf}_G = (\alpha_4 \tilde{z}^4 + \tilde{w} \tilde{z} + \tilde{t} = 0) \subset \mbP (1, 4, 3),
\end{split}
\]
are $(\mbQ$-)divisors on $\mbP^{\wf}$ with $H^{\wf}_{\tilde{t}} = (\tilde{t} = 0)$ and $H^{\wf}_{\tilde{w}} = (\tilde{w} = 0)$.
It is easy to see that any pair of curves $H^{\wf}_{\tilde{t}}, H^{\wf}_{\tilde{w}}$ and $\Gamma$ intersect transversally.
If $\alpha_4 \ne 0$, then $H^{\wf}_{\tilde{t}} \cap H^{\wf}_{\tilde{w}} \cap \Gamma = \emptyset$, and thus $\lct (\mbP^{\wf}, \Diff; \Gamma) = 1$.
If $\alpha_4 = 0$, then $H^{\wf}_{\tilde{t}} \cap H^{\wf}_{\tilde{w}} \cap \Gamma = \{\msp_{\tilde{z}}\}$.
In this case, by consider the the blowup at $\msp_{\tilde{z}}$, we can confirm the equality $\lct (\mbP^{\wf}, \Diff; \Gamma) = 5/6$.
Thus, we have $\lct_{\msp} (X; H_y) \ge 5/6$, and the proof is completed.

\subsubsection{The singular point of type $\frac{1}{4} (1,1,3)$}

Let $\msp = \msp_w$ be the singular point of type $\frac{1}{4} (1,1,3)$.
We can write
\[
F = w^2 x + w (t a_2 (y, z) + b_5 (y, z)) + f_9 (x, y, z, t),
\]
where $a_2 = a_2 (y, z)$, $b_5 = b_5 (y, z)$ and $f_9 =f_9 (x, y, z, t)$ are homogeneous polynomials of degrees $2$, $5$ and $9$, respectively.

Suppose that $a_2 \ne 0$ as a polynomial.
Then $\bar{F} := F (0, y, z, t, 1) \in (y, z, t)^3$ and its cubic part $t a_2 + t^3$ is not a cube of a linear form. 
It follows that $\lct_{\msp} (X;H_x) \ge 2/3$ by Lemma \ref{lem:lcttangcube}.
Let $D \in |A|_{\mbQ}$ be an irreducible $\mbQ$-divisor other than $H_x$.
Since the set $\{x, y, z\}$ isolates $\msp$, we can take a $\mbQ$-divisor $T \in |A|_{\mbQ}$ such that $\omult_{\msp} (T) \ge 1$ and $\Supp (T)$ does not contain any component of $D \cdot H_x$.
We have $\omult_{\msp} (H_x) = 3$.
It follows that
\[
3 \omult_{\msp} (D) \le (q^*D \cdot q^*H_x \cdot q^*T)_{\check{\msp}} \le 4 (D \cdot H_x \cdot T) = 3,
\]
where $q = q_{\msp}$ is the quotient morphism of $\msp \in X$ and $\check{\msp}$ is the preimage of $\msp$ via $q$.
This shows $\lct_{\msp} (X;D) \ge 1$ and thus $\alpha_{\msp} (X) \ge 1/2$.

Finally, suppose that $a_2 = 0$ and $b_5 \ne 0$.
Replacing $y$ and $z$, we may assume $z^5 \in b_5$ and $\coeff_{b_5} (z^5) = 1$.
We may also assume that $\coeff_{f_9} (t^3) = 1$ by rescaling $t$.
Then we have
\[
F (0,0,z,t,1) = z^5 + f_9 (0,0,z,t).
\]
The lowest weight part with respect to the weight $\wt (z,t) = (3,5)$ is $z^5 + t^3$ and thus
\[
\lct_{\msp} (X;H_x) \ge \lct_{\msp} (H_y;H_x|_{H_y}) = \frac{8}{15}.
\]
We have $\omult_{\msp} (H_x) = 3$ and the set $\{x, y, z\}$ isolates $\msp$.
Hence by the same argument as in the the case $a_2 \ne 0$, we have $\lct_{\msp} (X;D) \ge 1$ for any irreducible $\mbQ$-divisor $D \in |A|_{\mbQ}$ other than $H_x$.
Thus $\alpha_{\msp} (X) \ge 1/2$.

Suppose that $a_2 = b_5 = 0$.
Then we have $\alpha_{\msp} (X) = 5/9$ by Proposition \ref{prop:lctdegQI}, and the proof is completed.

\begin{Rem} \label{rem:QIcent-8}
The singular point $\msp \in X$ of type $\frac{1}{4} (1, 1, 3)$ is a QI center.
When $\msp$ is non-degenerate, the above proof shows that $\lct_{\msp} (X; D) \ge 1$ for any irreducible $\mbQ$-divisor $D \in |A|_{\mbQ}$ other than $H_x$ and $\lct_{\msp} (X; H_x) > 1/2$.
\end{Rem}

\section{Further results and discussion on related problems}
\label{chap:discuss}

\subsection{Birationally superrigid Fano $3$-folds of higher codimensions}

We can embed a Fano 3-fold into a weighted projective space by choosing (minimal) generators of the anticanonical graded ring.
We consider embedded Fano 3-folds.
We have satisfactory results on the classification of Fano 3-folds of low codimensions (\cite{IF00}, \cite{CCC11}, \cite{ABR02}) and the following are known for their birational (super)rigidity.

\begin{itemize}
\item Fano 3-folds of codimension $2$ are all weighted complete intersections and they consist of 85 families.
Among them, there are exactly $19$ families whose members are  birationally rigid (\cite{Oka14}, \cite{AZ16}).
\item Fano 3-folds of codimension $3$ consist of 69 families of so called Pfaffian Fano 3-folds and 1 family of complete intersections of three quadrics in $\mbP^6$.
Among them, there are exactly $3$ families whose members are birationally rigid (\cite{AO18}).
\item Constructions of many families of Fano 3-folds of codimension 4 has been known (see e.g.\ \cite{BKR12}, \cite{CD18}), but their classification is not completed.
There are at least $2$ families of birationally superrigid Fano 3-folds of codimension 4 (\cite{Oka20a}). 
\end{itemize}

For birationally rigid Fano 3-folds of codimension 2 and 3, K-stability and existence of KE metrics are known under some generality assumptions.

\begin{Thm}[{\cite{KOW18}}] \label{thm:KOW18}
Let $X$ be a general quasi-smooth Fano 3-folds of codimension $c \in \{2, 3\}$ which is birationally rigid.
We assume that $X$ is a complete intersection of a quadric and cubic in $\mbP^5$ when $c = 2$.
Then $\alpha (X) \ge 1$, $X$ is K-stable and admits a KE metric.
\end{Thm}

\begin{Thm}[{\cite[Theorem 1.3]{Zhu20b}}] \label{thm:Zhu20}
Let $X$ be a smooth complete intersection of a quadric and cubic in $\mbP^5$.
Then $X$ is K-stable and admits a KE metric.
\end{Thm}

\begin{Question}
Can we conclude K-stability for any quasi-smooth Fano 3-fold of codimension 2 and 3 which is birationally (super)rigid?
How about for Fano 3-folds of codimension 4 or higher?
\end{Question}

\subsection{Lower bound of alpha invariants}

In the context of Theorem \ref{thm:SZ}, the following is a very natural question to ask. 

\begin{Question}
Is it true that $\alpha (X) \ge 1/2$ (or $\alpha (X) > 1/2$) for any birationally superrigid Fano variety? 
If yes, can we find a lower bound better than $1/2$?
\end{Question}

The following example suggests that the number $1/2$ is optimal (or the lower bound can be even smaller).

\begin{Ex} \label{ex:bsrlcthalf}
For an integer $a \ge 2$, let $X_a$ be a weighted hypersurfaces of degree $2 a + 1$ in $\mbP (1^{a+2}, a) = \Proj \mbC [x_1,\dots,x_{a+2},y]$, given by the equation
\[
y^2 x_1 + f (x_1,\dots,x_{a+2}) = 0,
\]
where $f$ is a general homogeneous polynomial of degree $2 a + 1$.
Then $X_a$ is a quasi-smooth Fano weighted hypersurface of dimension $a+1$ and Picard number $1$ with the unique singular point $\msp$ of type
\[
\frac{1}{a} (\overbrace{1,\dots,1}^{a+1}).
\]
The singularity $\msp \in X$ is terminal.
By the same argument as in the proof of Proposition \ref{prop:lctdegQI}, we obtain
\[
\alpha (X) \le \alpha_{\msp} (X) = \lct_{\msp} (X; H_{x_1}) = \frac{a+1}{2 a + 1}.
\]
When $a = 2$, $X_a = X_2$ is a member of the family $\mcF_2$ and it is birationally superrigid.
We expect that $X_a$ is birationally superrigid although this is not proved at all when $a \ge 3$.
If $X_a$ is birationally superrigid for $a \gg 0$, then it follows that there exists a sequence of birationally superrigid Fano varieties whose alpha invariants are arbitrary close to (or less than) $\frac{1}{2}$.
\end{Ex}

\begin{Question}
Let $X_a$ be as in Example \ref{ex:bsrlcthalf}.
Is $X_a$ birationally superrigid for $a \ge 3$?
\end{Question}

\subsection{Existence of K\"ahler--Einstein metrics}
\label{sec:KEmetric}

For a quasi-smooth Fano 3-fold weighted hypersurface of index $1$ which is strictly birationally rigid, we are unable to conclude the existence of a K\"{a}hler--Einstein metric as a direct consequence of Theorem \ref{mainthm}.
However, for a Fano variety $X$ of dimension $n$ with only quotient singularities, the implication
\[
\alpha (X) > \frac{n}{n + 1} \ \Longrightarrow \ \text{existence of a KE metric on $X$}
\]
is proved in \cite[Section 6]{DK}.
The aim of this section is to prove the existence of KE metrics on quasi-smooth members of suitable families.

We set
\[
\msI'_{\mathrm{KE}} =  \{42, 44, 45, 61, 69, 74, 76, 79 \} \subset \msI_{\BR},
\]
and 
\[
\msI_{\mathrm{KE}} = \msI_{\BSR} \sqcup \msI'_{\mathrm{KE}}.
\]
Note that $|\msI_{\mathrm{KE}}| = 56$.
For a family $\mcF_{\msi}$ with $\msi \in \msI$, the mark ``KE" is given in the right-most column of Table \ref{table:main} if and only if $\msi \in \msI_{\mathrm{KE}}$.

\begin{Thm} \label{thm:qsmWHKE}
For a member $X$ of a family $\mcF_{\msi}$ with $\msi \in \msI'_{\mathrm{KE}}$, we have
\[
\alpha (X) > \frac{3}{4}.
\]
In particular, any member of a family $\mcF_{\msi}$ with $\msI_{\mathrm{KE}}$ admits a KE metric and is K-stable.
\end{Thm}

\begin{proof}
By Corollary \ref{maincor:Kst} and the above arguments, it is enough to prove the first assertion.
Let 
\[
X = X_d \subset \mbP (1, a_1, a_2, a_3, a_4)_{x, y, z, t, w}
\] 
be a member of a family $\mcF_{\msi}$ with $\msi \in \msI'_{\mathrm{KE}}$, where we assume $a_1 \le a_2 \le a_3 \le a_4$.
Note that $1 < a_1 < a_2$.

\begin{Claim} \label{clm:KEU1}
$\alpha_{\msp} (X) \ge 1$ for any $\msp \in U_1 \cap \Sm (X)$.
\end{Claim}

\begin{proof}[Proof of Claim \ref{clm:KEU1}]
Let $\msp$ be a smooth point of $X$ contaied in $U_1$.

Suppose $\msi = 42$.
Then $d = 20$ is divisible by $a_4 = 10$ and $a_2 a_3 (A^3) = 1$.
By Lemma \ref{lem:nsptU1-2}, we have $\alpha_{\msp} (X) \ge 1$.

Suppose $\msi \in \{ 69, 74, 76, 79 \}$.
Then $a_2 a_4 (A^3) \le 1$.
By Lemma \ref{lem:nsptU1-1}, we have $\alpha_{\msp} (X) \ge 1$ in this case.

Suppose $\msi \in \{44, 45, 61\}$.
Then $a_3 a_4 (A^3) \le 2$.
We may assume $\msp = \msp_x$.
Then we have $\alpha_{\msp} (X) \ge 2/a_3 a_4 (A^3) \ge 1$ by Lemma \ref{lem:complctsingtang}.
This completes the proof.
\end{proof}

\begin{Claim} \label{clm:KEH}
$\alpha_{\msp} (X) \ge 1$ for any $\msp \in (H_x \setminus L_{xy}) \cap \Sm (X)$.
\end{Claim}

\begin{proof}[Proof of Claim \ref{clm:KEH}]
This follows immediately from Proposition \ref{prop:smptHminusL}.
\end{proof}

\begin{Claim} \label{clm:KEL}
$\alpha_{\msp} (X) \ge 43/54 > 3/4$ for any $\msp \in L_{xy} \cap \Sm (X)$.
\end{Claim}

\begin{proof}[Proof of Claim \ref{clm:KEL}]
Let $\msp$ be a smooth point of $X$ contained in $L_{xy}$.
Suppose that $X$ is a member of one of the families listed in Tables \ref{table:Lsmooth} or \ref{table:Lsing}, that is, $X$ is a member of a family $\mcF_{\msi}$ with $\msi \in \{44, 45, 61, 69, 74, 76, 79\}$.
Then the claim follows immediately from Proposition \ref{prop:smptL1}.

Suppose $\msi = 42$.
Then, by the proof of Proposition \ref{prop:smptL2} (see \S \ref{sec:smptL2-42}), either $\alpha_{\msp} (X) \ge 1$ for any $\msp \in L_{xy} \cap \Sm (X)$ or $X$ satisfies the assumption of Lemma \ref{lem:Lredcp1}.
In the latter case we have $\alpha_{\msp} (X) \ge 43/54$ by Remark \ref{rem:Lredcp1}.
This completes the proof.
\end{proof}

By Claims \ref{clm:KEU1}, \ref{clm:KEH} and \ref{clm:KEL}, we have $\alpha_{\msp} (X) \ge 3/4$ for any smooth point $\msp \in X$. 
It remains to consider singular points.

\begin{Claim} \label{clm:KEsing}
$\alpha_{\msp} (X) > 3/4$ for any $\msp \in \Sing (X)$.
\end{Claim}

\begin{proof}[Proof of Claim \ref{clm:KEsing}]
Let $\msp \in X$ be a singular point.
If the subscript $\heartsuit$ (resp.\ $\diamondsuit$) is given in Table \ref{table:main}, then $\alpha_{\msp} (X) \ge 1$ by Proposition \ref{prop:singptCP} (resp.\ Proposition \ref{prop:lctsingptL}).
It remains to consider the case where $\msi = 42$ and $\msp$ is of type $\frac{1}{5} (1, 2, 3)$.  
In this case we have $\alpha_{\msp} (X) \ge 1$ by the proof of Proposition \ref{prop:ndQIcent} (see \S \ref{sec:ndQI2eqtypes}).
\end{proof}

This completes the proof of Theorem \ref{thm:qsmWHKE}.
\end{proof}

\subsection{Birational rigidity and K-stability} \label{section:BRKst}

\subsubsection{Generalizations of the conjecture}

Birational superrigidity is a very strong property.
It is natural to relax the assumption of birational superrigidity to birational rigidity in Conjecture \ref{conj:BSRKst}, and we still expect a positive answer to the following.

\begin{Conj}[{\cite[Conjecture 1.9]{KOW18}}] \label{conj:BRKst}
A birationally rigid Fano variety is K-stable.
\end{Conj}

We explain the situation for smooth Fano 3-folds.
There are exactly 2 families of smooth Fano 3-folds which is strictly birationally rigid: one is the family of complete intersections of a quadric and cubic in $\mbP^6$ (\cite{IP96}), and another is the family of double covers $V$ of a smooth quadric $Q$ of dimension 3 branched along a smooth surface degree $8$ on $Q$ (\cite{Isk80}).
Former Fano 3-folds are K-stable and admit KE metrics (\cite{Zhu20b}), and so are the latter Fano 3-folds (this follows from \cite{Der16a} since $Q$ is K-semistable).
Some more evidences are already provided by Theorems \ref{thm:BRWH} and \ref{thm:KOW18}, and we will provide further evidences in the next subsection (see Corollary \ref{cor:Kstqsm}).

It may be interesting to consider further generalization of Conjecture \ref{conj:BRKst}.
According to systematic studies of Fano 3-folds of codimension 2 \cite{Oka14, Oka18, Oka20b}, existence of many birationally bi-rigid Fano 3-folds are verified.
Here a Fano variety $X$ of Picard number $1$ is {\it birationally bi-rigid} if there exists a Fano variety $X'$ of Picard number $1$ which is birational but not isomorphic to $X$, and up to isomorphism $\{X, X'\}$ are all the Mori fiber space in the birational equivalence class of $X$.
Extending bi-rigidity, tri-rigidity and so on, notion of solid Fano variety in introduced in \cite{AO18}: a Fano variety of Picard number $1$ is {\it solid} if any Mori fiber space in the birational equivalence class is a Fano variety of Picard number $1$.  
Solid Fano varieties are expected to behave nicely in moduli (\cite{Zhu20a}).
Only a few evidences are known (\cite{KOW19}) for the following question.

\begin{Question}
Is it true that any solid Fano variety is K-stable?
\end{Question}

\subsubsection{On K-stability for 95 families}
\label{sec:Kstqsm}

For strictly birationally rigid members of the 95 families, we are unable to conclude K-stability by Theorem \ref{mainthm}, except for those treated in Theorem \ref{thm:qsmWHKE}.
The aim of this subsection is to prove K-stability for all the quasi-smooth members of suitable families indexed by $\msI_{\BR}$.
This will be done by combining the inequality $\alpha \ge 1/2$ obtained by Theorem \ref{mainthm} and an additional information on local movable alpha invariants which are introduced below.

\begin{Def}
Let $X$ be a Fano variety of Picard number $1$ and $\msp \in X$ a point.
For a non-empty linear system $\mcM$ on $X$, we define $\lambda_{\mcM} \in \mbQ_{> 0}$ to be the rational number such that $\mcM \sim_{\mbQ} - \lambda_{\mcM} K_X$.
For a movable linear system $\mcM$ on $X$ and a positive rational number $\mu$, we define the {\it movable log canonical threshold} of $(X, \mu \mcM)$ at $\msp$ to be the number
\[
\lct^{\mov}_{\msp} (X; \mu \mcM) = \sup \{\, c \in \mbQ_{\ge 0} \mid \text{$(X, c \mu \mcM)$ is log canonical at $\msp$} \, \},
\]
and then we define the {\it movable alpha invariant of $X$ at} $\msp$ as
\[
\alpha^{\mov}_{\msp} (X) = \inf \{\, \lct_{\msp}^{\mov} (X, \lambda_{\mcM}^{-1} \mcM) \mid \text{$\mcM$ is a movable linear system on $X$} \}.
\]
\end{Def}

\begin{Prop}[{cf.\ \cite[Corollary 3.1]{SZ19}}] \label{prop:Kstmovalpha}
Let $X$ be a quasi-smooth Fano $3$-fold weighted hypersurface of index $1$.
Assume that, for any maximal center $\msp \in X$, we have 
\[
\alpha^{\mov}_{\msp} (X) \ge 1 \quad \text{and} \quad
(\alpha^{\mov}_{\msp} (X), \alpha_{\msp} (X)) \ne (1, 1/2).
\]
Then $X$ is K-stable.
\end{Prop}

\begin{proof}
By the main result of \cite{CP17} (cf.\ Remark \ref{rem:maxcent}), we have $\alpha^{\mov}_{\msq} (X) \ge 1$ for any point $\msq \in X$ which is not a maximal center.
It follows that the pair $(X, \lambda_{\mcM}^{-1} \mcM)$ is log canonical for any movable linear system $\mcM$ on $X$.
Combining this with the inequality $\alpha (X) \ge 1/2$ obtained by Theorem \ref{mainthm}, we see that $X$ is K-semistable by \cite[Theorem 1.2]{SZ19}.

Suppose that $X$ is not K-stable.
Then, by \cite[Corollary 3.1]{SZ19}, there exists a prime divisor $E$ over $X$, a movable linear system $\mcM \sim_{\mbQ} - n K_X$ and an effective $\mbQ$-divisor $D \sim_{\mbQ} - K_X$ such that $E$ is a log canonical place of $(X, \frac{1}{n} \mcM)$ and $(X, \frac{1}{2} D)$.
Note that the center $\Gamma$ of $E$ on $X$ is necessarily a maximal center, and a maximal center on $X$ is a BI center.
Thus $\Gamma = \msp$ is a BI center, and this implies $(\alpha_{\msp}^{\mov} (X), \alpha_{\msp} (X)) = (1, 1/2)$. 
This is impossible by the assumption.
Therefore $X$ is K-stable.
\end{proof}

We define
\[
\msI'_{\mathrm{K}} = \{6, 8, 15, 16, 17, 26, 27, 30, 36, 41, 47, 48, 54, 56, 60, 65, 68\} \subset \msI_{\BR}.
\]

\begin{Thm} \label{thm:qsmWHKst}
Let $X$ be a member of a family $\mcF_{\msi}$ with $\msi \in \msI'_{\mathrm{K}}$.
Then, for any BI center $\msp \in X$, we have 
\begin{equation} \label{eq:qsmWHKst}
\alpha^{\operatorname{mov}}_{\msp} (X) \ge 1 \quad \text{and} \quad \alpha_{\msp} (X) > \frac{1}{2}.
\end{equation}
In particular, $X$ is K-stable.
\end{Thm}

\begin{proof}
Let $X$ be a member of  $\mcF_{\msi}$, where $\msi \in \msI'_{\mathrm{K}}$.
We first show that the inequalities \eqref{thm:qsmWHKst} are satisfied.

Suppose $\msi \in \{16, 17, 26, 27, 36, 47, 48, 54, 65\}$.
Then, the subscript $\diamondsuit$ is given in the 4th column of Table \ref{table:main} for any BI center on $X$.
By Proposition \ref{prop:lctsingptL}, we have $\alpha_{\msp} (X) \ge 1$, and hence $\alpha^{\mov}_{\msp} (X) \ge 1$, for any BI center $\msp \in X$.

Suppose $\msi \in \{6, 15, 30, 41, 68\}$.
In this case, $X$ admits two QI centers of equal singularity type, and does not admit any other BI center.
By the proof of Proposition \ref{prop:ndQIcent} (see \S \ref{sec:ndQI2eqtypes}), we have $\alpha_{\msp} (X) \ge 1$ for any QI center $\msp \in X$.
In particular, we have $\alpha^{\mov}_{\msp} (X) \ge \alpha_{\msp} (X) \ge 1$.

Suppose $\msi \in \{8, 56, 60\}$.
In this case $X$ admits a unique BI center and it is a QI center.
The inequalities \eqref{eq:qsmWHKst} follows from Remark \ref{rem:QIcent-8} and Propositions \ref{prop:No56QIex}, \ref{prop:No60QIex}. 
This completes the verifications for the inequalities \eqref{eq:qsmWHKst}.

The K-stability of $X$ follows from the inequalities \eqref{eq:qsmWHKst}, Theorem \ref{mainthm}, and Proposition \ref{prop:Kstmovalpha}.
\end{proof}

We define
\[
\msI_{\mathrm{K}} := \msI'_{\mathrm{K}} \sqcup \msI_{\mathrm{KE}}.
\]
Note that $|\msI_{\mathrm{K}}| = 73$.
Combining Theorems \ref{thm:qsmWHKE}, \ref{thm:qsmWHKst} and Corollary \ref{maincor:Kst}, we obtain the K-stability of arbitrary quasi-smooth member for families indexed by $\msI_{\mathrm{K}}$.

\begin{Cor} \label{cor:Kstqsm}
Let $X$ be a member of a family $\mcF_{\msi}$ with $\msi \in \msI_{\mathrm{K}}$.
Then $X$ is K-stable.
\end{Cor}

\subsection{Further computations of alpha invariants}

In this section, we compute local alpha invariants for a few families in order to give better lower bounds.
The results obtained in this section are used only in the proof of Theorem \ref{thm:qsmWHKst}. 

\begin{Prop} \label{prop:No56QIex}
Let $X$ be a member of the family $\mcF_{56}$ and $\msp = \msp_w \in X$ be the singular point of type $\frac{1}{11} (1, 3, 8)$.
Then,
\[
\alpha_{\msp} (X) \ge \frac{2}{3} \quad \text{and} \quad 
\alpha^{\mov}_{\msp} (X) \ge 1.
\]
\end{Prop}

\begin{proof}
We set $\msp = \msp_w$.
We can write the defining polynomial of $X$ as
\[
F = w^2 y + f_{13} + f_{24},
\]
where $f_i = f_i (x, y, z, t)$ is a quasi-homogeneous polynomial of degree $i$.
By the quasi-smoothness of $X$, we have $t^3 \in F$.
It is easy to see that $F (x, 0, z, t, 1) \in (x, z, t)^3$.
It follows that $\omult_{\msp} (H_y) = 3$, which in particular implies $\lct_{\msp} (X; \frac{1}{2} H_y) \ge 2/3$.

Let $D \in |A|_{\mbQ}$ be an irreducible $\mbQ$-divisor other than $\frac{1}{2} H_y$.
We can take a $\mbQ$-divisor $T \in |3 A|_{\mbQ}$ such that $\omult_{\msp} (T) \ge 1$ and $\Supp (T)$ does not contain any component of the effective $1$-cycle $D \cdot H_y$.
We have
\[
3 \omult_{\msp} (D) \le 11 (D \cdot H_y \cdot T) = 3.
\]
This shows $\lct_{\msp} (X;D) \ge 1$.
Therefore $\alpha^{\mov}_{\msp} (X) \ge 1$ and $\alpha_{\msp} (X) \ge 2/3$. 
\end{proof}

\begin{Prop} \label{prop:No60QIex}
Let $X$ be a member of the family $\mcF_{60}$ and let $\msp = \msp_w$ be the singular point of type $\frac{1}{9} (1, 4, 5)$.
Then 
\[
\alpha_{\msp} (X) = 1.
\]
\end{Prop}

\begin{proof}
We set $S = H_x \sim A$, $T = H_z$ and $\Gamma = S \cap T = (x = z = 0)_X$.
Let $\rho = \rho_{\msp} \colon \breve{U}_{\msp} \to U_{\msp}$ be the orbifold chart of $\msp \in X$, and we set $\breve{\Gamma} = (\breve{x} = \breve{z} = 0) \subset \breve{U}_{\msp}$.
We can write the defining polynomial of $X$ as
\[
F = w^2 t + w f_{15} + f_{24},
\]
where $f_i = f_i (x, y, z, t)$ is a quasi-homogeneous polynomial of degree $i$.
By the quasi-smoothness of $X$, we have $t^4, y^6 \in F$, and we may assume $\coeff_F (t^4) = \coeff_F (y^6) = 1$ by rescaling $y$ and $t$.
We set $\lambda = \coeff_F (t^2 y^3) \in \mbC$.
Then,  
\[
\begin{split}
\Gamma &\cong (w^2 t + t^4 + \lambda t^2 y^3 + y^6 = 0) \subset \mbP (3, 10, 17)_{y, t, w}, \\
\breve{\Gamma} & \cong (\breve{w}^2 \breve{t} + \breve{t}^4 + \lambda \breve{t}^{\hspace{0.3mm} 2} \breve{y}^3 + \breve{y}^6 = 0) \subset \mbA^3_{\breve{y}, \breve{t}, \breve{w}}.
\end{split}
\]
It is easy to see that $\Gamma$ is an irreducible and reduced curve, and $\mult_{\breve{\msp}} (\breve{\Gamma}) = 1$, where $\breve{\msp} = o \in \mbA^3$ is the preimage of $\msp$ via $\rho$.

We see that $H_x$ is quasi-smooth at $\msp$, and hence $\lct_{\msp} (X; H_x) = 1$.
Therefore, we have $\alpha_{\msp} (X) \ge 1$ by Lemma \ref{lem:exclL}.
\end{proof}

\section{The table}
\label{chap:table}

The list of the 93 families together with their basic information are summarized in Table \ref{table:main}, and we explain the contents.

The first two columns indicate basic information of each family and the anticanonical degree $(A^3) = (-K_X)^3$ is indicated in the 3rd column.

In the 4th column, the number and the singularities of $X$ are described.
The symbol $\frac{1}{r} [a, r-a]$ stands for the cyclic quotient singularity of type $\frac{1}{r} (1, a, r-a)$, where $1 \le a \le r/2$.
Moreover the symbols $\frac{1}{2}$, $\frac{1}{3}$ and $\frac{1}{4}$ stand for singularities of types $\frac{1}{2} (1, 1, 1)$, $\frac{1}{3} (1, 1, 2)$ and $\frac{1}{4} (1, 1, 3)$.
The superscript $\Q$ and $\E$ indicates that the corresponding singular point $\msp$ is a QI center and EI center, respectively (see \S \ref{sec:defBI} for definitions). 
The meaning of the subscripts is explained as follows.

\begin{itemize}
\item The subscript $\heartsuit$ indicates that $\alpha_{\msp} (X) \ge 1$ is proved by Proposition \ref{prop:singptCP}.
\item The subscript $\diamondsuit$ (resp.\ $\diamondsuit'$) indicates that $\alpha_{\msp} (X) \ge 1$ (resp.\ $\alpha_{\msp} (X) \ge 2/3$) is proved by Proposition \ref{prop:lctsingptL}.
\item The subscript $\clubsuit$ indicates that $\alpha_{\msp} (X) \ge 1/2$ is proved by Proposition \ref{prop:singptic}.
\item The subscript $\spadesuit$ indicates that $\alpha_{\msp} (X) \ge 1/2$ is proved by Proposition \ref{prop:singptrem}.
\end{itemize}

In Theorem \ref{mainthm}, any birational superrigid member of each of the 95 families is proved.
Apart from this main result, we have results on the existence of KE metrics or K-stability for any quasi-smooth member of suitable families. 
In the right-most column the mark ``KE" and ``K" are given and their meanings are as follows.
\begin{itemize}
\item The mark ``KE" in the right-most column means that any quasi-smooth member admits a KE metric and is K-stable (see \S \ref{sec:KEmetric}).
\item The mark ``K" in the right-most column means that any quasi-smooth member is K-stable (see \S \ref{sec:Kstqsm}).
\end{itemize}

\begin{Rem} \label{rem:whatisleft}
We explain what is left about K-stability of quasi-smooth Fano $3$-fold weighted hypersurfaces of index $1$.

As it is explained in \S \ref{sec:relresults}, the result \cite{LXZ22} obtained after this paper is written in particular implies that the K-stability of a quasi-smooth Fano $3$-fold weighted hypersurface is equivalent to the existence of a KE metric.
It follows that the meaning of the mark ``KE" and ``K" in the right-most column of Table \ref{table:main} are the same: it indicates that any quasi-smooth member is K-stable (and admits a KE metric). 
All in all, we obtain the following results in this article:
\begin{itemize}
\item Any quasi-smooth member in a family $\mcF_{\msi}$ with a mark ``K" or ``KE" in the right-most column of Table \ref{table:main} is K-stable.
\item Any quasi-smooth and birationally superrigid member in a family $\mcF_{\msi}$ with a blank right-most column in Table \ref{table:main} is K-stable.
\end{itemize}
Therefore, it remains to determine K-stability for quasi-smooth members in a family $\mcF_{\msi}$ with a blank right-most column that are not birationally superrigid.
\end{Rem}


\begingroup
\renewcommand{\arraystretch}{1.35}
\begin{longtable}{clclc}
\caption{The 93 families}
\label{table:main}
\\
\hline\hline
No.\ & $X_d \subset \mbP (1,a_1,a_2,a_3,a_4)$ & $(A^3)$ & Singular points & \\ 
\hline\hline
\endfirsthead
\hline\hline
No.\ & $X_d \subset \mbP (1,a_1,a_2,a_3,a_4)$ & $(A^3)$ & Singular points & \\ 
\hline\hline
\endhead
\rowcolor{lightgray}
2 & $X_5 \subset \mbP (1,1,1,1,2)$ & $\frac{5}{2}$ & $\frac{1}{2}^{\Q}$ & \\
4 & $X_6 \subset \mbP (1,1,1,2,2)$ & $\frac{3}{2}$ & $3 \ntimes \frac{1}{2}^{\Q}$ & \\
\rowcolor{lightgray}
5 & $X_7 \subset \mbP (1,1,1,2,3)$ & $\frac{7}{6}$ & $\frac{1}{2}^{\Q}, \ \frac{1}{3}^{\Q}$ & \\
6 & $X_8 \subset \mbP (1,1,1,2,4)$ & $1$ & $2 \ntimes \frac{1}{2}^{\Q}$ & K \\
\rowcolor{lightgray}
7 & $X_8 \subset \mbP (1,1,2,2,3)$ & $\frac{2}{3}$ & $4 \ntimes \frac{1}{2}^{\E}, \ \frac{1}{3}^{\Q}$ & \\
8 & $X_9 \subset \mbP (1,1,1,3,4)$ & $\frac{3}{4}$ & $\frac{1}{4}^{\Q}$ & K \\
\rowcolor{lightgray}
9 & $X_9 \subset \mbP (1,1,2,3,3)$ & $\frac{1}{2}$ & $\frac{1}{2}_{\heartsuit}, \ 3 \ntimes \frac{1}{3}^{\Q}$ & \\
10 & $X_{10} \subset \mbP (1,1,1,3,5)$ & $\frac{2}{3}$ & $\frac{1}{3}_{\clubsuit}$ & KE \\
\rowcolor{lightgray}
11 & $X_{10} \subset \mbP (1,1,2,2,5)$ & $\frac{1}{2}$ & $5 \ntimes \frac{1}{2}_{\heartsuit}$ & KE \\
12 & $X_{10} \subset \mbP (1,1,2,3,4)$ & $\frac{5}{12}$ & $2 \ntimes \frac{1}{2}_{\spadesuit}, \ \frac{1}{3}^{\Q}, \ \frac{1}{4}^{\Q}$ & \\
\rowcolor{lightgray}
13 & $X_{11} \subset \mbP (1,1,2,3,5)$ & $\frac{11}{30}$ & $\frac{1}{2}_{\spadesuit}, \ \frac{1}{3}^{\Q}, \ \frac{1}{5} [2, 3]^{\Q}$ & \\
14 & $X_{12} \subset \mbP (1,1,1,4,6)$ & $\frac{1}{2}$ & $2 \ntimes \frac{1}{2}_{\heartsuit}$ & KE \\
\rowcolor{lightgray}
15 & $X_{12} \subset \mbP (1,1,2,3,6)$ & $\frac{1}{3}$ & $2 \ntimes \frac{1}{2}_{\diamondsuit}, \ 2 \ntimes \frac{1}{3}_{\diamondsuit}^{\Q}$ & K \\ 
16 & $X_{12} \subset \mbP (1,1,2,4,5)$ & $\frac{3}{10}$ & $3 \ntimes \frac{1}{2}_{\heartsuit}, \ \frac{1}{5} [1, 4]^{\Q}_{\diamondsuit}$ & K \\
\rowcolor{lightgray}
17 & $X_{12} \subset \mbP (1,1,3,4,4)$ & $\frac{1}{4}$ & $3 \ntimes \frac{1}{4}^{\Q}_{\diamondsuit}$ & K \\
18 & $X_{12} \subset \mbP (1,2,2,3,5)$ & $\frac{1}{5}$ & $6 \ntimes \frac{1}{2}_{\heartsuit}, \ \frac{1}{5} [2, 3]^{\Q}$ & \\
\rowcolor{lightgray}
19 & $X_{12} \subset \mbP (1,2,3,3,4)$ & $\frac{1}{6}$ & $3 \ntimes \frac{1}{2}_{\heartsuit}, \ 4 \ntimes \frac{1}{3}_{\heartsuit}$ & KE \\
20 & $X_{13} \subset \mbP (1,1,3,4,5)$ & $\frac{13}{60}$ & $\frac{1}{3}^{\E}, \ \frac{1}{4}^{\Q}, \ \frac{1}{5} [1, 4]^{\Q}$ & \\
\rowcolor{lightgray}
21 & $X_{14} \subset \mbP (1,1,2,4,7)$ & $\frac{1}{4}$ & $3 \ntimes \frac{1}{2}_{\heartsuit}, \ \frac{1}{4}_{\diamondsuit}$ & KE \\
22 & $X_{14} \subset \mbP (1,2,2,3,7)$ & $\frac{1}{6}$ & $7 \ntimes \frac{1}{2}_{\heartsuit}, \ \frac{1}{3}_{\heartsuit}$ & KE \\
\rowcolor{lightgray}
23 & $X_{14} \subset \mbP (1,2,3,4,5)$ & $\frac{7}{60}$ & $3 \ntimes \frac{1}{2}_{\clubsuit}, \ \frac{1}{3}_{\clubsuit}, \ \frac{1}{4}^{\E}_{\clubsuit}, \ \frac{1}{5} [2, 3]^{\Q}$ & \\
24 & $X_{15} \subset \mbP (1,1,2,5,7)$ & $\frac{3}{14}$ & $\frac{1}{2}_{\spadesuit}, \ \frac{1}{7} [2, 5]^{\Q}$ & \\
\rowcolor{lightgray}
25 & $X_{15} \subset \mbP (1,1,3,4,7)$ & $\frac{5}{28}$ & $\frac{1}{4}^{\Q}, \ \frac{1}{7} [3, 4]^{\Q}$ & \\
26 & $X_{15} \subset \mbP (1,1,3,5,6)$ & $\frac{1}{6}$ & $2 \ntimes \frac{1}{3}_{\heartsuit}, \ \frac{1}{6} [1, 5]^{\Q}_{\diamondsuit}$ & K \\
\rowcolor{lightgray}
27 & $X_{15} \subset \mbP (1,2,3,5,5)$ & $\frac{1}{10}$ & $\frac{1}{2}_{\spadesuit}, 3 \ntimes \frac{1}{5} [2, 3]^{\Q}_{\diamondsuit}$ & K \\
28 & $X_{15} \subset \mbP (1,3,3,4,5)$ & $\frac{1}{12}$ & $5 \ntimes \frac{1}{3}_{\heartsuit}, \ \frac{1}{4}_{\heartsuit}$ & KE \\
\rowcolor{lightgray}
29 & $X_{16} \subset \mbP (1,1,2,5,8)$ & $\frac{1}{5}$ & $2 \ntimes \frac{1}{2}_{\clubsuit}, \ \frac{1}{5} [2, 3]_{\clubsuit}$ & KE \\
30 & $X_{16} \subset \mbP (1,1,3,4,8)$ & $\frac{1}{6}$ & $\frac{1}{3}_{\heartsuit}, \ 2 \ntimes \frac{1}{4}^{\Q}$ & K \\
\rowcolor{lightgray}
31 & $X_{16} \subset \mbP (1,1,4,5,6)$ & $\frac{2}{15}$ & $\frac{1}{2}_{\clubsuit}, \ \frac{1}{5} [1, 4]^{\Q}, \ \frac{1}{6} [1, 5]^{\Q}$ & \\
32 & $X_{16} \subset \mbP (1,2,3,4,7)$ & $\frac{2}{21}$ & $4 \ntimes \frac{1}{2}_{\heartsuit}, \ \frac{1}{3}_{\spadesuit}, \ \frac{1}{7} [3, 4]^{\Q}$ & \\
\rowcolor{lightgray}
33 & $X_{17} \subset \mbP (1,2,3,5,7)$ & $\frac{17}{210}$ & $\frac{1}{2}_{\spadesuit}, \ \frac{1}{3}_{\clubsuit}, \ \frac{1}{5} [2, 3]^{\Q}, \ \frac{1}{7} [2, 5]^{\Q}$ & \\
34 & $X_{18} \subset \mbP (1,1,2,6,9)$ & $\frac{1}{6}$ & $3 \ntimes \frac{1}{2}_{\heartsuit}, \ \frac{1}{3}_{\heartsuit}$ & KE \\
\rowcolor{lightgray}
35 & $X_{18} \subset \mbP (1,1,3,5,9)$ & $\frac{2}{15}$ & $2 \ntimes \frac{1}{3}_{\heartsuit}, \ \frac{1}{5} [1, 4]_{\diamondsuit}$ & KE \\
36 & $X_{18} \subset \mbP (1,1,4,6,7)$ & $\frac{3}{28}$ & $\frac{1}{2}_{\heartsuit}, \ \frac{1}{4}^{\E}_{\diamondsuit}, \ \frac{1}{7} [1, 6]^{\Q}_{\diamondsuit}$ & K \\
\rowcolor{lightgray}
37 & $X_{18} \subset \mbP (1,2,3,4,9)$ & $\frac{1}{12}$ & $4 \ntimes \frac{1}{2}_{\heartsuit}, \ 2 \ntimes \frac{1}{3}_{\clubsuit}, \  \frac{1}{4}_{\heartsuit}$ & KE \\
38 & $X_{18} \subset \mbP (1,2,3,5,8)$ & $\frac{3}{40}$ & $2 \ntimes \frac{1}{2}_{\heartsuit}, \ \frac{1}{5}  [2, 5]^{\Q}, \frac{1}{8} [3, 5]^{\Q}$ & \\
\rowcolor{lightgray}
39 & $X_{18} \subset \mbP (1,3,4,5,6)$ & $\frac{1}{20}$ & $ \frac{1}{2}_{\heartsuit}, \ 3 \ntimes \frac{1}{3}_{\clubsuit}, \ \frac{1}{4}_{\clubsuit}, \ \frac{1}{5} [1, 4]_{\heartsuit}$ & KE \\
40 & $X_{19} \subset \mbP (1,3,4,5,7)$ & $\frac{19}{420}$ & $\frac{1}{3}_{\spadesuit}, \ \frac{1}{4}_{\clubsuit}, \ \frac{1}{5} [2, 3]^{\E}_{\clubsuit}, \  \frac{1}{7} [3, 4]^{\Q}$ & \\
\rowcolor{lightgray}
41 & $X_{20} \subset \mbP (1,1,4,5,10)$ & $\frac{1}{10}$ & $\frac{1}{2}_{\heartsuit}, \ 2 \ntimes \frac{1}{5} [1, 4]^{\Q}_{\diamondsuit}$ & K \\
42 & $X_{20} \subset \mbP (1,2,3,5,10)$ & $\frac{1}{15}$ & $2 \ntimes \frac{1}{2}_{\heartsuit}, \ \frac{1}{3}_{\heartsuit}, \ 2 \ntimes \frac{1}{5} [2, 3]^{\Q}$ & KE \\
\rowcolor{lightgray}
43 & $X_{20} \subset \mbP (1,2,4,5,9)$ & $\frac{1}{18}$ & $5 \ntimes \frac{1}{2}_{\heartsuit}, \ \frac{1}{9} [4, 5]^{\Q}$ & \\
44 & $X_{20} \subset \mbP (1,2,5,6,7)$ & $\frac{1}{21}$ & $3 \ntimes \frac{1}{2}_{\heartsuit}, \ \frac{1}{6} [1, 5]^{\E}_{\diamondsuit}, \  \frac{1}{7} [2, 5]^{\Q}_{\diamondsuit}$ & KE \\
\rowcolor{lightgray}
45 & $X_{20} \subset \mbP (1,3,4,5,8)$ & $\frac{1}{24}$ & $\frac{1}{3}_{\heartsuit}, \ 2 \ntimes \frac{1}{4}_{\heartsuit}, \ \frac{1}{8} [3, 5]^{\Q}_{\diamondsuit}$ & KE \\
46 & $X_{21} \subset \mbP (1,1,3,7,10)$ & $\frac{1}{10}$ & $\frac{1}{10} [3, 7]^{\Q}$ & \\
\rowcolor{lightgray}
47 & $X_{21} \subset \mbP (1,1,5,7,8)$ & $\frac{3}{40}$ & $\frac{1}{5} [2, 3]_{\spadesuit}, \ \frac{1}{8} [1, 7]^{\Q}_{\diamondsuit}$ & K \\
48 & $X_{21} \subset \mbP (1,2,3,7,9)$ & $\frac{1}{18}$ & $\frac{1}{2}_{\spadesuit}, \  2 \ntimes \frac{1}{3}_{\heartsuit}, \ \frac{1}{9} [2, 7]^{\Q}_{\diamondsuit}$ & K \\
\rowcolor{lightgray}
49 & $X_{21} \subset \mbP (1,3,5,6,7)$ & $\frac{1}{30}$ & $3 \ntimes \frac{1}{3}_{\heartsuit}, \ \frac{1}{5} [2, 3]_{\spadesuit}, \ \frac{1}{6} [1, 5]_{\heartsuit}$ & KE \\
50 & $X_{22} \subset \mbP (1,1,3,7,11)$ & $\frac{2}{21}$ & $\frac{1}{3}_{\heartsuit}, \ \frac{1}{7} [3, 4]_{\clubsuit}$ & KE \\
\rowcolor{lightgray}
51 & $X_{22} \subset \mbP (1,1,4,6,11)$ & $\frac{1}{12}$ & $\frac{1}{2}_{\heartsuit}, \ \frac{1}{4}_{\heartsuit}, \ \frac{1}{6} [1, 5]_{\diamondsuit}$ & KE \\
52 & $X_{22} \subset \mbP (1,2,4,5,11)$ & $\frac{1}{20}$ & $5 \ntimes \frac{1}{2}_{\heartsuit}, \ \frac{1}{4}_{\heartsuit}, \ \frac{1}{5} [1, 4]_{\heartsuit}$ & KE \\
\rowcolor{lightgray}
53 & $X_{24} \subset \mbP (1,1,3,8,12)$ & $\frac{1}{12}$ & $2 \ntimes \frac{1}{3}_{\heartsuit}, \ \frac{1}{4}_{\heartsuit}$ & KE \\
54 & $X_{24} \subset \mbP (1,1,6,8,9)$ & $\frac{1}{18}$ & $\frac{1}{2}_{\heartsuit}, \ \frac{1}{3}_{\heartsuit}, \ \frac{1}{9} [1, 8]^{\Q}_{\diamondsuit}$ & K \\
\rowcolor{lightgray}
55 & $X_{24} \subset \mbP (1,2,3,7,12)$ & $\frac{1}{21}$ & $2 \ntimes \frac{1}{2}_{\heartsuit}, \ 2 \ntimes \frac{1}{3}_{\heartsuit}, \  \frac{1}{7} [2, 5]_{\diamondsuit}$ & KE \\
56 & $X_{24} \subset \mbP (1,2,3,8,11)$ & $\frac{1}{22}$ & $3 \ntimes \frac{1}{2}_{\heartsuit}, \ \frac{1}{11} [3, 8]^{\Q}$ & K \\
\rowcolor{lightgray}
57 & $X_{24} \subset \mbP (1,3,4,5,12)$ & $\frac{1}{30}$ & $2 \ntimes \frac{1}{3}_{\heartsuit}, \ 2 \ntimes \frac{1}{4}_{\heartsuit}, \  \frac{1}{5} [2, 3]_{\heartsuit}$ & KE \\
58 & $X_{24} \subset \mbP (1,3,4,7,10)$ & $\frac{1}{35}$ & $\frac{1}{2}_{\heartsuit}, \ \frac{1}{7} [3, 4]^{\Q}, \ \frac{1}{10} [3, 7]^{\Q}$ &  \\
\rowcolor{lightgray}
59 & $X_{24} \subset \mbP (1,3,6,7,8)$ & $\frac{1}{42}$ & $\frac{1}{2}_{\heartsuit}, \ 4 \ntimes \frac{1}{3}_{\heartsuit}, \ \frac{1}{7} [1, 6]_{\heartsuit}$ & KE \\
60 & $X_{24} \subset \mbP (1,4,5,6,9)$ & $\frac{1}{45}$ & $2 \ntimes \frac{1}{2}_{\heartsuit}, \ \frac{1}{3}_{\heartsuit}, \ \frac{1}{5} [1, 4]_{\heartsuit}, \ \frac{1}{9} [4, 5]^{\Q}$ & K \\
\rowcolor{lightgray}
61 & $X_{25} \subset \mbP (1,4,5,7,9)$ & $\frac{5}{252}$ & $\frac{1}{4}_{\clubsuit}, \ \frac{1}{7} [2, 5]^{\EI}_{\diamondsuit}, \ \frac{1}{9} [4, 5]_{\diamondsuit}$ & KE \\
62 & $X_{26} \subset \mbP (1,1,5,7,13)$ & $\frac{2}{35}$ & $\frac{1}{5} [2, 3]_{\spadesuit}, \ \frac{1}{7} [1, 6]_{\diamondsuit}$ & KE \\
\rowcolor{lightgray}
63 & $X_{26} \subset \mbP (1,2,3,8,13)$ & $\frac{1}{24}$ & $3 \ntimes \frac{1}{2}_{\heartsuit}, \ \frac{1}{3}_{\heartsuit}, \ \frac{1}{8} [3, 5]_{\clubsuit}$ & KE \\
64 & $X_{26} \subset \mbP (1,2,5,6,13)$ & $\frac{1}{30}$ & $4 \ntimes \frac{1}{2}_{\heartsuit}, \ \frac{1}{5} [2, 3]_{\clubsuit}, \  \frac{1}{6} [1, 5]_{\heartsuit}$ & KE \\
\rowcolor{lightgray}
65 & $X_{27} \subset \mbP (1,2,5,9,11)$ & $\frac{3}{110}$ & $\frac{1}{2}_{\spadesuit}, \ \frac{1}{5} [1, 4]_{\heartsuit}, \ \frac{1}{11} [2, 9]^{\Q}_{\diamondsuit}$ & K \\
66 & $X_{27} \subset \mbP (1,5,6,7,9)$ & $\frac{1}{70}$ & $\frac{1}{3}_{\heartsuit}, \ \frac{1}{5} [1, 4]_{\clubsuit}, \ \frac{1}{6} [1, 5]_{\heartsuit}, \ \frac{1}{7} [2, 5]_{\heartsuit}$ & KE \\
\rowcolor{lightgray}
67 & $X_{28} \subset \mbP (1,1,4,9,14)$ & $\frac{1}{18}$ & $\frac{1}{2}_{\heartsuit}, \ \frac{1}{9} [4, 5]_{\spadesuit}$ & KE \\
68 & $X_{28} \subset \mbP (1,3,4,7,14)$ & $\frac{1}{42}$ & $\frac{1}{2}_{\heartsuit}, \ \frac{1}{3}_{\clubsuit}, \ 2 \ntimes \frac{1}{7} [3, 4]^{\Q}_{\diamondsuit}$ & K \\
\rowcolor{lightgray}
69 & $X_{28} \subset \mbP (1,4,6,7,11)$ & $\frac{1}{66}$ & $2 \ntimes \frac{1}{2}_{\heartsuit}, \ \frac{1}{6} [1, 5]_{\heartsuit}, \  \frac{1}{11} [4, 7]^{\Q}_{\diamondsuit}$ & KE \\
70 & $X_{30} \subset \mbP (1,1,4,10,15)$ & $\frac{1}{20}$ & $\frac{1}{2}_{\heartsuit}, \ \frac{1}{4}_{\heartsuit}, \ \frac{1}{5} [1, 4]_{\heartsuit}$ & KE \\
\rowcolor{lightgray}
71 & $X_{30} \subset \mbP (1,1,6,8,15)$ & $\frac{1}{24}$ & $\frac{1}{2}_{\heartsuit}, \ \frac{1}{3}_{\heartsuit}, \ \frac{1}{8} [1, 7]_{\diamondsuit}$ & KE \\
72 & $X_{30} \subset \mbP (1,2,3,10,15)$ & $\frac{1}{30}$ & $3 \ntimes \frac{1}{2}_{\heartsuit}, \ 2 \ntimes \frac{1}{3}_{\heartsuit}, \  \frac{1}{5} [2, 3]_{\heartsuit}$ & KE \\
\rowcolor{lightgray}
73 & $X_{30} \subset \mbP (1,2,6,7,15)$ & $\frac{1}{42}$ & $5 \ntimes \frac{1}{2}_{\heartsuit}, \ \frac{1}{3}_{\heartsuit}, \ \frac{1}{7} [1, 6]_{\heartsuit}$ & KE \\
74 & $X_{30} \subset \mbP (1,3,4,10,13)$ & $\frac{1}{52}$ & $\frac{1}{2}_{\heartsuit}, \ \frac{1}{4}_{\heartsuit}, \ \frac{1}{13} [3, 10]^{\Q}_{\diamondsuit}$ & KE \\
\rowcolor{lightgray}
75 & $X_{30} \subset \mbP (1,4,5,6,15)$ & $\frac{1}{60}$ & $2 \ntimes \frac{1}{2}_{\heartsuit}, \ \frac{1}{3}_{\heartsuit}, \ \frac{1}{4}_{\heartsuit}, \ 2 \ntimes \frac{1}{5} [1, 4]_{\heartsuit}$ & KE \\ 
76 & $X_{30} \subset \mbP (1,5,6,8,11)$ & $\frac{1}{88}$ & $\frac{1}{2}_{\heartsuit}, \ \frac{1}{8} [3, 5]^{\E}_{\diamondsuit}, \ \frac{1}{11} [5, 6]^{\Q}_{\diamondsuit}$ & KE \\
\rowcolor{lightgray}
77 & $X_{32} \subset \mbP (1,2,5,9,16)$ & $\frac{1}{45}$ & $2 \ntimes \frac{1}{2}_{\heartsuit}, \ \frac{1}{5} [1, 4]_{\heartsuit}, \  \frac{1}{9} [2, 7]_{\diamondsuit}$ & KE \\
78 & $X_{32} \subset \mbP (1,4,5,7,16)$ & $\frac{1}{70}$ & $2 \ntimes \frac{1}{4}_{\heartsuit}, \ \frac{1}{5} [1, 4]_{\heartsuit}, \  \frac{1}{7} [2, 5]_{\heartsuit}$ & KE \\
\rowcolor{lightgray}
79 & $X_{33} \subset \mbP (1,3,5,11,14)$ & $\frac{1}{70}$ & $\frac{1}{5} [1, 4]_{\heartsuit}, \ \frac{1}{14} [3, 11]^{\Q}_{\diamondsuit}$ & KE \\
80 & $X_{34} \subset \mbP (1,3,4,10,17)$ & $\frac{1}{60}$ & $\frac{1}{2}_{\heartsuit}, \ \frac{1}{3}_{\clubsuit}, \ \frac{1}{4}_{\heartsuit}, \  \frac{1}{10} [3, 7]_{\diamondsuit}$ & KE \\
\rowcolor{lightgray}
81 & $X_{34} \subset \mbP (1,4,6,7,17)$ & $\frac{1}{84}$ & $2 \ntimes \frac{1}{2}_{\heartsuit}, \ \frac{1}{4}_{\heartsuit}, \ \frac{1}{6} [1, 5]_{\heartsuit}, \ \frac{1}{7} [3, 4]_{\heartsuit}$ & KE \\
82 & $X_{36} \subset \mbP (1,1,5,12,18)$ & $\frac{1}{30}$ & $\frac{1}{5} [2, 3]_{\spadesuit}, \frac{1}{6} [1, 5]_{\heartsuit}$ & KE \\
\rowcolor{lightgray}
83 & $X_{36} \subset \mbP (1,3,4,11,18)$ & $\frac{1}{66}$ & $\frac{1}{2}_{\heartsuit}, \ 2 \ntimes \frac{1}{3}_{\heartsuit}, \ \frac{1}{11} [4, 7]_{\diamondsuit}$ & KE \\
84 & $X_{36} \subset \mbP (1,7,8,9,12)$ & $\frac{1}{168}$ & $\frac{1}{3}_{\heartsuit}, \ \frac{1}{4}_{\heartsuit}, \ \frac{1}{7} [2, 5]_{\spadesuit}, \  \frac{1}{8} [1, 7]_{\heartsuit}$ & KE \\
\rowcolor{lightgray}
85 & $X_{38} \subset \mbP (1,3,5,11,19)$ & $\frac{2}{165}$ & $\frac{1}{3}_{\heartsuit}, \ \frac{1}{5} [1, 4]_{\heartsuit}, \ \frac{1}{11} [3, 8]_{\diamondsuit}$ & KE \\
86 & $X_{38} \subset \mbP (1,5,6,8,19)$ & $\frac{1}{120}$ & $\frac{1}{2}_{\heartsuit}, \ \frac{1}{5} [1, 4]_{\heartsuit}, \ \frac{1}{6} [1, 5]_{\heartsuit}, \ \frac{1}{8} [3, 5]_{\heartsuit}$ & KE \\
\rowcolor{lightgray}
87 & $X_{40} \subset \mbP (1,5,7,8,20)$ & $\frac{1}{140}$ & $\frac{1}{4}_{\heartsuit}, \ 2 \ntimes \frac{1}{5} [2, 3]_{\heartsuit}, \  \frac{1}{7} [1, 6]_{\heartsuit}$ & KE \\
88 & $X_{42} \subset \mbP (1,1,6,14,21)$ & $\frac{1}{42}$ & $\frac{1}{2}_{\heartsuit}, \ \frac{1}{3}_{\heartsuit}, \ \frac{1}{7} [1, 6]_{\heartsuit}$ & KE \\
\rowcolor{lightgray}
89 & $X_{42} \subset \mbP (1,2,5,14,21)$ & $\frac{1}{70}$ & $3 \ntimes \frac{1}{2}_{\heartsuit}, \ \frac{1}{5} [1, 4]_{\heartsuit}, \  \frac{1}{7} [2, 5]_{\heartsuit}$ & KE \\
90 & $X_{42} \subset \mbP (1,3,4,14,21)$ & $\frac{1}{84}$ & $\frac{1}{2}_{\heartsuit}, \ 2 \ntimes \frac{1}{3}_{\heartsuit}, \ \frac{1}{4}_{\heartsuit}, \ \frac{1}{7} [3, 4]_{\heartsuit}$ & KE \\
\rowcolor{lightgray}
91 & $X_{44} \subset \mbP (1,4,5,13,22)$ & $\frac{1}{130}$ & $\frac{1}{2}_{\heartsuit}, \ \frac{1}{5} [2, 3]_{\heartsuit}, \ \frac{1}{13} [4, 9]_{\diamondsuit}$ & KE \\
92 & $X_{48} \subset \mbP (1,3,5,16,24)$ & $\frac{1}{120}$ & $2 \ntimes \frac{1}{3}_{\heartsuit}, \ \frac{1}{5} [1, 4]_{\heartsuit}, \  \frac{1}{8} [3, 5]_{\heartsuit}$ & KE \\
\rowcolor{lightgray}
93 & $X_{50} \subset \mbP (1,7,8,10,25)$ & $\frac{1}{280}$ & $\frac{1}{2}_{\heartsuit}, \ \frac{1}{5} [2, 3]_{\heartsuit}, \ \frac{1}{7} [3, 4]_{\clubsuit}, \ \frac{1}{8} [1, 7]_{\heartsuit}$ & KE \\
94 & $X_{54} \subset \mbP (1,4,5,18,27)$ & $\frac{1}{180}$ & $\frac{1}{2}_{\heartsuit}, \ \frac{1}{4}_{\heartsuit}, \ \frac{1}{5} [2, 3]_{\heartsuit}, \ \frac{1}{9} [4, 5]_{\heartsuit}$ & KE \\
\rowcolor{lightgray}
95 & $X_{66} \subset \mbP (1,5,6,22,33)$ & $\frac{1}{330}$ & $\frac{1}{2}_{\heartsuit}, \ \frac{1}{3}_{\heartsuit}, \ \frac{1}{5} [2, 3]_{\clubsuit}, \ \frac{1}{11} [5, 6]_{\heartsuit}$ & KE
\end{longtable}
\endgroup


\end{document}